\documentclass{amsart}
\usepackage{amssymb}
\usepackage{amsfonts}
\usepackage{geometry}
\usepackage{color}

\newtheorem{theorem}{Theorem}
\theoremstyle{plain}
\newtheorem{acknowledgement}[theorem]{Acknowledgement}

\newtheorem{corollary}[theorem]{Corollary}

\newtheorem{definition}[theorem]{Definition}

\newtheorem{lemma}[theorem]{Lemma}
\newtheorem{notation}[theorem]{Notation}
\newtheorem{problem}[theorem]{Problem}
\newtheorem{proposition}[theorem]{Proposition}
\newtheorem{remark}[theorem]{Remark}

\numberwithin{equation}{section}
\input{tcilatex}
\begin{document}
\title[Two weight Sobolev norm inequalities]{A two weight Sobolev $T1$
theorem for fractional vector Riesz transforms and doubling weights}
\author[E. T. Sawyer]{Eric T. Sawyer$^\dagger$}
\address{Eric T. Sawyer, Department of Mathematics and Statistics\\
McMaster University\\
1280 Main Street West\\
Hamilton, Ontario L8S 4K1 Canada}
\thanks{$\dagger $ Research supported in part by a grant from the National
Science and Engineering Research Council of Canada.}
\email{sawyer@mcmaster.ca}
\author[B. D. Wick]{Brett D. Wick$^\ddagger$}
\address{Brett D. Wick, Department of Mathematics \& Statistics, Washington
University -- St. Louis, One Brookings Drive, St. Louis, MO USA 63130-4899.}
\email{wick@math.wustl.edu}
\thanks{$\ddagger $ B. D. Wick's research is supported in part by National
Science Foundation Grants DMS \# 1800057, \# 2054863, and \# 2000510 and
Australian Research Council -- DP 220100285.}
\date{\today }

\begin{abstract}
Let $\mu $ be a positive locally finite Borel measure on $\mathbb{R}^{n}$
that is doubling, and define the homogeneous $W^{s}\left( \mu \right) $%
-Sobolev norm squared $\left\Vert f\right\Vert _{W^{s}\left( \mu \right)
}^{2}$ of a function $f\in L_{\func{loc}}^{2}\left( \mu \right) $ by%
\begin{equation*}
\int_{\mathbb{R}^{n}}\int_{\mathbb{R}^{n}}\left( \frac{f\left( x\right)
-f\left( y\right) }{\left\vert x-y\right\vert ^{s}}\right) ^{2}\frac{d\mu
\left( x\right) d\mu \left( y\right) }{\left\vert B\left( \frac{x+y}{2},%
\frac{\left\vert x-y\right\vert }{2}\right) \right\vert _{\mu }},
\end{equation*}%
and denote by $W^{s}\left( \mu \right) $ the corresponding Hilbert space
completion (when $\mu $ is Lebesgue measure, this is the familiar Sobolev
space on $\mathbb{R}^{n}$).

We prove in particular that for $0\leq \alpha <n$, and $\sigma $ and $\omega 
$ doubling measures on $\mathbb{R}^{n}$, there is a positive constant $%
\theta $ such that for $0<s<\theta $, the $\alpha $-fractional vector Riesz
transform $\mathbf{R}^{\alpha }$ is bounded from\thinspace $W^{s}\left(
\sigma \right) $ to $W^{s}\left( \omega \right) $ \emph{if and only if} the
Sobolev $\mathbf{1}$-testing and $\mathbf{1}^{\ast }$-testing conditions
hold for the operator $\mathbf{R}^{\alpha }$, i.e.%
\begin{eqnarray*}
\left\Vert \mathbf{R}_{\sigma }^{\alpha }\mathbf{1}_{I}\right\Vert
_{W^{s}\left( \omega \right) } &\leq &\mathfrak{T}_{\mathbf{R}^{\alpha
}}\left( \sigma ,\omega \right) \sqrt{\left\vert I\right\vert _{\sigma }}%
\ell \left( I\right) ^{-s},\ \ \ \ \ I\in \mathcal{Q}^{n}, \\
\left\Vert \mathbf{R}_{\omega }^{\alpha ,\ast }\mathbf{1}_{I}\right\Vert
_{W^{s}\left( \sigma \right) ^{\ast }} &\leq &\mathfrak{T}_{\mathbf{R}%
^{\alpha ,\ast }}\left( \omega ,\sigma \right) \sqrt{\left\vert I\right\vert
_{\omega }}\ell \left( I\right) ^{s},\ \ \ \ \ I\in \mathcal{Q}^{n},
\end{eqnarray*}%
taken over the family of indicator test functions $\left\{ \mathbf{1}%
_{I}\right\} _{I\in \mathcal{Q}^{n}}$. Here $\mathcal{Q}^{n}$ is the
collection of all cubes with sides parallel to the coordinate axes, and $%
W^{s}\left( \mu \right) ^{\ast }$ denotes the dual of $W^{s}\left( \mu
\right) $ determined by the usual $L^{2}\left( \mu \right) $ bilinear
pairing, which we identify with a dyadic Sobolev space $W_{\limfunc{dyad}%
}^{-s}\left( \mu \right) $ of negative order.

Under the side assumption that the classical pivotal conditions hold, the
vector Riesz transform $\mathbf{R}^{\alpha }$ can be replaced by any $\alpha 
$-fractional smooth Calder\'{o}n-Zygmund operator $T^{\alpha }$.
\end{abstract}

\maketitle
\tableofcontents

\section{Introduction}

The Nazarov-Treil-Volberg $T1$ conjecture on the boundedness of the Hilbert
transform from one weighted space $L^{2}\left( \sigma \right) $ to another $%
L^{2}\left( \omega \right) $, was settled affirmatively in the two part
paper \cite{LaSaShUr3},\cite{Lac}. Since then there have been a number of
generalizations of boundedness of Calder\'{o}n-Zygmund operators from one
weighted $L^{2}$ space to another, both to higher dimensional Euclidean
spaces (see e.g. \cite{SaShUr7}, \cite{LaWi} and \cite{LaSaShUrWi}), and
also to spaces of homogeneous type (see e.g. \cite{DuLiSaVeWiYa}). In
addition there have been some generalizations to Sobolev spaces in place
of\thinspace $L^{2}$ spaces, but only in the setting of a single weight (see
e.g. \cite{DilWiWi} and \cite{KaLiPeWa}).

The purpose of this paper is to prove a \emph{two weight} $T1$ theorem on 
\emph{weighted Sobolev} spaces for vector Riesz transforms, but with \emph{%
doubling} measures\footnote{%
Weighted Sobolev spaces are not canonically defined for general weights, and
doubling is a convenient hypothesis that gives equivalence of the various
definitions.}. The proof also shows that this theorem extends to smooth $%
\alpha $-fractional singular integrals on $\mathbb{R}^{n}$, in the presence
of classical pivotal side conditions. In order to state our theorem, we need
a number of definitions, some of which are recalled and explained in detail
further below. Let $\mu $ be a positive locally finite Borel measure on $%
\mathbb{R}^{n}$ that is doubling, let $\mathcal{D}$ be a dyadic grid on $%
\mathbb{R}^{n}$, let $\kappa \in \mathbb{N}$ and let $\left\{ \bigtriangleup
_{Q;\kappa }^{\mu }\right\} _{Q\in \mathcal{D}}$ be the set of weighted
Alpert projections on $L^{2}\left( \mu \right) $ (see \cite{RaSaWi}). When $%
\kappa =1$, these are the familiar weighted Haar projections $\bigtriangleup
_{Q}^{\mu }=\bigtriangleup _{Q;1}^{\mu }$.

\begin{definition}
Let $\mu $ be a doubling measure on $\mathbb{R}^{n}$. Given $s\in \mathbb{R}$%
, we define the $\mathcal{D}$-dyadic homogeneous $W_{\mathcal{D};\kappa
}^{s}\left( \mu \right) $-Sobolev norm of a function $f\in L_{\func{loc}%
}^{2}\left( \mu \right) $ by%
\begin{equation*}
\left\Vert f\right\Vert _{W_{\mathcal{D};\kappa }^{s}\left( \mu \right)
}^{2}\equiv \sum_{Q\in \mathcal{D}}\ell \left( Q\right) ^{-2s}\left\Vert
\bigtriangleup _{Q;\kappa }^{\mu }f\right\Vert _{L^{2}\left( \mu \right)
}^{2}\ ,
\end{equation*}%
and we denote by $W_{\mathcal{D};\kappa }^{s}\left( \mu \right) $ the
corresponding Hilbert space completion\footnote{%
For general measures, the functional $\left\Vert {}\right\Vert _{W_{\mathcal{%
D};\kappa }^{s}\left( \mu \right) }$ may only be a seminorm, but this is
avoided for doubling measures.} of $f\in L_{\func{loc}}^{2}\left( \mu
\right) $ with 
\begin{equation*}
\left\Vert f\right\Vert _{W_{\mathcal{D};\kappa }^{s}\left( \mu \right)
}<\infty .
\end{equation*}
\end{definition}

Note that $W_{\mathcal{D};\kappa }^{0}\left( \mu \right) =L^{2}\left( \mu
\right) $. We will show below that $W_{\mathcal{D};\kappa }^{s}\left( \mu
\right) =W_{\mathcal{D}^{\prime };\kappa ^{\prime }}^{s}\left( \mu \right) $
for all $s\in \mathbb{R}$ with $\left\vert s\right\vert $ sufficiently
small, for all $\kappa ,\kappa ^{\prime }\geq 1$, and for all dyadic grids $%
\mathcal{D}$ and $\mathcal{D}^{\prime }$. Thus for a sufficiently small real 
$s$ depending on the doubling measure $\mu $, there is essentially just one
weighted `dyadic' Sobolev space of order $s$, which we will denote by $W_{%
\limfunc{dyad}}^{s}\left( \mu \right) $. Moreover, for $s>0$ and small
enough and $\mu $ doubling, there is a more familiar equivalent `continuous'
norm,%
\begin{equation}
\left\Vert f\right\Vert _{W^{s}\left( \mu \right) }=\sqrt{\int_{\mathbb{R}%
^{n}}\int_{\mathbb{R}^{n}}\left( \frac{f\left( x\right) -f\left( y\right) }{%
\left\vert x-y\right\vert ^{s}}\right) ^{2}\frac{d\mu \left( x\right) d\mu
\left( y\right) }{\left\vert B\left( \frac{x+y}{2},\frac{\left\vert
x-y\right\vert }{2}\right) \right\vert _{\mu }}}.  \label{cont norm}
\end{equation}%
We also show that the dual spaces $W_{\mathcal{D};\kappa }^{s}\left( \mu
\right) ^{\ast }$ under the $L^{2}\left( \mu \right) $ pairing are given by $%
W_{\mathcal{D};\kappa }^{-s}\left( \mu \right) $ for all grids $\mathcal{D}$
and integers $\kappa $. Thus when $\mu $ is doubling, we can identify $%
W^{s}\left( \mu \right) ^{\ast }$ with any of the spaces $W_{\mathcal{D}%
;\kappa }^{-s}\left( \mu \right) $ for $\left\vert s\right\vert $
sufficiently small, and it will be convenient to denote $W^{s}\left( \mu
\right) ^{\ast }$ by $W^{-s}\left( \mu \right) $, even though the above
formula for $\left\Vert f\right\Vert _{W^{-s}\left( \mu \right) }$ diverges
when $s\geq 0$.

Finally, we note that \emph{without} the doubling hypothesis on $\mu $, we
in general need to include additional Haar projections $\left\{ \mathbb{E}%
_{T}^{\mu }\right\} _{T\in \mathcal{T}}$ onto one-dimensional spaces of
constant functions on certain `tops' $T$ of the grid $\mathcal{D}$, where a
top is the union of a maximal tower in $\mathcal{D}$ (see \cite{AlSaUr2}).
Without these additional projections, we may not recover all of $L^{2}\left(
\mu \right) $ in general, and moreover, the spaces $W_{\mathcal{D};\kappa
}^{s}\left( \mu \right) $ defined above may actually depend on the dyadic
grid $\mathcal{D}$. For example, if $d\mu \left( x\right) =\mathbf{1}_{\left[
-1,1\right) }\left( x\right) dx$ and $f\left( x\right) =\mathbf{1}_{\left[
0,1\right) }\left( x\right) $, the reader can easily check that the
functional $\left\Vert f\right\Vert _{W_{\mathcal{D};1}^{s}\left( \mu
\right) }$ vanishes when $\mathcal{D}$ is the standard dyadic grid, but is
positive when $\mathcal{D}$ is any grid containing $\left[ -1,1\right) $%
\footnote{%
Moreover, even taking into account the behaviour at infinity, one can show
that $\left\Vert f\right\Vert _{W_{\mathcal{D};1}^{s}\left( \mu \right) }=0$
when $\mathcal{D}$ is the standard grid and $s>0$.}.

Note that we will not use the Hilbert space duality (that identifies the
dual of a Hilbert space with itself) to analyze the two weight boundedness
of $T_{\sigma }^{\alpha }:W_{\func{dyad}}^{s}\left( \sigma \right)
\rightarrow W_{\func{dyad}}^{s}\left( \omega \right) $, but rather we will
use the $L^{2}\left( \omega \right) $ and $L^{2}\left( \sigma \right) $
pairings in which case the dual of $W_{\func{dyad}}^{s}\left( \mu \right) $
is identified with $W_{\func{dyad}}^{-s}\left( \mu \right) $ as above. The
reason for this is that the weighted Alpert projections $\left\{
\bigtriangleup _{Q;\kappa }^{\mu }\right\} _{Q\in \mathcal{D}}$ satisfy
telescoping identities, while the orthogonal projections $\left\{ \ell
\left( Q\right) ^{s}\bigtriangleup _{Q;\kappa }^{\mu }\right\} _{Q\in 
\mathcal{D}}$ do not.

Denote by $\Omega _{\limfunc{dyad}}$ the collection of all dyadic grids in $%
\mathbb{R}^{n}$, and let $\mathcal{Q}^{n}$ denote the collection of all
cubes in $\mathbb{R}^{n}$ having sides parallel to the coordinate axes. A
positive locally finite Borel measure $\mu $ on $\mathbb{R}^{n}$ is said to
be doubling if there is a constant $C_{\limfunc{doub}}$, called the doubling
constant, such that%
\begin{equation*}
\left\vert 2Q\right\vert _{\mu }\leq C_{\limfunc{doub}}\left\vert
Q\right\vert _{\mu }\ ,\ \ \ \ \ \text{for all cubes }Q\in \mathcal{Q}^{n}.
\end{equation*}

For $0\leq \alpha <n$ we define a smooth $\alpha $-fractional Calder\'{o}%
n-Zygmund kernel $K^{\alpha }(x,y)$ to be a function $K^{\alpha }:\mathbb{R}%
^{n}\times \mathbb{R}^{n}\rightarrow \mathbb{R}$ satisfying the following
fractional size and smoothness conditions%
\begin{equation}
\left\vert \nabla _{x}^{j}K^{\alpha }\left( x,y\right) \right\vert
+\left\vert \nabla _{y}^{j}K^{\alpha }\left( x,y\right) \right\vert \leq
C_{\alpha ,j}\left\vert x-y\right\vert ^{\alpha -j-n},\ \ \ \ \ 0\leq
j<\infty ,  \label{sizeandsmoothness'}
\end{equation}%
and we denote by $T^{\alpha }$ the associated $\alpha $-fractional singular
integral on $\mathbb{R}^{n}$. We say that $T_{\sigma }^{\alpha }$, where $%
T_{\sigma }^{\alpha }f\equiv T^{\alpha }\left( f\sigma \right) $, is bounded
from\thinspace $W_{\func{dyad}}^{s}\left( \sigma \right) $ to $W_{\func{dyad}%
}^{s}\left( \omega \right) $ if for all admissible truncations $\widetilde{%
T^{\alpha }}$ we have%
\begin{equation*}
\left\Vert \widetilde{T_{\sigma }^{\alpha }}f\right\Vert _{W_{\func{dyad}%
}^{s}\left( \omega \right) }\leq \mathfrak{N}_{T^{\alpha }}\left( \sigma
,\omega \right) \left\Vert f\right\Vert _{W_{\func{dyad}}^{s}\left( \sigma
\right) },\ \ \ \ \ \text{for all }f\in W_{\func{dyad}}^{s}\left( \sigma
\right) .
\end{equation*}%
Here $\mathfrak{N}_{T^{\alpha }}\left( \sigma ,\omega \right) $ denotes the
best constant in these inequalities uniformly over all admissible
truncations of $T^{\alpha }$. See below for a precise definition of
admissible truncations, as well as the interpretation of the testing
conditions appearing in the next theorem. The case $s=0$ of Theorem \ref%
{pivotal theorem} is in \cite{LaWi} and later in \cite{SaShUr9}.

Finally, for $0\leq \alpha <n$ we define the $\alpha $-fractional vector
Riesz transform kernel $\mathbf{K}^{\alpha }(x,y)$ to be the standard
function $\mathbf{K}^{\alpha }:\mathbb{R}^{n}\times \mathbb{R}%
^{n}\rightarrow \mathbb{R}^{n}$ given by $c_{n,\alpha }\frac{z}{\left\vert
z\right\vert ^{n-\alpha }}$, and we denote by $\mathbf{R}^{\alpha }$ the
associated $\alpha $-fractional singular integral on $\mathbb{R}^{n}$.

\begin{theorem}[$T1$ for doubling measures]
\label{pivotal theorem}Let $0\leq \alpha <n$, and let $\mathbf{R}^{\alpha }$
be the $\alpha $-fractional vector Riesz transform on $\mathbb{R}^{n}$. Let $%
\sigma $ and $\omega $ be doubling Borel measures on $\mathbb{R}^{n}$. Then
there is a positive constant $\theta $, depending only on the doubling
constants of $\sigma $ and $\omega $, such that if $0<s<\theta $, then $%
\mathbf{R}_{\sigma }^{\alpha }$, where $\mathbf{R}_{\sigma }^{\alpha
}f\equiv \mathbf{R}^{\alpha }\left( f\sigma \right) $, is bounded
from\thinspace $W^{s}\left( \sigma \right) $ to $W^{s}\left( \omega \right) $%
, i.e.%
\begin{equation}
\left\Vert \mathbf{R}_{\sigma }^{\alpha }f\right\Vert _{W^{s}\left( \omega
\right) }\leq \mathfrak{N}_{\mathbf{R}^{\alpha }}\left\Vert f\right\Vert
_{W^{s}\left( \sigma \right) },  \label{boundedness}
\end{equation}%
provided the Sobolev $\mathbf{1}$-testing and $\mathbf{1}^{\ast }$-testing
conditions for the operator $T^{\alpha }$,%
\begin{eqnarray*}
\left\Vert \mathbf{R}_{\sigma }^{\alpha }\mathbf{1}_{I}\right\Vert
_{W^{s}\left( \omega \right) } &\leq &\mathfrak{T}_{\mathbf{R}^{\alpha
,s}}\left( \sigma ,\omega \right) \sqrt{\left\vert I\right\vert _{\sigma }}%
\ell \left( I\right) ^{-s},\ \ \ \ \ I\in \mathcal{Q}^{n}, \\
\left\Vert \mathbf{R}_{\omega }^{\alpha ,\ast }\mathbf{1}_{I}\right\Vert
_{W^{-s}\left( \sigma \right) } &\leq &\mathfrak{T}_{\mathbf{R}^{\alpha
,-s,\ast }}\left( \omega ,\sigma \right) \sqrt{\left\vert I\right\vert
_{\omega }}\ell \left( I\right) ^{s},\ \ \ \ \ I\in \mathcal{Q}^{n},
\end{eqnarray*}%
taken over the family of indicator test functions $\left\{ \mathbf{1}%
_{I}\right\} _{I\in \mathcal{Q}^{n}}$.\newline
If we assume in addition that the classical pivotal constants $\mathcal{V}%
_{2}^{\alpha ,1}\left( \sigma ,\omega \right) $ and $\mathcal{V}_{2}^{\alpha
,1,\ast }\left( \omega ,\sigma \right) $ hold, then for any smooth $\alpha $%
-fractional singular integral $T^{\alpha }$ on $\mathbb{R}^{n}$, then $%
T_{\sigma }^{\alpha }$, where $T_{\sigma }^{\alpha }f\equiv T^{\alpha
}\left( f\sigma \right) $, is bounded from\thinspace $W^{s}\left( \sigma
\right) $ to $W^{s}\left( \omega \right) $, i.e.%
\begin{equation}
\left\Vert T_{\sigma }^{\alpha }f\right\Vert _{W^{s}\left( \omega \right)
}\leq \mathfrak{N}_{T^{\alpha }}\left\Vert f\right\Vert _{W^{s}\left( \sigma
\right) },  \label{boundedness'}
\end{equation}%
provided the classical fractional Muckenhoupt condition on the measure pair
holds,%
\begin{equation*}
A_{2}^{\alpha }\equiv \sup_{Q\in \mathcal{Q}^{n}}\frac{\left\vert
Q\right\vert _{\omega }\left\vert Q\right\vert _{\sigma }}{\left\vert
Q\right\vert ^{2\left( 1-\frac{\alpha }{n}\right) }}<\infty ,
\end{equation*}%
as well as the Sobolev $\mathbf{1}$-testing and $\mathbf{1}^{\ast }$-testing
conditions for the operator $T^{\alpha }$,%
\begin{eqnarray*}
\left\Vert T_{\sigma }^{\alpha }\mathbf{1}_{I}\right\Vert _{W^{s}\left(
\omega \right) } &\leq &\mathfrak{T}_{T^{\alpha ,s}}\left( \sigma ,\omega
\right) \sqrt{\left\vert I\right\vert _{\sigma }}\ell \left( I\right)
^{-s},\ \ \ \ \ I\in \mathcal{Q}^{n}, \\
\left\Vert T_{\omega }^{\alpha ,\ast }\mathbf{1}_{I}\right\Vert
_{W^{-s}\left( \sigma \right) } &\leq &\mathfrak{T}_{T^{\alpha ,-s,\ast
}}\left( \omega ,\sigma \right) \sqrt{\left\vert I\right\vert _{\omega }}%
\ell \left( I\right) ^{s},\ \ \ \ \ I\in \mathcal{Q}^{n},
\end{eqnarray*}%
taken over the family of indicator test functions $\left\{ \mathbf{1}%
_{I}\right\} _{I\in \mathcal{Q}^{n}}$.\newline
Conversely, the testing conditions are necessary for (\ref{boundedness}),
and if in addition $T^{\alpha }$ is a smooth convolution operator with
homogeneous kernel that is nonvanishing in some coordinate direction, then $%
A_{2}^{\alpha }<\infty $ whenever the two weight norm inequality (\ref%
{boundedness'}) holds for some $s>0$.
\end{theorem}

\begin{remark}
One can weaken the smoothness assumption on the kernel $K$ depending on the
doubling constants of the measures $\sigma $ and $\omega $, but we will not
pursue this here. See \cite{Saw6} for sharper assumptions in the $L^{2}$
case.
\end{remark}

\begin{problem}
$T1$ theorems for Sobolev norms involving general measures, even for the
Hilbert transform on the line, remain open at this time.
\end{problem}

The proof of Theorem \ref{pivotal theorem} expands on that for $L^{2}$
spaces with doubling measures using weighted Alpert wavelets (\cite{AlSaUr}%
), but with a number of differences. For example:

\begin{enumerate}
\item The map $f\rightarrow \left\vert f\right\vert $ fails to be bounded on 
$W_{\func{dyad}}^{-s}\left( \omega \right) $ for $s>0$\footnote{%
For example, if $d\mu =dx$ and $f=\sum_{k=1}^{2N}\left( -1\right) ^{k}1_{%
\left[ k-1,k\right) }$, then $\left\Vert f\right\Vert _{W_{\limfunc{dyad}%
}^{-s}}^{2}\approx N$, while $\left\Vert \left\vert f\right\vert \right\Vert
_{W_{\limfunc{dyad}}^{-s}}^{2}\approx N^{1+2s}$.}, which gives rise to
significant obstacles in dealing with bilinear inequalities (using the $%
L^{2} $ inner product as a duality pairing) that require control of both $%
T:W_{\func{dyad}}^{s}\left( \sigma \right) \rightarrow W_{\func{dyad}%
}^{s}\left( \omega \right) $ and $T^{\ast }:W_{\func{dyad}}^{-s}\left(
\omega \right) \rightarrow W_{\func{dyad}}^{-s}\left( \sigma \right) $.

\item As a consequence, we are no longer able to use Calder\'{o}n-Zygmund
decompositions and Carleson embedding theorems that require use of the
modulus $\left\vert f\right\vert $ of a Sobolev function $f$. Instead, we
derive a new form of the pivotal condition, that permits a new Carleson
condition to circumvent these hurdles. These new conditions yield stronger
inequalities, but only over grids that are \emph{full}.

\item The estimation of Sobolev norms in the paraproduct form requires the
use both of Alpert and Haar wavelets, in connection with the new Carleson
condition. This in turn requires the identification of different wavelet
spaces.

\item In Proposition \ref{Int Prop}, we extend the Intertwining Proposition
from \cite{Saw6} to Sobolev spaces using a new stronger form of the $\kappa $%
-pivotal condition, which also results in an overall simplification of this
proof.

\item A power decay of doubling measures near zero sets of polynomials is
needed to estimate Sobolev norms of moduli of Alpert wavelets in Lemma \ref%
{mod Alpert} when $s<0$, where the logarithmic decay obtained in \cite{Saw6}
is insufficient.

\item Finally, we prove the comparability of the various Sobolev space norms
for a fixed $s$ and doubling measure (the case $s=0$ being trivial),
including the familiar continuous norm in (\ref{cont norm}) when $s>0$. This
equivalence is needed in particular to implement the $\func{good}/\func{bad}$
technology of Nazarov, Treil and Volberg.
\end{enumerate}

\begin{acknowledgement}
We thank a referee for pointing out a gap in our treatment of the stopping
form in the original version of this paper, whose fix here requires

\begin{enumerate}
\item weakening the $T1$ theorem to hold only for vector fractional Riesz
transforms $\mathbf{R}^{\alpha ,n}$, and

\item weakening the two weight norm inequalities for general fractional
Calder\'{o}n-Zygmund operators $T^{\alpha }$ by requiring a pivotal side
condition.
\end{enumerate}
\end{acknowledgement}

\section{Preliminaries: Sobolev spaces and doubling measures}

Denote by $\mathcal{Q}^{n}$ the collection of cubes in $\mathbb{R}^{n}$
having sides parallel to the coordinate axes. A positive locally finite
Borel measure $\mu $ on $\mathbb{R}^{n}$ is said to satisfy the\emph{\
doubling condition} if there is a pair of constants $\left( \beta ,\gamma
\right) \in \left( 0,1\right) ^{2}$, called doubling parameters, such that%
\begin{equation}
\left\vert \beta Q\right\vert _{\mu }\geq \gamma \left\vert Q\right\vert
_{\mu }\ ,\ \ \ \ \ \text{for all cubes }Q\in \mathcal{Q}^{n},
\label{def rev doub}
\end{equation}%
and the \emph{reverse doubling condition} if there is a pair of constants $%
\left( \beta ,\gamma \right) \in \left( 0,1\right) ^{2}$, called reverse
doubling parameters, such that%
\begin{equation}
\left\vert \beta Q\right\vert _{\mu }\leq \gamma \left\vert Q\right\vert
_{\mu }\ ,\ \ \ \ \ \text{for all cubes }Q\in \mathcal{Q}^{n}.
\label{def rev doub'}
\end{equation}%
Note that the inequality in (\ref{def rev doub'}) has been reversed from
that in the definition of the doubling condition in (\ref{def rev doub}). A
familiar equivalent reformulation of (\ref{def rev doub}) is that there is a
positive constant $C_{\limfunc{doub}}$, called the doubling constant, such
that $\left\vert 2Q\right\vert _{\mu }\leq C_{\limfunc{doub}}\left\vert
Q\right\vert _{\mu }$ for all cubes $Q\in \mathcal{Q}^{n}$. There is also a
positive constant $\theta _{\mu }^{\limfunc{doub}}$, called a\emph{\
doubling exponent}, such that%
\begin{equation*}
\sup_{Q\in \mathcal{Q}^{n}}\frac{\left\vert sQ\right\vert _{\mu }}{%
\left\vert Q\right\vert _{\mu }}\leq s^{\theta _{\mu }^{\limfunc{doub}}},\ \
\ \ \ \text{for all sufficiently large }s>0.
\end{equation*}

It is well known (see e.g. the introduction in \cite{SaUr}) that doubling
implies reverse doubling, and that $\mu $ is reverse doubling if and only if
there exists a positive constant $\theta _{\mu }^{\limfunc{rev}}$, called a 
\emph{reverse doubling exponent}, such that%
\begin{equation*}
\sup_{Q\in \mathcal{Q}^{n}}\frac{\left\vert sQ\right\vert _{\mu }}{%
\left\vert Q\right\vert _{\mu }}\leq s^{\theta _{\mu }^{\limfunc{rev}}},\ \
\ \ \ \text{for all sufficiently small }s>0.
\end{equation*}%
In particular, for $I\subset Q$, we have%
\begin{equation}
\frac{\left\vert I\right\vert _{\mu }}{\left\vert Q\right\vert _{\mu }}=%
\frac{\left\vert \frac{\ell \left( Q\right) }{\ell \left( I\right) }%
I\right\vert _{\mu }}{\left\vert Q\right\vert _{\mu }}\frac{\left\vert
I\right\vert _{\mu }}{\left\vert \frac{\ell \left( Q\right) }{\ell \left(
I\right) }I\right\vert _{\mu }}\leq \frac{\left\vert 3Q\right\vert _{\mu }}{%
\left\vert Q\right\vert _{\mu }}\left( \frac{\ell \left( I\right) }{\ell
\left( Q\right) }\right) ^{\theta _{\mu }^{\limfunc{rev}}}\leq C_{\limfunc{%
doub}}^{2}\left( \frac{\ell \left( I\right) }{\ell \left( Q\right) }\right)
^{\theta _{\mu }^{\limfunc{rev}}}.  \label{combine}
\end{equation}

\subsection{Decay of doubling measures near zero sets of polynomials}

In order to deal with Sobolev norms and doubling measures, we will need the
following estimate on doubling measures of `halos' of zero sets of
normalized polynomials, which follows the same plan of proof as in the case
of boundaries of cubes proved in \cite[Lemma 24]{Saw6}. We first recall a
slight variant of a remark from \cite{Saw6}.

For any polynomial $P$ and cube $Q$, we say that $P$ is $Q$\emph{-normalized}
if $\left\Vert P\right\Vert _{L^{\infty }\left( Q\right) }=1$.

\begin{remark}
\label{normalized}Since all norms on a finite dimensional vector space are
equivalent, we have upon rescaling the cube $Q$ to the unit cube, 
\begin{equation}
\left\Vert P\right\Vert _{L^{\infty }\left( Q\right) }\approx \left\vert
P\left( 0\right) \right\vert +\sqrt{n}\ell \left( Q\right) \left\Vert \nabla
P\right\Vert _{L^{\infty }\left( Q\right) },\ \ \ \ \ \deg P<\kappa ,
\label{fin equiv}
\end{equation}%
with implicit constants depending only on $n$ and $\kappa $. In particular
there is a positive constant $K_{n,\kappa }$ such that $\sqrt{n}\ell \left(
Q\right) \left\Vert \nabla P\right\Vert _{L^{\infty }\left( Q\right) }\leq
K_{n,\kappa }$ for all $Q$-normalized polynomials $P$. Then for every $Q$%
-normalized polynomial $P$ of degree less than $\kappa $, there is a ball $%
B\left( y,\frac{\sqrt{n}}{2K_{n,\kappa }}\ell \left( Q\right) \right)
\subset Q\ $on which $P$ is nonvanishing. Indeed, if there is no such ball,
then 
\begin{equation*}
1=\left\Vert P\right\Vert _{L^{\infty }\left( Q\right) }\leq \frac{\sqrt{n}}{%
2K_{n,\kappa }}\ell \left( Q\right) \left\Vert \nabla P\right\Vert
_{L^{\infty }\left( Q\right) }\leq \frac{1}{2}
\end{equation*}%
is a contradiction.
\end{remark}

Here is the result proved in \cite[Lemma 24]{Saw6}.

\begin{lemma}
Suppose $\mu $ is a doubling measure on $\mathbb{R}^{n}$ and that $Q\in 
\mathcal{Q}^{n}$. Then for $0<\delta <1$ we have%
\begin{equation*}
\left\vert Q\setminus \left( 1-\delta \right) Q\right\vert _{\mu }\leq \frac{%
C}{\ln \frac{1}{\delta }}\left\vert Q\right\vert _{\mu }\ .
\end{equation*}
\end{lemma}

We will need to improve significantly on this as follows. Without loss of
generality, suppose that $Q=\left[ 0,1\right] \times \left[ 0,1\right] $ in
the plane and $d\mu \left( x,y\right) =w\left( x,y\right) dxdy$. Define $W$
to be even on $\left[ -1,1\right] $ by%
\begin{equation*}
W\left( y\right) \equiv \int_{0}^{1}w\left( x,y\right) dx,\ \ \ \ \ 0\leq
y\leq 1.
\end{equation*}%
and note that $W\left( y\right) dy$ is a doubling measure on $\left[ 0,1%
\right] $, hence also reverse doubling with exponent $\theta ^{\func{rev}}$.
Thus from the reverse doubling property applied to the subinterval $\left[
0,t\right] $ of $\left[ -1,1\right] $ we have that%
\begin{equation*}
\int_{0}^{t}W\left( y\right) dy\leq Ct^{\theta ^{\func{rev}%
}}\int_{-1}^{1}W\left( y\right) dy\leq C^{\prime }t^{\theta ^{\func{rev}%
}}\int_{0}^{1}W\left( y\right) dy,
\end{equation*}%
which says that%
\begin{equation}
\left\vert \left[ 0,1\right] \times \left[ 0,t\right] \right\vert _{\mu
}=\int_{0}^{t}W\left( y\right) dy\leq C^{\prime }t^{\theta ^{\func{rev}%
}}\int_{0}^{1}W\left( y\right) dy=C^{\prime }t^{\theta ^{\func{rev}%
}}\left\vert \left[ 0,1\right] \times \left[ 0,1\right] \right\vert _{\mu }.
\label{improvement}
\end{equation}%
This gives power decay instead of logarithmic decay, which will prove
crucial below. The next lemma is a generalization of \cite[Lemma 24]{Saw6}.

\begin{lemma}
Let $\kappa \in \mathbb{N}$. Suppose $\mu $ is a doubling measure on $%
\mathbb{R}^{n}$ and that $Q\in \mathcal{Q}^{n}$. Let $Z$ denote the zero set
of a $Q$-normalized polynomial $P$ of degree less than $\kappa $, and for $%
0<\delta <1$, let 
\begin{equation*}
Z_{\delta }=\left\{ y\in \mathbb{R}^{n}:\left\vert y-z\right\vert <\delta 
\text{ for some }z\in Z\right\}
\end{equation*}%
denote the $\delta $-halo of $Z$. Then for a positive constant $C_{n,\kappa
} $ depending only on $n$ and $\kappa $, and not on $P$ itself, we have%
\begin{equation*}
\left\vert Q\cap Z_{\delta }\right\vert _{\mu }\leq \frac{C_{n,\kappa }}{\ln 
\frac{1}{\delta }}\left\vert Q\right\vert _{\mu }\ .
\end{equation*}
\end{lemma}

\begin{proof}
Let $\delta =2^{-m}$. Denote by $\mathfrak{C}^{\left( m\right) }\left(
Q\right) $ the set of $m^{th}$ generation dyadic children of $Q$, so that
each $I\in \mathfrak{C}^{\left( m\right) }\left( Q\right) $ has side length $%
\ell \left( I\right) =2^{-m}\ell \left( Q\right) $, and define the
collections%
\begin{eqnarray*}
\mathfrak{G}^{\left( m\right) }\left( Q\right) &\equiv &\left\{ I\in 
\mathfrak{C}^{\left( m\right) }\left( Q\right) :I\subset Q\text{ and }%
\partial I\cap Z\neq \emptyset \right\} , \\
\mathfrak{H}^{\left( m\right) }\left( Q\right) &\equiv &\left\{ I\in 
\mathfrak{C}^{\left( m\right) }\left( Q\right) :3I\subset Q\text{ and }%
\partial \left( 3I\right) \cap Z\neq \emptyset \right\} .
\end{eqnarray*}%
Then 
\begin{equation*}
Q\cap Z_{\delta }=\bigcup_{I\in \mathfrak{G}^{\left( m\right) }\left(
Q\right) }I\text{ and }Q\setminus Z_{\delta }=\overset{\cdot }{\dbigcup }%
_{k=2}^{m}\bigcup_{I\in \mathfrak{H}^{\left( k\right) }\left( Q\right) }I.
\end{equation*}%
From Remark \ref{normalized}, we obtain that the the union $\bigcup_{I\in 
\mathfrak{H}^{\left( k\right) }\left( Q\right) }RI$ contains $Q\cap
Z_{\delta }$ for $k\geq cm$ for some $c=c_{n,\kappa }\in \left( 0,1\right) $
depending only on $n$ and $\kappa $, and in particular independent of $m$.
Then from the doubling condition we have $\left\vert RI\right\vert _{\mu
}\leq D_{R}\left\vert I\right\vert _{\mu }$ for all cubes $I$ and some
constant $D_{R}$, and so for $k\geq cm$, 
\begin{eqnarray*}
\left\vert \mathfrak{H}^{\left( k\right) }\left( Q\right) \right\vert _{\mu
} &=&\sum_{I\in \mathfrak{H}^{\left( k\right) }\left( Q\right) }\left\vert
I\right\vert _{\mu }\geq \sum_{I\in \mathfrak{H}^{\left( k\right) }\left(
Q\right) }\frac{1}{D_{R}}\left\vert RI\right\vert _{\mu }=\frac{1}{D_{R}}%
\int \left( \sum_{I\in \mathfrak{H}^{\left( k\right) }\left( Q\right) }%
\mathbf{1}_{RI}\right) d\mu \\
&\geq &\frac{1}{D_{R}}\int \left( \sum_{I\in \mathfrak{G}^{\left( k\right)
}\left( Q\right) }\mathbf{1}_{I}\right) d\mu =\frac{1}{D_{R}}\left\vert 
\mathfrak{G}^{\left( k\right) }\left( Q\right) \right\vert _{\mu }\geq \frac{%
1}{D_{R}}\left\vert \mathfrak{G}^{\left( m\right) }\left( Q\right)
\right\vert _{\mu }=\frac{1}{D_{R}}\left\vert Q\cap Z_{\delta }\right\vert
_{\mu }\ .
\end{eqnarray*}%
Thus we have%
\begin{equation*}
\left\vert Q\right\vert _{\mu }\geq \sum_{k=cm}^{m}\left\vert \mathfrak{H}%
^{\left( k\right) }\left( Q\right) \right\vert _{\mu }\geq \frac{m\left(
1-c\right) }{D_{R}}\left\vert Q\cap Z_{\delta }\right\vert _{\mu }\ ,
\end{equation*}%
which proves the lemma.
\end{proof}

We can apply the method used in (\ref{improvement}) to obtain a power decay
instead of a logarithmic decay.

\begin{corollary}
\label{zero set doubling}Let $\kappa \in \mathbb{N}$. Suppose $\mu $ is a
doubling measure on $\mathbb{R}^{n}$ and that $Q\in \mathcal{Q}^{n}$. Let $Z$
denote the zero set of a $Q$-normalized polynomial $P$ of degree less than $%
\kappa $, and for $0<\delta <1$, let $Z_{\delta }$ denote the $\delta $-halo
of $Z$. Then for a positive constant $C_{n,\kappa }$ depending only on $n$
and $\kappa $, and not on $P$ itself, and for some $\theta >0$, we have%
\begin{equation*}
\left\vert Q\cap Z_{\delta }\right\vert _{\mu }\leq C_{n,\kappa }\delta
^{\theta }\left\vert Q\right\vert _{\mu }\ .
\end{equation*}%
In particular this holds for $Z=\partial Q$, which is a finite union of zero
sets of linear functions.
\end{corollary}

\begin{proof}
Without loss of generality $Q$ is the unit cube $\left[ 0,1\right] ^{n}$.
Define an even function $w\left( t\right) $ on $\left[ -1,1\right] $, that
is increasing on $\left[ 0,1\right] $, by the formula%
\begin{equation*}
w\left( t\right) \equiv \left\vert Z_{t}\right\vert _{\mu }\ ,\ \ \ \ \
0\leq t\leq 1.
\end{equation*}%
Since $P$ is a $Q$-normalized polynomial of degree less than $\kappa $,
there are positive constants $t_{0},\,c_{0},C_{0},A$ such that for every $%
0<t<t_{0}$, there is a collection of cubes $\left\{ Q_{i}^{t}\right\} _{i}$
with 
\begin{eqnarray*}
\ell \left( Q_{i}^{t}\right) &=&c_{0}t, \\
\mathbf{1}_{Z_{t}}\left( x\right) &\leq &\sum_{i}\mathbf{1}%
_{Q_{i}^{t}}\left( x\right) \leq A\mathbf{1}_{Z_{t}}\left( x\right) \\
\mathbf{1}_{Z_{2t}}\left( x\right) &\leq &\sum_{i}\mathbf{1}%
_{C_{0}Q_{i}^{t}}\left( x\right) \leq A.
\end{eqnarray*}%
Thus we have%
\begin{equation*}
w\left( 2t\right) \leq \sum_{i}\left\vert C_{0}Q_{i}^{t}\right\vert _{\mu
}\leq \sum_{i}C_{\limfunc{doub}}\left\vert Q_{i}^{t}\right\vert _{\mu }\leq
C_{\limfunc{doub}}M\left\vert Z_{t}\right\vert _{\mu }=C_{\limfunc{doub}%
}Aw\left( t\right) ,\ \ \ \ \ 0<t<t_{0},
\end{equation*}%
and hence there is a doubling exponent $\theta _{w}^{\limfunc{doub}}$ such
that%
\begin{equation*}
\frac{w\left( st\right) }{w\left( t\right) }\leq s^{\theta _{w}^{\limfunc{%
doub}}},\ \ \ \ \ \text{for all sufficiently large }s.
\end{equation*}%
We claim $w\left( t\right) $ also satisfies the reverse doubling condition%
\begin{equation*}
w\left( \delta \right) \leq C\delta ^{\varepsilon }w\left( 1\right) ,\ \ \ \
\ 0<\delta <t_{0}.
\end{equation*}%
Indeed, let $d\mu \equiv \frac{dw}{dt}$. Then assuming $s\geq 5$ in the
definition of $\theta _{w}^{\limfunc{doub}}$, we obtain for $Q=\left[ 0,t%
\right] $ that%
\begin{eqnarray*}
\left\vert 3Q\setminus Q\right\vert _{\mu } &=&\sum_{I\in \mathcal{D}:\
I\subset 3Q\setminus Q,\ell \left( I\right) =\ell \left( Q\right)
}\left\vert I\right\vert _{\mu }\geq \sum_{I\in \mathcal{D}:\ I\subset
3Q\setminus Q,\ell \left( I\right) =\ell \left( Q\right) }5^{-\theta _{\mu
}^{\limfunc{doub}}}\left\vert 5I\right\vert _{\mu }\geq \left(
3^{n}-1\right) 5^{-\theta _{\mu }^{\limfunc{doub}}}\left\vert Q\right\vert
_{\mu } \\
&\Longrightarrow &\left\vert Q\right\vert _{\mu }=\left\vert 3Q\right\vert
_{\mu }-\left\vert 3Q\setminus Q\right\vert _{\mu }\leq \left( 1-\frac{%
3^{n}-1}{5^{\theta _{\mu }^{\limfunc{doub}}}}\right) \left\vert
3Q\right\vert _{\mu }\ ,
\end{eqnarray*}%
which gives reverse doubling, and hence%
\begin{equation*}
\left\vert Q\cap Z_{\delta }\right\vert _{\mu }=\int_{0}^{\delta }w\left(
t\right) dt\leq C\delta ^{\varepsilon }\left\vert Q\right\vert _{\mu },\ \ \
\ \ \text{for }0<\delta <t_{0},
\end{equation*}%
and trivially this is extended to $\delta <1$ by possibly increasing the
constant $C$.
\end{proof}

\subsection{Weighted Alpert bases for $L^{2}\left( \protect\mu \right) $ and 
$L^{\infty }$ control of projections\label{Subsection Haar}}

The following theorem was proved in \cite{RaSaWi}, which establishes the
existence of Alpert wavelets, for $L^{2}\left( \mu \right) $ in all
dimensions, having the three important properties of orthogonality,
telescoping and moment vanishing. Since the statement is simplified for
doubling measures, and this is the only case considered in our main theorem,
we restrict ourselves to this case here.

We first recall the basic construction of weighted Alpert wavelets in \cite%
{RaSaWi} restricted to doubling measures. Let $\mu $ be a doubling measure
on $\mathbb{R}^{n}$, and fix $\kappa \in \mathbb{N}$. For $Q\in \mathcal{Q}%
^{n}$, the collection of cubes with sides parallel to the coordinate axes,
denote by $L_{Q;\kappa }^{2}\left( \mu \right) $ the finite dimensional
subspace of $L^{2}\left( \mu \right) $ that consists of linear combinations
of the indicators of\ the children $\mathfrak{C}\left( Q\right) $ of $Q$
multiplied by polynomials of degree less than $\kappa $, and such that the
linear combinations have vanishing $\mu $-moments on the cube $Q$ up to
order $\kappa -1$:%
\begin{equation*}
L_{Q;\kappa }^{2}\left( \mu \right) \equiv \left\{ f=\dsum\limits_{Q^{\prime
}\in \mathfrak{C}\left( Q\right) }\mathbf{1}_{Q^{\prime }}p_{Q^{\prime
};\kappa }\left( x\right) :\int_{Q}f\left( x\right) x^{\beta }d\mu \left(
x\right) =0,\ \ \ \text{for }0\leq \left\vert \beta \right\vert <\kappa
\right\} ,
\end{equation*}%
where $p_{Q^{\prime };\kappa }\left( x\right) =\sum_{\beta \in \mathbb{Z}%
_{+}^{n}:\left\vert \beta \right\vert \leq \kappa -1\ }a_{Q^{\prime };\beta
}x^{\beta }$ is a polynomial in $\mathbb{R}^{n}$ of degree less than $\kappa 
$. Here $x^{\beta }=x_{1}^{\beta _{1}}x_{2}^{\beta _{2}}...x_{n}^{\beta
_{n}} $. Let $d_{Q;\kappa }\equiv \dim L_{Q;\kappa }^{2}\left( \mu \right) $
be the dimension of the finite dimensional linear space $L_{Q;\kappa
}^{2}\left( \mu \right) $.

Let $\mathcal{D}$ denote a dyadic grid on $\mathbb{R}^{n}$ and for $Q\in 
\mathcal{D}$, let $\bigtriangleup _{Q;\kappa }^{\mu }$ denote orthogonal
projection onto the finite dimensional subspace $L_{Q;\kappa }^{2}\left( \mu
\right) $, and let $\mathbb{E}_{Q;\kappa }^{\mu }$ denote orthogonal
projection onto the finite dimensional subspace%
\begin{equation*}
\mathcal{P}_{Q;\kappa }^{n}\left( \sigma \right) \equiv \mathnormal{\limfunc{%
Span}}\{\mathbf{1}_{Q}x^{\beta }:0\leq \left\vert \beta \right\vert <\kappa
\}.
\end{equation*}

\begin{theorem}[Weighted Alpert Bases]
\label{main1}Let $\mu $ be a doubling measure on $\mathbb{R}^{n}$, fix $%
\kappa \in \mathbb{N}$, and fix a dyadic grid $\mathcal{D}$ in $\mathbb{R}%
^{n}$.

\begin{enumerate}
\item Then $\left\{ \bigtriangleup _{Q;\kappa }^{\mu }\right\} _{Q\in 
\mathcal{D}}$ is a complete set of orthogonal projections in $L^{2}\left(
\mu \right) $ and%
\begin{eqnarray}
f &=&\sum_{Q\in \mathcal{D}}\bigtriangleup _{Q;\kappa }^{\mu }f,\ \ \ \ \
f\in L^{2}\left( \mu \right) ,  \label{Alpert expan} \\
&&\left\langle \bigtriangleup _{P;\kappa }^{\mu }f,\bigtriangleup _{Q;\kappa
}^{\mu }f\right\rangle _{L^{2}\left( \mu \right) }=0\text{ for }P\neq Q, 
\notag
\end{eqnarray}%
where convergence in the first line holds both in $L^{2}\left( \mu \right) $
norm and pointwise $\mu $-almost everywhere.

\item Moreover we have the telescoping identities%
\begin{equation}
\mathbf{1}_{Q}\sum_{I:\ Q\subsetneqq I\subset P}\bigtriangleup _{I;\kappa
}^{\mu }=\mathbb{E}_{Q;\kappa }^{\mu }-\mathbf{1}_{Q}\mathbb{E}_{P;\kappa
}^{\mu }\ \text{ \ for }P,Q\in \mathcal{D}\text{ with }Q\subsetneqq P,
\label{telescoping}
\end{equation}

\item and the moment vanishing conditions%
\begin{equation}
\int_{\mathbb{R}^{n}}\bigtriangleup _{Q;\kappa }^{\mu }f\left( x\right) \
x^{\beta }d\mu \left( x\right) =0,\ \ \ \text{for }Q\in \mathcal{D},\text{ }%
\beta \in \mathbb{Z}_{+}^{n},\ 0\leq \left\vert \beta \right\vert <\kappa \ .
\label{mom con}
\end{equation}
\end{enumerate}
\end{theorem}

We can fix\ an orthonormal basis $\left\{ h_{Q;\kappa }^{\mu ,a}\right\}
_{a\in \Gamma _{Q,n,\kappa }}$ of $L_{Q;\kappa }^{2}\left( \mu \right) $
where $\Gamma _{Q,n,\kappa }$ is a convenient finite index set. Then 
\begin{equation*}
\left\{ h_{Q;\kappa }^{\mu ,a}\right\} _{a\in \Gamma _{Q,n,\kappa }\text{
and }Q\in \mathcal{D}}
\end{equation*}%
is an orthonormal basis for $L^{2}\left( \mu \right) $. In particular we
have 
\begin{eqnarray*}
\left\Vert f\right\Vert _{L^{2}\left( \mu \right) }^{2} &=&\sum_{Q\in 
\mathcal{D}}\left\Vert \bigtriangleup _{Q;\kappa }^{\mu }f\right\Vert
_{L^{2}\left( \mu \right) }^{2}=\sum_{Q\in \mathcal{D}}\left\vert \widehat{f}%
\left( Q\right) \right\vert ^{2}, \\
\text{where }\widehat{f}\left( Q\right) &=&\left\{ \left\langle
f,h_{Q;\kappa }^{\mu ,a}\right\rangle _{\mu }\right\} _{a\in \Gamma
_{Q,n,\kappa }}\text{ and }\left\vert \widehat{f}\left( Q\right) \right\vert
^{2}\equiv \sum_{a\in \Gamma _{Q,n,\kappa }\text{ }}\left\vert \left\langle
f,h_{Q;\kappa }^{\mu ,a}\right\rangle _{\mu }\right\vert ^{2}.
\end{eqnarray*}%
In terms of the Alpert coefficient vectors $\widehat{f_{\kappa }}\left(
I\right) \equiv \left\{ \left\langle f,a_{I;\kappa ,j}^{\mu }\right\rangle
\right\} _{j=1}^{\kappa }$, we have for the special case of a doubling
measure $\mu $ (see \cite[(4.7) on page 14]{Saw6}),%
\begin{equation}
\left\vert \widehat{f_{\kappa }}\left( I\right) \right\vert =\left\Vert
\bigtriangleup _{I;\kappa }^{\mu }f\right\Vert _{L^{2}\left( \mu \right)
}\leq \left\Vert \bigtriangleup _{I;\kappa }^{\mu }f\right\Vert _{L^{\infty
}\left( \mu \right) }\sqrt{\left\vert I\right\vert _{\mu }}\leq C\left\Vert
\bigtriangleup _{I;\kappa }^{\mu }f\right\Vert _{L^{2}\left( \mu \right)
}=C\left\vert \widehat{f_{\kappa }}\left( I\right) \right\vert ,
\label{analogue'}
\end{equation}%
and in particular,%
\begin{equation}
\left\Vert h_{Q;\kappa }^{\mu ,a}\right\Vert _{L^{\infty }\left( \mu \right)
}\approx \frac{1}{\sqrt{\left\vert I\right\vert _{\mu }}},  \label{in part}
\end{equation}%
since $\bigtriangleup _{I;\kappa }^{\mu }h_{Q;\kappa }^{\mu ,a}=h_{Q;\kappa
}^{\mu ,a}$.

\begin{notation}
For doubling measures $\mu $, the cardinality of $\Gamma _{Q,n,\kappa }$
depends only on $n$ and $\kappa $, which are usually known from context, and
so we will simply write $\Gamma $ when $\mu $ is doubling.
\end{notation}

From now on, all measures considered will be assumed to be doubling, and
often without explicit mention. We now inroduce Sobolev spaces defined by
weighted Alpert projections instead of weighted Haar projections. We show
below that these spaces are actually equivalent. For convenience we repeat
the definition of weighted Sobolev space used in this paper.

\begin{definition}
Given $\kappa \in \mathbb{N}$ and $s\in \mathbb{R}$ and $\mathcal{D}\in
\Omega _{\limfunc{dyad}}$ a dyadic grid, we define the $\kappa $-dyadic
homogeneous $W_{\mathcal{D};\kappa }^{s}\left( \mu \right) $-Sobolev norm of
a function $f\in L^{2}\left( \mu \right) $ by%
\begin{equation*}
\left\Vert f\right\Vert _{W_{\mathcal{D};\kappa }^{s}\left( \mu \right)
}^{2}\equiv \sum_{Q\in \mathcal{D}}\ell \left( Q\right) ^{-2s}\left\Vert
\bigtriangleup _{Q;\kappa }^{\mu }f\right\Vert _{L^{2}\left( \mu \right)
}^{2}\ ,
\end{equation*}%
and we denote by $W_{\mathcal{D};\kappa }^{s}\left( \mu \right) $ the
corresponding Hilbert space completion.
\end{definition}

\begin{lemma}
\label{complete set}The set $\left\{ \ell \left( Q\right) ^{s}h_{Q;\kappa
}^{\mu ,a}\right\} _{\left( Q,a\right) \in \mathcal{D}\times \Gamma }$ is an
orthonormal basis for $W_{\mathcal{D};\kappa }^{s}\left( \mu \right) $, and
thus for any subset $\mathcal{H}$ of the dyadic grid $\mathcal{D}$, we have,%
\begin{equation}
\left\Vert \sum_{\left( I,a\right) \in \mathcal{H}\times \Gamma }c_{I,a}\ell
\left( I\right) ^{s}h_{I;\kappa }^{\mu ,a}\right\Vert _{W_{\mathcal{D}%
;\kappa }^{s}\left( \mu \right) }^{2}=\sum_{\left( I,a\right) \in \mathcal{H}%
\times \Gamma }\left\vert c_{I,a}\right\vert ^{2}.  \label{ortho}
\end{equation}
\end{lemma}

\begin{proof}
We have%
\begin{eqnarray*}
\left\langle \ell \left( I\right) ^{s}h_{I;\kappa }^{\mu ,a},\ell \left(
J\right) ^{s}h_{J;\kappa }^{\mu ,b}\right\rangle _{W_{\mathcal{D};\kappa
}^{s}\left( \mu \right) } &=&\sum_{Q\in \mathcal{D}}\ell \left( Q\right)
^{-2s}\left\langle \bigtriangleup _{Q;\kappa }^{\mu }\ell \left( I\right)
^{s}h_{I;\kappa }^{\mu ,a},\bigtriangleup _{Q;\kappa }^{\mu }\ell \left(
J\right) ^{s}h_{J;\kappa }^{\mu ,b}\right\rangle _{L^{2}\left( \mu \right) }
\\
&=&\sum_{Q\in \mathcal{D}}\ell \left( Q\right) ^{-2s}\ell \left( I\right)
^{s}\ell \left( J\right) ^{s}\left\langle \bigtriangleup _{Q;\kappa }^{\mu
}h_{I;\kappa }^{\mu ,a},\bigtriangleup _{Q;\kappa }^{\mu }h_{J;\kappa }^{\mu
,b}\right\rangle _{L^{2}\left( \mu \right) } \\
&=&\left\{ 
\begin{array}{ccc}
1 & \text{ if } & I=J\text{ and }a=b \\ 
0 & \text{ if } & I\not=J\text{ or }a\not=b%
\end{array}%
\right. ,
\end{eqnarray*}%
since $\bigtriangleup _{Q;\kappa }^{\mu }h_{K;\kappa }^{\mu ,a}$ vanishes if 
$Q\not=K$, and equals $h_{K;\kappa }^{\mu ,a}$ if $Q=K$. Thus $\left\{ \ell
\left( Q\right) ^{s}h_{Q;\kappa }^{\mu ,a}\right\} _{\left( Q,a\right) \in 
\mathcal{D}\times \Gamma }$ is an orthonormal basis and we conclude that (%
\ref{ortho}) holds.
\end{proof}

\subsection{Equivalence of Sobolev spaces}

In \cite{Tri}, Triebel defines the usual homogeneous unweighted Sobolev
space $W^{s}$ with norm $\left\Vert f\right\Vert _{W^{s}}$ given by%
\begin{equation*}
\left\Vert f\right\Vert _{W^{s}}^{2}\equiv \int_{\mathbb{R}^{n}}\left\vert
\left( -\bigtriangleup \right) ^{\frac{s}{2}}f\left( x\right) \right\vert
^{2}dx=\int_{\mathbb{R}^{n}}\left\vert \left( \left\vert \xi \right\vert
^{2}\right) ^{\frac{s}{2}}\widehat{f}\left( \xi \right) \right\vert ^{2}d\xi
,
\end{equation*}%
and the corresponding inhomogeneous version with norm squared $\int_{\mathbb{%
R}^{n}}\left\vert \left( I-\bigtriangleup \right) ^{\frac{s}{2}}f\left(
x\right) \right\vert ^{2}dx$, which we will not consider here. Combining
results of Triebel \cite{Tri} with those of Seeger and Ullrich \cite{SeUl}
shows that

\begin{lemma}[\protect\cite{Tri},\protect\cite{SeUl}]
The following three statements are equivalent for$\ \mu $ equal to Lebesgue
measure and $s\in \mathbb{R}$.

\begin{enumerate}
\item $W_{\func{dyad}}^{s}=W^{s}$,

\item $\left\{ \ell \left( I\right) ^{s}h_{I}\right\} _{I\in \mathcal{D}}$
is an orthonormal basis for $W^{s}$,

\item $-\frac{1}{2}<s<\frac{1}{2}$.
\end{enumerate}
\end{lemma}

Here is the first step toward proving the equivalence of the different
dyadic Sobolev spaces over all grids $\mathcal{D}$ and integers $\kappa \in 
\mathbb{N}$, which in particular is used to implement the $\func{good}/\func{%
bad}$ cube technology of Nazarov, Treil and Volberg. The reader can notice
that the doubling property of the measure $\mu $ is not explicitly used in
this argument, rather only in the definition of the Sobolev spaces.

\begin{lemma}
\label{coin}Let $\mu $ be a doubling measure on $\mathbb{R}^{n}$ and let $%
\mathcal{D}$ be a dyadic grid on $\mathbb{R}^{n}$. Then for $\kappa
_{1},\kappa _{2}\in \mathbb{N}$ and $s\in \mathbb{R}$, we have,%
\begin{equation*}
W_{\mathcal{D};\kappa _{1}}^{s}\left( \mu \right) =W_{\mathcal{D};\kappa
_{2}}^{s}\left( \mu \right) ,
\end{equation*}%
with equivalence of norms.
\end{lemma}

\begin{proof}
We first claim that for $s<0$,%
\begin{equation}
W_{\mathcal{D};\kappa _{2}}^{s}\left( \mu \right) \subset W_{\mathcal{D}%
;\kappa _{1}}^{s}\left( \mu \right) \text{ for }1\leq \kappa _{1}\leq \kappa
_{2},  \label{now claim}
\end{equation}%
which by duality gives for $s>0$, 
\begin{equation}
W_{\mathcal{D};\kappa _{1}}^{s}\left( \mu \right) \subset W_{\mathcal{D}%
;\kappa _{2}}^{s}\left( \mu \right) \text{ for }s>0\text{ and }1\leq \kappa
_{1}\leq \kappa _{2}.  \label{now claim''}
\end{equation}%
Indeed, for any subset $\mathcal{H}\subset \mathcal{D}$, we have 
\begin{eqnarray*}
&&\left\Vert \sum_{I\in \mathcal{H};\ a\in \Gamma _{I,n,\kappa
}}c_{I,a}h_{I;\kappa _{2}}^{\mu ,a}\right\Vert _{W_{\mathcal{D};\kappa
_{1}}^{s}\left( \mu \right) }^{2}=\sum_{Q\in \mathcal{D}}\ell \left(
Q\right) ^{-2s}\left\Vert \sum_{I\in \mathcal{H};\ a\in \Gamma _{I,n,\kappa
}:\ Q\subset I}c_{I,a}\bigtriangleup _{Q;\kappa _{1}}^{\mu }h_{I;\kappa
_{2}}^{\mu ,a}\right\Vert _{L^{2}\left( \mu \right) }^{2} \\
&=&\sum_{Q\in \mathcal{D}}\ell \left( Q\right) ^{-2s}\sum_{I\in \mathcal{H}%
;\ a\in \Gamma _{I,n,\kappa }}\left( c_{I,a}\right) ^{2}\left\Vert
\bigtriangleup _{Q;\kappa _{1}}^{\mu }h_{I;\kappa _{2}}^{\mu ,a}\right\Vert
_{L^{2}\left( \mu \right) }^{2} \\
&&+\sum_{Q\in \mathcal{D}}\ell \left( Q\right) ^{-2s}\sum_{I,I^{\prime }\in 
\mathcal{H};\ a\in \Gamma _{I,n,\kappa },a^{\prime }\in \Gamma _{I^{\prime
},n,\kappa }:\ Q\subset I\cap I^{\prime }}\int_{\mathbb{R}^{n}}c_{I,a}\left(
\bigtriangleup _{Q;\kappa _{1}}^{\mu }h_{I;\kappa _{2}}^{\mu ,a}\right)
c_{I^{\prime },a^{\prime }}\left( \bigtriangleup _{Q;\kappa _{1}}^{\mu
}h_{I^{\prime };\kappa _{2}}^{\mu ,a^{\prime }}\right) d\mu \\
&\equiv &A+B,
\end{eqnarray*}%
where the first term satisfies%
\begin{eqnarray*}
&&A=\sum_{Q\in \mathcal{D}}\ell \left( Q\right) ^{-2s}\sum_{I\in \mathcal{H}%
;\ a\in \Gamma _{I,n,\kappa }}\left( c_{I,a}\right) ^{2}\left\Vert
\bigtriangleup _{Q;\kappa _{1}}^{\mu }h_{I;\kappa _{2}}^{\mu ,a}\right\Vert
_{L^{2}\left( \mu \right) }^{2}=\sum_{I\in \mathcal{H};\ a\in \Gamma
_{I,n,\kappa }}\left( c_{I,a}\right) ^{2}\sum_{Q\in \mathcal{D}}\ell \left(
Q\right) ^{-2s}\left\Vert \bigtriangleup _{Q;\kappa _{1}}^{\mu }h_{I;\kappa
_{2}}^{\mu ,a}\right\Vert _{L^{2}\left( \mu \right) }^{2} \\
&=&\sum_{I\in \mathcal{H};\ a\in \Gamma _{I,n,\kappa }}\left( c_{I,a}\right)
^{2}\left\Vert h_{I;\kappa _{2}}^{\mu ,a}\right\Vert _{W_{\mathcal{D};\kappa
_{1}}^{s}\left( \mu \right) }^{2}\leq C\sum_{I\in \mathcal{H}}\left(
c_{I,a}\right) ^{2}\left\Vert h_{I;\kappa _{2}}^{\mu ,a}\right\Vert _{W_{%
\mathcal{D};\kappa _{2}}^{s}\left( \mu \right) }^{2} \\
&=&C\sum_{I\in \mathcal{H};\ a\in \Gamma _{I,n,\kappa }}\left(
c_{I,a}\right) ^{2}\ell \left( I\right) ^{-2s}=C\left\Vert \sum_{I\in 
\mathcal{H};\ a\in \Gamma _{I,n,\kappa }}c_{I,a}h_{I;\kappa _{2}}^{\mu
,a}\right\Vert _{W_{\mathcal{D};\kappa _{2}}^{s}\left( \mu \right) }^{2},
\end{eqnarray*}%
and where the final equality follows from Lemma \ref{complete set}.

To handle the second term, we write 
\begin{eqnarray*}
B &=&\sum_{I\neq I^{\prime }\in \mathcal{H};\ a\in \Gamma _{I,n,\kappa
},a^{\prime }\in \Gamma _{I^{\prime },n,\kappa }}c_{I,a}c_{I^{\prime
},a^{\prime }}\int_{\mathbb{R}^{n}}\left\{ \sum_{Q\in \mathcal{D}:\ Q\subset
I\cap I^{\prime }}\ell \left( Q\right) ^{-2s}\bigtriangleup _{Q;1}^{\mu
}h_{I;\kappa _{2}}^{\mu ,a}\right\} h_{I^{\prime };\kappa _{2}}^{\mu
,a^{\prime }}d\mu \\
&=&\left\{ \sum_{I\subsetneqq I^{\prime }\in \mathcal{H};\ a\in \Gamma
_{I,n,\kappa },a^{\prime }\in \Gamma _{I^{\prime },n,\kappa
}}+\sum_{I^{\prime }\subsetneqq I\in \mathcal{H};\ a,a^{\prime }\in \Gamma
}\right\} c_{I,a}c_{I^{\prime },a^{\prime }}\int_{\mathbb{R}^{n}}\left\{
\sum_{Q\in \mathcal{D}:\ Q\subset I}\ell \left( Q\right)
^{-2s}\bigtriangleup _{Q;1}^{\mu }h_{I;\kappa _{2}}^{\mu ,a}\right\}
h_{I^{\prime };\kappa _{2}}^{\mu ,a^{\prime }}d\mu \\
&\equiv &B_{1}+B_{2}\ ,
\end{eqnarray*}%
where $B_{1}$ and $B_{2}$ are symmetric. So it suffices to estimate%
\begin{equation*}
B_{1}=\sum_{I\in \mathcal{H};\ a\in \Gamma _{I,n,\kappa }}c_{I,a}\int_{%
\mathbb{R}^{n}}\left\{ \sum_{Q\in \mathcal{D}:\ Q\subset I}\ell \left(
Q\right) ^{-2s}\bigtriangleup _{Q;1}^{\mu }h_{I;\kappa _{2}}^{\mu
,a}\right\} \left\{ \sum_{I^{\prime }\in \mathcal{H};\ a^{\prime }\in \Gamma
_{I^{\prime },n,\kappa }:\ I\subsetneqq I^{\prime }}c_{I^{\prime },a^{\prime
}}h_{I^{\prime };\kappa _{2}}^{\mu ,a^{\prime }}\right\} d\mu ,
\end{equation*}%
where 
\begin{eqnarray*}
\left\Vert \sum_{I^{\prime }\in \mathcal{H};\ a^{\prime }\in \Gamma
_{I^{\prime },n,\kappa }:\ I\subsetneqq I^{\prime }}c_{I^{\prime },a^{\prime
}}h_{I^{\prime };\kappa _{2}}^{\mu ,a}\right\Vert _{L^{2}\left( \mu \right)
}^{2} &=&\sum_{I^{\prime }\in \mathcal{H};\ a^{\prime }\in \Gamma
_{I^{\prime },n,\kappa }:\ I\subsetneqq I^{\prime }}\left\vert c_{I^{\prime
},a^{\prime }}\right\vert ^{2}, \\
\left\Vert \sum_{Q\in \mathcal{D}:\ Q\subset I}\ell \left( Q\right)
^{-2s}\bigtriangleup _{Q;1}^{\mu }h_{I;\kappa _{2}}^{\mu ,a}\right\Vert
_{L^{2}\left( \mu \right) }^{2} &=&\sum_{Q\in \mathcal{D}:\ Q\subset I}\ell
\left( Q\right) ^{-4s}\left\Vert \bigtriangleup _{Q;1}^{\mu }h_{I;\kappa
_{2}}^{\mu ,a}\right\Vert _{L^{2}\left( \mu \right) }^{2}.
\end{eqnarray*}%
So for $s<0$ we have the estimate,%
\begin{eqnarray*}
B_{1} &\leq &\sum_{I\in \mathcal{H};\ a\in \Gamma _{I,n,\kappa }}\left\vert
c_{I,a}\right\vert \ell \left( I\right) ^{-2s}\left( \sum_{I^{\prime }\in 
\mathcal{H};\ a^{\prime }\in \Gamma _{I^{\prime },n,\kappa }:\ I\subsetneqq
I^{\prime }}\left\vert c_{I^{\prime },a^{\prime }}\right\vert ^{2}\right) ^{%
\frac{1}{2}} \\
&\leq &\sqrt{\sum_{I\in \mathcal{H};\ a\in \Gamma _{I,n,\kappa }}\left\vert
c_{I,a}\right\vert ^{2}\ell \left( I\right) ^{-2s}}\sqrt{\sum_{I\in \mathcal{%
H};\ a\in \Gamma _{I,n,\kappa }}\ell \left( I\right) ^{-2s}\left(
\sum_{I^{\prime }\in \mathcal{H};\ a^{\prime }\in \Gamma _{I^{\prime
},n,\kappa }:\ I\subsetneqq I^{\prime }}\left\vert c_{I^{\prime },a^{\prime
}}\right\vert ^{2}\right) },
\end{eqnarray*}%
where the second factor squared is%
\begin{eqnarray*}
\sum_{I\in \mathcal{H};\ a\in \Gamma _{I,n,\kappa }}\ell \left( I\right)
^{-2s}\left( \sum_{I^{\prime }\in \mathcal{H};\ a^{\prime }\in \Gamma :\
I\subsetneqq I^{\prime }}\left\vert c_{I^{\prime },a^{\prime }}\right\vert
^{2}\right) &=&\sum_{I^{\prime }\in \mathcal{H};\ a^{\prime }\in \Gamma
_{I^{\prime },n,\kappa }:\ I\subsetneqq I^{\prime }}\left\vert c_{I^{\prime
},a^{\prime }}\right\vert ^{2}\sum_{I\in \mathcal{H};\ a\in \Gamma
_{I,n,\kappa }:\ I\subsetneqq I^{\prime }}\ell \left( I\right) ^{-2s} \\
&\approx &\sum_{I^{\prime }\in \mathcal{H};\ a^{\prime }\in \Gamma
_{I^{\prime },n,\kappa }:\ I\subsetneqq I^{\prime }}\left\vert c_{I^{\prime
},a^{\prime }}\right\vert ^{2}\ell \left( I^{\prime }\right) ^{-2s}
\end{eqnarray*}%
since $s<0$. Altogether we have%
\begin{equation*}
B_{1}\lesssim \sum_{I\in \mathcal{H};\ a\in \Gamma _{I,n,\kappa }}\left\vert
c_{I,a}\right\vert ^{2}\ell \left( I\right) ^{-2s}=C\left\Vert \sum_{I\in 
\mathcal{H};\ a\in \Gamma _{I,n,\kappa }}c_{I,a}h_{I;\kappa _{2}}^{\mu
}\right\Vert _{W_{\mathcal{D};\kappa _{2}}^{s}\left( \mu \right) }^{2},
\end{equation*}%
which together with the estimate for term $I$ proves our claim (\ref{now
claim}).

Now we claim that for $s>0$ we have%
\begin{equation}
W_{\mathcal{D};\kappa _{1}}^{s}\left( \mu \right) \subset W_{\mathcal{D}%
;\kappa _{2}}^{s}\left( \mu \right) \text{ for all }\kappa _{1}\geq \kappa
_{2}\geq 1.  \label{now claim'}
\end{equation}%
Indeed, for any subset $\mathcal{H}\subset \mathcal{D}$, we have 
\begin{eqnarray*}
&&\left\Vert \sum_{I\in \mathcal{H};\ a\in \Gamma _{I,n,\kappa
}}c_{I,a}h_{I;\kappa _{2}}^{\mu ,a}\right\Vert _{W_{\mathcal{D};\kappa
_{1}}^{s}\left( \mu \right) }^{2}=\sum_{Q\in \mathcal{D}}\ell \left(
Q\right) ^{-2s}\left\Vert \sum_{I\in \mathcal{H};\ a\in \Gamma _{I,n,\kappa
}:\ Q\subset I}c_{I,a}\bigtriangleup _{Q;\kappa _{1}}^{\mu }h_{I;\kappa
_{2}}^{\mu ,a}\right\Vert _{L^{2}\left( \mu \right) }^{2} \\
&=&\sum_{Q\in \mathcal{D}}\ell \left( Q\right) ^{-2s}\sum_{I\in \mathcal{H}%
;\ a\in \Gamma _{I,n,\kappa }}\left( c_{I,a}\right) ^{2}\left\Vert
\bigtriangleup _{Q;\kappa _{1}}^{\mu }h_{I;\kappa _{2}}^{\mu ,a}\right\Vert
_{L^{2}\left( \mu \right) }^{2} \\
&&+\sum_{Q\in \mathcal{D}}\ell \left( Q\right) ^{-2s}\sum_{I,I^{\prime }\in 
\mathcal{H};\ a\in \Gamma _{I,n,\kappa },a^{\prime }\in \Gamma _{I^{\prime
},n,\kappa }:\ Q\supset I\vee I^{\prime }}\int_{\mathbb{R}^{n}}c_{I,a}\left(
\bigtriangleup _{Q;\kappa _{1}}^{\mu }h_{I;\kappa _{2}}^{\mu ,a}\right)
c_{I^{\prime },a^{\prime }}\left( \bigtriangleup _{Q;\kappa _{1}}^{\mu
}h_{I^{\prime };\kappa _{2}}^{\mu ,a^{\prime }}\right) d\mu \\
&\equiv &A+B,
\end{eqnarray*}%
where $I\vee I^{\prime }$ denotes the smallest dyadic cube containing both $%
I $ and $I^{\prime }$ if it exists; otherwise the sum over $Q\supset I\vee
I^{\prime }$ is empty. Just as before, the first term satisfies%
\begin{eqnarray*}
&&A=\sum_{Q\in \mathcal{D}}\ell \left( Q\right) ^{-2s}\sum_{I\in \mathcal{H}%
;\ a\in \Gamma _{I,n,\kappa }}\left( c_{I,a}\right) ^{2}\left\Vert
\bigtriangleup _{Q;\kappa _{1}}^{\mu }h_{I;\kappa _{2}}^{\mu ,a}\right\Vert
_{L^{2}\left( \mu \right) }^{2}=\sum_{I\in \mathcal{H};\ a\in \Gamma
_{I,n,\kappa }}\left( c_{I,a}\right) ^{2}\sum_{Q\in \mathcal{D}}\ell \left(
Q\right) ^{-2s}\left\Vert \bigtriangleup _{Q;\kappa _{1}}^{\mu }h_{I;\kappa
_{2}}^{\mu ,a}\right\Vert _{L^{2}\left( \mu \right) }^{2} \\
&=&\sum_{I\in \mathcal{H};\ a\in \Gamma _{I,n,\kappa }}\left( c_{I,a}\right)
^{2}\left\Vert h_{I;\kappa _{2}}^{\mu ,a}\right\Vert _{W_{\mathcal{D};\kappa
_{1}}^{s}\left( \mu \right) }^{2}\leq C\sum_{I\in \mathcal{H};\ a\in \Gamma
_{I,n,\kappa }}\left( c_{I,a}\right) ^{2}\left\Vert h_{I;\kappa _{2}}^{\mu
,a}\right\Vert _{W_{\mathcal{D};\kappa _{2}}^{s}\left( \mu \right) }^{2} \\
&=&C\sum_{I\in \mathcal{H};\ a\in \Gamma _{I,n,\kappa }}\left(
c_{I,a}\right) ^{2}\ell \left( I\right) ^{-2s}=C\left\Vert \sum_{I\in 
\mathcal{H};\ a\in \Gamma _{I,n,\kappa }}c_{I,a}h_{I;\kappa _{2}}^{\mu
}\right\Vert _{W_{\mathcal{D};\kappa _{2}}^{s}\left( \mu \right) }^{2},
\end{eqnarray*}%
and where the final equality follows from Lemma \ref{complete set}.

To handle the second term $B$, we only need to consider the two cases $%
I^{\prime }\subset I$ and $I^{\prime }\cap I=\emptyset ,\ell \left(
I^{\prime }\right) \leq \ell \left( I\right) $. For the first case, we have
the estimate, 
\begin{eqnarray*}
B_{\func{case}1} &\equiv &\left\vert \sum_{I\in \mathcal{H};\ a\in \Gamma
_{I,n,\kappa }}c_{I,a}\int_{\mathbb{R}^{n}}\left\{ \sum_{Q\in \mathcal{D}:\
Q\supset I}\ell \left( Q\right) ^{-2s}\bigtriangleup _{Q;\kappa _{1}}^{\mu
}h_{I;\kappa _{2}}^{\mu ,a}\right\} \left\{ \sum_{I^{\prime }\in \mathcal{H}%
;\ a^{\prime }\in \Gamma _{I^{\prime },n,\kappa }:\ I^{\prime }\subsetneqq
I}c_{I^{\prime },a^{\prime }}h_{I^{\prime };\kappa _{2}}^{\mu ,a^{\prime
}}\right\} d\mu \right\vert \\
&\leq &\sum_{I\in \mathcal{H};\ a\in \Gamma _{I,n,\kappa }}\left\vert
c_{I,a}\right\vert \left\Vert \sum_{Q\in \mathcal{D}:\ Q\supset I}\ell
\left( Q\right) ^{-2s}\bigtriangleup _{Q;\kappa _{1}}^{\mu }h_{I;\kappa
_{2}}^{\mu ,a}\right\Vert _{L^{2}\left( \mu \right) }\left\Vert
\sum_{I^{\prime }\in \mathcal{H};\ a^{\prime }\in \Gamma _{I^{\prime
},n,\kappa }:\ I^{\prime }\subsetneqq I}c_{I^{\prime },a^{\prime
}}h_{I^{\prime };\kappa _{2}}^{\mu ,a^{\prime }}\right\Vert _{L^{2}\left(
\mu \right) },
\end{eqnarray*}%
where for $s>0$,%
\begin{equation*}
\left\Vert \sum_{Q\in \mathcal{D}:\ Q\supset I}\ell \left( Q\right)
^{-2s}\bigtriangleup _{Q;\kappa _{1}}^{\mu }h_{I;\kappa _{2}}^{\mu
,a}\right\Vert _{L^{2}\left( \mu \right) }\leq \ell \left( I\right)
^{-2s}\left\Vert \sum_{Q\in \mathcal{D}:\ Q\supset I}\bigtriangleup
_{Q;\kappa _{1}}^{\mu }h_{I;\kappa _{2}}^{\mu ,a}\right\Vert _{L^{2}\left(
\mu \right) }=\ell \left( I\right) ^{-2s},
\end{equation*}%
and%
\begin{equation*}
\left\Vert \sum_{I^{\prime }\in \mathcal{H};\ a^{\prime }\in \Gamma
_{I^{\prime },n,\kappa }:\ I^{\prime }\subsetneqq I}c_{I^{\prime },a^{\prime
}}h_{I^{\prime };\kappa _{2}}^{\mu ,a}\right\Vert _{L^{2}\left( \mu \right)
}=\sqrt{\sum_{I^{\prime }\in \mathcal{H};\ a^{\prime }\in \Gamma _{I^{\prime
},n,\kappa }:\ I^{\prime }\subsetneqq I}\left\vert c_{I^{\prime },a^{\prime
}}\right\vert ^{2}}.
\end{equation*}%
Thus we obtain the estimate%
\begin{eqnarray*}
B_{\func{case}1} &\leq &\sum_{I\in \mathcal{H};\ a\in \Gamma _{I,n,\kappa
}}\left\vert c_{I,a}\right\vert \ell \left( I\right) ^{-2s}\sqrt{%
\sum_{I^{\prime }\in \mathcal{H};\ a^{\prime }\in \Gamma _{I^{\prime
},n,\kappa }:\ I^{\prime }\subsetneqq I}\left\vert c_{I^{\prime },a^{\prime
}}\right\vert ^{2}} \\
&\leq &\sqrt{\sum_{I\in \mathcal{H};\ a\in \Gamma _{I,n,\kappa }}\left\vert
c_{I,a}\right\vert ^{2}\ell \left( I\right) ^{-2s}}\sqrt{\sum_{I\in \mathcal{%
H};\ a\in \Gamma _{I,n,\kappa }}\ell \left( I\right) ^{-2s}\sum_{I^{\prime
}\in \mathcal{H};\ a^{\prime }\in \Gamma _{I^{\prime },n,\kappa }:\
I^{\prime }\subsetneqq I}\left\vert c_{I^{\prime },a^{\prime }}\right\vert
^{2}}
\end{eqnarray*}%
where%
\begin{eqnarray*}
\sum_{I\in \mathcal{H};\ a\in \Gamma _{I,n,\kappa }}\ell \left( I\right)
^{-2s}\sum_{I^{\prime }\in \mathcal{H};\ a^{\prime }\in \Gamma _{I^{\prime
},n,\kappa }:\ I^{\prime }\subsetneqq I}\left\vert c_{I^{\prime },a^{\prime
}}\right\vert ^{2} &=&\sum_{I^{\prime }\in \mathcal{H};\ a^{\prime }\in
\Gamma _{I^{\prime },n,\kappa }}\left\vert c_{I^{\prime },a^{\prime
}}\right\vert ^{2}\sum_{I\in \mathcal{H};\ a\in \Gamma _{I,n,\kappa }:\
I^{\prime }\subsetneqq I}\ell \left( I\right) ^{-2s} \\
&\approx &\sum_{I^{\prime }\in \mathcal{H};\ a^{\prime }\in \Gamma
_{I^{\prime },n,\kappa }}\left\vert c_{I^{\prime },a^{\prime }}\right\vert
^{2}\ell \left( I^{\prime }\right) ^{-2s}.
\end{eqnarray*}%
Altogether we have%
\begin{equation*}
B_{\func{case}1}\lesssim \sum_{I\in \mathcal{H};\ a\in \Gamma _{I,n,\kappa
}}\left\vert c_{I,a}\right\vert ^{2}\ell \left( I\right) ^{-2s}\approx
\left\Vert \sum_{I\in \mathcal{H};\ a\in \Gamma _{I,n,\kappa
}}c_{I,a}h_{I;\kappa _{2}}^{\mu ,a}\right\Vert _{W_{\mathcal{D};\kappa
_{2}}^{s}\left( \mu \right) }^{2},
\end{equation*}%
which is the desired estimate for $B_{\func{case}1}$.

Turning finally to the second case $I^{\prime }\cap I=\emptyset ,\ell \left(
I\right) \leq \ell \left( I^{\prime }\right) $, we have%
\begin{eqnarray*}
&&B_{\func{case}2}\equiv \sum_{Q\in \mathcal{D}}\ell \left( Q\right)
^{-2s}\sum_{\substack{ I,I^{\prime }\in \mathcal{H};\ a\in \Gamma
_{I,n,\kappa },a^{\prime }\in \Gamma _{I^{\prime },n,\kappa }:\ Q\supset
I\vee I^{\prime }  \\ I^{\prime }\cap I=\emptyset ,\ell \left( I^{\prime
}\right) \leq \ell \left( I\right) }}\int_{\mathbb{R}^{n}}c_{I,a}\left(
\bigtriangleup _{Q;\kappa _{1}}^{\mu }h_{I;\kappa _{2}}^{\mu ,a}\right)
c_{I^{\prime },a^{\prime }}\left( \bigtriangleup _{Q;\kappa _{1}}^{\mu
}h_{I^{\prime };\kappa _{2}}^{\mu ,a^{\prime }}\right) d\mu \\
&=&\sum_{K\in \mathcal{D}}\sum_{Q\in \mathcal{D}:\ K\subset Q}\ell \left(
Q\right) ^{-2s}\sum_{m=1}^{\infty }\sum_{I\in \mathcal{H}\cap \mathfrak{C}%
^{\left( m\right) }\left( K\right) ;\ a\in \Gamma _{I,n,\kappa }}\sum 
_{\substack{ I^{\prime }\in \mathcal{H};\ a^{\prime }\in \Gamma _{I^{\prime
},n,\kappa }:\ \pi ^{\left( m\right) }\left( I^{\prime }\right) \subset K 
\\ \ell \left( I^{\prime }\right) \leq \ell \left( I\right) }} \\
&&\ \ \ \ \ \ \ \ \ \ \ \ \ \ \ \ \ \ \ \ \ \ \ \ \ \ \ \ \ \ \ \ \ \ \ \ \
\ \ \ \ \ \ \ \ \ \ \ \ \ \ \ \ \ \ \int_{\mathbb{R}^{n}}c_{I,a}\left(
\bigtriangleup _{Q;\kappa _{1}}^{\mu }h_{I;\kappa _{2}}^{\mu ,a}\right)
c_{I^{\prime },a^{\prime }}\left( \bigtriangleup _{Q;\kappa _{1}}^{\mu
}h_{I^{\prime };\kappa _{2}}^{\mu ,a^{\prime }}\right) d\mu \\
&=&\sum_{K\in \mathcal{D}}\sum_{m=1}^{\infty }\int_{\mathbb{R}%
^{n}}\sum_{Q\in \mathcal{D}:\ K\subset Q}\ell \left( Q\right)
^{-2s}\bigtriangleup _{Q;\kappa _{1}}^{\mu }\left( \sum_{I\in \mathcal{H}%
\cap \mathfrak{C}^{\left( m\right) }\left( K\right) ;\ a\in \Gamma
_{I,n,\kappa }}c_{I,a}h_{I;\kappa _{2}}^{\mu ,a}\right) \\
&&\ \ \ \ \ \ \ \ \ \ \ \ \ \ \ \ \ \ \ \ \ \ \ \ \ \ \ \ \ \ \ \ \ \ \ \ \
\ \ \ \ \ \ \ \ \ \ \ \ \ \ \ \ \ \ \bigtriangleup _{Q;\kappa _{1}}^{\mu
}\left( \sum_{\substack{ I^{\prime }\in \mathcal{H};\ a^{\prime }\in \Gamma
_{I^{\prime },n,\kappa }:\ \pi ^{\left( m\right) }\left( I^{\prime }\right)
\subset K  \\ \ell \left( I^{\prime }\right) \leq \ell \left( I\right) }}%
c_{I^{\prime },a^{\prime }}h_{I^{\prime };\kappa _{2}}^{\mu ,a^{\prime
}}\right) d\mu .
\end{eqnarray*}%
Now we compute that for each $K\in \mathcal{D}$,%
\begin{equation*}
\left\Vert \sum_{\substack{ I^{\prime }\in \mathcal{H};\ a^{\prime }\in
\Gamma _{I^{\prime },n,\kappa }:\ \pi ^{\left( m\right) }\left( I^{\prime
}\right) \subset K  \\ \ell \left( I^{\prime }\right) \leq \ell \left(
I\right) }}c_{I^{\prime },a^{\prime }}h_{I^{\prime };\kappa _{2}}^{\mu
,a}\right\Vert _{L^{2}\left( \mu \right) }^{2}=\sum_{\substack{ I^{\prime
}\in \mathcal{H};\ a^{\prime }\in \Gamma _{I^{\prime },n,\kappa }:\ \pi
^{\left( m\right) }\left( I^{\prime }\right) \subset K  \\ \ell \left(
I^{\prime }\right) \leq \ell \left( I\right) }}\left\vert c_{I^{\prime
},a^{\prime }}\right\vert ^{2},
\end{equation*}%
and%
\begin{eqnarray*}
&&\left\Vert \sum_{Q\in \mathcal{D}:\ K\subset Q}\ell \left( Q\right)
^{-2s}\bigtriangleup _{Q;\kappa _{1}}^{\mu }\left[ \sum_{I\in \mathcal{H}%
\cap \mathfrak{C}^{\left( m\right) }\left( K\right) ;\ a\in \Gamma
_{I,n,\kappa }}c_{I,a}h_{I;\kappa _{2}}^{\mu ,a}\right] \right\Vert
_{L^{2}\left( \mu \right) }^{2} \\
&=&\sum_{Q\in \mathcal{D}:\ K\subset Q}\ell \left( Q\right) ^{-4s}\left\Vert
\bigtriangleup _{Q;\kappa _{1}}^{\mu }\left[ \sum_{I\in \mathcal{H}\cap 
\mathfrak{C}^{\left( m\right) }\left( K\right) ;\ a\in \Gamma _{I,n,\kappa
}}c_{I,a}h_{I;\kappa _{2}}^{\mu ,a}\right] \right\Vert _{L^{2}\left( \mu
\right) }^{2} \\
&\leq &\sum_{Q\in \mathcal{D}:\ K\subset Q}\ell \left( Q\right)
^{-4s}\left\Vert \sum_{I\in \mathcal{H}\cap \mathfrak{C}^{\left( m\right)
}\left( K\right) ;\ a\in \Gamma _{I,n,\kappa }}c_{I,a}h_{I;\kappa _{2}}^{\mu
,a}\right\Vert _{L^{2}\left( \mu \right) }^{2} \\
&=&C\ell \left( K\right) ^{-4s}\sum_{I\in \mathcal{H}\cap \mathfrak{C}%
^{\left( m\right) }\left( K\right) ;\ a\in \Gamma _{I,n,\kappa }}\left\vert
c_{I,a}\right\vert ^{2}=C2^{-4sm}\sum_{I\in \mathcal{H}\cap \mathfrak{C}%
^{\left( m\right) }\left( K\right) ;\ a\in \Gamma _{I,n,\kappa }}\ell \left(
I\right) ^{-4s}\left\vert c_{I,a}\right\vert ^{2},
\end{eqnarray*}%
and hence%
\begin{eqnarray*}
B_{\func{case}2} &=&\sum_{K\in \mathcal{D}}\sum_{m=1}^{\infty }\int_{\mathbb{%
R}^{n}}\sum_{Q\in \mathcal{D}:\ K\subset Q}\ell \left( Q\right)
^{-2s}\bigtriangleup _{Q;\kappa _{1}}^{\mu }\left( \sum_{I\in \mathcal{H}%
\cap \mathfrak{C}^{\left( m\right) }\left( K\right) ;\ a\in \Gamma
_{I,n,\kappa }}c_{I,a}h_{I;\kappa _{2}}^{\mu ,a}\right) \\
&&\ \ \ \ \ \ \ \ \ \ \ \ \ \ \ \ \ \ \ \ \ \ \ \ \ \ \ \ \ \ \ \ \ \ \ \ \
\ \ \ \times \bigtriangleup _{Q;\kappa _{1}}^{\mu }\left( \sum_{\substack{ %
I^{\prime }\in \mathcal{H};\ a^{\prime }\in \Gamma _{I^{\prime },n,\kappa
}:\ \pi ^{\left( m\right) }\left( I^{\prime }\right) \subset K  \\ \ell
\left( I^{\prime }\right) \leq \ell \left( I\right) }}c_{I^{\prime
},a^{\prime }}h_{I^{\prime };\kappa _{2}}^{\mu ,a}\right) d\mu \\
&\leq &\sum_{K\in \mathcal{D}}\sum_{m=1}^{\infty }C2^{-sm}\sqrt{\sum_{I\in 
\mathcal{H}\cap \mathfrak{C}^{\left( m\right) }\left( K\right) ;\ a\in
\Gamma }\ell \left( I\right) ^{-2s}\left\vert c_{I,a}\right\vert ^{2}}\sqrt{%
\ell \left( K\right) ^{-2s}\sum_{\substack{ I^{\prime }\in \mathcal{H};\
a^{\prime }\in \Gamma _{I^{\prime },n,\kappa }:\ \pi ^{\left( m\right)
}\left( I^{\prime }\right) \subset K  \\ \ell \left( I^{\prime }\right) \leq
\ell \left( I\right) }}\left\vert c_{I^{\prime },a^{\prime }}\right\vert ^{2}%
} \\
&\leq &\sum_{m=1}^{\infty }C2^{-sm}\sqrt{\sum_{K\in \mathcal{D}}\sum_{I\in 
\mathcal{H}\cap \mathfrak{C}^{\left( m\right) }\left( K\right) ;\ a\in
\Gamma _{I,n,\kappa }}\ell \left( I\right) ^{-2s}\left\vert
c_{I,a}\right\vert ^{2}}\sqrt{\sum_{K\in \mathcal{D}}\ell \left( K\right)
^{-2s}\sum_{\substack{ I^{\prime }\in \mathcal{H};\ a^{\prime }\in \Gamma
_{I^{\prime },n,\kappa }:\ \pi ^{\left( m\right) }\left( I^{\prime }\right)
\subset K  \\ \ell \left( I^{\prime }\right) \leq \ell \left( I\right) }}%
\left\vert c_{I^{\prime },a^{\prime }}\right\vert ^{2}},
\end{eqnarray*}%
and regrouping we obtain%
\begin{eqnarray*}
B_{\func{case}2} &\leq &\sum_{m=1}^{\infty }C2^{-sm}\sqrt{\sum_{I\in 
\mathcal{H};\ a\in \Gamma _{I,n,\kappa }}\ell \left( I\right)
^{-2s}\left\vert c_{I,a}\right\vert ^{2}}\sqrt{\sum_{I^{\prime }\in \mathcal{%
H};\ a^{\prime }\in \Gamma _{I^{\prime },n,\kappa }}\left\vert c_{I^{\prime
},a^{\prime }}\right\vert ^{2}\sum_{K\in \mathcal{D}:\ \pi ^{\left( m\right)
}\left( I^{\prime }\right) \subset K}\ell \left( K\right) ^{-2s}} \\
&\leq &\sum_{m=1}^{\infty }C2^{-2sm}\sqrt{\sum_{I\in \mathcal{H};\ a\in
\Gamma _{I,n,\kappa }}\ell \left( I\right) ^{-2s}\left\vert
c_{I,a}\right\vert ^{2}}\sqrt{\sum_{I^{\prime }\in \mathcal{H};\ a^{\prime
}\in \Gamma _{I^{\prime },n,\kappa }}\ell \left( I^{\prime }\right)
^{-2s}\left\vert c_{I^{\prime },a^{\prime }}\right\vert ^{2}} \\
&\lesssim &\sum_{I\in \mathcal{H};\ a\in \Gamma _{I,n,\kappa }}\ell \left(
I\right) ^{-2s}\left\vert c_{I,a}\right\vert ^{2}\approx \left\Vert
\sum_{I\in \mathcal{H};\ a\in \Gamma _{I,n,\kappa }}c_{I,a}h_{I;\kappa
_{2}}^{\mu }\right\Vert _{W_{\mathcal{D};\kappa _{2}}^{s}\left( \mu \right)
}^{2},
\end{eqnarray*}%
which is the desired estimate for $B_{\func{case}2}$.

Thus from (\ref{now claim''}) and (\ref{now claim}) we obtain $W_{\mathcal{D}%
;\kappa _{1}}^{s}\left( \mu \right) =W_{\mathcal{D};\kappa _{2}}^{s}\left(
\mu \right) $ for all $s>0$, all grids $\mathcal{D}$ and all integers $%
\kappa _{1},\kappa _{2}\in \mathbb{N}$. Duality now establishes these
equalities for $s<0$ as well, and the case $s=0$ is automatic.
\end{proof}

Following Peetre \cite{Pee} and Stein \cite{Ste}, we define the homogeneous 
\emph{difference} Sobolev space $W_{\func{diff};\kappa }^{s}\left( \mu
\right) $ by%
\begin{equation*}
W_{\mathcal{D}_{\func{diff}};\kappa }^{s}\left( \mu \right) \equiv \left\{
f\in L^{2}\left( \mu \right) :\left\Vert f\right\Vert _{W_{\mathcal{D}_{%
\func{diff}};\kappa }^{s}}\left( \mu \right) <\infty \right\} ,\ \ \ \ \
s\in \mathbb{R}\text{ and }\kappa \in \mathbb{N},
\end{equation*}%
where%
\begin{eqnarray*}
\left\Vert f\right\Vert _{W_{\mathcal{D}_{\func{diff}};\kappa }^{s}\left(
\mu \right) }^{2} &\equiv &\sum_{Q\in \mathcal{D}}\int_{Q}\left\vert \frac{%
f\left( x\right) -\mathbb{E}_{Q;\kappa }^{\mu }f\left( x\right) }{\ell
\left( Q\right) ^{s}}\right\vert ^{2}d\mu \left( x\right) , \\
\text{and }\mathbb{E}_{Q;\kappa }^{\mu }f\left( x\right) &\equiv &\left(
E_{Q}^{\mu }f\right) \mathbf{1}_{Q}\left( x\right) =\left( \frac{1}{%
\left\vert Q\right\vert _{\mu }}\int_{Q}fd\mu \right) \mathbf{1}_{Q}\left(
x\right) .
\end{eqnarray*}%
The proof of the next lemma does not explicitly use the doubling property of 
$\mu $ either.

\begin{lemma}
Suppose $\mu $ is a doubling measure on $\mathbb{R}^{n}$ and $\mathcal{D}$
is a dyadic grid on $\mathbb{R}^{n}$. Then for $s>0$ and $\kappa \in \mathbb{%
N}$, we have 
\begin{equation*}
W_{\mathcal{D}_{\func{diff}};\kappa }^{s}\left( \mu \right) =W_{\mathcal{D}%
;\kappa }^{s}\left( \mu \right) ,
\end{equation*}%
with equivalence of norms.
\end{lemma}

\begin{proof}
We expand the function%
\begin{equation*}
f\left( x\right) -\mathbb{E}_{Q;\kappa }^{\mu }f\left( x\right)
=\sum_{I\subset Q}\bigtriangleup _{I;\kappa }^{\mu }f\left( x\right) ,
\end{equation*}%
and so obtain for $s>0$ that%
\begin{eqnarray*}
\left\Vert f\right\Vert _{W_{\mathcal{D}_{\func{diff}};\kappa }^{s}}^{2}
&=&\sum_{Q\in \mathcal{D}}\ell \left( Q\right) ^{-2s}\sum_{I\subset
Q}\left\Vert \bigtriangleup _{I;\kappa }^{\mu }f\right\Vert _{L^{2}\left(
\mu \right) }^{2} \\
&=&\sum_{I\in \mathcal{D}}\ell \left( I\right) ^{-2s}\left\Vert
\bigtriangleup _{I;\kappa }^{\mu }f\right\Vert _{L^{2}\left( \mu \right)
}^{2}\sum_{\substack{ Q\in \mathcal{D}  \\ Q\supset I}}\left( \frac{\ell
\left( Q\right) }{\ell \left( I\right) }\right) ^{-2s} \\
&\approx &\sum_{I\in \mathcal{D}}\ell \left( I\right) ^{-2s}\left\Vert
\bigtriangleup _{I;\kappa }^{\mu }f\right\Vert _{L^{2}\left( \mu \right)
}^{2}=\left\Vert f\right\Vert _{W_{\mathcal{D};\kappa }^{s}\left( \mu
\right) }^{2}\ .
\end{eqnarray*}
\end{proof}

We just showed in Lemma \ref{coin} above that the weighted Alpert Sobolev
spaces $W_{\mathcal{D};\kappa }^{s}\left( \mu \right) $ coincide for $s\in 
\mathbb{R}$ and $\kappa \geq 1$, and so we simply write $W_{\mathcal{D}%
}^{s}\left( \mu \right) $ for these spaces, and for specificity we use the
norm of $W_{\mathcal{D};1}^{s}\left( \mu \right) $. Now we will show that
the spaces $W_{\mathcal{D};\kappa }^{s}\left( \mu \right) $ are independent
of the dyadic grid $\mathcal{D}$ for $\left\vert s\right\vert $ sufficiently
small provided the measure $\mu $ is \emph{doubling}, and extend this to
include the difference spaces $W_{\mathcal{D}_{\func{diff}};\kappa
}^{s}\left( \mu \right) $ as well.

\begin{theorem}
\label{dyad equiv}Suppose $\mu $ is a doubling measure on $\mathbb{R}^{n}$
and $\mathcal{D}$ and $\mathcal{E}$ are dyadic grids on $\mathbb{R}^{n}$.
Then for $\kappa \in \mathbb{N}$, and $\left\vert s\right\vert $
sufficiently small, we have%
\begin{equation*}
W_{\mathcal{D};\kappa }^{s}\left( \mu \right) =W_{\mathcal{E};\kappa
}^{s}\left( \mu \right) .
\end{equation*}
\end{theorem}

\begin{proof}
Note that for $J\in \mathcal{E}$ and $0<\delta =2^{-m}\leq 1$, and $F$ is
any finite linear combination of Alpert wavelets,%
\begin{eqnarray*}
&&\sum_{I\in \mathcal{D}:\ \ell \left( I\right) =\delta \ell \left( J\right)
}\int_{\mathbb{R}^{n}}\left\vert \bigtriangleup _{I;\kappa }^{\mu
}\bigtriangleup _{J;\kappa }^{\mu }F\right\vert ^{2}d\mu =\sum_{I\in 
\mathcal{D}:\ \ell \left( I\right) =\delta \ell \left( J\right) }\int_{%
\mathbb{R}^{n}}\left\vert \bigtriangleup _{I;\kappa }^{\mu }\left( \mathbf{1}%
_{H_{\delta }\left( J\right) }\bigtriangleup _{J;\kappa }^{\mu }F\right)
\right\vert ^{2}d\mu \leq \int_{H_{\delta }\left( J\right) }\left\vert
\bigtriangleup _{J;\kappa }^{\mu }F\right\vert ^{2}d\mu \\
&\leq &\int_{H_{\delta }\left( J\right) }\left\Vert \bigtriangleup
_{J;\kappa }^{\mu }F\right\Vert _{\infty }^{2}d\mu =\left\Vert
\bigtriangleup _{J;\kappa }^{\mu }F\right\Vert _{\infty }^{2}\left\vert
H_{\delta }\left( J\right) \right\vert _{\mu }\leq \left\Vert \bigtriangleup
_{J;\kappa }^{\mu }F\right\Vert _{\infty }^{2}C\delta ^{\varepsilon
}\left\vert J\right\vert _{\mu }\leq C\delta ^{\varepsilon }\int_{\mathbb{R}%
^{n}}\left\vert \bigtriangleup _{J;\kappa }^{\mu }F\right\vert ^{2}d\mu ,
\end{eqnarray*}%
which gives for $\eta >0$%
\begin{eqnarray*}
&&\sum_{I\in \mathcal{D}:\ \ell \left( I\right) \leq \ell \left( J\right)
}\ell \left( I\right) ^{-\left( 2s+\eta \right) }\int_{\mathbb{R}%
^{n}}\left\vert \bigtriangleup _{I;\kappa }^{\mu }\bigtriangleup _{J;\kappa
}^{\mu }F\right\vert ^{2}d\mu =\sum_{m=1}^{\infty }\sum_{I\in \mathcal{D}:\
\ell \left( I\right) =2^{-m}\ell \left( J\right) }\ell \left( I\right)
^{-\left( 2s+\eta \right) }\int_{\mathbb{R}^{n}}\left\vert \bigtriangleup
_{I;\kappa }^{\mu }\bigtriangleup _{J;\kappa }^{\mu }F\right\vert ^{2}d\mu \\
&\leq &\sum_{m=1}^{\infty }C2^{m\left( 2s+\eta -\varepsilon \right) }\ell
\left( J\right) ^{-\left( 2s+\eta \right) }\int_{\mathbb{R}^{n}}\left\vert
\bigtriangleup _{J;\kappa }^{\mu }F\right\vert ^{2}d\mu \leq C_{\varepsilon
,\eta ,s}\ell \left( J\right) ^{-\left( 2s+\eta \right) }\int_{\mathbb{R}%
^{n}}\left\vert \bigtriangleup _{J;\kappa }^{\mu }F\right\vert ^{2}d\mu ,
\end{eqnarray*}%
provided $s<\frac{\varepsilon -\eta }{2}$.

On the other hand, for $J\in \mathcal{E}$ and $\delta =2^{m}>1$, there are
at most $2^{n}$ cubes $I$ such that $\ell \left( I\right) =\delta \ell
\left( J\right) $ and $I\cap J\neq \emptyset $, and then following the line
of reasoning in (\ref{ind}) we have,

\begin{eqnarray*}
&&\sum_{I\in \mathcal{D}:\ \ell \left( I\right) >\ell \left( J\right) }\ell
\left( I\right) ^{-\left( 2s+\eta \right) }\int_{\mathbb{R}^{n}}\left\vert
\bigtriangleup _{I;\kappa }^{\mu }\bigtriangleup _{J;\kappa }^{\mu
}F\right\vert ^{2}d\mu =\sum_{I\in \mathcal{D}:\ \ell \left( I\right) >\ell
\left( J\right) \text{ and }I\cap J\neq \emptyset }\ell \left( I\right)
^{-\left( 2s+\eta \right) }\int_{\mathbb{R}^{n}}\left\vert \bigtriangleup
_{I;\kappa }^{\mu }\bigtriangleup _{J;\kappa }^{\mu }F\right\vert ^{2}d\mu \\
&\approx &\sum_{I\in \mathcal{D}:\ \ell \left( I\right) >\ell \left(
J\right) \text{ and }I\cap J\neq \emptyset }\ell \left( I\right) ^{-\left(
2s+\eta \right) }\left\vert I\right\vert _{\mu }\left\Vert \bigtriangleup
_{I;\kappa }^{\mu }\bigtriangleup _{J;\kappa }^{\mu }F\right\Vert _{\infty
}^{2}\approx \left\vert J\right\vert _{\mu }\sum_{I\in \mathcal{D}:\ \ell
\left( I\right) >\ell \left( J\right) \text{ and }I\cap J\neq \emptyset
}\ell \left( I\right) ^{-\left( 2s+\eta \right) }\frac{\left\vert
J\right\vert _{\mu }}{\left\vert I\right\vert _{\mu }}\left\Vert
\bigtriangleup _{J;\kappa }^{\mu }F\right\Vert _{\infty }^{2} \\
&\lesssim &\left\vert J\right\vert _{\mu }\left\Vert \bigtriangleup
_{J;\kappa }^{\mu }F\right\Vert _{\infty }^{2}C_{\eta }\ell \left( J\right)
^{\eta }\sum_{m=1}^{\infty }2^{-m\left( 2s+\eta \right) }\ell \left(
J\right) ^{-\left( 2s+\eta \right) }\frac{\left\vert J\right\vert _{\mu }}{%
\left\vert \pi ^{n}J\right\vert _{\mu }} \\
&\leq &\ell \left( J\right) ^{-2s}\int_{\mathbb{R}^{n}}\left\vert
\bigtriangleup _{J;\kappa }^{\mu }F\right\vert ^{2}d\mu \sum_{m=1}^{\infty
}2^{-m\left( 2s+\eta \right) }c2^{-m\theta _{\mu }^{\func{rev}}}\leq C\ell
\left( J\right) ^{-2s}\int_{\mathbb{R}^{n}}\left\vert \bigtriangleup
_{J;\kappa }^{\mu }F\right\vert ^{2},
\end{eqnarray*}%
provided $s>\frac{\eta -\theta _{\mu }^{\func{rev}}}{2}$ where $\theta _{\mu
}^{\func{rev}}>0$ is the reverse doubling exponent of $\mu $, i.e. 
\begin{equation*}
\left\vert \pi ^{n}J\right\vert _{\mu }\geq c2^{n\theta _{\mu }^{\func{rev}%
}}\left\vert J\right\vert _{\mu }.
\end{equation*}

Now we compute

\begin{eqnarray*}
\left\Vert f\right\Vert _{W_{\mathcal{D};\kappa }^{s}\left( \mu \right)
}^{2} &=&\sum_{I\in \mathcal{D}}\ell \left( I\right) ^{-2s}\left\Vert
\bigtriangleup _{I;\kappa }^{\mu }f\right\Vert _{L^{2}\left( \mu \right)
}^{2}=\sum_{I\in \mathcal{D}}\ell \left( I\right) ^{-2s}\left\Vert
\bigtriangleup _{I;\kappa }^{\mu }\left( \sum_{J\in \mathcal{E}%
}\bigtriangleup _{J;\kappa }^{\mu }\right) f\right\Vert _{L^{2}\left( \mu
\right) }^{2} \\
&\leq &2\sum_{I\in \mathcal{D}}\ell \left( I\right) ^{-2s}\left\Vert
\sum_{J\in \mathcal{E}:\ \ell \left( I\right) \leq \ell \left( J\right)
}\bigtriangleup _{I;\kappa }^{\mu }\bigtriangleup _{J;\kappa }^{\mu
}f\right\Vert _{L^{2}\left( \mu \right) }^{2} \\
&&+2\sum_{I\in \mathcal{D}}\ell \left( I\right) ^{-2s}\left\Vert \sum_{J\in 
\mathcal{E}:\ \ell \left( I\right) >\ell \left( J\right) }\bigtriangleup
_{I;\kappa }^{\mu }\bigtriangleup _{J;\kappa }^{\mu }f\right\Vert
_{L^{2}\left( \mu \right) }^{2}.
\end{eqnarray*}%
We bound the first sum by%
\begin{eqnarray*}
&&C_{\eta }\sum_{I\in \mathcal{D}}\ell \left( I\right) ^{-2s}\sum_{J\in 
\mathcal{E}:\ \ell \left( I\right) \leq \ell \left( J\right) }\left( \frac{%
\ell \left( J\right) }{\ell \left( I\right) }\right) ^{\eta }\left\Vert
\bigtriangleup _{I;\kappa }^{\mu }\bigtriangleup _{J;\kappa }^{\mu
}f\right\Vert _{L^{2}\left( \mu \right) }^{2} \\
&\leq &C_{\eta }\sum_{J\in \mathcal{E}}\ell \left( J\right) ^{\eta
}\sum_{I\in \mathcal{D}:\ \ell \left( I\right) \leq \ell \left( J\right)
}\ell \left( I\right) ^{-\left( 2s+\eta \right) }\left\Vert \bigtriangleup
_{I;\kappa }^{\mu }\bigtriangleup _{J;\kappa }^{\mu }f\right\Vert
_{L^{2}\left( \mu \right) }^{2} \\
&\leq &C_{\eta }\sum_{J\in \mathcal{E}}\ell \left( J\right) ^{\eta
}C_{\varepsilon ,\eta ,s}\ell \left( J\right) ^{-\left( 2s+\eta \right)
}\int_{\mathbb{R}^{n}}\left\vert \bigtriangleup _{J;\kappa }^{\mu
}f\right\vert ^{2}d\mu \\
&\leq &C_{\varepsilon ,\eta ,s}\sum_{J\in \mathcal{E}}\ell \left( J\right)
^{-2s}\int_{\mathbb{R}^{n}}\left\vert \bigtriangleup _{J;\kappa }^{\mu
}f\right\vert ^{2}d\mu =C_{\varepsilon ,\eta ,s}\left\Vert f\right\Vert _{W_{%
\mathcal{E};\kappa }^{s}\left( \mu \right) }^{2},
\end{eqnarray*}%
and the second sum by%
\begin{equation*}
\sum_{I\in \mathcal{D}}\ell \left( I\right) ^{-2s}\sum_{J\in \mathcal{E}:\
\ell \left( I\right) >\ell \left( J\right) }\left( \frac{\ell \left(
I\right) }{\ell \left( J\right) }\right) ^{\eta }\left\Vert \bigtriangleup
_{I;\kappa }^{\mu }\bigtriangleup _{J;\kappa }^{\mu }f\right\Vert
_{L^{2}\left( \mu \right) }^{2}\leq C_{\eta }\left\Vert f\right\Vert _{W_{%
\mathcal{E};\kappa }^{s}\left( \mu \right) }^{2}
\end{equation*}%
provided $s>\frac{\eta -\theta _{\mu }^{\func{rev}}}{2}$. Altogether we
obtain%
\begin{equation*}
\left\Vert f\right\Vert _{W_{\mathcal{D};\kappa }^{s}\left( \mu \right)
}\leq C_{\eta }\left\Vert f\right\Vert _{W_{\mathcal{E};\kappa }^{s}\left(
\mu \right) }\ ,
\end{equation*}%
provided%
\begin{equation*}
\frac{\eta -\theta _{\mu }^{\func{rev}}}{2}<s<\frac{\varepsilon -\eta }{2},
\end{equation*}%
and interchanging the roles of the dyadic grids $\mathcal{D}$ and $\mathcal{E%
}$ completes the proof.
\end{proof}

Finally, we will explicitly compute the norm $\left\Vert f\right\Vert _{W_{%
\mathcal{D}_{\func{diff}};1}^{s}\left( \mu \right) }$ by starting with%
\begin{eqnarray*}
&&\frac{1}{\left\vert Q\right\vert _{\mu }}\int_{Q}\int_{Q}\left( f\left(
x\right) -f\left( y\right) \right) ^{2}d\mu \left( x\right) d\mu \left(
y\right) \\
&=&\frac{1}{\left\vert Q\right\vert _{\mu }}\int_{Q}\int_{Q}\left\{ f\left(
x\right) ^{2}-2f\left( x\right) f\left( y\right) +f\left( y\right)
^{2}\right\} d\mu \left( x\right) d\mu \left( y\right) \\
&=&2\int_{Q}f^{2}d\mu -2\left\vert Q\right\vert _{\mu }\left( E_{Q}^{\mu
}f\right) ^{2}=2\int_{Q}\left\{ f\left( x\right) ^{2}-\left( E_{Q}^{\mu
}f\right) ^{2}\right\} d\mu ,
\end{eqnarray*}%
to obtain the representation%
\begin{eqnarray*}
\left\Vert f\right\Vert _{W_{\mathcal{D}_{\func{diff}};1}^{s}\left( \mu
\right) }^{2} &=&\sum_{Q\in \mathcal{D}}\int_{Q}\left\vert \frac{f\left(
x\right) -E_{Q}^{\mu }f}{\ell \left( Q\right) ^{s}}\right\vert ^{2}d\mu
\left( x\right) \\
&=&\sum_{Q\in \mathcal{D}}\ell \left( Q\right) ^{-2s}\int_{Q}\left\{ f\left(
x\right) ^{2}-2\left( E_{Q}^{\mu }f\right) f\left( x\right) +\left(
E_{Q}^{\mu }f\right) ^{2}\right\} d\mu \left( x\right) \\
&=&\sum_{Q\in \mathcal{D}}\ell \left( Q\right) ^{-2s}\int_{Q}\left\{ f\left(
x\right) ^{2}-\left( E_{Q}^{\mu }f\right) ^{2}\right\} d\mu \left( x\right)
\\
&=&\frac{1}{2}\sum_{Q\in \mathcal{D}}\ell \left( Q\right) ^{-2s}\frac{1}{%
\left\vert Q\right\vert _{\mu }}\int_{Q}\int_{Q}\left( f\left( x\right)
-f\left( y\right) \right) ^{2}d\mu \left( x\right) d\mu \left( y\right) .
\end{eqnarray*}

We will next show that this last expression is comparable to the expression%
\begin{equation*}
\left\Vert f\right\Vert _{W^{s}\left( \mu \right) }^{2}\equiv \int_{\mathbb{R%
}^{n}}\int_{\mathbb{R}^{n}}\left( \frac{f\left( x\right) -f\left( y\right) }{%
\left\vert x-y\right\vert ^{s}}\right) ^{2}\frac{d\mu \left( x\right) d\mu
\left( y\right) }{\left\vert B\left( \frac{x+y}{2},\frac{\left\vert
x-y\right\vert }{2}\right) \right\vert _{\mu }}.
\end{equation*}

\begin{theorem}
\label{final equiv}Suppose that $\mu $ is doubling on $\mathbb{R}^{n}$. For $%
s>0$ sufficiently small, we have%
\begin{equation*}
\left\Vert f\right\Vert _{W_{\mathcal{D}_{\func{diff}};1}^{s}\left( \mu
\right) }^{2}\approx \left\Vert f\right\Vert _{W^{s}\left( \mu \right) }^{2}.
\end{equation*}
\end{theorem}

\begin{proof}
From the formula above we have%
\begin{eqnarray*}
&&\left\Vert f\right\Vert _{W_{\mathcal{D}_{\func{diff}};1}^{s}\left( \mu
\right) }^{2}=\frac{1}{2}\int_{\mathbb{R}^{n}}\int_{\mathbb{R}^{n}}\left\{
\sum_{Q\in \mathcal{D}}\frac{\mathbf{1}_{Q\times Q}\left( x,y\right) }{\ell
\left( Q\right) ^{2s}\left\vert Q\right\vert _{\mu }}\right\} \left( f\left(
x\right) -f\left( y\right) \right) ^{2}d\mu \left( x\right) d\mu \left(
y\right) \\
&\leq &\frac{1}{2}\int_{\mathbb{R}^{n}}\int_{\mathbb{R}^{n}}\frac{C}{%
\left\vert x-y\right\vert ^{2s}\left\vert B\left( \frac{x+y}{2},\frac{%
\left\vert x-y\right\vert }{2}\right) \right\vert _{\mu }}\left( f\left(
x\right) -f\left( y\right) \right) ^{2}d\mu \left( x\right) d\mu \left(
y\right) =\frac{C}{2}\left\Vert f\right\Vert _{W^{s}\left( \mu \right) }^{2},
\end{eqnarray*}%
since $\left\vert Q\right\vert _{\mu }\gtrsim \left\vert B\left( \frac{x+y}{2%
},\frac{\left\vert x-y\right\vert }{2}\right) \right\vert _{\mu }$ whenever $%
\left( x,y\right) \in Q\times Q$. Conversely we use the one third trick for
dyadic grids. Namely that there is a finite collection of dyadic grids $%
\left\{ \mathcal{D}_{m}\right\} _{m=1}^{3^{n}}$ so that for every $\left(
x,y\right) \in \mathbb{R}^{n}\times \mathbb{R}^{n}$, there is some $m$ and
some $Q\in \mathcal{D}_{m}$ such that 
\begin{equation*}
B\left( x,C\left\vert x-y\right\vert \right) ,B\left( y,C\left\vert
x-y\right\vert \right) \subset Q\text{ and }\ell \left( Q\right) \leq
C\left\vert x-y\right\vert ,
\end{equation*}%
where $C$ is a large constant that will be fixed below. In particular this
gives 
\begin{equation*}
\left\vert Q\right\vert _{\mu }\approx \left\vert B\left( x,c\left\vert
x-y\right\vert \right) \right\vert _{\mu }\approx \left\vert B\left(
y,c\left\vert x-y\right\vert \right) \right\vert _{\mu },\ \ \ \ \ \text{for 
}\frac{1}{100}<c<\frac{1}{2}.
\end{equation*}%
Then we cover the product space $\mathbb{R}^{n}\times \mathbb{R}^{n}$ with a
collection of product balls 
\begin{equation*}
\left\{ B\left( x_{k},c\left\vert x_{k}-y_{k}\right\vert \right) \times
B\left( y_{k},c\left\vert x_{k}-y_{k}\right\vert \right) \right\}
_{k=1}^{\infty }
\end{equation*}%
where $E\equiv \left\{ \left( x_{k},y_{k}\right) \right\} _{k=1}^{\infty }$
is a discrete subset of $\mathbb{R}^{n}\times \mathbb{R}^{n}$, and provided $%
c$ is chosen sufficiently small, this collection of product balls has
bounded overlap. Now denote by $Q_{k}$ the cube chosen above by the point $%
\left( x_{k},y_{k}\right) \in \mathbb{R}^{n}\times \mathbb{R}^{n}$. Then we
have%
\begin{eqnarray*}
&&\left\vert Q_{k}\right\vert _{\mu }\approx \left\vert B\left(
x,c\left\vert x-y\right\vert \right) \right\vert _{\mu }\approx \left\vert
B\left( y,c\left\vert x-y\right\vert \right) \right\vert _{\mu } \\
&&\text{for all }\left( x,y\right) \text{ in the product ball }B\left(
x_{k},c\left\vert x_{k}-y_{k}\right\vert \right) \times B\left(
y_{k},c\left\vert x_{k}-y_{k}\right\vert \right) ,
\end{eqnarray*}%
and so%
\begin{eqnarray*}
\left\Vert f\right\Vert _{W^{s}\left( \mu \right) }^{2} &\leq
&\sum_{k=1}^{\infty }\int_{B\left( x_{k},c\left\vert x_{k}-y_{k}\right\vert
\right) }\int_{B\left( y_{k},c\left\vert x_{k}-y_{k}\right\vert \right)
}\left( \frac{f\left( x\right) -f\left( y\right) }{\left\vert x-y\right\vert
^{s}}\right) ^{2}\frac{d\mu \left( x\right) d\mu \left( y\right) }{%
\left\vert B\left( \frac{x+y}{2},\frac{\left\vert x-y\right\vert }{2}\right)
\right\vert _{\mu }} \\
&\lesssim &\sum_{k=1}^{\infty }\int_{Q_{k}}\int_{Q_{k}}\left( \frac{f\left(
x\right) -f\left( y\right) }{\ell \left( Q_{k}\right) ^{s}}\right) ^{2}\frac{%
d\mu \left( x\right) d\mu \left( y\right) }{\left\vert Q_{k}\right\vert
_{\mu }} \\
&\leq &\sum_{m=1}^{3^{n}}\sum_{Q\in \mathcal{D}_{m}}\ell \left( Q\right)
^{-2s}\frac{1}{\left\vert Q\right\vert _{\mu }}\int_{Q}\int_{Q}\left(
f\left( x\right) -f\left( y\right) \right) ^{2}d\mu \left( x\right) d\mu
\left( y\right) \\
&=&2\sum_{m=1}^{3^{n}}\left\Vert f\right\Vert _{W_{\mathcal{D}_{m,\func{diff}%
};1}^{s}\left( \mu \right) }^{2}\leq C\left\Vert f\right\Vert _{W_{\mathcal{D%
}_{\func{diff}};1}^{s}\left( \mu \right) }^{2}
\end{eqnarray*}%
since $\left\Vert f\right\Vert _{W_{\mathcal{D}_{m,\func{diff}};1}^{s}\left(
\mu \right) }^{2}$ is independent of the dyadic grid $\mathcal{D}_{m}$ by
Theorem \ref{dyad equiv}.
\end{proof}

In particular, we have thus obtained one of the main results of this
subsection.

\begin{theorem}
\label{W spaces equal}For all grids $\mathcal{D}$ on $\mathbb{R}^{n}\,$, all
positive integers $\kappa $, and all sufficiently small $s>0$ depending\
only on the doubling constant of $\mu $, we have 
\begin{equation*}
W^{s}\left( \mu \right) =W_{\mathcal{D};\kappa }^{s}\left( \mu \right) =W_{%
\mathcal{D}_{\func{diff}};\kappa }^{s}\left( \mu \right) ,
\end{equation*}%
with equivalence of norms.
\end{theorem}

As a consequence of this theorem, there is essentially just one notion of a
weighted Sobolev space for a doubling measure $\mu $ provided $\left\vert
s\right\vert $ is sufficiently small, namely $W^{s}\left( \mu \right) $ when 
$s>0$, and any of the dyadic spaces $W_{\mathcal{D};\kappa }^{s}\left( \mu
\right) $ when $s<0$. For specificity we will use the norm of $W_{\mathcal{D}%
_{0};1}^{s}\left( \mu \right) $ on these spaces, where $\mathcal{D}_{0}$ is
the standard dyadic grid on $\mathbb{R}^{n}$.

\begin{definition}
Define $W_{\limfunc{dyad}}^{s}\left( \mu \right) =W_{\mathcal{D};\kappa
}^{s}\left( \mu \right) =W_{\mathcal{D}_{\func{diff}};\kappa }^{s}\left( \mu
\right) $ for $\left\vert s\right\vert $ sufficiently small, and norm $W_{%
\limfunc{dyad}}^{s}\left( \mu \right) $ with the norm of $W_{\mathcal{D}%
;1}^{s}\left( \mu \right) $.
\end{definition}

Note that $W_{\limfunc{dyad}}^{s}\left( \mu \right) =W^{s}\left( \mu \right) 
$ for $s>0$ sufficiently small.

\begin{remark}
The Sobolev space $W^{s}\left( \mu \right) $ used here is different from the
Sobolev space introduced on a space of homogeneous type in \cite{HaSa},
since one can show that the norm squared used in \cite{HaSa} is comparable\
to%
\begin{equation*}
\int_{\mathbb{R}^{n}}\int_{\mathbb{R}^{n}}\left( \frac{f\left( x\right)
-f\left( y\right) }{\left\vert B\left( \frac{x+y}{2},\frac{\left\vert
x-y\right\vert }{2}\right) \right\vert _{\mu }^{\alpha }}\right) ^{2}\frac{%
d\mu \left( x\right) d\mu \left( y\right) }{\left\vert B\left( \frac{x+y}{2},%
\frac{\left\vert x-y\right\vert }{2}\right) \right\vert _{\mu }}.
\end{equation*}%
It seems likely that our proof extends to the analogous $T1$ theorem for
these weighted Sobolev spaces using the doubling measure inequalities $\frac{%
\left\vert Q\right\vert _{\mu }}{\left\vert 2^{m}Q\right\vert _{\mu }}%
\lesssim \left( \frac{\ell \left( Q\right) }{\ell \left( 2^{m}Q\right) }%
\right) ^{\theta _{\mu }^{\func{rev}}}$ and $\frac{\left\vert
2^{m}Q\right\vert _{\mu }}{\left\vert Q\right\vert _{\mu }}\lesssim \left( 
\frac{\ell \left( 2^{m}Q\right) }{\ell \left( Q\right) }\right) ^{\theta
_{\mu }^{\limfunc{doub}}}$.
\end{remark}

\begin{problem}
Does a $T1$ theorem hold in the context of weighted Sobolev spaces with
doubling measures and norm squared given by%
\begin{equation*}
\int_{\mathbb{R}^{n}}\int_{\mathbb{R}^{n}}\left( \frac{f\left( x\right)
-f\left( y\right) }{\varphi \left( \left\vert x-y\right\vert \right) }%
\right) ^{2}\frac{d\mu \left( x\right) d\mu \left( y\right) }{\left\vert
B\left( \frac{x+y}{2},\frac{\left\vert x-y\right\vert }{2}\right)
\right\vert _{\mu }},
\end{equation*}%
where $\varphi :\left( 0,\infty \right) \rightarrow \left( 0,\infty \right) $
satisfies 
\begin{equation*}
\left( \frac{s}{t}\right) ^{\theta _{1}}\lesssim \frac{\varphi \left(
s\right) }{\varphi \left( t\right) }\lesssim \left( \frac{s}{t}\right)
^{\theta _{2}}\text{ for }0<s\leq t<\infty \ ?
\end{equation*}
\end{problem}

\begin{remark}
The inner product for the Hilbert space $W_{\func{dyad}}^{s}\left( \mu
\right) $ is given by%
\begin{equation*}
\left\langle f,g\right\rangle _{W_{\func{dyad}}^{s}\left( \mu \right)
}\equiv \sum_{Q\in \mathcal{D}}\ell \left( Q\right) ^{-2s}\left\langle
\bigtriangleup _{Q;\kappa }^{\mu }f,\bigtriangleup _{Q;\kappa }^{\mu
}g\right\rangle _{L^{2}\left( \mu \right) }
\end{equation*}%
where the inner product for $L^{2}\left( \mu \right) $ is given by%
\begin{equation*}
\left\langle f,g\right\rangle _{L^{2}\left( \mu \right) }\equiv \int_{%
\mathbb{R}^{n}}f\left( x\right) g\left( x\right) d\mu \left( x\right) .
\end{equation*}%
In Lemma \ref{complete set}\ we showed that $\left\{ \ell \left( Q\right)
^{s}\bigtriangleup _{Q;\kappa }^{\mu }\right\} _{Q\in \mathcal{D}}$ is a
complete set of orthogonal projections on $W_{\func{dyad}}^{s}\left( \mu
\right) $, nevertheless we will not use the Hilbert space duality that
identifies the dual of a Hilbert space with the conjugate of itself under
the inner product $\left\langle f,g\right\rangle _{W_{\func{dyad}}^{s}\left(
\mu \right) }$, but rather the $L^{2}\left( \mu \right) $ inner product
which identifies the dual of $W_{\func{dyad}}^{s}\left( \mu \right) $ with $%
W_{\func{dyad}}^{-s}\left( \mu \right) $. As mentioned in the introduction,
the reason for this is that the weighted Alpert projections $\left\{
\bigtriangleup _{Q;\kappa }^{\mu }\right\} _{Q\in \mathcal{D}}$ satisfy
telescoping identities, while the orthogonal projections $\left\{ \ell
\left( Q\right) ^{s}\bigtriangleup _{Q;\kappa }^{\mu }\right\} _{Q\in 
\mathcal{D}}$ do not.
\end{remark}

\subsection{Haar, Alpert and indicator functions}

The Alpert projections $\left\{ \bigtriangleup _{I;\kappa }^{\mu }\right\}
_{I\in \mathcal{D}}$ form a complete family of orthogonal projections on $%
L^{2}\left( \mu \right) $, where%
\begin{equation*}
\bigtriangleup _{I;\kappa }^{\mu }f\equiv \mathbb{E}_{I;\kappa }^{\mu
}f-\sum_{I^{\prime }\in \mathfrak{C}_{\mathcal{D}}\left( I\right) }\mathbb{E}%
_{I^{\prime };\kappa }^{\mu }f=\sum_{a\in \Gamma }\left\langle f,h_{I;\kappa
}^{\mu ,a}\right\rangle _{L^{2}\left( \mu \right) }h_{I;\kappa }^{\mu ,a}\ .
\end{equation*}%
Thus we have%
\begin{eqnarray*}
&&\sum_{a\in \Gamma }\left\vert \left\langle f,h_{I;\kappa }^{\mu
,a}\right\rangle _{L^{2}\left( \mu \right) }\right\vert ^{2}\left\Vert
h_{I;\kappa }^{\mu ,a}\right\Vert _{W_{\func{dyad}}^{s}\left( \mu \right)
}^{2}=\sum_{a\in \Gamma }\left\Vert \left\langle f,h_{I;\kappa }^{\mu
,a}\right\rangle _{L^{2}\left( \mu \right) }h_{I;\kappa }^{\mu
,a}\right\Vert _{W_{\func{dyad}}^{s}\left( \mu \right) }^{2}=\left\Vert
\bigtriangleup _{I;\kappa }^{\mu }f\right\Vert _{W_{\func{dyad}}^{s}\left(
\mu \right) }^{2} \\
&=&\sum_{Q\in \mathcal{D}}\ell \left( Q\right) ^{-2s}\left\Vert
\bigtriangleup _{Q;\kappa }^{\mu }\bigtriangleup _{I;\kappa }^{\mu
}f\right\Vert _{L^{2}\left( \mu \right) }^{2}=\ell \left( I\right)
^{-2s}\left\Vert \bigtriangleup _{I;\kappa }^{\mu }f\right\Vert
_{L^{2}\left( \mu \right) }^{2}=\ell \left( I\right) ^{-2s}\sum_{a\in \Gamma
}\left\vert \left\langle f,h_{I;\kappa }^{\mu ,a}\right\rangle _{L^{2}\left(
\mu \right) }\right\vert ^{2}\ ,
\end{eqnarray*}%
for all $f\in W_{\func{dyad}}^{s}\left( \mu \right) $, which implies upon
taking $f=h_{I;\kappa }^{\mu ,a}$, that%
\begin{equation*}
\left\Vert h_{I;\kappa }^{\mu ,a}\right\Vert _{W_{\func{dyad}}^{s}\left( \mu
\right) }^{2}=\ell \left( I\right) ^{-2s},\ \ \ \ \ I\in \mathcal{D}.
\end{equation*}

Now we compute the weighted Alpert Sobolev norms of indicators. By
independence of $\kappa $, we may assume $\kappa =1$. Using (\ref{analogue'}%
) and (\ref{in part}), we then note that%
\begin{equation*}
\left\Vert \bigtriangleup _{Q}^{\mu }\mathbf{1}_{I}\right\Vert _{\infty
}^{2}\approx \frac{\left\Vert \bigtriangleup _{Q}^{\mu }\mathbf{1}%
_{I}\right\Vert _{L^{2}\left( \mu \right) }^{2}}{\left\vert Q\right\vert
_{\mu }}\approx \frac{\frac{\left\vert I\right\vert _{\mu }^{2}}{\left\vert
Q\right\vert _{\mu }}}{\left\vert Q\right\vert _{\mu }}=\left( \frac{%
\left\vert I\right\vert _{\mu }}{\left\vert Q\right\vert _{\mu }}\right)
^{2}\ \ \ \ \ \text{for }I\subset Q^{\prime }\in \mathfrak{C}_{\mathcal{D}%
}\left( Q\right) ,
\end{equation*}%
and hence from (\ref{analogue'}) we obtain 
\begin{eqnarray}
\left\Vert \mathbf{1}_{I}\right\Vert _{W_{\func{dyad}}^{s}\left( \mu \right)
}^{2} &=&\sum_{Q\in \mathcal{D}}\ell \left( Q\right) ^{-2s}\left\Vert
\bigtriangleup _{Q;1}^{\mu }\mathbf{1}_{I}\right\Vert _{L^{2}\left( \mu
\right) }^{2}=\sum_{Q\in \mathcal{D}:\ Q\varsupsetneqq I}\ell \left(
Q\right) ^{-2s}\left\Vert \bigtriangleup _{Q;1}^{\mu }\mathbf{1}%
_{I}\right\Vert _{L^{2}\left( \mu \right) }^{2}  \label{ind} \\
&\approx &\sum_{Q\in \mathcal{D}:\ Q\varsupsetneqq I}\ell \left( Q\right)
^{-2s}\left\vert Q\right\vert _{\mu }\left\Vert \bigtriangleup _{Q;1}^{\mu }%
\mathbf{1}_{I}\right\Vert _{\infty }^{2}\approx \left\vert I\right\vert
_{\mu }\sum_{Q\in \mathcal{D}:\ Q\varsupsetneqq I}\ell \left( Q\right) ^{-2s}%
\frac{\left\vert I\right\vert _{\mu }}{\left\vert Q\right\vert _{\mu }} 
\notag \\
&=&\left\vert I\right\vert _{\mu }\sum_{n=1}^{\infty }2^{-2ns}\ell \left(
I\right) ^{-2s}\frac{\left\vert I\right\vert _{\mu }}{\left\vert \pi
^{n}I\right\vert _{\mu }}=\ell \left( I\right) ^{-2s}\left\vert I\right\vert
_{\mu }\sum_{n=1}^{\infty }2^{-2ns}\frac{\left\vert I\right\vert _{\mu }}{%
\left\vert \pi ^{n}I\right\vert _{\mu }}.  \notag
\end{eqnarray}%
Since $\mu $ is doubling, it also satisfies a \emph{dyadic reverse doubling}
condition with reverse doubling exponent $\theta _{\mu }^{\func{rev}}>0$
depending on the doubling constant, i.e. 
\begin{equation*}
\left\vert \pi ^{n}I\right\vert _{\mu }\geq c2^{n\theta _{\mu }^{\func{rev}%
}}\left\vert I\right\vert _{\mu }.
\end{equation*}%
Then for $s>-\theta _{\mu }^{\func{rev}}$ we have $\sum_{n=1}^{\infty
}2^{-2ns}\frac{\left\vert I\right\vert _{\mu }}{\left\vert \pi
^{n}I\right\vert _{\mu }}\leq \sum_{n=1}^{\infty }2^{-2n\left( s+\theta
_{\mu }^{\func{rev}}\right) }<\infty $, and so%
\begin{equation*}
\left\Vert \mathbf{1}_{I}\right\Vert _{W_{\func{dyad}}^{s}\left( \mu \right)
}^{2}\approx \ell \left( I\right) ^{-2s}\left\vert I\right\vert _{\mu }\ .
\end{equation*}%
Altogether we have proved the following lemma.

\begin{lemma}
Suppose $\mu $ is a locally finite positive Borel measure on $\mathbb{R}^{n}$%
. Then%
\begin{equation*}
\left\Vert h_{I;\kappa }^{\mu ,a}\right\Vert _{W_{\func{dyad}}^{s}\left( \mu
\right) }=\ell \left( I\right) ^{-s},\ \ \ \ \ \text{for all }I\in \mathcal{D%
},a\in \Gamma _{I,n,\kappa },\kappa \geq 1\text{ and }s\in \mathbb{R},
\end{equation*}%
and if $\mu $ is a doubling measure,%
\begin{equation*}
\left\Vert \mathbf{1}_{I}\right\Vert _{W_{\func{dyad}}^{s}\left( \mu \right)
}\approx \ell \left( I\right) ^{-s}\sqrt{\left\vert I\right\vert _{\mu }},\
\ \ \ \ \text{for all }I\in \mathcal{D},\kappa \geq 1\text{ and }s>-\theta
_{\mu }^{\func{rev}}.
\end{equation*}
\end{lemma}

\subsubsection{Sharpness}

Here we construct measures for which $\left\Vert \mathbf{1}_{I}\right\Vert
_{W_{\func{dyad}}^{s}\left( \mu \right) }^{2}=\infty $ for all intervals $I$
and $s<0$, and thus are \emph{not} dyadic reverse doubling. A trivial
example is any finite measure $\mu $, and an infinite example is $d\mu
\left( x\right) =\mathbf{1}_{\left[ e,\infty \right) }\left( x\right) \frac{1%
}{x\ln x}$. In fact we have the following lemma.

\begin{lemma}
If there is $s<0$ such that 
\begin{equation*}
\left\Vert \mathbf{1}_{I}\right\Vert _{W_{\func{dyad}}^{s}\left( \mu \right)
}^{2}\leq C\ell \left( I\right) ^{-2s}\left\vert I\right\vert _{\mu }
\end{equation*}%
for all dyadic intervals $I$, then $\mu $ is a dyadic reverse doubling
measure with exponent $\left\vert s\right\vert $.
\end{lemma}

\begin{proof}
We have 
\begin{equation*}
\sum_{n=0}^{\infty }2^{\left( n+1\right) \left\vert s\right\vert }\frac{1}{%
\left\vert \pi ^{n}I\right\vert _{\mu }}\lesssim \left\vert I\right\vert
_{\mu }^{-2}\ell \left( I\right) ^{2s}\left\Vert \mathbf{1}_{I}\right\Vert
_{W_{\func{dyad}}^{s}\left( \mu \right) }^{2}\leq C\left\vert I\right\vert
_{\mu }^{-1},\text{ \ \ \ \ for all intervals }I\text{,}
\end{equation*}%
which shows that $\left\vert \pi ^{n}I\right\vert _{\mu }\geq \frac{1}{C}%
2^{\left( n+1\right) \left\vert s\right\vert }\left\vert I\right\vert _{\mu
} $ for all intervals $I$, which is the dyadic reverse doubling condition
with $t=\left\vert s\right\vert $.
\end{proof}

\begin{remark}
There is an asymmetry inherent in the homogeneous Sobolev two weight
inequality (\ref{hom Sob inequ}) for general measures. If we wish to use 
\emph{cube testing} to characterize the Sobolev inequality (\ref{hom Sob
inequ}) for some $s>0$ \emph{assuming} the estimate $\left\Vert \mathbf{1}%
_{I}\right\Vert _{W_{\func{dyad}}^{-s}\left( \omega \right) }^{2}\leq C\ell
\left( I\right) ^{2s}\left\vert I\right\vert _{\omega }$, then from the
equivalence with the bilinear inequality (\ref{hom bil inequ}) and the
discussion and lemma above, we see that $\omega $ needs to be restricted, in
fact by reverse dyadic doubling with exponent essentially greater than $%
2\left\vert s\right\vert $, while no restriction needs to be made on $\sigma 
$.
\end{remark}

\subsubsection{Norms of moduli of Alpert wavelets}

Let $\mathbf{h}_{J;\kappa }^{\omega }=\left( h_{J;\kappa }^{\omega
,a}\right) _{a\in \Gamma }$ be the vector of Alpert wavelets associated with
the cube $J$. Note that for $\omega $ doubling and $0\leq s<1$, we have 
\begin{equation*}
\frac{\left\Vert \mathbf{h}_{J;\kappa }^{\omega }\right\Vert _{W_{\func{dyad}%
}^{-s}\left( \omega \right) }^{2}}{\ell \left( J\right) ^{2s}}=1,
\end{equation*}%
and we now show that the same sort of Sobolev estimate holds for the
absolute value $\left\vert \mathbf{h}_{J;\kappa }^{\omega }\right\vert $ of
the vector Alpert wavelet (which remains trivial in the case $\kappa =1$
since $\mathbf{h}_{J;\kappa }^{\omega }$is then constant on dyadic children
of $J$).

\begin{lemma}
\label{mod Alpert}Let $\mu $ be a doubling measure on $\mathbb{R}^{n}$. Then
the modulus of a vector of Alpert wavelets $\mathbf{h}_{J;\kappa }^{\mu
}=\left\{ h_{J;\kappa }^{\mu ,a}\right\} _{a\in \Gamma _{J,n,\kappa }}$
satisfies%
\begin{equation*}
\frac{\left\Vert \left\vert \mathbf{h}_{J;\kappa }^{\mu }\right\vert
\right\Vert _{W_{\func{dyad}}^{-s}\left( \mu \right) }^{2}}{\ell \left(
J\right) ^{2s}}\lesssim C,\ \ \ \ \ \text{for }J\in \mathcal{D}.
\end{equation*}
\end{lemma}

\begin{proof}
We expand%
\begin{equation*}
\left\Vert \left\vert h_{J;\kappa }^{\mu ,a}\right\vert \right\Vert _{W_{%
\func{dyad}}^{-s}\left( \mu \right) }^{2}=\sum_{Q}\ell \left( Q\right)
^{2s}\left\Vert \bigtriangleup _{Q;\kappa }^{\mu }\left\vert h_{J;\kappa
}^{\mu ,a}\right\vert \right\Vert _{L^{2}\left( \mu \right) }^{2}.
\end{equation*}%
For a cube $Q$ contained in a child $J^{\prime }$ of $J$ that is disjoint
from the zero set of the polynomial $\mathbf{1}_{J^{\prime }}h_{J;\kappa
}^{\mu ,a}$ on the child $J^{\prime }$, the absolute values on the Alpert
wavelet can be removed, and we obtain that $\bigtriangleup _{Q;\kappa }^{\mu
}\left\vert h_{J;\kappa }^{\mu ,a}\right\vert =\pm \bigtriangleup _{Q;\kappa
}^{\mu }h_{J;\kappa }^{\mu ,a}$ vanishes. On the other hand, if $Q\subset
J^{\prime }$ intersects the zero set $Z$ of the polynomial $\mathbf{1}%
_{J^{\prime }}h_{J;\kappa }^{\mu ,a}$, then we use $\left\Vert
\bigtriangleup _{Q;\kappa }^{\mu }f\right\Vert _{L^{2}\left( \mu \right)
}\leq C\frac{\left\vert \widehat{f}\left( Q\right) \right\vert }{\sqrt{%
\left\vert Q\right\vert _{\mu }}}$ to obtain the crude estimate, 
\begin{eqnarray*}
\left\Vert \bigtriangleup _{Q;\kappa }^{\mu }\left\vert h_{J;\kappa }^{\mu
,a}\right\vert \right\Vert _{L^{2}\left( \mu \right) }^{2} &=&\left\vert
\left\langle h_{Q;\kappa }^{\mu ,a},\left\vert h_{J;\kappa }^{\mu
,a}\right\vert \right\rangle _{L^{2}\left( \mu \right) }\right\vert
^{2}\int_{\mathbb{R}^{n}}\left\vert h_{J;\kappa }^{\mu ,a}\right\vert
^{2}d\mu =\left\vert \left\langle h_{Q;\kappa }^{\mu ,a},\left\vert
h_{J;\kappa }^{\mu ,a}\right\vert \right\rangle _{L^{2}\left( \mu \right)
}\right\vert ^{2} \\
&\lesssim &\left( \int_{Q}\left\vert h_{J;\kappa }^{\mu ,a}\right\vert
^{2}d\mu \right) \left\Vert h_{J;\kappa }^{\mu ,a}\right\Vert _{\infty
}^{2}\left\vert Q\right\vert _{\mu }\lesssim \frac{\left\vert Q\right\vert
_{\mu }}{\left\vert J\right\vert _{\mu }},
\end{eqnarray*}%
and together with Corollary \ref{zero set doubling}, we obtain%
\begin{eqnarray*}
&&\sum_{Q\subset J^{\prime }:\ Q\cap Z\neq \emptyset }\ell \left( Q\right)
^{2s}\left\Vert \bigtriangleup _{Q;\kappa }^{\mu }\left\vert h_{J;\kappa
}^{\mu ,a}\right\vert \right\Vert _{L^{2}\left( \mu \right) }^{2}\lesssim
\sum_{Q\subset J^{\prime }:\ Q\cap Z\neq \emptyset }\ell \left( Q\right)
^{2s}\frac{\left\vert Q\right\vert _{\mu }}{\left\vert J\right\vert _{\mu }}
\\
&=&\sum_{m=1}^{\infty }\left( 2^{-m}\ell \left( J\right) \right) ^{2s}\sum 
_{\substack{ Q\subset J^{\prime }:\ Q\cap Z\neq \emptyset  \\ \ell \left(
Q\right) =2^{-m}\ell \left( J\right) }}\frac{\left\vert Q\right\vert _{\mu }%
}{\left\vert J\right\vert _{\mu }}\lesssim \left( \sum_{m=1}^{\infty
}2^{-2ms}\sum_{\substack{ Q\subset J^{\prime }:\ Q\cap Z\neq \emptyset  \\ %
\ell \left( Q\right) =2^{-m}\ell \left( J\right) }}\frac{\left\vert
Q\right\vert _{\mu }}{\left\vert J\right\vert _{\mu }}\right) \ \ell \left(
J\right) ^{2s} \\
&\lesssim &\left( \sum_{m=1}^{\infty }2^{-2ms}2^{-\varepsilon m}\right) \
\ell \left( J\right) ^{2s}\ ,
\end{eqnarray*}%
which is the estimate we want when $s>0$ or $-\frac{\varepsilon }{2}<s\leq 0$%
, where $\varepsilon =\varepsilon \left( \mu \right) >0$.

Finally, for big cubes $Q$ containing $J$ there is only the tower above $J$
to consider, and trivial estimates work:%
\begin{eqnarray*}
&&\sum_{Q\in \mathcal{D}:\ Q\supset J}\ell \left( Q\right) ^{2s}\left\Vert
\bigtriangleup _{Q;\kappa }^{\mu }\left\vert h_{J;\kappa }^{\mu
,a}\right\vert \right\Vert _{L^{2}\left( \mu \right) }^{2}=\sum_{Q\in 
\mathcal{D}:\ Q\supset J}\ell \left( Q\right) ^{2s}\left\vert \left\langle
h_{Q;\kappa }^{\mu ,a},\left\vert h_{J;\kappa }^{\mu }\right\vert
\right\rangle _{L^{2}\left( \mu \right) }\right\vert ^{2}\int_{\mathbb{R}%
^{n}}\left\vert h_{J;\kappa }^{\mu ,a}\right\vert ^{2}d\mu \\
&\leq &\sum_{m=1}^{\infty }2^{2ms}\left\vert \left\langle h_{\pi ^{\left(
m\right) }J;\kappa }^{\mu ,a},\left\vert h_{J;\kappa }^{\mu ,a}\right\vert
\right\rangle _{L^{2}\left( \mu \right) }\right\vert ^{2}\ell \left(
J\right) ^{2s}\lesssim \sum_{m=1}^{\infty }2^{2ms}\left( \int_{J}\left\vert
h_{\pi ^{\left( m\right) }J;\kappa }^{\mu ,a}\right\vert ^{2}d\mu \right)
\ell \left( J\right) ^{2s} \\
&\lesssim &\sum_{m=1}^{\infty }2^{2ms}\frac{\left\vert J\right\vert _{\mu }}{%
\left\vert \pi ^{\left( m\right) }J\right\vert _{\mu }}\ell \left( J\right)
^{2s}\lesssim \ell \left( J\right) ^{2s},
\end{eqnarray*}%
provided $s<0$ or $s$ is small enough depending on the doubling constant of $%
\mu $.
\end{proof}

\subsection{Duality}

Here we compute the dual space of $W_{\func{dyad}}^{s}\left( \mu \right) $
under the $L^{2}\left( \mu \right) $ pairing 
\begin{equation*}
\left\langle f,g\right\rangle _{L^{2}\left( \mu \right) }=\int_{\mathbb{R}%
^{n}}f\left( x\right) g\left( x\right) d\mu \left( x\right) =\sum_{I\in 
\mathcal{D},J\in \mathcal{D}}\int_{\mathbb{R}^{n}}\bigtriangleup _{I}^{\mu
}f\bigtriangleup _{J}^{\mu }gd\mu =\sum_{I\in \mathcal{D}}\int_{\mathbb{R}%
^{n}}\bigtriangleup _{I}^{\mu }f\bigtriangleup _{I}^{\mu }gd\mu .
\end{equation*}

\begin{lemma}
\label{identify dual}Let $-1<s<1$. Then%
\begin{equation*}
\left( W_{\func{dyad}}^{s}\left( \mu \right) \right) ^{\ast }=W_{\func{dyad}%
}^{-s}\left( \mu \right) ,
\end{equation*}%
holds in the sense that if $g\in W_{\func{dyad}}^{-s}\left( \mu \right) $
then $f\rightarrow \left\langle f,g\right\rangle _{L^{2}\left( \mu \right) }$
defines a bounded linear functional on $W_{\func{dyad}}^{s}\left( \mu
\right) $, and conversely that every bounded linear functional on $W_{\func{%
dyad}}^{s}\left( \mu \right) $ arises in this way.
\end{lemma}

\begin{proof}
For $\kappa $ sufficiently large, Cauchy-Schwarz gives%
\begin{eqnarray*}
\left\vert \left\langle f,g\right\rangle _{L^{2}\left( \mu \right)
}\right\vert &=&\left\vert \left\langle \sum_{I\in \mathcal{D}%
}\bigtriangleup _{I;\kappa }^{\mu }f,\sum_{J\in \mathcal{D}}\bigtriangleup
_{J;\kappa }^{\mu }g\right\rangle _{L^{2}\left( \mu \right) }\right\vert
=\left\vert \sum_{I\in \mathcal{D}}\sum_{J\in \mathcal{D}}\int_{\mathbb{R}%
^{n}}\left( \bigtriangleup _{I;\kappa }^{\mu }f\right) \left( \bigtriangleup
_{J;\kappa }^{\mu }g\right) d\mu \right\vert \\
&=&\left\vert \sum_{I\in \mathcal{D}}\int_{\mathbb{R}^{n}}\left(
\bigtriangleup _{I;\kappa }^{\mu }f\right) \left( \bigtriangleup _{I;\kappa
}^{\mu }g\right) d\mu \right\vert =\left\vert \sum_{I\in \mathcal{D}}\int_{%
\mathbb{R}^{n}}\ell \left( I\right) ^{-s}\left( \bigtriangleup _{I;\kappa
}^{\mu }f\right) \ell \left( I\right) ^{s}\left( \bigtriangleup _{I;\kappa
}^{\mu }g\right) d\mu \right\vert \\
&\leq &\sqrt{\int_{\mathbb{R}^{n}}\sum_{I\in \mathcal{D}}\left\vert \ell
\left( I\right) ^{-s}\left( \bigtriangleup _{I;\kappa }^{\mu }f\right)
\right\vert ^{2}d\mu }\sqrt{\int_{\mathbb{R}^{n}}\sum_{I\in \mathcal{D}%
}\left\vert \ell \left( I\right) ^{s}\left( \bigtriangleup _{I;\kappa }^{\mu
}g\right) \right\vert ^{2}d\mu }=\left\Vert f\right\Vert _{W_{\func{dyad}%
}^{s}\left( \mu \right) }\left\Vert g\right\Vert _{W_{\func{dyad}%
}^{-s}\left( \mu \right) }.
\end{eqnarray*}%
Conversely, if $\Lambda \in W_{\func{dyad}}^{s}\left( \mu \right) ^{\ast }$\
is a continuous linear functional on $W_{\func{dyad}}^{s}\left( \mu \right) $%
, then for $\kappa $ sufficiently large%
\begin{eqnarray*}
&&\left\vert \sum_{I\in \mathcal{D}}\ell \left( I\right) ^{-s}\left\langle
f,h_{I;\kappa }^{\mu }\right\rangle _{L^{2}\left( \mu \right) }\ell \left(
I\right) ^{s}\Lambda h_{I;\kappa }^{\mu }\right\vert =\left\vert \sum_{I\in 
\mathcal{D}}\left\langle f,h_{I;\kappa }^{\mu }\right\rangle _{L^{2}\left(
\mu \right) }\Lambda h_{I;\kappa }^{\mu }\right\vert =\left\vert \sum_{I\in 
\mathcal{D}}\Lambda \left( \bigtriangleup _{I;\kappa }^{\mu }f\right)
\right\vert =\left\vert \Lambda f\right\vert \\
&\leq &\left\Vert \Lambda \right\Vert \left\Vert f\right\Vert _{W_{\func{dyad%
}}^{s}\left( \mu \right) }=\left\Vert \Lambda \right\Vert \sqrt{\int_{%
\mathbb{R}^{n}}\sum_{I\in \mathcal{D}}\left\vert \ell \left( I\right)
^{-s}\left( \bigtriangleup _{I;\kappa }^{\mu }f\right) \right\vert ^{2}d\mu }%
=\left\Vert \Lambda \right\Vert \sqrt{\sum_{I\in \mathcal{D}}\ell \left(
I\right) ^{-2s}\left\vert \left\langle f,h_{I;\kappa }^{\mu }\right\rangle
_{L^{2}\left( \mu \right) }\right\vert ^{2}}
\end{eqnarray*}%
for all choices of coefficients $\left\{ \ell \left( I\right)
^{-s}\left\langle f,h_{I;\kappa }^{\mu }\right\rangle _{L^{2}\left( \mu
\right) }\right\} _{I\in \mathcal{D}}\in \ell ^{2}\left( \mathcal{D}\right) $%
, and so we have $\ell \left( I\right) ^{s}\Lambda h_{I;\kappa }^{\mu }\in
\ell ^{2}\left( \mathcal{D}\right) $, i.e.%
\begin{equation*}
\sqrt{\sum_{I\in \mathcal{D}}\ell \left( I\right) ^{2s}\left\vert \Lambda
h_{I;\kappa }^{\mu }\right\vert ^{2}}\leq \left\Vert \Lambda \right\Vert .
\end{equation*}%
Thus if we define $g$ to have Alpert coefficients $\Lambda h_{I;\kappa
}^{\mu ,a}$, i.e.%
\begin{equation*}
g=\sum_{\substack{ I\in \mathcal{D}  \\ a\in \Gamma }}\left\langle
g,h_{I;\kappa }^{\mu ,a}\right\rangle _{L^{2}\left( \mu \right) }h_{I;\kappa
}^{\mu ,a}\equiv \sum_{\substack{ I\in \mathcal{D}  \\ a\in \Gamma }}\left(
\Lambda h_{I;\kappa }^{\mu ,a}\right) h_{I;\kappa }^{\mu ,a}\ ,
\end{equation*}%
then $g\in W_{\func{dyad}}^{-s}\left( \mu \right) $ since%
\begin{equation*}
\left\Vert g\right\Vert _{W_{\func{dyad}}^{-s}\left( \mu \right) }=\sqrt{%
\int_{\mathbb{R}^{n}}\sum_{I\in \mathcal{D}}\left\vert \ell \left( I\right)
^{s}\left( \bigtriangleup _{I;\kappa }^{\mu }g\right) \right\vert ^{2}d\mu }=%
\sqrt{\sum_{\substack{ I\in \mathcal{D}  \\ a\in \Gamma }}\ell \left(
I\right) ^{2s}\left\vert \left\langle g,h_{I;\kappa }^{\mu ,a}\right\rangle
_{L^{2}\left( \mu \right) }\right\vert ^{2}}=\sqrt{\sum_{\substack{ I\in 
\mathcal{D}  \\ a\in \Gamma }}\ell \left( I\right) ^{2s}\left\vert \Lambda
h_{I;\kappa }^{\mu ,a}\right\vert ^{2}}\leq \left\Vert \Lambda \right\Vert
<\infty ,
\end{equation*}%
and finally we have 
\begin{equation*}
\Lambda f=\sum_{\substack{ I\in \mathcal{D}  \\ a\in \Gamma }}\left\langle
f,h_{I;\kappa }^{\mu ,a}\right\rangle _{L^{2}\left( \mu \right) }\Lambda
h_{I;\kappa }^{\mu ,a}=\sum_{\substack{ I\in \mathcal{D}  \\ a\in \Gamma }}%
\left\langle f,h_{I;\kappa }^{\mu ,a}\right\rangle _{L^{2}\left( \mu \right)
}\left\langle g,h_{I;\kappa }^{\mu ,a}\right\rangle _{L^{2}\left( \mu
\right) }=\int_{\mathbb{R}^{n}}f\left( x\right) g\left( x\right) d\mu \left(
x\right) =\left\langle f,g\right\rangle _{L^{2}\left( \mu \right) }.
\end{equation*}
\end{proof}

\subsection{Quasiorthogonality in weighted Sobolev spaces}

Let $\bigtriangleup _{I;\kappa _{1}}^{\mu ,s}=\ell \left( I\right)
^{-s}\bigtriangleup _{I;\kappa _{1}}^{\mu }$and $\mathbb{E}_{I;\kappa
_{1}}^{\mu ,s}=\ell \left( I\right) ^{-s}\mathbb{E}_{I;\kappa _{1}}^{\mu }$.
Since $\left\{ \bigtriangleup _{I;\kappa _{1}}^{\mu ,s}\right\} _{I\in 
\mathcal{D}}$ is a complete set of orthogonal projections on $W_{\limfunc{%
dyad}}^{s}\left( \mathbb{R}^{n}\right) $, we have%
\begin{equation*}
\sum_{I\in \mathcal{D}}\ell \left( I\right) ^{-2s}\left\Vert \bigtriangleup
_{I;\kappa _{1}}^{\mu }f\right\Vert _{L^{2}\left( \mu \right)
}^{2}=\sum_{I\in \mathcal{D}}\left\Vert \bigtriangleup _{I;\kappa _{1}}^{\mu
,s}f\right\Vert _{W_{\limfunc{dyad}}^{s}\left( \mu \right) }^{2}=\left\Vert
\sum_{I\in \mathcal{D}}\bigtriangleup _{I;\kappa _{1}}^{\mu ,s}f\right\Vert
_{W_{\limfunc{dyad}}^{2}\left( \mu \right) }^{2}=\left\Vert f\right\Vert
_{W_{\limfunc{dyad}}^{s}\left( \mu \right) }^{2}\ ,
\end{equation*}%
and then if $\left\{ \mathbb{E}_{F;\kappa _{1}}^{\mu ,s}\right\} _{F\in 
\mathcal{F}}$ is a collection of projections, indexed by a subgrid $\mathcal{%
F}$ of $\mathcal{D}$ satisfying an appropriate Carleson condition, we expect
to have%
\begin{equation*}
\sum_{F\in \mathcal{F}}\ell \left( F\right) ^{-2s}\left\Vert \mathbb{E}%
_{F;\kappa _{1}}^{\mu }f\right\Vert _{L^{2}\left( \mu \right)
}^{2}=\sum_{F\in \mathcal{F}}\left\Vert \mathbb{E}_{F;\kappa _{1}}^{\mu
,s}f\right\Vert _{L^{2}\left( \mu \right) }^{2}\lesssim \left\Vert
f\right\Vert _{W_{\limfunc{dyad}}^{s}\left( \mu \right) }^{2}=\sum_{Q\in 
\mathcal{D}}\ell \left( Q\right) ^{-2s}\left\Vert \bigtriangleup _{Q;\kappa
_{1}}^{\mu }f\right\Vert _{L^{2}\left( \mu \right) }^{2}\ .
\end{equation*}%
This inequality says we can replace the collection of moment vanishing
projections $\left\{ \bigtriangleup _{Q;\kappa _{1}}^{\mu }\right\} _{Q\in 
\mathcal{D}}$ with a collection of averaging projections $\left\{ \mathbb{E}%
_{F;\kappa _{1}}^{\mu }\right\} _{F\in \mathcal{F}}$ provided the subset $%
\mathcal{F}$ of $\mathcal{D}$ is sufficiently small that an appropriate
Carleson condition holds. Here is the quasiorthogonality lemma with
appropriate Carleson condition that is suitable for use with Sobolev
spaces,and in which $\left\vert f\right\vert $ does not appear. Finally, it
can be viewed as a Sobolev space version of the Carleson Embedding Theorem.

\begin{definition}
A subgrid $\mathcal{F}\subset \mathcal{D}$ satisfies the $\varepsilon $\emph{%
-strong }$\mu $\emph{-Carleson condition} if,%
\begin{equation}
\sum_{\substack{ F\in \mathcal{F}  \\ F\subset F^{\prime }}}\left( \frac{%
\ell \left( F^{\prime }\right) }{\ell \left( F\right) }\right) ^{\varepsilon
}\left\vert F\right\vert _{\mu }\leq C\left\vert F^{\prime }\right\vert
_{\mu }\ ,\ \ \ \ \ F^{\prime }\in \mathcal{F}.  \label{eps Carleson}
\end{equation}
\end{definition}

\begin{lemma}[Quasiorthogonality Lemma]
\label{q lemma}Let $\mu $ be a doubling measure on $\mathbb{R}^{n}$. Suppose
that for some $\varepsilon >0$, the subgrid $\mathcal{F}\subset \mathcal{D}$
satisfies the $\varepsilon $-strong $\mu $-Carleson condition (\ref{eps
Carleson}). Then for $s<\frac{\varepsilon }{2}$ we have 
\begin{equation*}
\sum_{F\in \mathcal{F}}\ell \left( F\right) ^{-2s}\left\Vert \mathbb{E}%
_{F;\kappa }^{\mu }f\right\Vert _{L^{2}\left( \mu \right) }^{2}\lesssim
\left\Vert f\right\Vert _{W_{\limfunc{dyad}}^{s}\left( \mu \right)
}^{2}=\sum_{Q\in \mathcal{D}}\ell \left( Q\right) ^{-2s}\left\Vert
\bigtriangleup _{Q;\kappa }^{\mu }f\right\Vert _{L^{2}\left( \mu \right)
}^{2}\ .
\end{equation*}
\end{lemma}

\begin{proof}
Since $\frac{\bigtriangleup _{\pi ^{\left( m\right) }F;\kappa }^{\mu }f}{%
\left\Vert \bigtriangleup _{\pi ^{\left( m\right) }F;\kappa }^{\mu
}f\right\Vert _{\infty }}$ is a normalized polynomial of degree less than $%
\kappa $ on the $\mathcal{D}$-child $\left( \pi ^{\left( m\right) }F\right)
_{F}$ of $\pi ^{\left( m\right) }F$ that contains $F$, we have by (\ref%
{analogue'}) and (\ref{in part}),%
\begin{equation*}
\left\vert F\right\vert _{\mu }\left\Vert \mathbb{E}_{F;\kappa }^{\sigma
}\bigtriangleup _{\pi ^{\left( m\right) }F;\kappa }^{\mu }f\right\Vert
_{\infty }^{2}\approx \int_{F}\left\vert \mathbb{E}_{F;\kappa }^{\sigma
}\bigtriangleup _{\pi ^{\left( m\right) }F;\kappa }^{\mu }f\right\vert
^{2}d\mu \leq \int_{F}\left\vert \bigtriangleup _{\pi ^{\left( m\right)
}F;\kappa }^{\mu }f\right\vert ^{2}d\mu \approx \left\vert F\right\vert
_{\mu }\left\Vert \bigtriangleup _{\pi ^{\left( m\right) }F;\kappa }^{\mu
}f\right\Vert _{\infty }^{2}\ ,
\end{equation*}%
and so for any $t<0$ we have%
\begin{eqnarray*}
&&\sum_{F\in \mathcal{F}}\ell \left( F\right) ^{-2s}\left\Vert \mathbb{E}%
_{F;\kappa }^{\mu }f\right\Vert _{L^{2}\left( \mu \right) }^{2}=\sum_{F\in 
\mathcal{F}}\ell \left( F\right) ^{-2s}\left\Vert \mathbb{E}_{F;\kappa
_{1}}^{\mu }\left( \sum_{I\in \mathcal{D}}\bigtriangleup _{I;\kappa }^{\mu
}f\right) \right\Vert _{L^{2}\left( \mu \right) }^{2} \\
&=&\sum_{F\in \mathcal{F}}\ell \left( F\right) ^{-2s}\left\Vert
\sum_{m=1}^{\infty }\mathbb{E}_{F;\kappa _{1}}^{\mu }\bigtriangleup _{\pi
^{\left( m\right) }F;\kappa }^{\mu }f\right\Vert _{L^{2}\left( \mu \right)
}^{2}\lesssim \sum_{F\in \mathcal{F}}\ell \left( F\right) ^{-2s}\left\vert
F\right\vert _{\mu }\left( \sum_{m=1}^{\infty }\left\Vert \mathbb{E}%
_{F;\kappa }^{\mu }\bigtriangleup _{\pi ^{\left( m\right) }F;\kappa }^{\mu
}f\right\Vert _{\infty }\right) ^{2} \\
&\leq &\sum_{F\in \mathcal{F}}\ell \left( F\right) ^{-2s}\left\vert
F\right\vert _{\mu }\left( \sum_{m=1}^{\infty }\left\Vert \bigtriangleup
_{\pi ^{\left( m\right) }F;\kappa }^{\mu }f\right\Vert _{\infty }\right) ^{2}
\\
&\leq &\sum_{F\in \mathcal{F}}\ell \left( F\right) ^{-2s}\left\vert
F\right\vert _{\mu }\left( \sum_{m=0}^{\infty }\ell \left( \pi ^{\left(
m\right) }F\right) ^{2t}\right) \left( \sum_{m=0}^{\infty }\ell \left( \pi
^{\left( m\right) }F\right) ^{-2t}\left\Vert \bigtriangleup _{\pi ^{\left(
m\right) }F;\kappa }^{\mu }f\right\Vert _{\infty }^{2}\right) \\
&\approx &\sum_{F\in \mathcal{F}}\ell \left( F\right) ^{2t-2s}\left\vert
F\right\vert _{\mu }\sum_{m=0}^{\infty }\ell \left( \pi ^{\left( m\right)
}F\right) ^{-2t}\left\Vert \bigtriangleup _{\pi ^{\left( m\right) }F;\kappa
}^{\mu }f\right\Vert _{\infty }^{2}.
\end{eqnarray*}%
Substituting $F^{\prime }$ for $\pi ^{\left( m\right) }F$, and letting $t=s-%
\frac{\varepsilon }{2}<0$, we obtain from the $\varepsilon $-strong Carleson
condition (\ref{eps Carleson}) that%
\begin{eqnarray*}
&&\sum_{F\in \mathcal{F}}\ell \left( F\right) ^{-2s}\left\Vert \mathbb{E}%
_{F;\kappa }^{\mu }f\right\Vert _{L^{2}\left( \mu \right) }^{2}\lesssim
\sum_{F\in \mathcal{F}}\ell \left( F\right) ^{2t-2s}\sum_{m=0}^{\infty }%
\frac{\left\vert F\right\vert _{\mu }}{\left\vert \left( \pi ^{\left(
m\right) }F\right) _{F}\right\vert _{\mu }}\ell \left( \pi ^{\left( m\right)
}F\right) ^{-2t}\int_{\left( \pi ^{\left( m\right) }F\right) _{F}}\left\Vert
\bigtriangleup _{\pi ^{\left( m\right) }F;\kappa }^{\mu }f\right\Vert
_{\infty }^{2}d\mu \\
&\approx &\sum_{F\in \mathcal{F}}\ell \left( F\right)
^{2t-2s}\sum_{m=0}^{\infty }\frac{\left\vert F\right\vert _{\mu }}{%
\left\vert \left( \pi ^{\left( m\right) }F\right) _{F}\right\vert _{\mu }}%
\ell \left( \pi ^{\left( m\right) }F\right) ^{-2t}\int_{\left( \pi ^{\left(
m\right) }F\right) _{F}}\left\vert \bigtriangleup _{\pi ^{\left( m\right)
}F;\kappa }^{\mu }f\right\vert ^{2}d\mu \\
&=&\sum_{F^{\prime }\in \mathcal{F}}\left( \sum_{\substack{ F\in \mathcal{F} 
\\ F\subset F^{\prime }}}\ell \left( F\right) ^{-\varepsilon }\frac{%
\left\vert F\right\vert _{\mu }}{\left\vert \left( F^{\prime }\right)
_{F}\right\vert _{\mu }}\right) \ell \left( F^{\prime }\right) ^{\varepsilon
-2s}\left\Vert \bigtriangleup _{F^{\prime };\kappa }^{\mu }f\right\Vert
_{L^{2}\left( \mu \right) }^{2} \\
&\leq &\sum_{F^{\prime }\in \mathcal{F}}\left( C\ell \left( F^{\prime
}\right) ^{-\varepsilon }\frac{\left\vert F^{\prime }\right\vert _{\mu }}{%
\left\vert \left( F^{\prime }\right) _{F}\right\vert _{\mu }}\right) \ell
\left( F^{\prime }\right) ^{\varepsilon -2s}\left\Vert \bigtriangleup
_{F^{\prime };\kappa }^{\mu }f\right\Vert _{L^{2}\left( \mu \right)
}^{2}\lesssim \sum_{F^{\prime }\in \mathcal{F}}\ell \left( F^{\prime
}\right) ^{-2s}\left\Vert \bigtriangleup _{F^{\prime };\kappa }^{\mu
}f\right\Vert _{L^{2}\left( \mu \right) }^{2}\lesssim \left\Vert
f\right\Vert _{W_{\limfunc{dyad}}^{s}\left( \mu \right) }^{2}\ .
\end{eqnarray*}
\end{proof}

\begin{remark}
We can replace $f$ by its modulus $\left\vert f\right\vert $ in the above
lemma when $\kappa =1$ and $s>0$ is sufficiently small. Indeed, by the
reverse triangle inequality we have 
\begin{eqnarray*}
\left\Vert \left\vert f\right\vert \right\Vert _{W_{\func{diff};1}^{s}\left(
\mu \right) }^{2} &=&\sum_{Q\in \mathcal{D}}\int_{Q}\left\vert \frac{f\left(
x\right) -\mathbb{E}_{Q;1}^{\mu }f\left( x\right) }{\ell \left( Q\right) ^{s}%
}\right\vert ^{2}d\mu \left( x\right) =\sum_{Q\in \mathcal{D}%
}\int_{Q}\left\vert \frac{f\left( x\right) -\left( \frac{1}{\left\vert
Q\right\vert _{\mu }}\int_{Q}fd\mu \right) \mathbf{1}_{Q}\left( x\right) }{%
\ell \left( Q\right) ^{s}}\right\vert ^{2}d\mu \left( x\right) \\
&=&2\sum_{Q\in \mathbb{D}}\frac{1}{\left\vert Q\right\vert _{\mu }^{2}}%
\int_{Q}\int_{Q}\left\vert \frac{\left\vert f\left( x\right) \right\vert
-\left\vert f\left( y\right) \right\vert }{\ell \left( Q\right) ^{s}}%
\right\vert ^{2}d\mu \left( x\right) d\mu \left( y\right) \\
&\leq &2\sum_{Q\in \mathbb{D}}\frac{1}{\left\vert Q\right\vert _{\mu }^{2}}%
\int_{Q}\int_{Q}\left\vert \frac{f\left( x\right) -f\left( y\right) }{\ell
\left( Q\right) ^{s}}\right\vert ^{2}d\mu \left( x\right) d\mu \left(
y\right) =\left\Vert f\right\Vert _{W_{\func{diff};1}^{s}\left( \mu \right)
}^{2},
\end{eqnarray*}%
and now we use the equivalence $\left\Vert \cdot \right\Vert _{W_{\func{diff}%
;1}^{s}\left( \mu \right) }^{2}\approx \left\Vert \cdot \right\Vert _{W_{%
\limfunc{dyad};1}^{s}\left( \mu \right) }^{2}$ for $\left\vert s\right\vert $
sufficiently small.
\end{remark}

\section{Preliminaries: weighted Sobolev norm inequalities}

Duality shows the equivalence of weighted norm inequalities with bilinear
inequalities.

\begin{lemma}
The Sobolev norm inequality%
\begin{equation}
\left\Vert T_{\sigma }^{\alpha }f\right\Vert _{W_{\func{dyad}}^{s}\left(
\omega \right) }\leq \left\Vert T^{\alpha }\right\Vert _{\func{op}%
}\left\Vert f\right\Vert _{W_{\func{dyad}}^{s}\left( \sigma \right) }\ ,
\label{hom Sob inequ}
\end{equation}%
is equivalent to the bilinear inequality 
\begin{equation}
\left\vert \sum_{I\in \mathcal{D}}\sum_{J\in \mathcal{D}}\int_{\mathbb{R}%
^{n}}\left( T_{\sigma }^{\alpha }\bigtriangleup _{I}^{\sigma }f\right)
\bigtriangleup _{J}^{\omega }gd\omega \right\vert \leq \left\Vert T^{\alpha
}\right\Vert _{\func{bil}}\left\Vert f\right\Vert _{W_{\func{dyad}%
}^{s}\left( \sigma \right) }\left\Vert g\right\Vert _{W_{\func{dyad}%
}^{-s}\left( \omega \right) }\ ,\text{ \ \ for }f,g\in L^{2}\left( \mu
\right) .  \label{hom bil inequ}
\end{equation}
\end{lemma}

\begin{proof}
Indeed, if the bilinear inequality holds and $f=\sum_{I\in \mathcal{D}%
}\bigtriangleup _{I}^{\sigma }f$ and $g=\sum_{J\in \mathcal{D}%
}\bigtriangleup _{J}^{\omega }g$, then%
\begin{equation*}
\left\vert \int_{\mathbb{R}^{n}}\left( T_{\sigma }^{\alpha }f\right)
gd\omega \right\vert =\left\vert \sum_{I\in \mathcal{D}}\sum_{J\in \mathcal{D%
}}\int_{\mathbb{R}^{n}}\left( T_{\sigma }^{\alpha }\bigtriangleup
_{I}^{\sigma }f\right) \bigtriangleup _{J}^{\omega }gd\omega \right\vert
\leq \left\Vert T^{\alpha }\right\Vert _{\func{bil}}\left\Vert f\right\Vert
_{W_{\func{dyad}}^{s}\left( \sigma \right) }\left\Vert g\right\Vert _{W_{%
\func{dyad}}^{-s}\left( \omega \right) }
\end{equation*}%
shows that%
\begin{eqnarray*}
\left\Vert T_{\sigma }^{\alpha }f\right\Vert _{W_{\func{dyad}}^{s}\left(
\omega \right) } &=&\sup_{\left\Vert g\right\Vert _{W_{\func{dyad}%
}^{s}\left( \omega \right) ^{\ast }}\leq 1}\left\vert \int_{\mathbb{R}%
^{n}}\left( T_{\sigma }^{\alpha }f\right) gd\omega \right\vert
=\sup_{\left\Vert g\right\Vert _{W_{\func{dyad}}^{-s}\left( \omega \right)
}\leq 1}\left\vert \int_{\mathbb{R}^{n}}\left( T_{\sigma }^{\alpha }f\right)
gd\omega \right\vert \leq \left\Vert T^{\alpha }\right\Vert _{\func{bil}%
}\left\Vert f\right\Vert _{W_{\func{dyad}}^{s}\left( \sigma \right) }\ , \\
&\Longrightarrow &\left\Vert T^{\alpha }\right\Vert _{\func{op}}\leq
\left\Vert T^{\alpha }\right\Vert _{\func{bil}}\text{ since }L^{2}\left( \mu
\right) \text{ is dense in }W_{\func{dyad}}^{s}\left( \mu \right) .
\end{eqnarray*}%
Conversely, if the norm inequality holds, then%
\begin{eqnarray*}
\left\vert \int_{\mathbb{R}^{n}}\left( T_{\sigma }^{\alpha }f\right)
gd\omega \right\vert &\leq &\left\Vert T_{\sigma }^{\alpha }f\right\Vert
_{W_{\func{dyad}}^{s}\left( \omega \right) }\left\Vert g\right\Vert _{W_{%
\func{dyad}}^{s}\left( \omega \right) ^{\ast }}\leq \left\Vert T^{\alpha
}\right\Vert _{\func{op}}\left\Vert f\right\Vert _{W_{\func{dyad}}^{s}\left(
\sigma \right) }\left\Vert g\right\Vert _{W_{\func{dyad}}^{s}\left( \omega
\right) ^{\ast }} \\
&=&\left\Vert T^{\alpha }\right\Vert _{\func{op}}\left\Vert f\right\Vert
_{W_{\func{dyad}}^{s}\left( \sigma \right) }\left\Vert g\right\Vert _{W_{%
\func{dyad}}^{-s}\left( \omega \right) }\ , \\
&\Longrightarrow &\left\Vert T^{\alpha }\right\Vert _{\func{bil}}\leq
\left\Vert T^{\alpha }\right\Vert _{\func{op}}.
\end{eqnarray*}
\end{proof}

\subsection{The good-bad decomposition}

Here we follow the random grid idea of Nazarov, Treil and Volberg. Denote by 
$\Omega _{\limfunc{dyad}}$ the collection of all dyadic grids $\mathcal{D}$.
For a weight $\mu $, we consider a random choice of dyadic grid $\mathcal{D}$
on the natural probability space $\Omega _{\limfunc{dyad}}$.

\begin{definition}
\label{d.good} For a positive integer $r$ and $0<\varepsilon <1$, a cube $%
J\in \mathcal{D}$ is said to be\emph{\ }$\left( r\emph{,}\varepsilon \right) 
$\emph{-bad} if there is a cube $I\in \mathcal{D}$ with $\lvert I\rvert \geq
2^{r}\lvert J\rvert $, and 
\begin{equation*}
\func{dist}(e(I),J)\leq \tfrac{1}{2}\lvert J\rvert ^{\varepsilon }\lvert
I\rvert ^{1-\varepsilon }\,.
\end{equation*}%
Here, $e(J)$ is the union of the boundaries of the children of the cube $J$.
(This contains the set of discontinuities of $h_{J;\kappa }^{\mu }$ and its
derivatives less than order $\kappa $.) Otherwise, $J$ is said to be $\left(
r\emph{,}\varepsilon \right) $\emph{-}$\QTR{up}{\func{good}}$.
\end{definition}

The basic proposition here is this, see e.g. \cite{Vol} and e.g. \cite{LaWi}
or \cite{SaShUr7} for higher dimensions.

\begin{proposition}
There is the conditional probability estimate 
\begin{equation*}
\mathbb{P}_{\limfunc{cond}}^{\Omega _{\limfunc{dyad}}}\left( J\text{ is }%
\left( r\emph{,}\varepsilon \right) \text{-}\QTR{up}{\func{bad}}:J\in 
\mathcal{D}\right) \leq C_{\varepsilon }2^{-\varepsilon r}.
\end{equation*}
\end{proposition}

Define projections 
\begin{equation}
\mathsf{P}_{\QTR{up}{\func{good};}\mathcal{D}}^{\mu }f=\mathsf{P}_{\QTR{up}{%
\func{good}}}^{\mu }f\equiv \sum_{I\text{ is }\left( r,\varepsilon \right) 
\text{-}\QTR{up}{\func{good}}\text{ }\in \mathcal{D}}\Delta _{I}^{\mu }f\,%
\text{\ and }\mathsf{P}_{\QTR{up}{\func{bad};}\mathcal{D}}^{\mu }f=\mathsf{P}%
_{\QTR{up}{\func{bad}}}^{\mu }f\equiv f-\mathsf{P}_{\QTR{up}{\func{good}}%
}^{\mu }f.  \label{e.Pgood}
\end{equation}%
Recall that%
\begin{eqnarray*}
\left\Vert f\right\Vert _{W_{\mathcal{D}}^{s}\left( \mu \right) }^{2}
&\equiv &\sum_{Q\in \mathcal{D}}\ell \left( Q\right) ^{-2s}\left\Vert
\bigtriangleup _{Q}^{\mu }f\right\Vert _{L^{2}\left( \mu \right) }^{2}, \\
\left\Vert f\right\Vert _{W_{\limfunc{dyad}}^{s}\left( \mu \right) }^{2}
&\approx &\left\Vert f\right\Vert _{W_{\mathcal{D}}^{s}\left( \mu \right)
}^{2},\ \ \ \ \ \text{for all }\mathcal{D}\in \Omega _{\limfunc{dyad}}.
\end{eqnarray*}%
The basic Proposition is then this.

\begin{proposition}
\label{p.Pgood}(cf. Theorem 17.1 in \cite{Vol} where the middle line below
is treated) We have the estimates 
\begin{eqnarray*}
\mathbb{E}_{\Omega _{\limfunc{dyad}}}^{\mathcal{D}}\left\Vert \mathsf{P}_{%
\QTR{up}{\func{bad};}\mathcal{D}}^{\mu }f\right\Vert _{W_{\mathcal{D}%
}^{s}\left( \mu \right) } &\leq &C_{\varepsilon }2^{-\frac{\varepsilon }{2}%
r}\left\Vert f\right\Vert _{W_{\func{dyad}}^{s}\left( \mu \right) }, \\
\mathbb{E}_{\Omega _{\limfunc{dyad}}}^{\mathcal{D}}\left\Vert \mathsf{P}_{%
\QTR{up}{\func{bad};}\mathcal{D}}^{\mu }f\right\Vert _{L^{p}\left( \mu
\right) } &\leq &C_{\varepsilon }2^{-\frac{\varepsilon r}{p}}\left\Vert
f\right\Vert _{L^{p}\left( \mu \right) }, \\
\mathbb{E}_{\Omega _{\limfunc{dyad}}}^{\mathcal{D}}\left\Vert \mathsf{P}_{%
\QTR{up}{\func{bad};}\mathcal{D}}^{\mu }f\right\Vert _{W_{\mathcal{D}%
}^{-s}\left( \mu \right) } &\leq &C_{\varepsilon }2^{-\frac{\varepsilon }{2}%
r}\left\Vert f\right\Vert _{W_{\func{dyad}}^{-s}\left( \mu \right) }.
\end{eqnarray*}
\end{proposition}

\begin{proof}
We have 
\begin{align*}
\mathbb{E}_{\Omega _{\limfunc{dyad}}}^{\mathcal{D}}\left( \left\Vert \mathsf{%
P}_{\QTR{up}{\func{bad};}\mathcal{D}}^{\mu }f\right\Vert _{W_{\func{dyad}%
}^{s}\left( \mu \right) }^{2}\right) & =\mathbb{E}_{\Omega _{\limfunc{dyad}%
}}^{\mathcal{D}}\sum_{I\in \mathcal{D}\text{ is }\left( r,\varepsilon
\right) \text{-}\QTR{up}{\func{bad}}}\ell \left( I\right) ^{-2s}\left\langle
f,h_{I}^{\mu }\right\rangle _{L^{2}\left( \mu \right) }^{2} \\
& \leq C_{\varepsilon }\mathbb{E}_{\Omega _{\limfunc{dyad}}}^{\mathcal{D}%
}2^{-\varepsilon r}\sum_{I\in \mathcal{D}}\ell \left( I\right)
^{-2s}\left\langle f,h_{I}^{\mu }\right\rangle _{L^{2}\left( \mu \right)
}^{2}=C_{\varepsilon }2^{-\varepsilon r}\left\Vert f\right\Vert _{W_{\func{%
dyad}}^{s}\left( \mu \right) }^{2}\,,
\end{align*}%
and then%
\begin{equation*}
\mathbb{E}_{\Omega _{\limfunc{dyad}}}^{\mathcal{D}}\left( \left\Vert \mathsf{%
P}_{\QTR{up}{\func{bad};}\mathcal{D}}^{\mu }f\right\Vert _{W_{\func{dyad}%
}^{s}\left( \mu \right) }\right) \leq \sqrt{\mathbb{E}_{\Omega _{\limfunc{%
dyad}}}^{\mathcal{D}}\left( \left\Vert \mathsf{P}_{\QTR{up}{\func{bad};}%
\mathcal{D}}^{\mu }f\right\Vert _{W_{\func{dyad}}^{s}\left( \mu \right)
}^{2}\right) }\leq C_{\varepsilon }2^{-\frac{\varepsilon }{2}r}\left\Vert
f\right\Vert _{W_{\func{dyad}}^{s}\left( \mu \right) }\ .
\end{equation*}%
Similarly for $L^{p}\left( \mu \right) $ and $W_{\func{dyad}}^{-s}\left( \mu
\right) $ in place of $W_{\func{dyad}}^{s}\left( \mu \right) $.
\end{proof}

From this we conclude the following: Given any $0<\varepsilon <1$, there is
a choice of $r$, depending on $\varepsilon $, so that the following holds.
Let $T:W_{\func{dyad}}^{s}\left( \sigma \right) \rightarrow W_{\func{dyad}%
}^{s}\left( \omega \right) $ be a bounded linear operator, where for
specificity we take $W_{\func{dyad}}^{s}=W_{\mathcal{D}_{0}}^{s}$, and $%
\mathcal{D}_{0}$ is the standard dyadic grid on $\mathbb{R}^{n}$. We then
have 
\begin{equation}
\left\Vert T\right\Vert _{W_{\func{dyad}}^{s}\left( \sigma \right)
\rightarrow W_{\func{dyad}}^{s}\left( \omega \right) }\leq 2\sup_{\left\Vert
f\right\Vert _{W_{\func{dyad}}^{s}\left( \sigma \right) }=1}\sup_{\left\Vert
\phi \right\Vert _{W_{\func{dyad}}^{-s}\left( \omega \right) }=1}\mathbb{E}%
_{\Omega _{\limfunc{dyad}}}^{\mathcal{D}}\mathbb{E}_{\Omega _{\limfunc{dyad}%
}}^{\mathcal{E}}\lvert \left\langle T\mathsf{P}_{\QTR{up}{\func{good};}%
\mathcal{D}}^{\sigma }f,\mathsf{P}_{\QTR{up}{\func{good};}\mathcal{D}%
}^{\omega }g\right\rangle _{\omega }\rvert \,.  \label{e.Tgood}
\end{equation}%
Indeed, we can choose $f\in W_{\func{dyad}}^{s}\left( \sigma \right) $ of
norm one, and $g\in W_{\func{dyad}}^{-s}\left( \omega \right) $ of norm one,
and we can write 
\begin{equation*}
f=\mathsf{P}_{\QTR{up}{\func{good};}\mathcal{D}}^{\sigma }f+\mathsf{P}_{%
\QTR{up}{\func{bad};}\mathcal{D}}^{\sigma }f\,
\end{equation*}%
and similarly for $g$ and $\mathcal{E}$, so that 
\begin{align*}
& \left\Vert T\right\Vert _{W_{\func{dyad}}^{s}\left( \sigma \right)
\rightarrow W_{\func{dyad}}^{s}\left( \omega \right) }=\left\vert
\left\langle T_{\sigma }f,g\right\rangle _{\omega }\right\vert \\
& \leq \mathbb{E}_{\Omega _{\limfunc{dyad}}}^{\mathcal{D}}\mathbb{E}_{\Omega
_{\limfunc{dyad}}}^{\mathcal{E}}\lvert \left\langle T_{\sigma }\mathsf{P}_{%
\QTR{up}{\func{good};}\mathcal{D}}^{\sigma }f,\mathsf{P}_{\QTR{up}{\func{good%
};}\mathcal{E}}^{\omega }g\right\rangle _{\omega }\rvert +\mathbb{E}_{\Omega
_{\limfunc{dyad}}}^{\mathcal{D}}\mathbb{E}_{\Omega _{\limfunc{dyad}}}^{%
\mathcal{E}}\lvert \left\langle T_{\sigma }\mathsf{P}_{\QTR{up}{\func{bad};}%
\mathcal{D}}^{\sigma }f,\mathsf{P}_{\QTR{up}{\func{good};}\mathcal{E}%
}^{\omega }g\right\rangle _{\omega }\rvert \\
& \quad +\mathbb{E}_{\Omega _{\limfunc{dyad}}}^{\mathcal{D}}\mathbb{E}%
_{\Omega _{\limfunc{dyad}}}^{\mathcal{E}}\lvert \left\langle T_{\sigma }%
\mathsf{P}_{\QTR{up}{\func{good};}\mathcal{D}}^{\sigma }f,\mathsf{P}_{%
\QTR{up}{\func{bad};}\mathcal{E}}^{\omega }g\right\rangle _{\omega }\rvert +%
\mathbb{E}_{\Omega _{\limfunc{dyad}}}^{\mathcal{D}}\mathbb{E}_{\Omega _{%
\limfunc{dyad}}}^{\mathcal{E}}\lvert \left\langle T_{\sigma }\mathsf{P}_{%
\QTR{up}{\func{bad};}\mathcal{D}}^{\sigma }f,\mathsf{P}_{\QTR{up}{\func{bad};%
}\mathcal{E}}^{\omega }g\right\rangle _{\omega }\rvert \\
& \leq \mathbb{E}_{\Omega _{\limfunc{dyad}}}^{\mathcal{D}}\mathbb{E}_{\Omega
_{\limfunc{dyad}}}^{\mathcal{E}}\lvert \left\langle T_{\sigma }\mathsf{P}_{%
\QTR{up}{\func{good};}\mathcal{D}}^{\sigma }f,\mathsf{P}_{\QTR{up}{\func{good%
};}\mathcal{E}}^{\omega }g\right\rangle _{\omega }\rvert +3C_{\varepsilon
}2^{-\frac{\varepsilon }{2}r}\left\Vert T\right\Vert _{W_{\func{dyad}%
}^{s}\left( \sigma \right) \rightarrow W_{\func{dyad}}^{s}\left( \omega
\right) }\ .
\end{align*}%
And this proves \eqref{e.Tgood} for $r$ sufficiently large.

This has the following implication for us: \emph{Given any linear operator }$%
T$\emph{\ and }$0<\varepsilon <1$\emph{, it suffices to consider only }$%
\left( \emph{r,}\varepsilon \right) $\emph{-}$\QTR{up}{\func{good}}$\emph{\
cubes for r sufficiently large, and prove an estimate for $\left\Vert
T\right\Vert _{W_{\func{dyad}}^{s}\left( \sigma \right) \rightarrow W_{\func{%
dyad}}^{s}\left( \omega \right) }$ that is independent of this assumption.}
Accordingly, we will call $\left( r\emph{,}\varepsilon \right) $-$\QTR{up}{%
\func{good}}$ cubes just $\QTR{up}{\func{good}}$\emph{\ cubes} from now on.
At certain points in the arguments below, such as in the treatment of the
neighbour form for $W_{\func{dyad}}^{s}\left( \sigma \right) $, we will need
to further restrict the parameter $\varepsilon $ (and accordingly $r$ as
well).

\subsection{Defining the norm inequality\label{Subsubsection norm}}

We now turn to a precise definition of the weighted norm inequality%
\begin{equation}
\left\Vert T_{\sigma }^{\alpha }f\right\Vert _{W_{\func{dyad}}^{s}\left(
\omega \right) }\leq \mathfrak{N}_{T^{\alpha }}\left\Vert f\right\Vert _{W_{%
\func{dyad}}^{s}\left( \sigma \right) },\ \ \ \ \ f\in W_{\func{dyad}%
}^{s}\left( \sigma \right) ,  \label{two weight'}
\end{equation}%
where $W_{\func{dyad}}^{s}\left( \sigma \right) $ is the Hilbert space
completion of the space of functions $f\in L_{\func{loc}}^{2}\left( \sigma
\right) $ for which 
\begin{equation*}
\left\Vert f\right\Vert _{W_{\func{dyad}}^{s}\left( \sigma \right) }<\infty .
\end{equation*}%
A similar definition holds for $W_{\func{dyad}}^{s}\left( \omega \right) $.
For a precise definition of (\ref{two weight'}), it is possible to proceed
with the notion of associating operators and kernels through an identity for
functions with disjoint support as in \cite{Ste2}. However, we choose to
follow the approach in \cite[see page 314]{SaShUr9}. So we suppose that $%
K^{\alpha }$ is a smooth $\alpha $-fractional Calder\'{o}n-Zygmund kernel,
and we introduce a family $\left\{ \eta _{\delta ,R}^{\alpha }\right\}
_{0<\delta <R<\infty }$ of nonnegative functions on $\left[ 0,\infty \right) 
$ so that the truncated kernels $K_{\delta ,R}^{\alpha }\left( x,y\right)
=\eta _{\delta ,R}^{\alpha }\left( \left\vert x-y\right\vert \right)
K^{\alpha }\left( x,y\right) $ are bounded with compact support for fixed $x$
or $y$, and uniformly satisfy (\ref{sizeandsmoothness'}). Then the truncated
operators 
\begin{equation*}
T_{\sigma ,\delta ,R}^{\alpha }f\left( x\right) \equiv \int_{\mathbb{R}%
^{n}}K_{\delta ,R}^{\alpha }\left( x,y\right) f\left( y\right) d\sigma
\left( y\right) ,\ \ \ \ \ x\in \mathbb{R}^{n},
\end{equation*}%
are pointwise well-defined, and we will refer to the pair $\left( K^{\alpha
},\left\{ \eta _{\delta ,R}^{\alpha }\right\} _{0<\delta <R<\infty }\right) $
as an $\alpha $-fractional singular integral operator, which we typically
denote by $T^{\alpha }$, suppressing the dependence on the truncations.

\begin{definition}
\label{def bounded}We say that an $\alpha $-fractional singular integral
operator $T^{\alpha }=\left( K^{\alpha },\left\{ \eta _{\delta ,R}^{\alpha
}\right\} _{0<\delta <R<\infty }\right) $ satisfies the norm inequality (\ref%
{two weight'}) provided%
\begin{equation*}
\left\Vert T_{\sigma ,\delta ,R}^{\alpha }f\right\Vert _{W_{\func{dyad}%
}^{s}\left( \omega \right) }\leq \mathfrak{N}_{T^{\alpha }}\left( \sigma
,\omega \right) \left\Vert f\right\Vert _{W_{\func{dyad}}^{s}\left( \sigma
\right) },\ \ \ \ \ f\in W_{\func{dyad}}^{s}\left( \sigma \right) ,0<\delta
<R<\infty .
\end{equation*}
\end{definition}

\begin{description}
\item[Independence of Truncations] \label{independence}In the presence of
the classical Muckenhoupt condition $A_{2}^{\alpha }$, the norm inequality (%
\ref{two weight'}) is essentially independent of the choice of truncations
used, including \emph{nonsmooth} truncations as well - see \cite{LaSaShUr3}.
However, in dealing with the Monotonicity Lemma \ref{mono} below, where $%
\kappa ^{th}$ order Taylor approximations are made on the truncated kernels,
it is necessary to use sufficiently smooth truncations. Similar comments
apply to the Cube Testing conditions (\ref{def Kappa polynomial'}) and (\ref%
{full testing}) below.
\end{description}

\subsubsection{Ellipticity of kernels}

Modifying slightly the definition in \cite[(39) on page 210]{Ste}, we say
that an $\alpha $-fractional Calder\'{o}n-Zygmund kernel $K^{\alpha }$ is 
\emph{elliptic in the sense of Stein} if there is a unit coordinate vector $%
\mathbf{e}_{k}\in \mathbb{R}^{n}$ for some $1\leq k\leq n$, and a positive
constant $c>0$ such that%
\begin{equation*}
\left\vert K^{\alpha }\left( x,x+t\mathbf{e}_{k}\right) \right\vert \geq
c\left\vert t\right\vert ^{\alpha -n},\ \ \ \ \ \text{for all }t\in \mathbb{R%
}.
\end{equation*}%
For example, the Beurling, Cauchy and Riesz transform kernels, as well as
those for $k$-iterated Riesz transforms are elliptic in the sense of Stein
for any $k\geq 1$.

\subsubsection{Cube testing}

While the next more general testing conditions with $\kappa >1$, introduced
in \cite{RaSaWi} and \cite{Saw6}, are not used in the statements of our
theorems, they will be used in the course of our proof.

The $\kappa $\emph{-cube testing conditions} associated with an $\alpha $%
-fractional singular integral operator $T^{\alpha }$, introduced in \cite%
{RaSaWi} for $s=0$, are given by%
\begin{eqnarray}
\left( \mathfrak{T}_{T^{\alpha }}^{\kappa ,s}\left( \sigma ,\omega \right)
\right) ^{2} &\equiv &\sup_{Q\in \mathcal{Q}^{n}}\max_{0\leq \left\vert
\beta \right\vert <\kappa }\frac{1}{\ell \left( Q\right) ^{-2s}\left\vert
Q\right\vert _{\sigma }}\left\Vert \mathbf{1}_{Q}T_{\sigma }^{\alpha }\left( 
\mathbf{1}_{Q}m_{Q}^{\beta }\right) \right\Vert _{W_{\func{dyad}}^{s}\left(
\omega \right) }^{2}<\infty ,  \label{def Kappa polynomial'} \\
\left( \mathfrak{T}_{\left( T^{\alpha }\right) ^{\ast }}^{\kappa ,-s}\left(
\omega ,\sigma \right) \right) ^{2} &\equiv &\sup_{Q\in \mathcal{Q}%
^{n}}\max_{0\leq \left\vert \beta \right\vert <\kappa }\frac{1}{\ell \left(
Q\right) ^{2s}\left\vert Q\right\vert _{\omega }}\left\Vert \mathbf{1}%
_{Q}T_{\omega }^{\alpha ,\ast }\left( \mathbf{1}_{Q}m_{Q}^{\beta }\right)
\right\Vert _{W_{\func{dyad}}^{-s}\left( \sigma \right) }^{2}<\infty , 
\notag
\end{eqnarray}%
where $\left( T^{\alpha ,\ast }\right) _{\omega }=\left( T_{\sigma }^{\alpha
}\right) ^{\ast }$, with $m_{Q}^{\beta }\left( x\right) \equiv \left( \frac{%
x-c_{Q}}{\ell \left( Q\right) }\right) ^{\beta }$ for any cube $Q$ and
multiindex $\beta $, where $c_{Q}$ is the center of the cube $Q$, and where
we interpret the right hand sides as holding uniformly over all sufficiently
smooth truncations of $T^{\alpha }$. Equivalently, in the presence of $%
A_{2}^{\alpha }$, we can take a single suitable truncation, see Independence
of Truncations in Subsubsection \ref{independence} above.

We also use the larger \emph{triple }$\kappa $\emph{-cube testing conditions}
in which the integrals over $Q$ are extended to the triple $3Q$ of $Q$:%
\begin{eqnarray}
\left( \mathfrak{TR}_{T^{\alpha }}^{\kappa ,s}\left( \sigma ,\omega \right)
\right) ^{2} &\equiv &\sup_{Q\in \mathcal{Q}^{n}}\max_{0\leq \left\vert
\beta \right\vert <\kappa }\frac{1}{\ell \left( Q\right) ^{-2s}\left\vert
Q\right\vert _{\sigma }}\left\Vert \mathbf{1}_{3Q}T_{\sigma }^{\alpha
}\left( \mathbf{1}_{Q}m_{Q}^{\beta }\right) \right\Vert _{W_{\func{dyad}%
}^{s}\left( \omega \right) }^{2}<\infty ,  \label{full testing} \\
\left( \mathfrak{TR}_{\left( T^{\alpha }\right) ^{\ast }}^{\kappa ,-s}\left(
\omega ,\sigma \right) \right) ^{2} &\equiv &\sup_{Q\in \mathcal{Q}%
^{n}}\max_{0\leq \left\vert \beta \right\vert <\kappa }\frac{1}{\ell \left(
Q\right) ^{2s}\left\vert Q\right\vert _{\omega }}\left\Vert \mathbf{1}%
_{3Q}T_{\omega }^{\alpha ,\ast }\left( \mathbf{1}_{Q}m_{Q}^{\beta }\right)
\right\Vert _{W_{\func{dyad}}^{-s}\left( \sigma \right) }^{2}<\infty . 
\notag
\end{eqnarray}

\subsection{Necessity of the classical Muckenhoupt condition}

Suppose that $T^{\alpha }f=K\ast f$ where $K^{\alpha }\left( x\right) =\frac{%
\Omega \left( x\right) }{\left\vert x\right\vert ^{n-\alpha }}$, and $\Omega
\left( x\right) $ is homogeneous of degree $0$ and smooth away from the
origin. Note that we do not require any cancellation properties on $\Omega $%
, except that when $\alpha =0$ we suppose $\int \Omega \left( x\right)
d\sigma _{n-1}\left( x\right) =0$ where $\sigma _{n-1}$ is surface measure
on the sphere (see e.g. \cite{Ste} page 68 for the case $\alpha =0$). We
assume $\Omega $ is nontrivial in the sense that there is a\ coordinate
direction $\Theta \in \mathbb{S}^{n-1}$ such that $\Omega \left( \Theta
\right) \neq 0$. Then there is a cone $\Gamma $ centered on $\Theta $ on
which $K\left( x\right) =\frac{\Omega \left( x\right) }{\left\vert
x\right\vert ^{n-\alpha }}$ and $\Omega \left( x\right) \geq c>0$ for $x\in
\Gamma $. Consider pairs of separated dyadic cubes in direction $\Theta $, 
\begin{equation*}
\mathcal{SP}_{\Theta }\equiv \left\{ \left( I,I^{\prime }\right) :\limfunc{%
dist}\left( I,I^{\prime }\right) \approx \ell \left( I\right) =\ell \left(
I^{\prime }\right) \text{, }R_{1}\ell \left( I\right) \leq \limfunc{dist}%
\left( I,I^{\prime }\right) \approx R\ell \left( I\right) \text{, and }%
I^{\prime }\text{ has direction }\Theta \text{ from }I\right\} ,
\end{equation*}%
where $R$ is chosen large enough that if the cone $\Gamma $ is translated to
any point in $I$, then it contains any cube $I^{\prime }$ for which $\left(
I,I^{\prime }\right) \in \mathcal{SP}_{\Theta }$.

We first derive the `separated' Muckenhoupt condition from the full testing
condition for $T^{\alpha }$, i.e.%
\begin{equation*}
\left( \frac{1}{\left\vert I^{\prime }\right\vert }\int_{I^{\prime }}d\omega
\right) \left( \frac{1}{\left\vert I\right\vert }\int_{I}d\sigma \right)
\leq \mathfrak{FT}_{T^{\alpha }}\left( s;\sigma ,\omega \right) ^{2},\ \ \ \
\ \left( I,I^{\prime }\right) \in \mathcal{SP}_{\Theta }\text{.}
\end{equation*}

We may assume without loss of generality that $\Theta =e_{1}$, the unit
vector in the direction of the positive $x_{1}$-axis. Now we choose a
special unit Haar function $h_{I^{\prime }}^{\omega }$, i.e. $\bigtriangleup
_{I^{\prime }}^{\omega }h_{I^{\prime }}^{\omega }=h_{I^{\prime }}^{\omega }$
and $\left\Vert h_{I^{\prime }}^{\omega }\right\Vert _{L^{2}\left( \omega
\right) }=1$, satisfying%
\begin{equation*}
h_{I^{\prime }}^{\omega }\left( x\right) =\sum_{K\in \mathfrak{C}\left(
I\right) }a_{K}\mathbf{1}_{K}\left( x\right) \text{, where }\left\{ 
\begin{array}{ccc}
a_{K}>0 & \text{ if } & K\text{ lies to the right of center} \\ 
a_{K}<0 & \text{ if } & K\text{ lies to the left of center}%
\end{array}%
\right. ,
\end{equation*}%
where for a cube $Q$ centered at the origin, we\ say a child $K$ lies to the
right of center if $K$ is contained in the half space where $x_{1}\geq 0$.
We now compute%
\begin{eqnarray*}
&&\left\Vert \mathbf{1}_{I^{\prime }}T^{\alpha }\left( \mathbf{1}_{I}\sigma
\right) \right\Vert _{W_{\func{dyad}}^{s}\left( \omega \right)
}^{2}=\sum_{J\in \mathcal{D}}\ell \left( J\right) ^{-2s}\left\Vert
\bigtriangleup _{J}^{\omega }\left( \mathbf{1}_{I^{\prime }}T^{\alpha }%
\mathbf{1}_{I}\sigma \right) \right\Vert _{L^{2}\left( \omega \right)
}^{2}\geq \ell \left( I^{\prime }\right) ^{-2s}\left\Vert \bigtriangleup
_{I^{\prime }}^{\omega }\left( T^{\alpha }\mathbf{1}_{I}\sigma \right)
\right\Vert _{L^{2}\left( \omega \right) }^{2} \\
&\geq &\ell \left( I^{\prime }\right) ^{-2s}\left\vert \left\langle
T^{\alpha }\mathbf{1}_{I}\sigma ,h_{I^{\prime }}^{\omega }\right\rangle
_{\omega }\right\vert ^{2}=\ell \left( I^{\prime }\right) ^{-2s}\left\vert
\int_{I^{\prime }}\left( \int_{I}K^{\alpha }\left( x,y\right) d\sigma \left(
y\right) \right) h_{I^{\prime }}^{\omega }\left( x\right) d\omega \left(
x\right) \right\vert ^{2} \\
&=&\ell \left( I^{\prime }\right) ^{-2s}\left\vert \int_{I^{\prime }}\left(
\int_{I}\left[ K^{\alpha }\left( x-y\right) -K^{\alpha }\left( c_{I^{\prime
}}-y\right) \right] d\sigma \left( y\right) \right) h_{I^{\prime }}^{\omega
}\left( x\right) d\omega \left( x\right) \right\vert ^{2} \\
&=&\ell \left( I^{\prime }\right) ^{-2s}\left\vert \int_{I^{\prime
}}\int_{I} \left[ K^{\alpha }\left( x-y\right) -K^{\alpha }\left(
c_{I^{\prime }}-y\right) \right] h_{I^{\prime }}^{\omega }\left( x\right)
d\sigma \left( y\right) d\omega \left( x\right) \right\vert ^{2} \\
&\gtrsim &c\ell \left( I^{\prime }\right) ^{-2s}\left( \int_{I^{\prime
}}\int_{I}\left\vert \frac{1}{\sqrt{\left\vert I\right\vert _{\omega }}}%
\frac{\ell \left( I^{\prime }\right) }{\limfunc{dist}\left( I,I^{\prime
}\right) ^{n+1-\alpha }}\right\vert d\sigma \left( y\right) d\omega \left(
x\right) \right) ^{2}=c\ell \left( I^{\prime }\right) ^{-2s}\frac{\left\vert
I\right\vert _{\sigma }^{2}\left\vert I^{\prime }\right\vert _{\omega }}{%
\limfunc{dist}\left( I,I^{\prime }\right) ^{2n-2\alpha }},
\end{eqnarray*}%
since $\omega $ is doubling. Thus we have%
\begin{eqnarray*}
\ell \left( I^{\prime }\right) ^{-2s}\frac{\left\vert I\right\vert _{\sigma
}^{2}\left\vert I^{\prime }\right\vert _{\omega }}{\limfunc{dist}\left(
I,I^{\prime }\right) ^{2n-2\alpha }} &\lesssim &\left\Vert \mathbf{1}%
_{I^{\prime }}T^{\alpha }\left( \mathbf{1}_{I}\sigma \right) \right\Vert
_{W_{\func{dyad}}^{s}\left( \omega \right) }^{2} \\
&\leq &\mathfrak{FT}_{T^{\alpha }}\left( s;\sigma ,\omega \right)
^{2}\left\Vert \mathbf{1}_{I}\right\Vert _{W_{\func{dyad}}^{s}\left( \sigma
\right) }^{2}\approx \mathfrak{FT}_{T^{\alpha }}\left( s;\sigma ,\omega
\right) ^{2}\ell \left( I\right) ^{-2s}\left\vert I\right\vert _{\sigma }\ ,
\end{eqnarray*}%
which gives the desired inequality,%
\begin{equation*}
\frac{\left\vert I\right\vert _{\sigma }}{\left\vert I\right\vert ^{n-\alpha
}}\frac{\left\vert I^{\prime }\right\vert _{\omega }}{\left\vert I^{\prime
}\right\vert ^{n-\alpha }}\lesssim \mathfrak{FT}_{T^{\alpha }}\left(
s;\sigma ,\omega \right) ^{2}.
\end{equation*}

Now using \cite{LaSaUr1} or \cite{SaShUr7}, we can derive from this the full
Muckenhoupt inequality when $\omega $ and $\sigma $ contain no common point
masses, which is certainly the case for doubling measures,%
\begin{equation*}
A_{2}^{\alpha }\left( \sigma ,\omega \right) =\sup_{I}\frac{\left\vert
I\right\vert _{\sigma }\left\vert I\right\vert _{\omega }}{\left\vert
I\right\vert ^{2\left( n-\alpha \right) }}\lesssim \mathfrak{FT}_{T^{\alpha
}}\left( s;\sigma ,\omega \right) ^{2}.
\end{equation*}%
Thus we have proved the following lemma.

\begin{lemma}
\label{S ell}If $T^{\alpha }f=K^{\alpha }\ast f$ where $K^{\alpha }\left(
x\right) =\frac{\Omega \left( x\right) }{\left\vert x\right\vert ^{n-\alpha }%
}$, and $\Omega \left( x\right) $ is homogeneous of degree $0$ and smooth
away from the origin and $\Omega \left( \mathbf{e}_{k}\right) \neq 0$ for
some $1\leq k\leq n$, then boundedness of $T^{\alpha }$ from $W_{\func{dyad}%
}^{s}\left( \sigma \right) $ to $W_{\func{dyad}}^{s}\left( \omega \right) $,
implies the $A_{2}^{\alpha }\left( \sigma ,\omega \right) $ condition, more
precisely,%
\begin{equation*}
\sqrt{A_{2}^{\alpha }\left( \sigma ,\omega \right) }\lesssim \mathfrak{FT}%
_{T^{\alpha }}\left( s;\sigma ,\omega \right) \leq \mathfrak{N}_{T^{\alpha
}}\left( s;\sigma ,\omega \right) .
\end{equation*}
\end{lemma}

\subsection{Necessity of the strong $\protect\kappa ^{th}$ order Pivotal
Condition with full grids and doubling weights}

We will need the following definitions that capture a critical property of
stopping times arising from pivotal criteria.

\begin{definition}
Let $\sigma $ be a measure on $\mathbb{R}^{n}$ and $0\leq \beta <1$.

\begin{enumerate}
\item We say that a collection of dyadic cubes $\left\{ F_{r}\right\}
_{r=1}^{\infty }$ is a $\left( \beta ,\sigma \right) $-subdecomposition of a
dyadic cube $F$ if $F_{r}\subset F$ for all $1<r<\infty $ and 
\begin{equation*}
\sum_{r=1}^{\infty }\left( \frac{\left\vert F_{r}\right\vert _{\sigma }}{%
\left\vert F\right\vert _{\sigma }}\right) ^{1-\beta }\leq C.
\end{equation*}%
\newline

\item We say that a grid $\mathcal{F}\subset \mathcal{D}$ is $\left( \beta
,\sigma \right) $\emph{-full} if $\mathfrak{C}_{\mathcal{F}}\left( F\right) $
is a $\left( \beta ,\sigma \right) $-subdecomposition of $F$, uniformly over
all $F\in \mathcal{F}$ (with respect to a constant $C$ we suppress).
\end{enumerate}
\end{definition}

The smaller fractional Poisson integrals $\mathrm{P}_{\kappa }^{\alpha
}\left( Q,\mu \right) $ used here, in \cite{RaSaWi} and elsewhere, are given
by 
\begin{equation}
\mathrm{P}_{\kappa }^{\alpha }\left( Q,\mu \right) =\int_{\mathbb{R}^{n}}%
\frac{\ell \left( Q\right) ^{\kappa }}{\left( \ell \left( Q\right)
+\left\vert y-c_{Q}\right\vert \right) ^{n+\kappa -\alpha }}d\mu \left(
y\right) ,\ \ \ \ \ \kappa \geq 1,  \label{def kappa Poisson}
\end{equation}%
and the $\kappa ^{th}$-order fractional pivotal constants $\mathcal{V}%
_{2}^{\alpha ,\kappa },\mathcal{V}_{2}^{\alpha ,\kappa ,\ast }<\infty $, $%
\kappa \geq 1$, are given by%
\begin{eqnarray}
\left( \mathcal{V}_{2}^{\alpha ,\kappa }\left( \sigma ,\omega \right)
\right) ^{2} &=&\sup_{Q\supset \dot{\cup}Q_{r}}\frac{1}{\left\vert
Q\right\vert _{\sigma }}\sum_{r=1}^{\infty }\mathrm{P}_{\kappa }^{\alpha
}\left( Q_{r},\mathbf{1}_{Q}\sigma \right) ^{2}\left\vert Q_{r}\right\vert
_{\omega }\ ,  \label{both pivotal k} \\
\left( \mathcal{V}_{2}^{\alpha ,\kappa ,\ast }\left( \sigma ,\omega \right)
\right) ^{2} &=&\sup_{Q\supset \dot{\cup}Q_{r}}\frac{1}{\left\vert
Q\right\vert _{\omega }}\sum_{r=1}^{\infty }\mathrm{P}_{\kappa }^{\alpha
}\left( Q_{r},\mathbf{1}_{Q}\omega \right) ^{2}\left\vert Q_{r}\right\vert
_{\sigma }=\left( \mathcal{V}_{2}^{\alpha ,\kappa }\left( \omega ,\sigma
\right) \right) ^{2}\ ,  \notag
\end{eqnarray}%
and for $0<\beta ,\varepsilon \leq 1$, the $\left( \beta ,\varepsilon
\right) $-strong $\kappa $-pivotal constants $\mathcal{V}_{2,\beta
,\varepsilon }^{\alpha ,\kappa },\mathcal{V}_{2,\beta ,\varepsilon }^{\alpha
,\kappa ,\ast }<\infty $, $\kappa \geq 1$ are given by%
\begin{eqnarray}
\left( \mathcal{V}_{2,\beta ,\varepsilon }^{\alpha ,\kappa }\left( \sigma
,\omega \right) \right) ^{2} &=&\sup_{Q\supset \dot{\cup}Q_{r}\ \text{is }%
\left( \beta ,\sigma \right) \text{-full}}\frac{1}{\left\vert Q\right\vert
_{\sigma }}\sum_{r=1}^{\infty }\mathrm{P}_{\kappa }^{\alpha }\left( Q_{r},%
\mathbf{1}_{Q}\sigma \right) ^{2}\left( \frac{\ell \left( Q\right) }{\ell
\left( Q_{r}\right) }\right) ^{\varepsilon }\left\vert Q_{r}\right\vert
_{\omega }\ ,  \label{strong pivotal} \\
\left( \mathcal{V}_{2,\beta ,\varepsilon }^{\alpha ,\kappa ,\ast }\left(
\sigma ,\omega \right) \right) ^{2} &=&\sup_{Q\supset \dot{\cup}Q_{r}\ \text{%
is }\left( \beta ,\sigma \right) \text{-full}}\frac{1}{\left\vert
Q\right\vert _{\omega }}\sum_{r=1}^{\infty }\mathrm{P}_{\kappa }^{\alpha
}\left( Q_{r},\mathbf{1}_{Q}\omega \right) ^{2}\left( \frac{\ell \left(
Q\right) }{\ell \left( Q_{r}\right) }\right) ^{\varepsilon }\left\vert
Q_{r}\right\vert _{\sigma }=\left( \mathcal{V}_{2,\beta ,\varepsilon
}^{\alpha ,\kappa }\left( \omega ,\sigma \right) \right) ^{2}\ ,  \notag
\end{eqnarray}%
and where the suprema are taken over all subdecompositions of a cube $Q\in 
\mathcal{Q}^{n}$ into pairwise disjoint dyadic subcubes $Q_{r}$ for which $%
\left\{ Q_{r}\right\} _{r=1}^{\infty }$ is a $\left( \beta ,\sigma \right) $%
-subdecomposition of $Q$.

\begin{lemma}
\label{full}Let $\varepsilon >0$ be as in (\ref{e.Jsimeq}), and $0\leq
\alpha <n$. Let $\gamma >1$ and suppose $\mathcal{F}$ is a stopping time
grid constructed from a pair of doubling measures $\sigma ,\omega $ using
the stopping criterion,%
\begin{equation}
\mathrm{P}_{\kappa }^{\alpha }\left( I,\mathbf{1}_{F}\sigma \right)
^{2}\left\vert I\right\vert _{\omega }\geq \gamma \left\vert I\right\vert
_{\sigma }\ ,\ \ \ \ \ I\in \mathcal{D}\left( F\right) ,  \label{stop crit}
\end{equation}%
i.e. for each $F\in \mathcal{F}$, the set $\mathfrak{C}_{\mathcal{F}}\left(
F\right) $ of $\mathcal{F}$-children of $F$ consists of the maximal dyadic
subcubes $I$ of $F$ satisfying (\ref{stop crit}). Then $\mathcal{F}$ is $%
\left( \beta ,\sigma \right) $-full provided $\kappa >\frac{n}{2}$ and $%
\beta ,\varepsilon \ $ are sufficiently small. More precisely, under these
conditions we have%
\begin{equation*}
\sum_{r=1}^{\infty }\left( \frac{\left\vert F_{r}\right\vert _{\sigma }}{%
\left\vert F\right\vert _{\sigma }}\right) ^{1-\beta }\lesssim A_{2}^{\alpha
}\left( \sigma ,\omega \right) ^{1-\beta }.
\end{equation*}
\end{lemma}

\begin{proof}
Suppose $\left\{ F_{r}\right\} _{r=1}^{\infty }=\mathfrak{C}_{\mathcal{F}%
}\left( F\right) $ is the set of stopping children of $F$ in the grid $%
\mathcal{F}$. Then by maximality of the cubes $F_{r}$, they also satisfy%
\begin{equation*}
\mathrm{P}_{\kappa }^{\alpha }\left( \pi _{\mathcal{D}}F_{r},\mathbf{1}%
_{F}\sigma \right) ^{2}\left\vert \pi _{\mathcal{D}}F_{r}\right\vert
_{\omega }<\gamma \left\vert \pi _{\mathcal{D}}F_{r}\right\vert _{\sigma }\ ,
\end{equation*}%
and since both measures are doubling, we have%
\begin{equation*}
\mathrm{P}_{\kappa }^{\alpha }\left( F_{r},\mathbf{1}_{F}\sigma \right)
^{2}\left\vert F_{r}\right\vert _{\omega }\approx \left\vert
F_{r}\right\vert _{\sigma }\ .
\end{equation*}%
Thus the inequality $\lesssim $ in the above equivalence is from the
stopping criterion, while the opposite inequality $\gtrsim $ is from
doubling. Using doubling once more it follows easily that%
\begin{equation*}
\mathrm{P}_{\kappa }^{\alpha }\left( F_{r},\mathbf{1}_{F\setminus \gamma
F_{r}}\sigma \right) ^{2}\left\vert F_{r}\right\vert _{\omega }\approx
\left\vert F_{r}\right\vert _{\sigma }\ .
\end{equation*}%
Now using (\ref{e.Jsimeq}), namely%
\begin{equation*}
\mathrm{P}_{\kappa }^{\alpha }(J,\sigma \mathbf{1}_{K\setminus I})\lesssim
\left( \frac{\ell \left( J\right) }{\ell \left( I\right) }\right) ^{\kappa
-\varepsilon \left( n+m-\alpha \right) }\mathrm{P}_{\kappa }^{\alpha
}(I,\sigma \mathbf{1}_{K\setminus I}),
\end{equation*}%
we have%
\begin{eqnarray*}
&&\sum_{r=1}^{\infty }\left[ \frac{\left\vert F_{r}\right\vert _{\sigma }}{%
\left\vert F\right\vert _{\sigma }}\right] ^{1-\beta }\approx
\sum_{r=1}^{\infty }\left[ \frac{\mathrm{P}_{\kappa }^{\alpha }\left( F_{r},%
\mathbf{1}_{F\setminus \gamma F_{r}}\sigma \right) ^{2}\left\vert
F_{r}\right\vert _{\omega }}{\left\vert F\right\vert _{\sigma }}\right]
^{1-\beta } \\
&\lesssim &\sum_{r=1}^{\infty }\left[ \left( \frac{\ell \left( F_{r}\right) 
}{\ell \left( F\right) }\right) ^{2\kappa -2\varepsilon \left( n+\kappa
-\alpha \right) }\mathrm{P}_{\kappa }^{\alpha }(F,\sigma \mathbf{1}_{F})^{2}%
\frac{\left\vert F_{r}\right\vert _{\omega }}{\left\vert F\right\vert
_{\sigma }}\right] ^{1-\beta } \\
&\approx &\sum_{r=1}^{\infty }\left[ \left( \frac{\ell \left( F_{r}\right) }{%
\ell \left( F\right) }\right) ^{2\kappa -2\varepsilon \left( n+\kappa
-\alpha \right) }\frac{\left\vert F\right\vert _{\sigma }\left\vert
F_{r}\right\vert _{\omega }}{\left\vert F\right\vert ^{2-2\alpha }}\right]
^{1-\beta }\lesssim \sum_{r=1}^{\infty }\left[ \left( \frac{\ell \left(
F_{r}\right) }{\ell \left( F\right) }\right) ^{2\kappa -2\varepsilon \left(
n+\kappa -\alpha \right) +\theta _{\omega }^{\limfunc{rev}}}\frac{\left\vert
F\right\vert _{\sigma }\left\vert F\right\vert _{\omega }}{\left\vert
F\right\vert ^{2-2\alpha }}\right] ^{1-\beta } \\
&\lesssim &\sum_{r=1}^{\infty }\left[ \left( \frac{\ell \left( F_{r}\right) 
}{\ell \left( F\right) }\right) ^{2\kappa -2\varepsilon \left( n+\kappa
-\alpha \right) +\theta _{\omega }^{\limfunc{rev}}}A_{2}^{\alpha }\left(
\sigma ,\omega \right) \right] ^{1-\beta } \\
&\lesssim &A_{2}^{\alpha }\left( \sigma ,\omega \right) ^{1-\beta
}\sum_{r=1}^{\infty }\left( \frac{\ell \left( F_{r}\right) }{\ell \left(
F\right) }\right) ^{\left[ 2\kappa -2\varepsilon \left( n+\kappa -\alpha
\right) +\theta _{\omega }^{\limfunc{rev}}\right] \left( 1-\beta \right)
}\lesssim A_{2}^{\alpha }\left( \sigma ,\omega \right) ^{1-\beta },
\end{eqnarray*}%
provided 
\begin{equation*}
\left[ 2\kappa -2\varepsilon \left( n+\kappa -\alpha \right) +\theta
_{\omega }^{\limfunc{rev}}\right] \left( 1-\beta \right) \geq n,
\end{equation*}%
which in turn holds for any $\kappa >\frac{n}{2}$ and $\theta _{\omega }^{%
\limfunc{rev}}>0$ if $\varepsilon $ and $\beta $ are chosen sufficiently
small. Thus we conclude that $\left\{ F_{r}\right\} _{r=1}^{\infty }$ is a $%
\left( \beta ,\sigma \right) $-subdecomposition of $F$ for every $F\in 
\mathcal{F}$, which completes the proof that $\mathcal{F}$ is $\left( \beta
,\sigma \right) $-full.
\end{proof}

The following lemma is a variant of one in \cite[Subsection 4.1 on pages
12-13, especially Remark 15]{Saw6}, where it was the point of departure for
freeing the theory from reliance on energy conditions when the measures are
doubling. This variant exploits the critical notion of a full grid.

\begin{lemma}
\label{doub piv}Let $0\leq \alpha <n$ and $0<\varepsilon <1$ and suppose $%
\sigma $ is a doubling measure satisfying (\ref{combine}), i.e.%
\begin{equation*}
\left\vert I_{r}\right\vert _{\sigma }\lesssim \left( \frac{\ell \left(
I_{r}\right) }{\ell \left( I\right) }\right) ^{\theta _{\sigma }^{\func{rev}%
}}\left\vert I\right\vert _{\sigma }\ .
\end{equation*}%
Then for $\kappa >\theta _{\sigma }^{\limfunc{doub}}+\alpha -n$ and $\beta
=\varepsilon ^{\prime }=\frac{\varepsilon }{\theta _{\sigma }^{\func{rev}}}$%
, we have 
\begin{equation*}
\mathcal{V}_{2,\beta ,\varepsilon ^{\prime }}^{\alpha ,\kappa }\left( \sigma
,\omega \right) \leq C_{\kappa ,\alpha ,n,\varepsilon ,\theta _{\sigma }^{%
\limfunc{doub}},\theta _{\sigma }^{\func{rev}}}\sqrt{A_{2}^{\alpha }\left(
\sigma ,\omega \right) }.
\end{equation*}
\end{lemma}

\begin{proof}
A doubling measure $\sigma $ has a `doubling exponent' $\theta _{\sigma }^{%
\limfunc{doub}}>0$ and a positive constant $c$ that satisfies the condition
(see e.g. \cite{Saw6}),%
\begin{equation*}
\left\vert 2^{-j}Q\right\vert _{\sigma }\geq c2^{-j\theta _{\sigma }^{%
\limfunc{doub}}}\left\vert Q\right\vert _{\sigma }\ ,\ \ \ \ \ \text{for all 
}j\in \mathbb{N}.
\end{equation*}%
We can then exploit the doubling exponents $\theta _{\sigma }^{\limfunc{doub}%
}$ and reverse doubling exponents $\theta _{\sigma }^{\func{rev}}$ of the
doubling measure $\sigma $ in order to derive certain $\kappa ^{th}$ order
pivotal conditions $\mathcal{V}_{2,\beta ,\varepsilon ^{\prime }}^{\alpha
,\kappa }\left( \sigma ,\omega \right) <\infty $. Indeed, if $\sigma $ has
doubling exponent $\theta _{\sigma }^{\limfunc{doub}}$ and $\kappa >\theta
_{\sigma }^{\limfunc{doub}}+\alpha -n$, we have%
\begin{eqnarray}
&&\int_{\mathbb{R}^{n}\setminus I}\frac{\ell \left( I\right) ^{\kappa }}{%
\left( \ell \left( I\right) +\left\vert x-c_{I}\right\vert \right)
^{n+\kappa -\alpha }}d\sigma \left( x\right) =\sum_{j=1}^{\infty }\ell
\left( I\right) ^{\alpha -n}\int_{2^{j}I\setminus 2^{j-1}I}\frac{1}{\left( 1+%
\frac{\left\vert x-c_{I}\right\vert }{\ell \left( I\right) }\right)
^{n+\kappa -\alpha }}d\sigma \left( x\right)  \label{kappa large} \\
&\lesssim &\left\vert I\right\vert ^{\frac{\alpha }{n}-1}\sum_{j=1}^{\infty
}2^{-j\left( n+\kappa -\alpha \right) }\left\vert 2^{j}I\right\vert _{\sigma
}\lesssim \left\vert I\right\vert ^{\frac{\alpha }{n}-1}\sum_{j=1}^{\infty
}2^{-j\left( n+\kappa -\alpha \right) }\frac{1}{c2^{-j\theta _{\sigma }^{%
\limfunc{doub}}}}\left\vert I\right\vert _{\sigma }\leq C_{n,\kappa ,\alpha
,\beta ,\gamma }\left\vert I\right\vert ^{\frac{\alpha }{n}-1}\left\vert
I\right\vert _{\sigma }\ ,  \notag
\end{eqnarray}%
provided $n+\kappa -\alpha -\theta _{\sigma }^{\limfunc{doub}}>0$, i.e. $%
\kappa >\theta _{\sigma }^{\limfunc{doub}}+\alpha -n$. It follows that if $%
I\supset \overset{\cdot }{\dbigcup }_{r=1}^{\infty }I_{r}$ is a
subdecomposition of $I$ into pairwise disjoint cubes $I_{r}$, and $\kappa
>\theta _{\sigma }^{\limfunc{doub}}+\alpha -n$, then 
\begin{eqnarray*}
&&\sum_{r=1}^{\infty }\mathrm{P}_{\kappa }^{\alpha }\left( I_{r},\mathbf{1}%
_{I}\sigma \right) ^{2}\left( \frac{\ell \left( I\right) }{\ell \left(
I_{r}\right) }\right) ^{\varepsilon }\left\vert I_{r}\right\vert _{\omega
}\lesssim \sum_{r=1}^{\infty }\left( \left\vert I_{r}\right\vert ^{\frac{%
\alpha }{n}-1}\left\vert I_{r}\right\vert _{\sigma }\right) ^{2}\left( \frac{%
\ell \left( I\right) }{\ell \left( I_{r}\right) }\right) ^{\varepsilon
}\left\vert I_{r}\right\vert _{\omega } \\
&=&\sum_{r=1}^{\infty }\left( \frac{\ell \left( I\right) }{\ell \left(
I_{r}\right) }\right) ^{\varepsilon }\frac{\left\vert I_{r}\right\vert
_{\sigma }\left\vert I_{r}\right\vert _{\omega }}{\left\vert
I_{r}\right\vert ^{2\left( 1-\frac{\alpha }{n}\right) }}\left\vert
I_{r}\right\vert _{\sigma }\lesssim A_{2}^{\alpha }\left( \sigma ,\omega
\right) \sum_{r=1}^{\infty }\left( \frac{\ell \left( I\right) }{\ell \left(
I_{r}\right) }\right) ^{\varepsilon }\left\vert I_{r}\right\vert _{\sigma }\
.
\end{eqnarray*}

Now $\left( \frac{\ell \left( I\right) }{\ell \left( I_{r}\right) }\right)
^{\theta _{\sigma }^{\func{rev}}}\lesssim \frac{\left\vert I\right\vert
_{\sigma }}{\left\vert I_{r}\right\vert _{\sigma }}$, and so we conclude that%
\begin{eqnarray*}
\sum_{r=1}^{\infty }\mathrm{P}_{\kappa }^{\alpha }\left( I_{r},\mathbf{1}%
_{I}\sigma \right) ^{2}\left( \frac{\ell \left( I\right) }{\ell \left(
I_{r}\right) }\right) ^{\varepsilon }\left\vert I_{r}\right\vert _{\omega }
&\lesssim &A_{2}^{\alpha }\left( \sigma ,\omega \right) \sum_{r=1}^{\infty
}\left( \frac{\left\vert I\right\vert _{\sigma }}{\left\vert
I_{r}\right\vert _{\sigma }}\right) ^{\frac{\varepsilon }{\theta _{\sigma }^{%
\func{rev}}}}\left\vert I_{r}\right\vert _{\sigma } \\
A_{2}^{\alpha }\left( \sigma ,\omega \right) \sum_{r=1}^{\infty }\left( 
\frac{\left\vert I\right\vert _{\sigma }}{\left\vert I_{r}\right\vert
_{\sigma }}\right) ^{\frac{\varepsilon }{\theta _{\sigma }^{\func{rev}}}}%
\frac{\left\vert I_{r}\right\vert _{\sigma }}{\left\vert I\right\vert
_{\sigma }}\left\vert I\right\vert _{\sigma } &=&A_{2}^{\alpha }\left(
\sigma ,\omega \right) \sum_{r=1}^{\infty }\left( \frac{\left\vert
I_{r}\right\vert _{\sigma }}{\left\vert I\right\vert _{\sigma }}\right) ^{1-%
\frac{\varepsilon }{\theta _{\sigma }^{\func{rev}}}}\left\vert I\right\vert
_{\sigma }\leq CA_{2}^{\alpha }\left( \sigma ,\omega \right) \left\vert
I\right\vert _{\sigma }\;,
\end{eqnarray*}%
for $\varepsilon $ sufficiently small, if we assume in addition that $%
\left\{ I_{r}\right\} _{r=1}^{\infty }$ is a $\left( \beta ,\sigma \right) $%
-subdecomposition of $F$ with $\beta =\frac{\varepsilon }{\theta _{\sigma }^{%
\func{rev}}}$.

This then gives, 
\begin{eqnarray}
\left( \mathcal{V}_{2,\frac{\varepsilon }{\theta _{\sigma }^{\func{rev}}},%
\frac{\varepsilon }{\theta _{\sigma }^{\func{rev}}}}^{\alpha ,\kappa
}\right) ^{2} &=&\sup_{\left\{ I\right\} \cup \left\{ I_{r}\right\}
_{r=1}^{\infty }\text{ is }\left( \beta ,\sigma \right) \text{-full}}\frac{1%
}{\left\vert I\right\vert _{\sigma }}\sum_{r=1}^{\infty }\mathrm{P}_{\kappa
}^{\alpha }\left( I_{r},\mathbf{1}_{I}\sigma \right) ^{2}\left( \frac{\ell
\left( I\right) }{\ell \left( I_{r}\right) }\right) ^{\varepsilon
}\left\vert I_{r}\right\vert _{\omega }\lesssim A_{2}^{\alpha }\left( \sigma
,\omega \right) \ ,  \label{piv control} \\
\ \ \ \ \ \text{for }\kappa &>&\theta _{\sigma }^{\limfunc{doub}}+\alpha -n%
\text{ and }\varepsilon >0\text{ sufficiently small}\ .  \notag
\end{eqnarray}
\end{proof}

A similar result holds for $\mathcal{V}_{2,\beta ,\varepsilon ^{\prime
}}^{\alpha ,\kappa ,\ast }$ if $\kappa +n-\alpha >\theta _{\omega }^{%
\limfunc{doub}}$ and $\varepsilon >0$ is sufficiently small, where $%
\varepsilon ^{\prime }=\frac{\varepsilon }{\theta _{\omega }^{\func{rev}}}$.

\subsection{The energy lemma}

For $0\leq \alpha <n$ and $m\in \mathbb{R}_{+}$, we recall from (\ref{def
kappa Poisson}) the $m^{th}$-order fractional Poisson integral%
\begin{equation*}
\mathrm{P}_{m}^{\alpha }\left( J,\mu \right) \equiv \int_{\mathbb{R}^{n}}%
\frac{\ell \left( J\right) ^{m}}{\left( \ell \left( J\right) +\left\vert
y-c_{J}\right\vert \right) ^{m+n-\alpha }}d\mu \left( y\right) ,
\end{equation*}%
where $\mathrm{P}_{1}^{\alpha }\left( J,\mu \right) =\mathrm{P}^{\alpha
}\left( J,\mu \right) $ is the standard Poisson integral. The case $s=0$ of
the following extension of the `energy lemma' is due to Rahm, Sawyer and
Wick \cite{RaSaWi}, and is proved in detail in \cite[Lemmas 28 and 29 on
pages 27-30]{Saw6}.

\begin{definition}
Given a subset $\mathcal{J}\subset \mathcal{D}$, define the projection $%
\mathsf{P}_{\mathcal{J}}^{\omega }\equiv \sum_{J^{\prime }\in \mathcal{J}%
}\bigtriangleup _{J^{\prime };\kappa }^{\omega }$, and given a cube $J\in 
\mathcal{D}$, define the projection $\mathsf{P}_{J}^{\omega }\equiv
\sum_{J^{\prime }\in \mathcal{D}:\ J^{\prime }\subset J}\bigtriangleup
_{J^{\prime };\kappa }^{\omega }$.
\end{definition}

\begin{lemma}[\textbf{Energy Lemma}]
Fix $\kappa \geq 1$. Let $J\ $be a cube in $\mathcal{D}$, and let $\Psi
_{J}\in W_{\func{dyad}}^{-s}\left( \omega \right) $ be supported in $J$ with
vanishing $\omega $-means up to order less than $\kappa $. Let $\nu $ be a
positive measure supported in $\mathbb{R}^{n}\setminus \gamma J$ with $%
\gamma >1$, and let $T^{\alpha }$ be a smooth $\alpha $-fractional singular
integral operator with $0\leq \alpha <n$. Then for $\left\vert s\right\vert $
sufficiently small, we have the `pivotal' bound%
\begin{equation}
\left\vert \left\langle T^{\alpha }\left( \varphi \nu \right) ,\Psi
_{J}\right\rangle _{L^{2}\left( \omega \right) }\right\vert \lesssim
C_{\gamma }\mathrm{P}_{\kappa }^{\alpha }\left( J,\nu \right) \ell \left(
J\right) ^{-s}\sqrt{\left\vert J\right\vert _{\omega }}\left\Vert \Psi
_{J}\right\Vert _{W_{\func{dyad}}^{-s}\left( \omega \right) }\ ,
\label{piv bound}
\end{equation}%
for any function $\varphi $ with $\left\vert \varphi \right\vert \leq 1$.
\end{lemma}

We also recall from \cite[Lemma 33]{Saw6} the following Poisson estimate,
that is a straightforward extension of the case $m=1$ due to Nazarov, Treil
and Volberg in \cite{NTV4}.

\begin{lemma}
\label{Poisson inequality}Fix $m\geq 1$. Suppose that $J\subset I\subset K$
and that $\limfunc{dist}\left( J,\partial I\right) >2\sqrt{n}\ell \left(
J\right) ^{\varepsilon }\ell \left( I\right) ^{1-\varepsilon }$. Then%
\begin{equation}
\mathrm{P}_{m}^{\alpha }(J,\sigma \mathbf{1}_{K\setminus I})\lesssim \left( 
\frac{\ell \left( J\right) }{\ell \left( I\right) }\right) ^{m-\varepsilon
\left( n+m-\alpha \right) }\mathrm{P}_{m}^{\alpha }(I,\sigma \mathbf{1}%
_{K\setminus I}).  \label{e.Jsimeq}
\end{equation}
\end{lemma}

We now give Sobolev modifications to several known arguments. The next lemma
was proved in \cite{RaSaWi} for $s=0$.

\begin{lemma}
\label{mono}Let $0\leq \alpha <n$, $\kappa \in \mathbb{N}$ and $0<\delta <1$%
. Suppose that$\ I$ and $J$ are cubes in $\mathbb{R}^{n}$ such that $%
J\subset 2J\subset I$, and that $\mu $ is a signed measure on $\mathbb{R}%
^{n} $ supported outside $I$. Finally suppose that $T^{\alpha }$ is a smooth
fractional singular integral on $\mathbb{R}^{n}$ with kernel $K^{\alpha
}\left( x,y\right) =K_{y}^{\alpha }\left( x\right) $, and that $\omega $ is
a locally finite positive Borel measure on $\mathbb{R}^{n}$. Then%
\begin{equation}
\left\Vert \bigtriangleup _{J;\kappa }^{\omega }T^{\alpha }\mu \right\Vert
_{W_{\func{dyad}}^{s}\left( \omega \right) }^{2}\lesssim \Phi _{\kappa
,s}^{\alpha }\left( J,\mu \right) ^{2}+\Psi _{\kappa ,s}^{\alpha }\left(
J,\left\vert \mu \right\vert \right) ^{2},  \label{estimate}
\end{equation}%
where for a measure $\nu $,%
\begin{eqnarray*}
&&\Phi _{\kappa ,s}^{\alpha }\left( J,\nu \right) ^{2}\equiv
\sum_{\left\vert \beta \right\vert =\kappa }\left\vert \int_{\mathbb{R}%
^{n}}\left( K_{y}^{\alpha }\right) ^{\left( \kappa \right) }\left(
m_{J}^{\kappa }\right) d\nu \left( y\right) \right\vert ^{2}\left\Vert
\bigtriangleup _{J;\kappa }^{\omega }x^{\beta }\right\Vert _{W_{\func{dyad}%
}^{s}\left( \omega \right) }^{2}\ , \\
&&\Psi _{\kappa ,s}^{\alpha }\left( J,\left\vert \nu \right\vert \right)
^{2}\equiv \left( \frac{\mathrm{P}_{\kappa +\delta }^{\alpha }\left(
J,\left\vert \nu \right\vert \right) }{\left\vert J\right\vert ^{\frac{%
\kappa }{n}}}\right) ^{2}\left\Vert \left\vert x-m_{J}^{\kappa }\right\vert
^{\kappa }\right\Vert _{W_{\func{dyad}}^{s}\left( \mathbf{1}_{J}\omega
\right) }^{2}\frac{\left\Vert \left\vert \mathbf{h}_{J;\kappa }^{\omega
}\right\vert \right\Vert _{W_{\func{dyad}}^{-s}\left( \omega \right) }^{2}}{%
\ell \left( J\right) ^{2s}}\ , \\
&&\text{where }m_{J}^{\kappa }\in J\text{ satisfies }\left\Vert \left\vert
x-m_{J}^{\kappa }\right\vert ^{\kappa }\right\Vert _{W_{\func{dyad}%
}^{s}\left( \mathbf{1}_{J}\omega \right) }^{2}=\inf_{m\in J}\left\Vert
\left\vert x-m\right\vert ^{\kappa }\right\Vert _{W_{\func{dyad}}^{s}\left( 
\mathbf{1}_{J}\omega \right) }^{2}.
\end{eqnarray*}
\end{lemma}

\begin{remark}
Note that when $s=0$, we have $\frac{\left\Vert \left\vert \mathbf{h}%
_{J;\kappa }^{\omega }\right\vert \right\Vert _{W_{\func{dyad}}^{-s}\left(
\omega \right) }^{2}}{\ell \left( J\right) ^{2s}}=\frac{\left\Vert \mathbf{h}%
_{J;\kappa }^{\omega }\right\Vert _{L^{2}\left( \omega \right) }^{2}}{1}=1$,
and so the above inequality becomes the familiar Monotonicity Lemma. For $s$
close to zero, Lemma \ref{mod Alpert} shows that $\frac{\left\Vert
\left\vert \mathbf{h}_{J;\kappa }^{\omega }\right\vert \right\Vert _{W_{%
\func{dyad}}^{-s}\left( \omega \right) }^{2}}{\ell \left( J\right) ^{2s}}%
\lesssim 1$, which gives the same familiar form.
\end{remark}

\begin{proof}[Proof of Lemma \protect\ref{mono}]
The proof is an easy adaptation of the one-dimensional proof in \cite{RaSaWi}%
, which was in turn adapted from the proofs in \cite{LaWi} and \cite{SaShUr7}%
, but using a $\kappa ^{th}$ order Taylor expansion instead of a first order
expansion on the kernel $\left( K_{y}^{\alpha }\right) \left( x\right)
=K^{\alpha }\left( x,y\right) $. Due to the importance of this lemma, as
explained above, we repeat the short argument.

Let $\left\{ h_{J;\kappa }^{\omega ,a}\right\} _{a\in \Gamma _{J,n,\kappa }}$
be an orthonormal basis of $L_{J;\kappa }^{2}\left( \omega \right) $
consisting of Alpert functions as above. Now we use the Calder\'{o}n-Zygmund
smoothness estimate (\ref{sizeandsmoothness'}), together with Taylor's
formula 
\begin{eqnarray*}
K_{y}^{\alpha }\left( x\right) &=&\limfunc{Tay}\left( K_{y}^{\alpha }\right)
\left( x,c\right) +\frac{1}{\kappa !}\sum_{\left\vert \beta \right\vert
=\kappa }\left( K_{y}^{\alpha }\right) ^{\left( \beta \right) }\left( \theta
\left( x,c\right) \right) \left( x-c\right) ^{\beta }; \\
\limfunc{Tay}\left( K_{y}^{\alpha }\right) \left( x,c\right) &\equiv
&K_{y}^{\alpha }\left( c\right) +\left[ \left( x-c\right) \cdot \nabla %
\right] K_{y}^{\alpha }\left( c\right) +...+\frac{1}{\left( \kappa -1\right)
!}\left[ \left( x-c\right) \cdot \nabla \right] ^{\kappa -1}K_{y}^{\alpha
}\left( c\right) ,
\end{eqnarray*}%
and the vanishing means of the Alpert functions $h_{J;\kappa }^{\omega ,a}$
for $a\in \Gamma _{J,n,\kappa }$, to obtain 
\begin{eqnarray*}
&&\left\langle T^{\alpha }\mu ,h_{J;\kappa }^{\omega ,a}\right\rangle
_{L^{2}\left( \omega \right) }=\int_{\mathbb{R}^{n}}\left\{ \int_{\mathbb{R}%
^{n}}K^{\alpha }\left( x,y\right) h_{J;\kappa }^{\omega ,a}\left( x\right)
d\omega \left( x\right) \right\} d\mu \left( y\right) =\int_{\mathbb{R}%
^{n}}\left\langle K_{y}^{\alpha },h_{J;\kappa }^{\omega ,a}\right\rangle
_{L^{2}\left( \omega \right) }d\mu \left( y\right) \\
&=&\int_{\mathbb{R}^{n}}\left\langle K_{y}^{\alpha }\left( x\right) -%
\limfunc{Tay}\left( K_{y}^{\alpha }\right) \left( x,m_{J}^{\kappa }\right)
,h_{J;\kappa }^{\omega ,a}\left( x\right) \right\rangle _{L^{2}\left( \omega
\right) }d\mu \left( y\right) \\
&=&\int_{\mathbb{R}^{n}}\left\langle \frac{1}{\kappa !}\sum_{\left\vert
\beta \right\vert =\kappa }\left( K_{y}^{\alpha }\right) ^{\left( \beta
\right) }\left( \theta \left( x,m_{J}^{\kappa }\right) \right) \left(
x-m_{J}^{\kappa }\right) ^{\beta },h_{J;\kappa }^{\omega ,a}\left( x\right)
\right\rangle _{L^{2}\left( \omega \right) }d\mu \left( y\right) \ \ \ \ \ 
\text{(some }\theta \left( x,m_{J}^{\kappa }\right) \in J\text{) } \\
&=&\sum_{\left\vert \beta \right\vert =\kappa }\left\langle \left[ \int_{%
\mathbb{R}^{n}}\frac{1}{\kappa !}\left( K_{y}^{\alpha }\right) ^{\left(
\beta \right) }\left( m_{J}^{\kappa }\right) d\mu \left( y\right) \right]
\left( x-m_{J}^{\kappa }\right) ^{\beta },h_{J;\kappa }^{\omega
,a}\right\rangle _{L^{2}\left( \omega \right) } \\
&&+\sum_{\left\vert \beta \right\vert =\kappa }\left\langle \left[ \int_{%
\mathbb{R}^{n}}\frac{1}{\kappa !}\left[ \left( K_{y}^{\alpha }\right)
^{\left( \beta \right) }\left( \theta \left( x,m_{J}^{\kappa }\right)
\right) -\sum_{\left\vert \beta \right\vert =\kappa }\left( K_{y}^{\alpha
}\right) ^{\left( \beta \right) }\left( m_{J}^{\kappa }\right) \right] d\mu
\left( y\right) \right] \left( x-m_{J}^{\kappa }\right) ^{\beta
},h_{J;\kappa }^{\omega ,a}\right\rangle _{L^{2}\left( \omega \right) }\ .
\end{eqnarray*}%
Then using that $\int_{\mathbb{R}^{n}}\left( K_{y}^{\alpha }\right) ^{\left(
\beta \right) }\left( m_{J}^{\kappa }\right) d\mu \left( y\right) $ is
independent of $x\in J$, and that $\left\langle \left( x-m_{J}^{\kappa
}\right) ^{\beta },\mathbf{h}_{J;\kappa }^{\omega }\right\rangle
_{L^{2}\left( \omega \right) }=\left\langle x^{\beta },\mathbf{h}_{J;\kappa
}^{\omega }\right\rangle _{L^{2}\left( \omega \right) }$ by moment vanishing
of the Alpert wavelets, we can continue with 
\begin{eqnarray*}
&&\left\langle T^{\alpha }\mu ,h_{J;\kappa }^{\omega ,a}\right\rangle
_{L^{2}\left( \omega \right) }=\sum_{\left\vert \beta \right\vert =\kappa } 
\left[ \int_{\mathbb{R}^{n}}\frac{1}{\kappa !}\left( K_{y}^{\alpha }\right)
^{\left( \beta \right) }\left( m_{J}^{\kappa }\right) d\mu \left( y\right) %
\right] \cdot \left\langle x^{\beta },h_{J;\kappa }^{\omega ,a}\right\rangle
_{L^{2}\left( \omega \right) } \\
&&\ \ \ \ \ +\frac{1}{\kappa !}\sum_{\left\vert \beta \right\vert =\kappa
}\left\langle \left[ \int_{\mathbb{R}^{n}}\left[ \left( K_{y}^{\alpha
}\right) ^{\left( \beta \right) }\left( \theta \left( x,m_{J}^{\kappa
}\right) \right) -\left( K_{y}^{\alpha }\right) ^{\left( \beta \right)
}\left( m_{J}^{\kappa }\right) \right] d\mu \left( y\right) \right] \left(
x-m_{J}^{\kappa }\right) ^{\beta },h_{J;\kappa }^{\omega ,a}\right\rangle
_{L^{2}\left( \omega \right) }\ .
\end{eqnarray*}%
Hence%
\begin{eqnarray*}
&&\left\vert \left\langle T^{\alpha }\mu ,h_{J;\kappa }^{\omega
,a}\right\rangle _{L^{2}\left( \omega \right) }-\sum_{\left\vert \beta
\right\vert =\kappa }\left[ \int_{\mathbb{R}^{n}}\frac{1}{\kappa !}\left(
K_{y}^{\alpha }\right) ^{\left( \beta \right) }\left( m_{J}^{\kappa }\right)
d\mu \left( y\right) \right] \cdot \left\langle x^{\beta },h_{J;\kappa
}^{\omega ,a}\right\rangle _{L^{2}\left( \omega \right) }\right\vert \\
&\leq &\frac{1}{\kappa !}\sum_{\left\vert \beta \right\vert =\kappa
}\left\vert \left\langle \left[ \int_{\mathbb{R}^{n}}\sup_{\theta \in
J}\left\vert \left( K_{y}^{\alpha }\right) ^{\left( \beta \right) }\left(
\theta \right) -\left( K_{y}^{\alpha }\right) ^{\left( \beta \right) }\left(
m_{J}^{\kappa }\right) \right\vert d\left\vert \mu \right\vert \left(
y\right) \right] \left\vert x-m_{J}^{\kappa }\right\vert ^{\kappa
},\left\vert h_{J;\kappa }^{\omega ,a}\right\vert \right\rangle
_{L^{2}\left( \omega \right) }\right\vert \\
&\lesssim &\left\Vert C_{CZ}\frac{\mathrm{P}_{\kappa +\delta }^{\alpha
}\left( J,\left\vert \mu \right\vert \right) }{\left\vert J\right\vert
^{\kappa }}\left\vert x-m_{J}^{\kappa }\right\vert ^{\kappa }\right\Vert
_{W_{\func{dyad}}^{s}\left( \omega \right) }\left\Vert \left\vert
h_{J;\kappa }^{\omega ,a}\right\vert \right\Vert _{W_{\func{dyad}%
}^{-s}\left( \omega \right) } \\
&\lesssim &C_{CZ}\frac{\mathrm{P}_{\kappa +\delta }^{\alpha }\left(
J,\left\vert \mu \right\vert \right) }{\left\vert J\right\vert ^{\kappa }}%
\left\Vert \left\vert x-m_{J}^{\kappa }\right\vert ^{\kappa }\right\Vert
_{W_{\func{dyad}}^{s}\left( \mathbf{1}_{J}\omega \right) }\left\Vert
\left\vert h_{J;\kappa }^{\omega ,a}\right\vert \right\Vert _{W_{\func{dyad}%
}^{-s}\left( \omega \right) }
\end{eqnarray*}%
where in the last line we have used%
\begin{eqnarray*}
&&\int_{\mathbb{R}^{n}}\sup_{\theta \in J}\left\vert \left( K_{y}^{\alpha
}\right) ^{\left( \beta \right) }\left( \theta \right) -\left( K_{y}^{\alpha
}\right) ^{\left( \beta \right) }\left( m_{J}^{\kappa }\right) \right\vert
d\left\vert \mu \right\vert \left( y\right) \\
&\lesssim &C_{CZ}\int_{\mathbb{R}^{n}}\left( \frac{\left\vert J\right\vert }{%
\left\vert y-c_{J}\right\vert }\right) ^{\delta }\frac{d\left\vert \mu
\right\vert \left( y\right) }{\left\vert y-c_{J}\right\vert ^{\kappa
+1-\alpha }}=C_{CZ}\frac{\mathrm{P}_{\kappa +\delta }^{\alpha }\left(
J,\left\vert \mu \right\vert \right) }{\left\vert J\right\vert ^{\kappa }}.
\end{eqnarray*}

Thus with $v_{J}^{\beta }=\frac{1}{\kappa !}\int_{\mathbb{R}^{n}}\left(
K_{y}^{\alpha }\right) ^{\left( \beta \right) }\left( m_{J}^{\kappa }\right)
d\mu \left( y\right) $, and noting that the functions $\left\{ v_{J}^{\beta
}h_{J;\kappa }^{\omega ,a}\right\} _{a\in \Gamma _{J,n,\kappa }}$ are
orthonormal in $a\in \Gamma _{J,n,\kappa }$ for each $\beta $ and $J$, we
have%
\begin{eqnarray*}
\left\vert v_{J}^{\beta }\left\langle x^{\beta },h_{J;\kappa }^{\omega
,a}\right\rangle _{L^{2}\left( \omega \right) }\right\vert ^{2}
&=&\sum_{a\in \Gamma _{J,n,\kappa }}\left\vert \left\langle x^{\beta
},v_{J}^{\beta }\cdot h_{J;\kappa }^{\omega ,a}\right\rangle _{L^{2}\left(
\omega \right) }\right\vert ^{2}=\left\Vert \bigtriangleup _{J;\kappa
}^{\omega }v_{J}^{\beta }x^{\beta }\right\Vert _{L^{2}\left( \omega \right)
}^{2} \\
&=&\left\vert v_{J}^{\beta }\right\vert ^{2}\left\Vert \bigtriangleup
_{J;\kappa }^{\omega }x^{\beta }\right\Vert _{L^{2}\left( \omega \right)
}^{2}=\left\vert v_{J}^{\beta }\right\vert ^{2}\ell \left( J\right)
^{2s}\left\Vert \bigtriangleup _{J;\kappa }^{\omega }x^{\beta }\right\Vert
_{W_{\func{dyad}}^{s}\left( \omega \right) }^{2}\ ,
\end{eqnarray*}%
and hence%
\begin{eqnarray*}
&&\left\Vert \bigtriangleup _{J;\kappa }^{\omega }T^{\alpha }\mu \right\Vert
_{W_{\func{dyad}}^{s}\left( \omega \right) }^{2}=\left\vert \widehat{%
T^{\alpha }\mu }\left( J;\kappa \right) \right\vert ^{2}\ell \left( J\right)
^{-2s}=\ell \left( J\right) ^{-2s}\sum_{a\in \Gamma _{J,n,\kappa
}}\left\vert \left\langle T^{\alpha }\mu ,h_{J;\kappa }^{\omega
,a}\right\rangle _{L^{2}\left( \omega \right) }\right\vert ^{2} \\
&=&\ell \left( J\right) ^{-2s}\sum_{\left\vert \beta \right\vert =\kappa
}\left\vert v_{J}^{\beta }\right\vert ^{2}\ell \left( J\right)
^{2s}\left\Vert \bigtriangleup _{J;\kappa }^{\omega }x^{\kappa }\right\Vert
_{W_{\func{dyad}}^{s}\left( \omega \right) }^{2} \\
&&+O\left( \frac{\mathrm{P}_{\kappa +\delta }^{\alpha }\left( J,\left\vert
\mu \right\vert \right) }{\left\vert J\right\vert ^{\frac{\kappa }{n}}}%
\right) ^{2}\ell \left( J\right) ^{-2s}\left\Vert \left\vert x-m_{J}^{\kappa
}\right\vert ^{\kappa }\right\Vert _{W_{\func{dyad}}^{s}\left( \mathbf{1}%
_{J}\omega \right) }^{2}\sum_{a\in \Gamma _{J,n,\kappa }}\left\Vert
\left\vert h_{J;\kappa }^{\omega ,a}\right\vert \right\Vert _{W_{\func{dyad}%
}^{-s}\left( \omega \right) }^{2}.
\end{eqnarray*}

Thus we conclude that%
\begin{eqnarray*}
\left\Vert \bigtriangleup _{J;\kappa }^{\omega }T^{\alpha }\mu \right\Vert
_{W_{\func{dyad}}^{s}\left( \omega \right) }^{2} &\leq
&C_{1}\sum_{\left\vert \beta \right\vert =\kappa }\left\vert \frac{1}{\kappa
!}\int_{\mathbb{R}^{n}}\left( K_{y}^{\alpha }\right) ^{\left( \beta \right)
}\left( m_{J}\right) d\mu \left( y\right) \right\vert ^{2}\left\Vert
\bigtriangleup _{J;\kappa }^{\omega }x^{\kappa }\right\Vert _{W_{\func{dyad}%
}^{s}\left( \omega \right) }^{2} \\
&&+C_{2}\left( \frac{\mathrm{P}_{\kappa +\delta }^{\alpha }\left(
J,\left\vert \mu \right\vert \right) }{\left\vert J\right\vert ^{\frac{%
\kappa }{n}}}\right) ^{2}\left\Vert \left\vert x-m_{J}^{\kappa }\right\vert
^{\kappa }\right\Vert _{W_{\func{dyad}}^{s}\left( \mathbf{1}_{J}\omega
\right) }^{2}\frac{\left\Vert \left\vert \mathbf{h}_{J;\kappa }^{\omega
}\right\vert \right\Vert _{W_{\func{dyad}}^{-s}\left( \omega \right) }^{2}}{%
\ell \left( J\right) ^{2s}}\ ,
\end{eqnarray*}%
where $\mathbf{h}_{J;\kappa }^{\omega }\equiv \left\{ h_{J;\kappa }^{\omega
,a}\right\} _{a\in \Gamma _{J,n,\kappa }}$ and%
\begin{equation*}
\sum_{\left\vert \beta \right\vert =\kappa }\left\vert \frac{1}{\kappa !}%
\int_{\mathbb{R}^{n}}\left( K_{y}^{\alpha }\right) ^{\left( \beta \right)
}\left( m_{J}\right) d\mu \left( y\right) \right\vert ^{2}\lesssim \left( 
\frac{\mathrm{P}_{\kappa }^{\alpha }\left( J,\left\vert \mu \right\vert
\right) }{\left\vert J\right\vert ^{\frac{\kappa }{n}}}\right) ^{2}.
\end{equation*}
\end{proof}

The following Energy Lemma follows from the above Monotonicity Lemma in a
standard way, see e.g. \cite{SaShUr7}. Recall that for a subset $\mathcal{J}%
\subset \mathcal{D}$, and for a cube $J\in \mathcal{D}$, there are
projections $\mathsf{P}_{\mathcal{J}}^{\omega }\equiv \sum_{J^{\prime }\in 
\mathcal{J}}\bigtriangleup _{J^{\prime };\kappa }^{\omega }$and $\mathsf{P}%
_{J}^{\omega }\equiv \sum_{J^{\prime }\in \mathcal{D}:\ J^{\prime }\subset
J}\bigtriangleup _{J^{\prime };\kappa }^{\omega }$. Recall also that $%
\mathbf{h}_{J;\kappa }^{\omega }\equiv \left\{ h_{J;\kappa }^{\omega
,a}\right\} _{a\in \Gamma _{J,n,\kappa }}$ is the vector of Alpert wavelets
associated with the cube $J$.

\begin{lemma}[\textbf{Energy Lemma}]
\label{ener}Fix $\kappa \geq 1$ and a locally finite positive Borel measure $%
\omega $. Let $J\ $be a cube in $\mathcal{D}$. Let $\Psi _{J}\in W_{\func{%
dyad}}^{-s}\left( \omega \right) $ be supported in $J$ with vanishing $%
\omega $-means up to order less than $\kappa $. Let $\nu $ be a positive
measure supported in $\mathbb{R}^{n}\setminus \gamma J$ with $\gamma >1$.
Let $T^{\alpha }$ be a smooth $\alpha $-fractional singular integral
operator with $0\leq \alpha <n$. Then we have the `pivotal' bound%
\begin{equation}
\left\vert \left\langle T^{\alpha }\left( \varphi \nu \right) ,\Psi
_{J}\right\rangle _{L^{2}\left( \omega \right) }\right\vert \lesssim
C_{\gamma }\mathrm{P}_{\kappa }^{\alpha }\left( J,\nu \right) \ell \left(
J\right) ^{-s}\sqrt{\left\vert J\right\vert _{\omega }}\left\Vert \Psi
_{J}\right\Vert _{W_{\func{dyad}}^{-s}\left( \omega \right) }\frac{%
\left\Vert \left\vert \mathbf{h}_{J;\kappa }^{\omega }\right\vert
\right\Vert _{W_{\func{dyad}}^{-s}\left( \omega \right) }^{2}}{\ell \left(
J\right) ^{2s}}\ ,  \label{piv lemma}
\end{equation}%
for any function $\varphi $ with $\left\vert \varphi \right\vert \leq 1$.
\end{lemma}

\section{The strong $\protect\kappa $-pivotal corona decomposition}

To set the stage for control of the stopping form below in the absence of
the energy condition, we construct the \emph{strong }$\kappa $\emph{-pivotal}
corona decomposition for $f\in W_{\func{dyad}}^{s}\left( \mu \right) $, in
analogy with the energy version for $L^{2}\left( \sigma \right) $ and $%
L^{2}\left( \omega \right) $ used in the two part paper \cite{LaSaShUr3},%
\cite{Lac} and in \cite{SaShUr7}.

Fix $\gamma >1$ and define $\mathcal{G}_{0}=\left\{ F_{1}^{0}\right\} $ to
consist of the single cube $F_{1}^{0}$, and define the first generation $%
\mathcal{G}_{1}=\left\{ F_{k}^{1}\right\} _{k}$ of $\kappa $\emph{-pivotal
stopping children} of $F_{1}^{0}$ to be the \emph{maximal} dyadic subcubes $%
I $ of $F_{0}$ satisfying%
\begin{equation*}
\mathrm{P}_{\kappa }^{\alpha }\left( I,\mathbf{1}_{F_{1}^{0}}\sigma \right)
^{2}\left\vert I\right\vert _{\omega }\geq \gamma \left\vert I\right\vert
_{\sigma }.
\end{equation*}%
Then define the second generation $\mathcal{G}_{2}=\left\{ F_{k}^{2}\right\}
_{k}$ of CZ $\kappa $-pivotal\emph{\ }$s$-stopping children of $F_{1}^{0}$
to be the \emph{maximal} dyadic subcubes $I$ of some $F_{k}^{1}\in \mathcal{G%
}_{1}$ satisfying%
\begin{equation*}
\mathrm{P}_{\kappa }^{\alpha }\left( I,\mathbf{1}_{F_{k}^{1}}\sigma \right)
^{2}\left\vert I\right\vert _{\omega }\geq \gamma \left\vert I\right\vert
_{\sigma }.
\end{equation*}%
Continue by recursion to define $\mathcal{G}_{n}$ for all $n\geq 0$, and
then set 
\begin{equation*}
\mathcal{F\equiv }\dbigcup\limits_{n=0}^{\infty }\mathcal{G}_{n}=\left\{
F_{k}^{n}:n\geq 0,k\geq 1\right\}
\end{equation*}%
to be the set of all $\kappa $-pivotal stopping intervals in $F_{1}^{0}$
obtained in this way.

\subsection{Carleson condition for stopping cubes and corona controls}

The $\varepsilon $-strong $\sigma $-Carleson condition (\ref{eps Carleson})
holds for a $\left( \beta ,\sigma \right) $-full grid $\mathcal{F}$ by the
usual calculation,%
\begin{equation*}
\sum_{F^{\prime }\in \mathfrak{C}_{\mathcal{F}}\left( F\right) }\left( \frac{%
\ell \left( F\right) }{\ell \left( F^{\prime }\right) }\right) ^{\varepsilon
}\left\vert F^{\prime }\right\vert _{\sigma }\leq \frac{1}{\gamma }\left\{
\sum_{F^{\prime }\in \mathfrak{C}_{\mathcal{F}}\left( F\right) }\mathrm{P}%
_{\kappa }^{\alpha }\left( F^{\prime },\mathbf{1}_{F}\sigma \right)
^{2}\left( \frac{\ell \left( F\right) }{\ell \left( F^{\prime }\right) }%
\right) ^{\varepsilon }\left\vert F^{\prime }\right\vert _{\omega }\right\}
\leq \frac{1}{\gamma }\mathcal{V}_{2,\beta ,\varepsilon }^{\alpha ,\kappa
}\left( \sigma ,\omega \right) \left\vert F\right\vert _{\sigma }\ .
\end{equation*}%
Indeed, set $\mathfrak{C}_{\mathcal{F}}^{\left( \ell \right) }\left(
F\right) $ to be the $\ell ^{th}$ generation of $\mathcal{F}$-subcubes of $F$%
, and define $\mathcal{F}\left( F\right) =\dbigcup\limits_{\ell =0}^{\infty }%
\mathfrak{C}_{\mathcal{F}}^{\left( \ell \right) }\left( F\right) $ to be the
collection of all $\mathcal{F}$-subcubes of $F$. Then if $\frac{\mathcal{V}%
_{2,\beta ,\varepsilon }^{\alpha ,\kappa }+1}{\gamma }<\frac{1}{2}$ and $%
\mathcal{F}$ is a $\left( \beta ,\sigma \right) $-full grid, we have the $%
\varepsilon $-strong $\sigma $-Carleson condition,%
\begin{equation}
\sum_{F^{\prime }\in \mathcal{F}\left( F\right) }\left( \frac{\ell \left(
F\right) }{\ell \left( F^{\prime }\right) }\right) ^{\varepsilon }\left\vert
F^{\prime }\right\vert _{\sigma }=\sum_{\ell =0}^{\infty }\sum_{F^{\prime
}\in \mathfrak{C}_{\mathcal{F}}^{\left( \ell \right) }\left( F\right)
}\left( \frac{\ell \left( F\right) }{\ell \left( F^{\prime }\right) }\right)
^{\varepsilon }\left\vert F^{\prime }\right\vert _{\sigma }\leq \sum_{\ell
=0}^{\infty }\left( \frac{\mathcal{V}_{2,\beta ,\varepsilon }^{\alpha
,\kappa }+1}{\gamma }\right) ^{\ell }\left\vert F\right\vert _{\sigma }\leq
2\left\vert F\right\vert _{\sigma }\ .  \label{sigma Car}
\end{equation}

Using Lemma \ref{q lemma}, this Carleson condition delivers a basic method
of control by quasiorthogonality (see \cite{LaSaShUr3} and \cite{SaShUr7}
for the case $s=0$),%
\begin{equation}
\sum_{F\in \mathcal{F}}\left\vert F\right\vert _{\sigma }\left( \ell \left(
F\right) ^{-s}E_{F}^{\sigma }f\right) ^{2}\lesssim \left\Vert f\right\Vert
_{W_{\func{dyad}}^{s}\left( \sigma \right) }^{2},  \label{quasi orth}
\end{equation}%
provided $\mathcal{F}$ is a $\beta $-full grid and $\left\vert s\right\vert $
is sufficiently small depending on $\beta $, which is used repeatedly in
conjunction with orthogonality of Sobolev projections $\bigtriangleup
_{J;\kappa }^{\omega }g$,%
\begin{equation}
\sum_{J\in \mathcal{D}}\left\Vert \bigtriangleup _{J;\kappa }^{\omega
}g\right\Vert _{W_{\func{dyad}}^{-s}\left( \omega \right) }^{2}=\left\Vert
g\right\Vert _{W_{\func{dyad}}^{-s}\left( \omega \right) }^{2}.  \label{orth}
\end{equation}

Recall that for an arbitrary grid $\mathcal{F}\subset \mathcal{D}$, the
coronas $\mathcal{C}_{F}$ for $F\in \mathcal{F}$ are defined by 
\begin{equation*}
\mathcal{C}_{F}\equiv \left\{ I\in \mathcal{D}:I\subset F\text{ and }%
I\not\subset F^{\prime }\text{ for any }F^{\prime }\in \mathcal{F}\text{
with }F^{\prime }\varsubsetneqq F\right\} .
\end{equation*}%
We have, from the definition of the stopping times $\mathcal{F}$ above, that
the $\left( \beta ,\varepsilon \right) $-strong $\kappa $-pivotal control
holds for $\beta >0$ sufficiently small, i.e.%
\begin{equation}
\mathrm{P}_{\kappa }^{\alpha }\left( I,\mathbf{1}_{F}\sigma \right)
^{2}\left( \frac{\ell \left( F\right) }{\ell \left( I\right) }\right)
^{\varepsilon }\left\vert I\right\vert _{\omega }<\Gamma \left\vert
I\right\vert _{\sigma },\ \ \ \ \ I\in \mathcal{C}_{F}\text{ and }F\in 
\mathcal{F}.  \label{energy control}
\end{equation}

\section{Reduction of the proof to local forms}

To prove Theorem \ref{pivotal theorem}, we begin by proving the bilinear
form bound,%
\begin{equation*}
\left\vert \left\langle T_{\sigma }^{\alpha }f,g\right\rangle _{\omega
}\right\vert \lesssim \left( \sqrt{A_{2}^{\alpha }\left( \sigma ,\omega
\right) }+\mathfrak{TR}_{T^{\alpha }}^{\left( \kappa \right) }\left( \sigma
,\omega \right) +\mathfrak{TR}_{\left( T^{\alpha }\right) ^{\ast }}^{\left(
\kappa \right) }\left( \omega ,\sigma \right) +\mathcal{V}_{2,\beta
,\varepsilon }^{\alpha ,\kappa }+\mathcal{V}_{2,\beta ,\varepsilon }^{\alpha
,\kappa ,\ast }\right) \ \left\Vert f\right\Vert _{W_{\limfunc{dyad}%
}^{s}\left( \sigma \right) }\left\Vert g\right\Vert _{W_{\limfunc{dyad}%
}^{-s}\left( \omega \right) },
\end{equation*}%
for $0<\beta ,\varepsilon <1$ sufficiently small. Following the weighted
Haar expansions of Nazarov, Treil and Volberg, we write $f$ and $g$ in
weighted Alpert wavelet expansions,%
\begin{equation}
\left\langle T_{\sigma }^{\alpha }f,g\right\rangle _{\omega }=\left\langle
T_{\sigma }^{\alpha }\left( \sum_{I\in \mathcal{D}}\bigtriangleup _{I;\kappa
_{1}}^{\sigma }f\right) ,\left( \sum_{J\in \mathcal{D}}\bigtriangleup
_{J;\kappa _{2}}^{\omega }g\right) \right\rangle _{\omega }.
\label{expansion}
\end{equation}%
Then following \cite{SaShUr7} and many others, the $L^{2}$ inner product in (%
\ref{expansion}) can be expanded as%
\begin{equation*}
\left\langle T_{\sigma }^{\alpha }f,g\right\rangle _{\omega }=\left\langle
T_{\sigma }^{\alpha }\left( \sum_{I\in \mathcal{D}}\bigtriangleup _{I;\kappa
_{1}}^{\sigma }f\right) ,\left( \sum_{J\in \mathcal{D}}\bigtriangleup
_{J;\kappa _{2}}^{\omega }g\right) \right\rangle _{\omega }=\sum_{I\in 
\mathcal{D}\ \text{and }J\in \mathcal{D}}\left\langle T_{\sigma }^{\alpha
}\left( \bigtriangleup _{I;\kappa _{1}}^{\sigma }f\right) ,\left(
\bigtriangleup _{J;\kappa _{2}}^{\omega }g\right) \right\rangle _{\omega }\ .
\end{equation*}%
Then the sum is further decomposed by first the \emph{Cube Size Splitting},
then using the \emph{Shifted Corona Decomposition}, according to the \emph{%
Canonical Splitting}. We assume the reader is familiar with the notation and
arguments in the first eight sections of \cite{SaShUr7}. The $n$-dimensional
decompositions used in \cite{SaShUr7} are in spirit the same as the
one-dimensional decompositions in \cite{LaSaShUr3}, as well as the $n$%
-dimensional decompositions in \cite{LaWi}, but differ in significant
details.

A fundamental result of Nazarov, Treil and Volberg \cite{NTV4} is that all
the cubes $I$ and $J$ appearing in the bilinear form above may be assumed to
be $\left( r,\varepsilon \right) -\limfunc{good}$, where a dyadic interval $%
K $ is $\left( r,\varepsilon \right) -\limfunc{good}$, or simply $\limfunc{%
good}$, if for \emph{every} dyadic supercube $L$ of $K$, it is the case that 
\textbf{either} $K$ has side length at least $2^{1-r}$ times that of $L$, 
\textbf{or} $K\Subset _{\left( r,\varepsilon \right) }L$. We say that a
dyadic cube $K$ is $\left( r,\varepsilon \right) $-\emph{deeply embedded} in
a dyadic cube $L$, or simply $r$\emph{-deeply embedded} in $L$, which we
write as $K\Subset _{r,\varepsilon }L$, when $K\subset L$ and both 
\begin{eqnarray}
\ell \left( K\right) &\leq &2^{-r}\ell \left( L\right) ,
\label{def deep embed} \\
\limfunc{dist}\left( K,\dbigcup\limits_{L^{\prime }\in \mathfrak{C}_{%
\mathcal{D}}L}\partial L^{\prime }\right) &\geq &2\ell \left( K\right)
^{\varepsilon }\ell \left( L\right) ^{1-\varepsilon }.  \notag
\end{eqnarray}

Here is a brief schematic diagram as in \cite{AlSaUr}, summarizing the
shifted corona decompositions as used in \cite{AlSaUr} and \cite{SaShUr7}
for Alpert and Haar wavelet expansions of $f$ and $g$. We first introduce
parameters as in \cite{AlSaUr} and \cite{SaShUr7}. We will choose $%
\varepsilon >0$ sufficiently small later in the argument, and then $r$ must
be chosen sufficiently large depending on $\varepsilon $ in order to reduce
matters to $\left( r,\varepsilon \right) -\func{good}$ functions by the
Nazarov, Treil and Volberg argument.

\begin{definition}
\label{def parameters}The parameters $\tau $ and $\rho $ are fixed to
satisfy 
\begin{equation*}
\tau >r\text{ and }\rho >r+\tau ,
\end{equation*}%
where $r$ is the goodness parameter already fixed.
\end{definition}

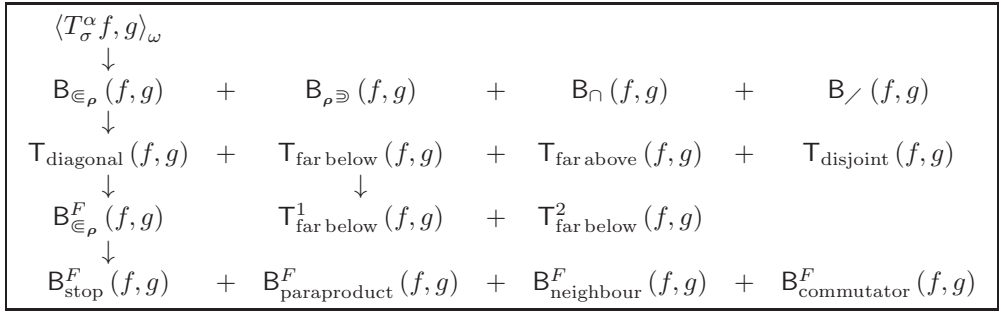
\begin{equation*}
\fbox{$%
\begin{array}{ccccccc}
\left\langle T_{\sigma }^{\alpha }f,g\right\rangle _{\omega } &  &  &  &  & 
&  \\ 
\downarrow &  &  &  &  &  &  \\ 
\mathsf{B}_{\Subset _{\mathbf{\rho }}}\left( f,g\right) & + & \mathsf{B}_{_{%
\mathbf{\rho }}\Supset }\left( f,g\right) & + & \mathsf{B}_{\cap }\left(
f,g\right) & + & \mathsf{B}_{\diagup }\left( f,g\right) \\ 
\downarrow &  &  &  &  &  &  \\ 
\mathsf{T}_{\limfunc{diagonal}}\left( f,g\right) & + & \mathsf{T}_{\limfunc{%
far}\limfunc{below}}\left( f,g\right) & + & \mathsf{T}_{\limfunc{far}%
\limfunc{above}}\left( f,g\right) & + & \mathsf{T}_{\limfunc{disjoint}%
}\left( f,g\right) \\ 
\downarrow &  & \downarrow &  &  &  &  \\ 
\mathsf{B}_{\Subset _{\mathbf{\rho }}}^{F}\left( f,g\right) &  & \mathsf{T}_{%
\limfunc{far}\limfunc{below}}^{1}\left( f,g\right) & + & \mathsf{T}_{%
\limfunc{far}\limfunc{below}}^{2}\left( f,g\right) &  &  \\ 
\downarrow &  &  &  &  &  &  \\ 
\mathsf{B}_{\func{stop}}^{F}\left( f,g\right) & + & \mathsf{B}_{\func{%
paraproduct}}^{F}\left( f,g\right) & + & \mathsf{B}_{\func{neighbour}%
}^{F}\left( f,g\right) & + & \mathsf{B}_{\limfunc{commutator}}^{F}\left(
f,g\right)%
\end{array}%
$}
\end{equation*}

\subsection{Cube size splitting}

The Nazarov, Treil and Volberg \emph{Cube Size Splitting} of the inner
product $\left\langle T_{\sigma }^{\alpha }f,g\right\rangle _{\omega }$
splits the pairs of cubes $\left( I,J\right) $ in a simultaneous Alpert
decomposition of $f$ and $g$ into four groups determined by relative
position, is given by%
\begin{eqnarray*}
\left\langle T_{\sigma }^{\alpha }f,g\right\rangle _{\omega }
&=&\dsum\limits_{I,J\in \mathcal{D}}\left\langle T_{\sigma }^{\alpha }\left(
\bigtriangleup _{I;\kappa }^{\sigma }f\right) ,\left( \bigtriangleup
_{J;\kappa }^{\omega }g\right) \right\rangle _{\omega } \\
&=&\dsum\limits_{\substack{ I,J\in \mathcal{D}  \\ J\Subset _{\rho
,\varepsilon }I}}\left\langle T_{\sigma }^{\alpha }\left( \bigtriangleup
_{I;\kappa }^{\sigma }f\right) ,\left( \bigtriangleup _{J;\kappa }^{\omega
}g\right) \right\rangle _{\omega }+\dsum\limits_{\substack{ I,J\in \mathcal{D%
}  \\ J_{\rho ,\varepsilon }\Supset I}}\left\langle T_{\sigma }^{\alpha
}\left( \bigtriangleup _{I;\kappa }^{\sigma }f\right) ,\left( \bigtriangleup
_{J;\kappa }^{\omega }g\right) \right\rangle _{\omega } \\
&&+\dsum\limits_{\substack{ I,J\in \mathcal{D}  \\ J\cap I=\emptyset \text{
and }\frac{\ell \left( J\right) }{\ell \left( I\right) }\notin \left[
2^{-\rho },2^{\rho }\right] }}\left\langle T_{\sigma }^{\alpha }\left(
\bigtriangleup _{I;\kappa }^{\sigma }f\right) ,\left( \bigtriangleup
_{J;\kappa }^{\omega }g\right) \right\rangle _{\omega } \\
&&+\dsum\limits_{\substack{ I,J\in \mathcal{D}  \\ 2^{-\rho }\leq \frac{\ell
\left( J\right) }{\ell \left( I\right) }\leq 2^{\rho }}}\left\langle
T_{\sigma }^{\alpha }\left( \bigtriangleup _{I;\kappa }^{\sigma }f\right)
,\left( \bigtriangleup _{J;\kappa }^{\omega }g\right) \right\rangle _{\omega
} \\
&=&\mathsf{B}_{\Subset _{\rho ,\varepsilon }}\left( f,g\right) +\mathsf{B}%
_{_{\rho ,\varepsilon }\Supset }\left( f,g\right) +\mathsf{B}_{\cap }\left(
f,g\right) +\mathsf{B}_{\diagup }\left( f,g\right) .
\end{eqnarray*}%
Note however that the assumption the cubes $I$ and $J$ are $\left(
r,\varepsilon \right) -\limfunc{good}$ remains in force throughout the proof.

We will now make use of the $\kappa $-cube testing and triple testing
constants, defined in (\ref{def Kappa polynomial'}) and (\ref{full testing}%
), to prove the following bound in the Sobolev setting, which in the case $%
s=0$ was proved in \cite[see Lemma 31]{Saw6} following the Nazarov, Treil
and Volberg arguments for Haar wavelets in \cite[see the proof of Lemma 7.1]%
{SaShUr7} (see also \cite{LaSaShUr3}),%
\begin{equation}
\left\vert \mathsf{B}_{\cap }\left( f,g\right) +\mathsf{B}_{\diagup }\left(
f,g\right) \right\vert \leq C\left( \mathfrak{T}_{T^{\alpha }}^{\kappa ,s}+%
\mathfrak{T}_{T^{\alpha ,\ast }}^{\kappa ,-s}+\mathcal{WBP}_{T^{\alpha
}}^{\left( \kappa _{1},\kappa _{2}\right) ,s}\left( \sigma ,\omega \right) +%
\sqrt{A_{2}^{\alpha }}\right) \left\Vert f\right\Vert _{W_{\limfunc{dyad}%
}^{s}\left( \sigma \right) }\left\Vert g\right\Vert _{W_{\limfunc{dyad}%
}^{-s}\left( \omega \right) },  \label{routine}
\end{equation}%
where if $\Omega _{\limfunc{dyad}}$ is the set of all dyadic grids, 
\begin{equation*}
\mathcal{WBP}_{T^{\alpha }}^{\left( \kappa _{1},\kappa _{2}\right) ,s}\left(
\sigma ,\omega \right) \equiv \sup_{\mathcal{D}\in \Omega }\sup_{\substack{ %
Q,Q^{\prime }\in \mathcal{D}  \\ Q\subset 3Q^{\prime }\setminus Q^{\prime }%
\text{ or }Q^{\prime }\subset 3Q\setminus Q}}\frac{\ell \left( Q\right)
^{s}\ell \left( Q^{\prime }\right) ^{-s}}{\sqrt{\left\vert Q\right\vert
_{\sigma }\left\vert Q^{\prime }\right\vert _{\omega }}}\sup_{\substack{ %
f\in \left( \mathcal{P}_{Q}^{\kappa _{1}}\right) _{\limfunc{norm}}\left(
\sigma \right)  \\ g\in \left( \mathcal{P}_{Q^{\prime }}^{\kappa
_{2}}\right) _{\limfunc{norm}}\left( \omega \right) }}\left\vert
\int_{Q^{\prime }}T_{\sigma }^{\alpha }\left( \mathbf{1}_{Q}f\right) \
gd\omega \right\vert <\infty
\end{equation*}%
is a weak boundedness constant that in the case $s=0$ was introduced in \cite%
{Saw6}. Here we will use the case $\kappa _{1}=\kappa _{2}=\kappa $.
However, we only use that this constant is removed in the final section
below using the following bound proved in \cite[see (6.25) in Subsection 6.7
and note that only triple testing is needed there by choosing $\ell \left(
Q^{\prime }\right) \leq \ell \left( Q\right) $ (using duality and $T^{\alpha
,\ast }$ if needed)]{Saw6}, and which holds also in the Sobolev setting
using Cauchy-Schwarz and triple testing, 
\begin{equation}
\mathcal{WBP}_{T^{\alpha }}^{\left( \kappa ,\kappa \right) ,s}\left( \sigma
,\omega \right) \leq C_{\kappa }\left( \mathfrak{TR}_{T^{\alpha }}^{\kappa
,s}\left( \sigma ,\omega \right) +\mathfrak{TR}_{T^{\alpha ,\ast }}^{\kappa
,-s}\left( \omega ,\sigma \right) \right) .  \label{unif bound'}
\end{equation}

In fact the stronger bound with absolute values inside the sums in (\ref%
{routine}) was proved in the case $s=0$ in the previous references,%
\begin{eqnarray}
&&\dsum\limits_{\substack{ I,J\in \mathcal{D}  \\ J\cap I=\emptyset \text{
and }\frac{\ell \left( J\right) }{\ell \left( I\right) }\notin \left[
2^{-\rho },2^{\rho }\right] }}\left\vert \left\langle T_{\sigma }^{\alpha
}\left( \bigtriangleup _{I;\kappa }^{\sigma }f\right) ,\left( \bigtriangleup
_{J;\kappa }^{\omega }g\right) \right\rangle _{\omega }\right\vert
+\dsum\limits_{\substack{ I,J\in \mathcal{D}  \\ 2^{-\rho }\leq \frac{\ell
\left( J\right) }{\ell \left( I\right) }\leq 2^{\rho }}}\left\vert
\left\langle T_{\sigma }^{\alpha }\left( \bigtriangleup _{I;\kappa }^{\sigma
}f\right) ,\left( \bigtriangleup _{J;\kappa }^{\omega }g\right)
\right\rangle _{\omega }\right\vert  \label{routine'} \\
&\leq &C\left( \mathfrak{T}_{T^{\alpha }}^{\kappa ,s}+\mathfrak{T}%
_{T^{\alpha ,\ast }}^{\kappa ,-s}+\mathcal{WBP}_{T^{\alpha }}^{\left( \kappa
,\kappa \right) ,s}\left( \sigma ,\omega \right) +\sqrt{A_{2}^{\alpha }}%
\right) \left\Vert f\right\Vert _{W_{\limfunc{dyad}}^{s}\left( \sigma
\right) }\left\Vert g\right\Vert _{W_{\limfunc{dyad}}^{-s}\left( \omega
\right) }.  \notag
\end{eqnarray}%
This bound will be useful later since it yields the same bound for the sum
of any subcollection of the index set, and for the convenience of the
reader, we prove (\ref{routine'}) below. Since the \emph{below} and \emph{%
above} forms $\mathsf{B}_{\Subset _{\mathbf{\rho },\varepsilon }}\left(
f,g\right) ,\mathsf{B}_{_{\mathbf{\rho },\varepsilon }\Supset }\left(
f,g\right) $ are symmetric, matters are then reduced to proving%
\begin{equation}
\left\vert \mathsf{B}_{\Subset _{\mathbf{\rho },\varepsilon }}\left(
f,g\right) \right\vert \lesssim \left( \mathfrak{T}_{T^{\alpha }}^{s}+%
\mathfrak{T}_{T^{\alpha ,\ast }}^{-s}+\sqrt{A_{2}^{\alpha }}\right)
\left\Vert f\right\Vert _{W_{\limfunc{dyad}}^{s}\left( \sigma \right)
}\left\Vert g\right\Vert _{W_{\limfunc{dyad}}^{-s}\left( \omega \right) }\ .
\label{red proving}
\end{equation}

We introduce some notation in order to prove (\ref{routine'}). For weighted
Alpert wavelet projections $\bigtriangleup _{I;\kappa }^{\sigma }$, we write
the projection $\mathbb{E}_{I^{\prime };\kappa }^{\sigma }\bigtriangleup
_{I;\kappa }^{\sigma }f$ onto the child $I^{\prime }\in \mathfrak{C}_{%
\mathcal{D}}\left( I\right) $ as $M_{I^{\prime };\kappa }^{\sigma }\mathbf{1}%
_{I_{\pm }}$, where $M_{I^{\prime };\kappa }^{\sigma }$ is a polynomial of
degree less than $\kappa $ restricted to $I^{\prime }$. Then we let $%
P_{I^{\prime };\kappa }^{\sigma }\equiv \frac{M_{I^{\prime };\kappa
}^{\sigma }}{\left\Vert M_{I^{\prime };\kappa }^{\sigma }\right\Vert
_{\infty }}$ be its normalization on $I^{\prime }$. From (\ref{analogue'})
we have the estimate,%
\begin{equation}
\left\Vert \mathbb{E}_{I^{\prime };\kappa }^{\sigma }\bigtriangleup
_{I;\kappa }^{\sigma }f\right\Vert _{\infty }\leq C\frac{\left\Vert
\bigtriangleup _{I;\kappa }^{\sigma }f\right\Vert _{L^{2}\left( \sigma
\right) }}{\sqrt{\left\vert I^{\prime }\right\vert _{\sigma }}}\approx C%
\frac{\left\vert \widehat{f_{\kappa }}\left( I\right) \right\vert }{\sqrt{%
\left\vert I^{\prime }\right\vert _{\sigma }}}.  \label{analogue''}
\end{equation}

\begin{proof}[Proof of (\protect\ref{routine'})]
To handle the second term in (\ref{routine'}) we first decompose it into%
\begin{eqnarray*}
&&\left\{ \sum_{\substack{ I,J\in \mathcal{D}:\ J\subset 3I  \\ 2^{-\rho
}\ell \left( I\right) \leq \ell \left( J\right) \leq 2^{\rho }\ell \left(
I\right) }}+\sum_{\substack{ I,J\in \mathcal{D}:\ I\subset 3J  \\ 2^{-\rho
}\ell \left( I\right) \leq \ell \left( J\right) \leq 2^{\rho }\ell \left(
I\right) }}+\sum_{\substack{ I,J\in \mathcal{D}  \\ 2^{-\rho }\ell \left(
I\right) \leq \ell \left( J\right) \leq 2^{\rho }\ell \left( I\right)  \\ %
J\not\subset 3I\text{ and }I\not\subset 3J}}\right\} \left\vert \left\langle
T_{\sigma }^{\alpha }\left( \bigtriangleup _{I;\kappa }^{\sigma }f\right)
,\bigtriangleup _{J;\kappa }^{\omega }g\right\rangle _{\omega }\right\vert \\
&&\ \ \ \ \ \ \ \ \ \ \equiv A_{1}+A_{2}+A_{3}.
\end{eqnarray*}%
The proof of the bound for term $A_{3}$ is similar to that of the bound for
the first term in (\ref{routine'}), and so we will defer its proof until
after the second term has been proved.

We now consider term $A_{1}$ as term $A_{2}$ is symmetric. To handle this
term we will write the Alpert functions $h_{I;\kappa }^{\sigma }$ and $%
h_{J;\kappa }^{\omega }$ as linear combinations of polynomials times
indicators of the children of their supporting cubes, denoted $I_{\theta }$
and $J_{\theta ^{\prime }}$ respectively. Then we use the testing condition
on $I_{\theta }$ and $J_{\theta ^{\prime }}$ when they \emph{overlap}, i.e.
their interiors intersect; we use the weak boundedness property on $%
I_{\theta }$ and $J_{\theta ^{\prime }}$ when they \emph{touch}, i.e. their
interiors are disjoint but their closures intersect (even in just a point);
and finally we use the $A_{2}^{\alpha }$ condition when $I_{\theta }$ and $%
J_{\theta ^{\prime }}$ are \emph{separated}, i.e. their closures are
disjoint. We will suppose initially that the side length of $J$ is at most
the side length $I$, i.e. $\ell \left( J\right) \leq \ell \left( I\right) $,
the proof for $J=\pi I$ being similar but for one point mentioned below.

So suppose that $I_{\theta }$ is a child of $I$ and that $J_{\theta ^{\prime
}}$ is a child of $J$. If $J_{\theta ^{\prime }}\subset I_{\theta }$ we have
using (\ref{analogue''}), 
\begin{eqnarray*}
&&\left\vert \left\langle T_{\sigma }^{\alpha }\left( \mathbf{1}_{I_{\theta
}}\bigtriangleup _{I;\kappa }^{\sigma }f\right) ,\mathbf{1}_{J_{\theta
^{\prime }}}\bigtriangleup _{J;\kappa }^{\omega }g\right\rangle _{\omega
}\right\vert \lesssim \frac{\left\vert \widehat{f_{\kappa }}\left( I\right)
\right\vert }{\sqrt{\left\vert I\right\vert _{\sigma }}}\left\vert
\left\langle T_{\sigma }^{\alpha }\left( P_{I_{\theta };\kappa }^{\sigma }%
\mathbf{1}_{I_{\theta }}\right) ,P_{J_{\theta ^{\prime }};\kappa }^{\omega }%
\mathbf{1}_{J_{\theta ^{\prime }}}\right\rangle _{\omega }\right\vert \frac{%
\left\vert \widehat{g_{\kappa }}\left( J\right) \right\vert }{\sqrt{%
\left\vert J\right\vert _{\omega }}} \\
&\lesssim &\frac{\left\vert \widehat{f_{\kappa }}\left( I\right) \right\vert 
}{\sqrt{\left\vert I\right\vert _{\sigma }}}\left( \int_{J_{\theta ^{\prime
}}}\left\vert T_{\sigma }^{\alpha }\left( P_{I_{\theta };\kappa }^{\sigma }%
\mathbf{1}_{I_{\theta }}\right) \right\vert ^{2}d\omega \right) ^{\frac{1}{2}%
}\left\vert \widehat{g_{\kappa }}\left( J\right) \right\vert \lesssim \frac{%
\left\vert \widehat{f_{\kappa }}\left( I\right) \right\vert }{\sqrt{%
\left\vert I\right\vert _{\sigma }}}\mathfrak{T}_{T_{\alpha }}^{\kappa
,s}\left\vert I\right\vert _{\sigma }^{\frac{1}{2}}\left\vert \widehat{%
g_{\kappa }}\left( J\right) \right\vert \lesssim \mathfrak{T}_{T_{\alpha
}}^{\kappa ,s}\left\vert \widehat{f_{\kappa }}\left( I\right) \right\vert
\left\vert \widehat{g_{\kappa }}\left( J\right) \right\vert ,
\end{eqnarray*}%
where $\widehat{f_{\kappa }}\left( I\right) $ denotes the vector Alpert
coefficient of $f$ at the dyadic cube $I$. The point referred to above is
that when $J=\pi I$ we write 
\begin{equation*}
\left\langle T_{\sigma }^{\alpha }\left( P_{I_{\theta };\kappa }^{\sigma }%
\mathbf{1}_{I_{\theta }}\right) ,P_{J_{\theta ^{\prime }};\kappa }^{\omega }%
\mathbf{1}_{J_{\theta ^{\prime }}}\right\rangle _{\omega }=\left\langle
P_{I_{\theta };\kappa }^{\sigma }\mathbf{1}_{I_{\theta }},T_{\omega
}^{\alpha ,\ast }\left( P_{J_{\theta ^{\prime }};\kappa }^{\omega }\mathbf{1}%
_{J_{\theta ^{\prime }}}\right) \right\rangle _{\sigma }\ ,
\end{equation*}%
and get the dual testing constant $\mathfrak{T}_{T_{\alpha }^{\ast
}}^{\kappa ,s}$. If $J_{\theta ^{\prime }}$ and $I_{\theta }$ touch, then $%
\ell \left( J_{\theta ^{\prime }}\right) \leq \ell \left( I_{\theta }\right) 
$ and we have $J_{\theta ^{\prime }}\subset 3I_{\theta }\setminus I_{\theta
} $, and so%
\begin{eqnarray*}
&&\left\vert \left\langle T_{\sigma }^{\alpha }\left( \mathbf{1}_{I_{\theta
}}\bigtriangleup _{I;\kappa }^{\sigma }f\right) ,\mathbf{1}_{J_{\theta
^{\prime }}}\bigtriangleup _{J;\kappa }^{\omega }g\right\rangle _{\omega
}\right\vert \lesssim \frac{\left\vert \widehat{f_{\kappa }}\left( I\right)
\right\vert }{\sqrt{\left\vert I\right\vert _{\sigma }}}\left\vert
\left\langle T_{\sigma }^{\alpha }\left( P_{I_{\theta };\kappa }^{\sigma }%
\mathbf{1}_{I_{\theta }}\right) ,P_{J_{\theta ^{\prime }};\kappa }^{\omega }%
\mathbf{1}_{J_{\theta ^{\prime }}}\right\rangle _{\omega }\right\vert \frac{%
\left\vert \widehat{g_{\kappa }}\left( J\right) \right\vert }{\sqrt{%
\left\vert J\right\vert _{\omega }}} \\
&\lesssim &\frac{\left\vert \widehat{f_{\kappa }}\left( I\right) \right\vert 
}{\sqrt{\left\vert I\right\vert _{\sigma }}}\mathcal{WBP}_{T^{\alpha
}}^{\left( \kappa ,\kappa \right) ,s}\sqrt{\left\vert I\right\vert _{\sigma
}\left\vert J\right\vert _{\omega }}\frac{\left\vert \widehat{g_{\kappa }}%
\left( J\right) \right\vert }{\sqrt{\left\vert J\right\vert _{\omega }}}=%
\mathcal{WBP}_{T^{\alpha }}^{\left( \kappa ,\kappa \right) ,s}\left\vert 
\widehat{f_{\kappa }}\left( I\right) \right\vert \left\vert \widehat{%
g_{\kappa }}\left( J\right) \right\vert .
\end{eqnarray*}%
Finally, if $J_{\theta ^{\prime }}$ and $I_{\theta }$ are separated, and if $%
K$ is the smallest (not necessarily dyadic) cube containing both $J_{\theta
^{\prime }}$ and $I_{\theta }$, then $\limfunc{dist}\left( I_{\theta
},J_{\theta ^{\prime }}\right) \approx \ell \left( K\right) $ and we have%
\begin{eqnarray*}
&&\left\vert \left\langle T_{\sigma }^{\alpha }\left( \mathbf{1}_{I_{\theta
}}\bigtriangleup _{I;\kappa }^{\sigma }f\right) ,\mathbf{1}_{J_{\theta
^{\prime }}}\bigtriangleup _{J;\kappa }^{\omega }g\right\rangle _{\omega
}\right\vert \lesssim \frac{\left\vert \widehat{f_{\kappa }}\left( I\right)
\right\vert }{\sqrt{\left\vert I\right\vert _{\sigma }}}\left\vert
\left\langle T_{\sigma }^{\alpha }\left( \mathbf{1}_{I_{\theta }}\right) ,%
\mathbf{1}_{J_{\theta ^{\prime }}}\right\rangle _{\omega }\right\vert \frac{%
\left\vert \widehat{g_{\kappa }}\left( J\right) \right\vert }{\sqrt{%
\left\vert J\right\vert _{\omega }}} \\
&\lesssim &\frac{\left\vert \widehat{f_{\kappa }}\left( I\right) \right\vert 
}{\sqrt{\left\vert I\right\vert _{\sigma }}}\frac{1}{\limfunc{dist}\left(
I_{\theta },J_{\theta ^{\prime }}\right) ^{n-\alpha }}\left\vert I_{\theta
}\right\vert _{\sigma }\left\vert J_{\theta ^{\prime }}\right\vert _{\omega }%
\frac{\left\vert \widehat{g_{\kappa }}\left( J\right) \right\vert }{\sqrt{%
\left\vert J\right\vert _{\omega }}}=\frac{\sqrt{\left\vert I_{\theta
}\right\vert _{\sigma }\left\vert J_{\theta ^{\prime }}\right\vert _{\omega }%
}}{\limfunc{dist}\left( I_{\theta },J_{\theta ^{\prime }}\right) ^{n-\alpha }%
}\left\vert \widehat{f_{\kappa }}\left( I\right) \right\vert \left\vert 
\widehat{g_{\kappa }}\left( J\right) \right\vert \lesssim \sqrt{%
A_{2}^{\alpha }}\left\vert \widehat{f_{\kappa }}\left( I\right) \right\vert
\left\vert \widehat{g_{\kappa }}\left( J\right) \right\vert .
\end{eqnarray*}%
Now we sum over all the children of $J$ and $I$ satisfying $2^{-\rho }\ell
\left( I\right) \leq \ell \left( J\right) \leq 2^{\rho }\ell \left( I\right) 
$ for which $J\subset 3I$ to obtain that%
\begin{equation*}
A_{1}\lesssim \left( \mathfrak{T}_{T^{\alpha }}^{\kappa ,s}+\mathfrak{T}%
_{T^{\alpha ,\ast }}^{\kappa ,-s}+\mathcal{WBP}_{T^{\alpha }}^{\left( \kappa
,\kappa \right) ,s}\left( \sigma ,\omega \right) +\sqrt{A_{2}^{\alpha }}%
\right) \sum_{\substack{ I,J\in \mathcal{D}:\ J\subset 3I  \\ 2^{-\rho }\ell
\left( I\right) \leq \ell \left( J\right) \leq 2^{\rho }\ell \left( I\right) 
}}\left\vert \widehat{f_{\kappa }}\left( I\right) \right\vert \left\vert 
\widehat{g_{\kappa }}\left( J\right) \right\vert \ .
\end{equation*}%
It is at this point that the Sobolev norms make their appearance, through an
application of the Cauchy-Schwarz inequality to obtain%
\begin{eqnarray*}
&&\sum_{\substack{ I,J\in \mathcal{D}:\ J\subset 3I  \\ 2^{-\rho }\ell
\left( I\right) \leq \ell \left( J\right) \leq 2^{\rho }\ell \left( I\right) 
}}\left\vert \widehat{f_{\kappa }}\left( I\right) \right\vert \left\vert 
\widehat{g_{\kappa }}\left( J\right) \right\vert \leq C_{\rho }\left( \sum 
_{\substack{ I,J\in \mathcal{D}:\ J\subset 3I  \\ 2^{-\rho }\ell \left(
I\right) \leq \ell \left( J\right) \leq 2^{\rho }\ell \left( I\right) }}%
\left\vert \widehat{f_{\kappa }}\left( I\right) \right\vert ^{2}\ell \left(
I\right) ^{-2s}\right) ^{\frac{1}{2}}\left( \sum_{\substack{ I,J\in \mathcal{%
D}:\ J\subset 3I  \\ 2^{-\rho }\ell \left( I\right) \leq \ell \left(
J\right) \leq 2^{\rho }\ell \left( I\right) }}\left\vert \widehat{g_{\kappa }%
}\left( J\right) \right\vert ^{2}\ell \left( J\right) ^{2s}\right) ^{\frac{1%
}{2}} \\
&&\ \ \ \ \ \ \ \ \ \ \ \ \ \ \ \ \ \ \ \ \ \ \ \ \ \ \ \ \ \ \lesssim
\left\Vert f\right\Vert _{W_{\func{dyad}}^{s}\left( \sigma \right)
}\left\Vert g\right\Vert _{W_{\func{dyad}}^{-s}\left( \omega \right) }\ .
\end{eqnarray*}%
This completes our proof of the bound for the second term in (\ref{routine'}%
), save for the deferral of term $A_{3}$, which we bound below.

\bigskip

Now we turn to the sum of separated cubes in (\ref{routine'}). We split the
pairs $\left( I,J\right) \in \mathcal{D}^{\sigma }\times \mathcal{D}^{\omega
}$ occurring in the first term in (\ref{routine'}) into two groups, those
with side length of $J$ smaller than side length of $I$, and those with side
length of $I$ smaller than side length of $J$, treating only the former
case, the latter being symmetric. Thus we prove the following bound:%
\begin{equation*}
\mathcal{A}\left( f,g\right) \equiv \sum_{\substack{ I,J\in \mathcal{D}  \\ %
I\cap J=\emptyset \text{ and }\ell \left( J\right) \leq 2^{-\rho }\ell
\left( I\right) }}\left\vert \left\langle T_{\sigma }^{\alpha }\left(
\bigtriangleup _{I;\kappa }^{\sigma }f\right) ,\bigtriangleup _{J;\kappa
}^{\omega }g\right\rangle _{\omega }\right\vert \lesssim \sqrt{A_{2}^{\alpha
}}\left\Vert f\right\Vert _{W_{\func{dyad}}^{s}\left( \sigma \right)
}\left\Vert g\right\Vert _{W_{\func{dyad}}^{-s}\left( \omega \right) }.
\end{equation*}%
We apply the `pivotal' bound (\ref{piv bound}) from the Energy Lemma to
estimate the inner product $\left\langle T_{\sigma }^{\alpha }\left(
\bigtriangleup _{I;\kappa }^{\sigma }f\right) ,\bigtriangleup _{J;\kappa
}^{\omega }g\right\rangle _{\omega }$ and obtain,%
\begin{equation*}
\left\vert \left\langle T_{\sigma }^{\alpha }\left( \bigtriangleup
_{I;\kappa }^{\sigma }f\right) ,\bigtriangleup _{J;\kappa }^{\omega
}g\right\rangle _{\omega }\right\vert \lesssim \mathrm{P}_{\kappa }^{\alpha
}\left( J,\left\vert \bigtriangleup _{I;\kappa }^{\sigma }f\right\vert
\sigma \right) \ell \left( J\right) ^{-s}\sqrt{\left\vert J\right\vert
_{\omega }}\left\Vert \bigtriangleup _{J;\kappa }^{\omega }g\right\Vert _{W_{%
\func{dyad}}^{-s}\left( \omega \right) }\,.
\end{equation*}%
Denote by $\limfunc{dist}$ the $\ell ^{\infty }$ distance in $\mathbb{R}^{n}$%
: $\limfunc{dist}\left( x,y\right) =\max_{1\leq j\leq n}\left\vert
x_{j}-y_{j}\right\vert $. We now estimate separately the long-range and
mid-range cases where $\limfunc{dist}\left( J,I\right) \geq \ell \left(
I\right) $ holds or not, and we decompose $\mathcal{A}$ accordingly:%
\begin{equation*}
\mathcal{A}\left( f,g\right) \equiv \mathcal{A}^{\limfunc{long}}\left(
f,g\right) +\mathcal{A}^{\limfunc{mid}}\left( f,g\right) .
\end{equation*}

\bigskip

\textbf{The long-range case}: We begin with the case where $\limfunc{dist}%
\left( J,I\right) $ is at least $\ell \left( I\right) $, i.e. $J\cap
3I=\emptyset $. Since $J$ and $I$ are separated by at least $\max \left\{
\ell \left( J\right) ,\ell \left( I\right) \right\} $, we have the inequality%
\begin{equation*}
\mathrm{P}_{\kappa }^{\alpha }\left( J,\left\vert \bigtriangleup _{I;\kappa
}^{\sigma }f\right\vert \sigma \right) \approx \int_{I}\frac{\ell \left(
J\right) }{\left\vert y-c_{J}\right\vert ^{n+1-\alpha }}\left\vert
\bigtriangleup _{I;\kappa }^{\sigma }f\left( y\right) \right\vert d\sigma
\left( y\right) \lesssim \left\Vert \bigtriangleup _{I;\kappa }^{\sigma
}f\right\Vert _{W_{\func{dyad}}^{s}\left( \sigma \right) }\ell \left(
I\right) ^{s}\frac{\ell \left( J\right) \sqrt{\left\vert I\right\vert
_{\sigma }}}{\limfunc{dist}\left( I,J\right) ^{n+1-\alpha }},
\end{equation*}%
since $\int_{I}\left\vert \bigtriangleup _{I;\kappa }^{\sigma }f\left(
y\right) \right\vert d\sigma \left( y\right) \leq \left\Vert \bigtriangleup
_{I;\kappa }^{\sigma }f\right\Vert _{L^{2}\left( \sigma \right) }\sqrt{%
\left\vert I\right\vert _{\sigma }}$ and $\left\Vert \bigtriangleup
_{I;\kappa }^{\sigma }f\right\Vert _{W_{\func{dyad}}^{s}\left( \sigma
\right) }=\ell \left( I\right) ^{-s}\left\Vert \bigtriangleup _{I;\kappa
}^{\sigma }f\right\Vert _{L^{2}\left( \sigma \right) }$. Thus with $A\left(
f,g\right) =\mathcal{A}^{\limfunc{long}}\left( f,g\right) $ we have%
\begin{eqnarray*}
A\left( f,g\right) &\lesssim &\sum_{I\in \mathcal{D}}\sum_{J\;:\;\ell \left(
J\right) \leq \ell \left( I\right) :\ \limfunc{dist}\left( I,J\right) \geq
\ell \left( I\right) }\left\Vert \bigtriangleup _{I;\kappa }^{\sigma
}f\right\Vert _{W_{\func{dyad}}^{s}\left( \sigma \right) }\left\Vert
\bigtriangleup _{J;\kappa }^{\omega }g\right\Vert _{W_{\func{dyad}%
}^{-s}\left( \omega \right) } \\
&&\ \ \ \ \ \ \ \ \ \ \ \ \ \ \ \times \left( \frac{\ell \left( I\right) }{%
\ell \left( J\right) }\right) ^{s}\frac{\ell \left( J\right) }{\limfunc{dist}%
\left( I,J\right) ^{n+1-\alpha }}\sqrt{\left\vert I\right\vert _{\sigma }}%
\sqrt{\left\vert J\right\vert _{\omega }} \\
&\equiv &\sum_{\left( I,J\right) \in \mathcal{P}}\left\Vert \bigtriangleup
_{I;\kappa }^{\sigma }f\right\Vert _{W_{\func{dyad}}^{s}\left( \sigma
\right) }\left\Vert \bigtriangleup _{J;\kappa }^{\omega }g\right\Vert _{W_{%
\func{dyad}}^{-s}\left( \omega \right) }A\left( I,J\right) ; \\
\text{with }A\left( I,J\right) &\equiv &\left( \frac{\ell \left( I\right) }{%
\ell \left( J\right) }\right) ^{s}\frac{\ell \left( J\right) }{\limfunc{dist}%
\left( I,J\right) ^{n+1-\alpha }}\sqrt{\left\vert I\right\vert _{\sigma }}%
\sqrt{\left\vert J\right\vert _{\omega }}; \\
\text{ and }\mathcal{P} &\equiv &\left\{ \left( I,J\right) \in \mathcal{D}%
\times \mathcal{D}:\ell \left( J\right) \leq \ell \left( I\right) \text{ and 
}\limfunc{dist}\left( I,J\right) \geq \ell \left( I\right) \right\} .
\end{eqnarray*}%
Now let $\mathcal{D}_{N}\equiv \left\{ K\in \mathcal{D}:\ell \left( K\right)
=2^{N}\right\} $ for each $N\in \mathbb{Z}$. For $N\in \mathbb{Z}$ and $t\in 
\mathbb{Z}_{+}$, we further decompose $A\left( f,g\right) $ by pigeonholing
the sidelengths of $I$ and $J$ by $2^{N}$ and $2^{N-t}$ respectively: 
\begin{eqnarray*}
A\left( f,g\right) &=&\sum_{t=0}^{\infty }\sum_{N\in \mathbb{Z}%
}A_{N}^{t}\left( f,g\right) ; \\
A_{N}^{t}\left( f,g\right) &\equiv &\sum_{\left( I,J\right) \in \mathcal{P}%
_{N}^{t}}\left\Vert \bigtriangleup _{I;\kappa }^{\sigma }f\right\Vert _{W_{%
\func{dyad}}^{s}\left( \sigma \right) }\left\Vert \bigtriangleup _{J;\kappa
}^{\omega }g\right\Vert _{W_{\func{dyad}}^{-s}\left( \omega \right) }A\left(
I,J\right) \\
\text{where }\mathcal{P}_{N}^{t} &\equiv &\left\{ \left( I,J\right) \in 
\mathcal{D}_{N}\times \mathcal{D}_{N-t}:\limfunc{dist}\left( I,J\right) \geq
\ell \left( I\right) \right\} .
\end{eqnarray*}%
Now $A_{N}^{t}\left( f,g\right) =A_{N}^{t}\left( \mathsf{P}_{N;\kappa
}^{\sigma }f,\mathsf{P}_{N-t;\kappa }^{\omega }g\right) $ where $\mathsf{P}%
_{M;\kappa }^{\mu }=\dsum\limits_{K\in \mathcal{D}_{M}}\bigtriangleup
_{K;\kappa }^{\mu }$ denotes Alpert projection onto the linear span $%
\limfunc{Span}\left\{ h_{K;\kappa }^{\mu ,a}\right\} _{K\in \mathcal{D}%
_{M},a\in \Gamma _{K,n,\kappa }}$, and so by orthogonality of the
projections $\left\{ \mathsf{P}_{M;\kappa }^{\mu }\right\} _{M\in \mathbb{Z}%
} $ we have%
\begin{eqnarray*}
\left\vert \sum_{N\in \mathbb{Z}}A_{N}^{t}\left( f,g\right) \right\vert
&=&\sum_{N\in \mathbb{Z}}\left\vert A_{N}^{t}\left( \mathsf{P}_{N;\kappa
}^{\sigma }f,\mathsf{P}_{N-t;\kappa }^{\omega }g\right) \right\vert \leq
\sum_{N\in \mathbb{Z}}\left\Vert A_{N}^{t}\right\Vert \left\Vert \mathsf{P}%
_{N;\kappa }^{\sigma }f\right\Vert _{W_{\func{dyad}}^{s}\left( \sigma
\right) }\left\Vert \mathsf{P}_{N-t;\kappa }^{\omega }g\right\Vert _{W_{%
\func{dyad}}^{-s}\left( \omega \right) } \\
&\leq &\left\{ \sup_{N\in \mathbb{Z}}\left\Vert A_{N}^{t}\right\Vert
\right\} \left( \sum_{N\in \mathbb{Z}}\left\Vert \mathsf{P}_{N;\kappa
}^{\sigma }f\right\Vert _{W_{\func{dyad}}^{s}\left( \sigma \right)
}^{2}\right) ^{\frac{1}{2}}\left( \sum_{N\in \mathbb{Z}}\left\Vert \mathsf{P}%
_{N-t;\kappa }^{\omega }g\right\Vert _{W_{\func{dyad}}^{-s}\left( \omega
\right) }^{2}\right) ^{\frac{1}{2}} \\
&\leq &\left\{ \sup_{N\in \mathbb{Z}}\left\Vert A_{N}^{t}\right\Vert
\right\} \left\Vert f\right\Vert _{W_{\func{dyad}}^{s}\left( \sigma \right)
}\left\Vert g\right\Vert _{W_{\func{dyad}}^{-s}\left( \omega \right) }.
\end{eqnarray*}%
Thus it suffices to show an estimate uniform in $N$ with geometric decay in $%
t$, and we will show%
\begin{equation}
\left\vert A_{N}^{t}\left( f,g\right) \right\vert \leq C2^{-t}\sqrt{%
A_{2}^{\alpha }}\left\Vert f\right\Vert _{W_{\func{dyad}}^{s}\left( \sigma
\right) }\left\Vert g\right\Vert _{W_{\func{dyad}}^{-s}\left( \omega \right)
},\ \ \ \ \ \text{for }t\geq 0\text{ and }N\in \mathbb{Z}.  \label{AsN}
\end{equation}

We now pigeonhole the distance between $I$ and $J$:%
\begin{eqnarray*}
A_{N}^{t}\left( f,g\right) &=&\dsum\limits_{\ell =0}^{\infty }A_{N,\ell
}^{t}\left( f,g\right) ; \\
A_{N,\ell }^{t}\left( f,g\right) &\equiv &\sum_{\left( I,J\right) \in 
\mathcal{P}_{N,\ell }^{t}}\left\Vert \bigtriangleup _{I;\kappa }^{\sigma
}f\right\Vert _{W_{\func{dyad}}^{s}\left( \sigma \right) }\left\Vert
\bigtriangleup _{J;\kappa }^{\omega }g\right\Vert _{W_{\func{dyad}%
}^{-s}\left( \omega \right) }A\left( I,J\right) \\
\text{where }\mathcal{P}_{N,\ell }^{t} &\equiv &\left\{ \left( I,J\right)
\in \mathcal{D}_{N}\times \mathcal{D}_{N-s}:\limfunc{dist}\left( I,J\right)
\approx 2^{N+\ell }\right\} .
\end{eqnarray*}%
If we define $\mathcal{H}\left( A_{N,\ell }^{t}\right) $ to be the bilinear
form on $\ell ^{2}\times \ell ^{2}$ with matrix $\left[ A\left( I,J\right) %
\right] _{\left( I,J\right) \in \mathcal{P}_{N,\ell }^{t}}$, then it remains
to show that the norm $\left\Vert \mathcal{H}\left( A_{N,\ell }^{t}\right)
\right\Vert _{\ell ^{2}\rightarrow \ell ^{2}}$ is bounded by $C2^{-t\left(
1-s\right) -\ell }\sqrt{A_{2}^{\alpha }}$. In turn, this is equivalent to
showing that the norm $\left\Vert \mathcal{H}\left( B_{N,\ell }^{t}\right)
\right\Vert _{\ell ^{2}\rightarrow \ell ^{2}}$ of the bilinear form $%
\mathcal{H}\left( B_{N,\ell }^{t}\right) \equiv \mathcal{H}\left( A_{N,\ell
}^{t}\right) ^{\limfunc{tr}}\mathcal{H}\left( A_{N,\ell }^{t}\right) $ on
the sequence space $\ell ^{2}$ is bounded by $C^{2}2^{-2t\left( 1-s\right)
-2\ell }A_{2}^{\alpha }$. Here $\mathcal{H}\left( B_{N,\ell }^{t}\right) $
is the quadratic form with matrix kernel $\left[ B_{N,\ell }^{t}\left(
J,J^{\prime }\right) \right] _{J,J^{\prime }\in \mathcal{D}_{N-s}}$ having
entries:%
\begin{equation*}
B_{N,\ell }^{t}\left( J,J^{\prime }\right) \equiv \sum_{I\in \mathcal{D}%
_{N}:\ \limfunc{dist}\left( I,J\right) \approx \limfunc{dist}\left(
I,J^{\prime }\right) \approx 2^{N+\ell }}A\left( I,J\right) A\left(
I,J^{\prime }\right) ,\ \ \ \ \ \text{for }J,J^{\prime }\in \mathcal{D}%
_{N-t}.
\end{equation*}

We are reduced to showing,%
\begin{equation*}
\left\Vert \mathcal{H}\left( B_{N,\ell }^{t}\right) \right\Vert _{\ell
^{2}\rightarrow \ell ^{2}}\leq C2^{-2t\left( 1-s\right) -2\ell
}A_{2}^{\alpha }\ \ \ \ \text{for }t\geq 0\text{, }\ell \geq 0\text{ and }%
N\in \mathbb{Z},
\end{equation*}%
which is an estimate in which Alpert projections no longer play a role, and
this estimate is proved as in \cite{NTV4}, and more precisely as in \cite%
{SaShUr7}. Note that the only arithmetic difference in the argument here is
that in the estimates, the parameter $t>0$ is replaced by $t\left(
1-s\right) >0$, which has no effect on the conclusion. This completes our
proof of the long-range estimate%
\begin{equation*}
\mathcal{A}^{\limfunc{long}}\left( f,g\right) \lesssim \sqrt{A_{2}^{\alpha }}%
\left\Vert f\right\Vert _{W_{\func{dyad}}^{s}\left( \sigma \right)
}\left\Vert g\right\Vert _{W_{\func{dyad}}^{-s}\left( \omega \right) }\ .
\end{equation*}

\bigskip

At this point we pause to complete the bound for $A_{3}$ in the second term
in (\ref{routine'}). Indeed, the deferred term $A_{3}$ can be handled using
the above argument since $3J\cap I=\emptyset =J\cap 3I$ implies that we can
use the Energy Lemma as we did above.

\bigskip

\textbf{The mid range case}: Let%
\begin{equation*}
\mathcal{P}\equiv \left\{ \left( I,J\right) \in \mathcal{D}\times \mathcal{D}%
:J\text{ is good},\ \ell \left( J\right) \leq 2^{-\rho }\ell \left( I\right)
,\text{ }J\subset 3I\setminus I\right\} .
\end{equation*}%
For $\left( I,J\right) \in \mathcal{P}$, the `pivotal' bound (\ref{piv bound}%
) from the Energy Lemma gives%
\begin{equation*}
\left\vert \left\langle T_{\sigma }^{\alpha }\left( \bigtriangleup
_{I;\kappa }^{\sigma }f\right) ,\bigtriangleup _{J}^{\omega }g\right\rangle
_{\omega }\right\vert \lesssim \mathrm{P}_{\kappa }^{\alpha }\left(
J,\left\vert \bigtriangleup _{I;\kappa }^{\sigma }f\right\vert \sigma
\right) \ell \left( J\right) ^{-s}\sqrt{\left\vert J\right\vert _{\omega }}%
\left\Vert \bigtriangleup _{J;\kappa }^{\omega }g\right\Vert _{W_{\limfunc{%
dyad}}^{-s}\left( \omega \right) }\,.
\end{equation*}%
Now we pigeonhole the lengths of $I$ and $J$ and the distance between them
by defining%
\begin{equation*}
\mathcal{P}_{N,d}^{t}\equiv \left\{ \left( I,J\right) \in \mathcal{D}\times 
\mathcal{D}:J\text{ is good},\ \ell \left( I\right) =2^{N},\ \ell \left(
J\right) =2^{N-t},\text{ }J\subset 3I\setminus I,\ 2^{d-1}\leq \limfunc{dist}%
\left( I,J\right) \leq 2^{d}\right\} .
\end{equation*}%
Note that the closest a good cube $J$ can come to $I$ is determined by the
goodness inequality, which gives this bound for $2^{d}\geq \limfunc{dist}%
\left( I,J\right) $: 
\begin{eqnarray*}
&&2^{d}\geq \frac{1}{2}\ell \left( I\right) ^{1-\varepsilon }\ell \left(
J\right) ^{\varepsilon }=\frac{1}{2}2^{N\left( 1-\varepsilon \right)
}2^{\left( N-t\right) \varepsilon }=\frac{1}{2}2^{N-\varepsilon t}; \\
&&\text{which implies }N-\varepsilon t-1\leq d\leq N,
\end{eqnarray*}%
where the last inequality holds because we are in the case of the mid-range
term. Thus we have%
\begin{eqnarray*}
&&\dsum\limits_{\left( I,J\right) \in \mathcal{P}}\left\vert \left\langle
T_{\sigma }^{\alpha }\left( \bigtriangleup _{I;\kappa }^{\sigma }f\right)
,\bigtriangleup _{J;\kappa }^{\omega }g\right\rangle _{\omega }\right\vert
\lesssim \dsum\limits_{\left( I,J\right) \in \mathcal{P}}\left\Vert
\bigtriangleup _{J;\kappa }^{\omega }g\right\Vert _{W_{\limfunc{dyad}%
}^{-s}\left( \omega \right) }\mathrm{P}_{\kappa }^{\alpha }\left(
J,\left\vert \bigtriangleup _{I;\kappa }^{\sigma }f\right\vert \sigma
\right) \ell \left( J\right) ^{-s}\sqrt{\left\vert J\right\vert _{\omega }}
\\
&&\ \ \ \ \ =\dsum\limits_{t=\rho }^{\infty }\ \sum_{N\in \mathbb{Z}}\
\sum_{d=N-\varepsilon t-1}^{N}\ \sum_{\left( I,J\right) \in \mathcal{P}%
_{N,d}^{t}}\ \left\Vert \bigtriangleup _{J;\kappa }^{\omega }g\right\Vert
_{W_{\limfunc{dyad}}^{-s}\left( \omega \right) }\mathrm{P}_{\kappa }^{\alpha
}\left( J,\left\vert \bigtriangleup _{I;\kappa }^{\sigma }f\right\vert
\sigma \right) \ell \left( J\right) ^{-s}\sqrt{\left\vert J\right\vert
_{\omega }}.
\end{eqnarray*}%
Now we use%
\begin{eqnarray*}
\mathrm{P}^{\alpha }\left( J,\left\vert \bigtriangleup _{I;\kappa }^{\sigma
}f\right\vert \sigma \right) &=&\int_{I}\frac{\ell \left( J\right) }{\left(
\ell \left( J\right) +\left\vert y-c_{J}\right\vert \right) ^{n+1-\alpha }}%
\left\vert \bigtriangleup _{I;\kappa }^{\sigma }f\left( y\right) \right\vert
d\sigma \left( y\right) \\
&\lesssim &\frac{2^{N-t}}{2^{d\left( n+1-\alpha \right) }}\left\Vert
\bigtriangleup _{I}^{\sigma }f\right\Vert _{W_{\limfunc{dyad}}^{s}\left(
\sigma \right) }\ell \left( I\right) ^{s}\sqrt{\left\vert I\right\vert
_{\sigma }}
\end{eqnarray*}%
and apply Cauchy-Schwarz in $J$ and use $J\subset 3I\setminus I$ to get%
\begin{eqnarray*}
&&\dsum\limits_{\left( I,J\right) \in \mathcal{P}}\left\vert \left\langle
T_{\sigma }^{\alpha }\left( \bigtriangleup _{I;\kappa }^{\sigma }f\right)
,\bigtriangleup _{J;\kappa }^{\omega }g\right\rangle _{\omega }\right\vert \\
&\lesssim &\dsum\limits_{t=\rho }^{\infty }\ \sum_{N\in \mathbb{Z}}\
\sum_{d=N-\varepsilon t-1}^{N}\ \sum_{I\in \mathcal{D}_{N}}\frac{%
2^{N-t\left( 1-s\right) }2^{N\left( n-\alpha \right) }}{2^{d\left(
n+1-\alpha \right) }}\left\Vert \bigtriangleup _{I;\kappa }^{\sigma
}f\right\Vert _{W_{\limfunc{dyad}}^{s}\left( \sigma \right) }\frac{\sqrt{%
\left\vert I\right\vert _{\sigma }}\sqrt{\left\vert 3I\setminus I\right\vert
_{\omega }}}{2^{N\left( n-\alpha \right) }} \\
&&\ \ \ \ \ \ \ \ \ \ \ \ \ \ \ \ \ \ \ \ \ \ \ \ \ \ \ \ \ \ \times \sqrt{%
\sum_{\substack{ J\in \mathcal{D}_{N-t}  \\ J\subset 3I\setminus I\text{ and 
}\limfunc{dist}\left( I,J\right) \approx 2^{d}}}\left\Vert \bigtriangleup
_{J;\kappa }^{\omega }g\right\Vert _{W_{\limfunc{dyad}}^{-s}\left( \omega
\right) }^{2}} \\
&\lesssim &\left( 1+\varepsilon t\right) \dsum\limits_{t=\rho }^{\infty }\
\sum_{N\in \mathbb{Z}}\frac{2^{N-t\left( 1-s\right) }2^{N\left( n-\alpha
\right) }}{2^{\left( N-\varepsilon t\right) \left( n+1-\alpha \right) }}%
\sqrt{A_{2}^{\alpha }}\sum_{I\in \mathcal{D}_{N}}\left\Vert \bigtriangleup
_{I;\kappa }^{\sigma }f\right\Vert _{W_{\limfunc{dyad}}^{s}\left( \sigma
\right) }\sqrt{\sum_{\substack{ J\in \mathcal{D}_{N-t}  \\ J\subset
3I\setminus I}}\left\Vert \bigtriangleup _{J;\kappa }^{\omega }g\right\Vert
_{W_{\limfunc{dyad}}^{-s}\left( \omega \right) }^{2}} \\
&\lesssim &\left( 1+\varepsilon t\right) \dsum\limits_{t=\rho }^{\infty
}2^{-t\left[ 1-s-\varepsilon \left( n+1-\alpha \right) \right] }\sqrt{%
A_{2}^{\alpha }}\left\Vert f\right\Vert _{W_{\limfunc{dyad}}^{s}\left(
\sigma \right) }\left\Vert g\right\Vert _{W_{\limfunc{dyad}}^{-s}\left(
\omega \right) }\lesssim \sqrt{A_{2}^{\alpha }}\left\Vert f\right\Vert _{W_{%
\limfunc{dyad}}^{s}\left( \sigma \right) }\left\Vert g\right\Vert _{W_{%
\limfunc{dyad}}^{-s}\left( \omega \right) },
\end{eqnarray*}%
where in the third line above we have used $\sum_{d=N-\varepsilon
t-1}^{N}1\lesssim 1+\varepsilon t$, and in the last line 
\begin{equation*}
\frac{2^{N-t\left( 1-s\right) }2^{N\left( n-\alpha \right) }}{2^{\left(
N-\varepsilon t\right) \left( n+1-\alpha \right) }}=2^{-t\left[
1-s-\varepsilon \left( n+1-\alpha \right) \right] },
\end{equation*}%
followed by Cauchy-Schwarz in $I$ and $N$, using that we have bounded
overlap in the triples of $I$ for $I\in \mathcal{D}_{N}$. We have also
assumed here that $0<\varepsilon <\frac{1-s}{n+1-\alpha }$, and this
completes the proof of (\ref{routine'}).
\end{proof}

\subsection{Shifted corona decomposition}

To prove (\ref{red proving}), we recall the \emph{Shifted Corona
Decomposition}, as opposed to the \emph{parallel} corona decomposition used
in \cite{Saw6}, associated with the Calder\'{o}n-Zygmund $\kappa $-pivotal
stopping cubes $\mathcal{F}$ introduced above. But first we must invoke
standard arguments, using the full $\kappa $-cube testing conditions (\ref%
{full testing}),\ to permit us to assume that $f$ and $g$ are supported in a
finite union of dyadic cubes $F_{0}$ on which they have vanishing moments of
order less than $\kappa $.

\subsubsection{The initial reduction using full testing}

For this construction, we will follow the treatment as given in \cite%
{SaShUr12}. We first restrict $f$ and $g$ to be supported in a large common
cube $Q_{\infty }$. Then we cover $Q_{\infty }$ with $2^{n}$ pairwise
disjoint cubes $I_{\infty }\in \mathcal{D}$ with $\ell \left( I_{\infty
}\right) =\ell \left( Q_{\infty }\right) $. We now claim we can reduce
matters to consideration of the $2^{2n}$ forms%
\begin{equation*}
\sum_{I\in \mathcal{D}:\ I\subset I_{\infty }}\sum_{J\in \mathcal{D}:\
J\subset J_{\infty }}\int_{\mathbb{R}^{n}}\left( T_{\sigma }^{\alpha
}\bigtriangleup _{I;\kappa }^{\sigma }f\right) \bigtriangleup _{J;\kappa
}^{\omega }gd\omega ,
\end{equation*}%
as both $I_{\infty }$ and $J_{\infty }$ range over the dyadic cubes as
above. First we note that when $I_{\infty }$ and $J_{\infty }$ are distinct,
the corresponding form is included in the sum $\mathsf{B}_{\cap }\left(
f,g\right) +\mathsf{B}_{\diagup }\left( f,g\right) $, and hence controlled.
Thus it remains to consider the forms with $I_{\infty }=J_{\infty }$ and use
the cubes $I_{\infty }$ as the starting cubes in our corona construction
below. Indeed, we have from (\ref{Alpert expan}) that%
\begin{eqnarray*}
f &=&\sum_{I\in \mathcal{D}:\ I\subset I_{\infty }}\bigtriangleup _{I;\kappa
}^{\sigma }f+\mathbb{E}_{I_{\infty };\kappa }^{\sigma }f, \\
g &=&\sum_{J\in \mathcal{D}:\ J\subset I_{\infty }}\bigtriangleup _{J;\kappa
}^{\omega }g+\mathbb{E}_{I_{\infty };\kappa }^{\omega }g,
\end{eqnarray*}%
which can then be used to write the bilinear form $\int \left( T_{\sigma
}^{\alpha }f\right) gd\omega $ as a sum of the forms%
\begin{eqnarray}
&&\ \ \ \ \ \ \ \ \ \ \ \ \ \ \ \int_{\mathbb{R}^{n}}\left( T_{\sigma
}^{\alpha }f\right) gd\omega =\sum_{I_{\infty }}\left\{ \sum_{I,J\in 
\mathcal{D}:\ I,J\subset I_{\infty }}\int_{\mathbb{R}^{n}}\left( T_{\sigma
}^{\alpha }\bigtriangleup _{I;\kappa }^{\sigma }f\right) \bigtriangleup
_{J;\kappa }^{\omega }gd\omega \right.  \label{sum of forms} \\
&&\left. +\sum_{I\in \mathcal{D}:\ I\subset I_{\infty }}\int_{\mathbb{R}%
^{n}}\left( T_{\sigma }^{\alpha }\bigtriangleup _{I;\kappa }^{\sigma
}f\right) \mathbb{E}_{I_{\infty };\kappa }^{\omega }gd\omega +\sum_{J\in 
\mathcal{D}:\ J\subset I_{\infty }}\int_{\mathbb{R}^{n}}\left( T_{\sigma
}^{\alpha }\mathbb{E}_{I_{\infty };\kappa }^{\sigma }f\right) \bigtriangleup
_{J;\kappa }^{\omega }gd\omega +\int_{\mathbb{R}^{n}}\left( T_{\sigma
}^{\alpha }\mathbb{E}_{I_{\infty };\kappa }^{\sigma }f\right) \mathbb{E}%
_{I_{\infty };\kappa }^{\omega }gd\omega \right\} ,  \notag
\end{eqnarray}%
taken over the $2^{n}$ cubes $I_{\infty }$ above.

The second, third and fourth sums in (\ref{sum of forms}) can be controlled
by the full testing conditions (\ref{full testing}), e.g.%
\begin{eqnarray}
&&\left\vert \sum_{I\in \mathcal{D}:\ I\subset I_{\infty }}\int_{\mathbb{R}%
^{n}}\left( T_{\sigma }^{\alpha }\bigtriangleup _{I;\kappa }^{\sigma
}f\right) \mathbb{E}_{I_{\infty };\kappa }^{\omega }gd\omega \right\vert
=\left\vert \int_{I_{\infty }}\left( \sum_{I\in \mathcal{D}:\ I\subset
I_{\infty }}\bigtriangleup _{I;\kappa }^{\sigma }f\right) T_{\omega
}^{\alpha ,\ast }\left( \mathbb{E}_{I_{\infty };\kappa }^{\omega }g\right)
d\sigma \right\vert  \label{top control} \\
&\leq &\left\Vert \sum_{I\in \mathcal{D}:\ I\subset I_{\infty
}}\bigtriangleup _{I;\kappa }^{\sigma }f\right\Vert _{W_{\limfunc{dyad}%
}^{s}\left( \sigma \right) }\left\Vert \mathbf{1}_{I_{\infty }}T_{\omega
}^{\alpha ,\ast }\left( \mathbb{E}_{I_{\infty };\kappa }^{\omega }g\right)
\right\Vert _{W_{\limfunc{dyad}}^{-s}\left( \sigma \right) }  \notag \\
&=&\left\Vert \sum_{I\in \mathcal{D}:\ I\subset I_{\infty }}\bigtriangleup
_{I;\kappa }^{\sigma }f\right\Vert _{W_{\limfunc{dyad}}^{s}\left( \sigma
\right) }\left\Vert \mathbb{E}_{I_{\infty };\kappa }^{\omega }g\right\Vert
_{L^{\infty }}\left\Vert \mathbf{1}_{I_{\infty }}T_{\omega }^{\alpha ,\ast
}\left( \frac{\mathbb{E}_{I_{\infty };\kappa }^{\omega }g}{\left\Vert 
\mathbb{E}_{I_{\infty };\kappa }^{\omega }g\right\Vert _{L^{\infty }}}%
\right) \right\Vert _{W_{\limfunc{dyad}}^{-s}\left( \sigma \right) }  \notag
\\
&\lesssim &\mathfrak{T}_{T_{\omega }^{\alpha ,\ast }}^{\kappa }\left\Vert
f\right\Vert _{W_{\limfunc{dyad}}^{s}\left( \sigma \right) }\left\Vert
g\right\Vert _{W_{\limfunc{dyad}}^{-s}\left( \omega \right) }\ ,  \notag
\end{eqnarray}%
and similarly for the third and fourth sum.

\subsubsection{The corona and shifted corona projections}

Given a grid $\mathcal{F}\subset \mathcal{D}$ and its associated coronas $%
\mathcal{C}_{F}$, define the two Alpert corona projections, 
\begin{equation*}
\mathsf{P}_{\mathcal{C}_{F}}^{\sigma }\equiv \sum_{I\in \mathcal{C}%
_{F}}\bigtriangleup _{I;\kappa _{1}}^{\sigma }\text{ and }\mathsf{P}_{%
\mathcal{C}_{F}^{\tau -\limfunc{shift}}}^{\omega }\equiv \sum_{J\in \mathcal{%
C}_{F}^{\tau -\limfunc{shift}}}\bigtriangleup _{J;\kappa _{2}}^{\omega }\ ,
\end{equation*}%
where%
\begin{eqnarray}
\mathcal{C}_{F}^{\tau -\limfunc{shift}} &\equiv &\left[ \mathcal{C}%
_{F}\setminus \mathcal{N}_{\mathcal{D}}^{\tau }\left( F\right) \right] \cup
\dbigcup\limits_{F^{\prime }\in \mathfrak{C}_{\mathcal{F}}\left( F\right) }%
\left[ \mathcal{N}_{\mathcal{D}}^{\tau }\left( F^{\prime }\right) \setminus 
\mathcal{N}_{\mathcal{D}}^{\tau }\left( F\right) \right] ;  \label{def shift}
\\
\text{where }\mathcal{N}_{\mathcal{D}}^{\tau }\left( F\right) &\equiv
&\left\{ J\in \mathcal{D}:J\subset F\text{ and }\ell \left( J\right)
>2^{-\tau }\ell \left( F\right) \right\} .  \notag
\end{eqnarray}%
Thus the \emph{shifted} corona $\mathcal{C}_{F}^{\tau -\limfunc{shift}}$ has
the top $\tau $ levels from $\mathcal{C}_{F}$ removed, and includes the
first $\tau $ levels from each of its $\mathcal{F}$-children, except if they
have already been removed. We must restrict the Alpert supports of $f$ and $%
g $ to $\limfunc{good}$ cubes, as defined e.g. in \cite{Saw6}, so that with
the superscript $\limfunc{good}$ denoting this restriction,%
\begin{equation*}
\mathsf{P}_{\mathcal{C}_{F}}^{\sigma }f=\sum_{I\in \mathcal{C}_{F}^{\limfunc{%
good}}}\bigtriangleup _{I;\kappa _{1}}^{\sigma }\text{ and }\mathsf{P}_{%
\mathcal{C}_{F}^{\tau -\limfunc{shift}}}^{\omega }g=\sum_{J\in \mathcal{C}%
_{F}^{\limfunc{good},\tau -\limfunc{shift}}}\bigtriangleup _{J;\kappa
_{2}}^{\omega }\ ,
\end{equation*}%
where $\mathcal{C}_{F}^{\limfunc{good}}\equiv \mathcal{C}_{F}\cap \mathcal{D}%
^{\limfunc{good}}$ and $\mathcal{C}_{F}^{\limfunc{good},\tau -\limfunc{shift}%
}\equiv \mathcal{C}_{F}^{\tau -\limfunc{shift}}\cap \mathcal{D}^{\limfunc{%
good}}$, and $\mathcal{D}^{\limfunc{good}}$ consists of the $\left(
r,\varepsilon \right) -\func{good}$ cubes in $\mathcal{D}$.

A simple but important property is the fact that the $\tau $-shifted coronas 
$\mathcal{C}_{F}^{\tau -\limfunc{shift}}$ have overlap bounded by $\tau $:%
\begin{equation}
\sum_{F\in \mathcal{F}}\mathbf{1}_{\mathcal{C}_{F}^{\tau -\limfunc{shift}%
}}\left( J\right) \leq \tau ,\ \ \ \ \ J\in \mathcal{D}.  \label{tau overlap}
\end{equation}%
It is convenient, for use in the canonical splitting below, to introduce the
following shorthand notation for $F,G\in \mathcal{F}$:%
\begin{equation*}
\left\langle T_{\sigma }^{\alpha }\left( \mathsf{P}_{\mathcal{C}%
_{F}}^{\sigma }f\right) ,\mathsf{P}_{\mathcal{C}_{G}^{\tau -\limfunc{shift}%
}}^{\omega }g\right\rangle _{\omega }^{\Subset _{\mathbf{\rho }}}\equiv \sum 
_{\substack{ I\in \mathcal{C}_{F}\text{ and }J\in \mathcal{C}_{G}^{\tau -%
\limfunc{shift}}  \\ J\Subset _{\rho }I}}\left\langle T_{\sigma }^{\alpha
}\left( \bigtriangleup _{I;\kappa }^{\sigma }f\right) ,\left( \bigtriangleup
_{J;\kappa }^{\omega }g\right) \right\rangle _{\omega }\ .
\end{equation*}

\subsection{Canonical splitting}

We then proceed with the \emph{Canonical Splitting} as in \cite{SaShUr7},
but with Alpert wavelets in place of Haar wavelets,%
\begin{eqnarray*}
&&\mathsf{B}_{\Subset _{\rho }}\left( f,g\right) =\sum_{F,G\in \mathcal{F}%
}\left\langle T_{\sigma }\left( \mathsf{P}_{\mathcal{C}_{F}}^{\sigma
}f\right) ,\mathsf{P}_{\mathcal{C}_{G}^{\tau -\limfunc{shift}}}^{\omega
}g\right\rangle _{\omega }^{\Subset _{\mathbf{\rho }}} \\
&=&\sum_{F\in \mathcal{F}}\left\langle T_{\sigma }\left( \mathsf{P}_{%
\mathcal{C}_{F}}^{\sigma }f\right) ,\mathsf{P}_{\mathcal{C}_{F}^{\tau -%
\limfunc{shift}}}^{\omega }g\right\rangle _{\omega }^{\Subset _{\mathbf{\rho 
}}}+\sum_{\substack{ F,G\in \mathcal{F}  \\ G\subsetneqq F}}\left\langle
T_{\sigma }\left( \mathsf{P}_{\mathcal{C}_{F}}^{\sigma }f\right) ,\mathsf{P}%
_{\mathcal{C}_{G}^{\tau -\limfunc{shift}}}^{\omega }g\right\rangle _{\omega
}^{\Subset _{\mathbf{\rho }}} \\
&&+\sum_{\substack{ F,G\in \mathcal{F}  \\ G\supsetneqq F}}\left\langle
T_{\sigma }\left( \mathsf{P}_{\mathcal{C}_{F}}^{\sigma }f\right) ,\mathsf{P}%
_{\mathcal{C}_{G}^{\tau -\limfunc{shift}}}^{\omega }g\right\rangle _{\omega
}^{\Subset _{\mathbf{\rho }}}+\sum_{\substack{ F,G\in \mathcal{F}  \\ F\cap
G=\emptyset }}\left\langle T_{\sigma }\left( \mathsf{P}_{\mathcal{C}%
_{F}}^{\sigma }f\right) ,\mathsf{P}_{\mathcal{C}_{G}^{\tau -\limfunc{shift}%
}}^{\omega }g\right\rangle _{\omega }^{\Subset _{\mathbf{\rho }}} \\
&\equiv &\mathsf{T}_{\limfunc{diagonal}}\left( f,g\right) +\mathsf{T}_{%
\limfunc{far}\limfunc{below}}\left( f,g\right) +\mathsf{T}_{\limfunc{far}%
\limfunc{above}}\left( f,g\right) +\mathsf{T}_{\limfunc{disjoint}}\left(
f,g\right) .
\end{eqnarray*}

The two forms $\mathsf{T}_{\limfunc{far}\limfunc{above}}\left( f,g\right) $
and $\mathsf{T}_{\limfunc{disjoint}}\left( f,g\right) $ each vanish just as
in \cite{SaShUr7}, since there are no pairs $\left( I,J\right) \in \mathcal{C%
}_{F}\times \mathcal{C}_{G}^{\mathbf{\tau }-\limfunc{shift}}$ with both (%
\textbf{i}) $J\Subset _{\rho }I$ and (\textbf{ii}) either $F\subsetneqq G$
or $G\cap F=\emptyset $.

\subsubsection{The far below form}

Here is a generalization to weighted Sobolev spaces of the Intertwining
Proposition from \cite[Proposition 36 on page 35]{Saw6}, that uses strong $%
\kappa $-pivotal conditions with Alpert wavelets. Recall that $0<\varepsilon
<1$ and $r$ is chosen sufficiently large depending on $\varepsilon $. The
argument given here is considerably simpler than that in \cite{Saw6}.

\begin{proposition}[The Intertwining Proposition]
\label{Int Prop}Suppose $\sigma ,\omega $ are positive locally finite Borel
measures on $\mathbb{R}^{n}$, that $\sigma $ is doubling, and that $\mathcal{%
F}$ satisfies an $\theta _{\sigma }^{\func{rev}}$-strong $\sigma $-Carleson
condition. Then for a smooth $\alpha $-fractional singular integral $%
T^{\alpha }$, and for $\limfunc{good}$ functions $f\in W_{\limfunc{dyad}%
}^{s}\left( \sigma \right) $ and $g\in W_{\limfunc{dyad}}^{-s}\left( \omega
\right) $, and with $\kappa _{1},\kappa _{2}\geq 1$ sufficiently large, we
have the following bound for $\mathsf{T}_{\limfunc{far}\limfunc{below}%
}\left( f,g\right) =\sum_{F\in \mathcal{F}}\ \sum_{I:\ I\supsetneqq F}\
\left\langle T_{\sigma }^{\alpha }\bigtriangleup _{I;\kappa _{1}}^{\sigma }f,%
\mathsf{P}_{\mathcal{C}_{F}^{\mathbf{\tau }-\limfunc{shift}}}^{\omega
}g\right\rangle _{\omega }$: 
\begin{equation}
\left\vert \mathsf{T}_{\limfunc{far}\limfunc{below}}\left( f,g\right)
\right\vert \lesssim \left( \mathcal{V}_{2,\beta ,\varepsilon ^{\prime
}}^{\alpha ,\kappa _{1}}+\sqrt{A_{2}^{\alpha }}\right) \ \left\Vert
f\right\Vert _{W_{\limfunc{dyad}}^{s}\left( \sigma \right) }\left\Vert
g\right\Vert _{W_{\limfunc{dyad}}^{-s}\left( \omega \right) },
\label{far below est}
\end{equation}%
where $\beta $ is as in Lemma \ref{full} and $\varepsilon ^{\prime }>0$ is
sufficiently small.
\end{proposition}

\begin{proof}
We write%
\begin{eqnarray*}
f_{F} &\equiv &\sum_{I:\ I\supsetneqq F}\bigtriangleup _{I;\kappa
_{1}}^{\sigma }f=\sum_{m=1}^{\infty }\sum_{I:\ \pi _{\mathcal{F}%
}^{m}F\subsetneqq I\subset \pi _{\mathcal{F}}^{m+1}F}\bigtriangleup
_{I;\kappa _{1}}^{\sigma }f \\
&=&\sum_{m=1}^{\infty }\sum_{I:\ \pi _{\mathcal{F}}^{m}F\subsetneqq I\subset
\pi _{\mathcal{F}}^{m+1}F}\mathbf{1}_{\theta \left( I\right) }\left( \mathbb{%
E}_{I;\kappa _{1}}^{\sigma }f-\mathbb{E}_{\pi _{\mathcal{F}}^{m+1}F;\kappa
_{1}}^{\sigma }f\right) \\
&=&\sum_{m=1}^{\infty }\sum_{I:\ \pi _{\mathcal{F}}^{m}F\subsetneqq I\subset
\pi _{\mathcal{F}}^{m+1}F}\mathbf{1}_{\theta \left( I\right) }\left( \mathbb{%
E}_{I;\kappa _{1}}^{\sigma }f\right) -\sum_{m=1}^{\infty }\mathbf{1}_{\pi _{%
\mathcal{F}}^{m+1}F\setminus \pi _{\mathcal{F}}^{m}F}\left( \mathbb{E}_{\pi
_{\mathcal{F}}^{m+1}F;\kappa _{1}}^{\sigma }f\right) \\
&\equiv &\beta _{F}-\gamma _{F}\ ,
\end{eqnarray*}%
and then%
\begin{equation*}
\sum_{F\in \mathcal{F}}\ \left\langle T_{\sigma }^{\alpha
}f_{F},g_{F}\right\rangle _{\omega }=\sum_{F\in \mathcal{F}}\ \left\langle
T_{\sigma }^{\alpha }\beta _{F},g_{F}\right\rangle _{\omega }+\sum_{F\in 
\mathcal{F}}\ \left\langle T_{\sigma }^{\alpha }\gamma
_{F},g_{F}\right\rangle _{\omega }\ .
\end{equation*}%
Now we use the pivotal bound (\ref{piv bound}),%
\begin{equation*}
\left\vert \left\langle T^{\alpha }\left( \varphi \nu \right) ,\Psi
_{J}\right\rangle _{L^{2}\left( \omega \right) }\right\vert \lesssim \mathrm{%
P}_{\kappa }^{\alpha }\left( J,\nu \right) \ell \left( J\right) ^{-s}\sqrt{%
\left\vert J\right\vert _{\omega }}\left\Vert \Psi _{J}\right\Vert _{W_{%
\func{dyad}}^{-s}\left( \omega \right) },
\end{equation*}%
the pivotal stopping control (\ref{energy control}),%
\begin{equation*}
\mathrm{P}_{\kappa }^{\alpha }\left( I,\mathbf{1}_{F}\sigma \right)
^{2}\left( \frac{\ell \left( F\right) }{\ell \left( I\right) }\right)
^{\varepsilon }\left\vert I\right\vert _{\omega }<\Gamma \left\vert
I\right\vert _{\sigma },\ \ \ \ \ I\in \mathcal{C}_{F}\text{ and }F\in 
\mathcal{F},
\end{equation*}%
and (\ref{e.Jsimeq}), namely 
\begin{equation*}
\mathrm{P}_{k}^{\alpha }\left( J,\sigma \mathbf{1}_{K\setminus I}\right)
\lesssim \left( \frac{\ell \left( J\right) }{\ell \left( I\right) }\right)
^{k-\varepsilon \left( n+k-\alpha \right) }\mathrm{P}_{k}^{\alpha }\left(
I,\sigma \mathbf{1}_{K\setminus I}\right) ,
\end{equation*}%
to obtain that%
\begin{eqnarray*}
&&\left\vert \sum_{F\in \mathcal{F}}\left\langle T_{\sigma }^{\alpha }\gamma
_{F},g_{F}\right\rangle _{\omega }\right\vert \lesssim \sum_{F\in \mathcal{F}%
}\mathrm{P}_{\kappa }^{\alpha }\left( F,\sum_{m=1}^{\infty }\mathbf{1}_{\pi
_{\mathcal{F}}^{m+1}F\setminus \pi _{\mathcal{F}}^{m}F}\left\vert \mathbb{E}%
_{\pi _{\mathcal{F}}^{m+1}F;\kappa _{1}}^{\sigma }f\right\vert \sigma
\right) \ell \left( F\right) ^{-s}\sqrt{\left\vert F\right\vert _{\omega }}%
\left\Vert g_{F}\right\Vert _{W_{\limfunc{dyad}}^{-s}\left( \omega \right) }
\\
&=&\sum_{m=1}^{\infty }\sum_{F\in \mathcal{F}}\left\Vert \mathbb{E}_{\pi _{%
\mathcal{F}}^{m+1}F;\kappa _{1}}^{\sigma }f\right\Vert _{\infty }\mathrm{P}%
_{\kappa }^{\alpha }\left( F,\mathbf{1}_{\pi _{\mathcal{F}}^{m+1}F\setminus
\pi _{\mathcal{F}}^{m}F}\sigma \right) \ell \left( F\right) ^{-s}\sqrt{%
\left\vert F\right\vert _{\omega }}\left\Vert g_{F}\right\Vert _{W_{\limfunc{%
dyad}}^{-s}\left( \omega \right) } \\
&\leq &\sum_{m=1}^{\infty }\sum_{F\in \mathcal{F}}\left\Vert \mathbb{E}_{\pi
_{\mathcal{F}}^{m+1}F;\kappa _{1}}^{\sigma }f\right\Vert _{\infty }\left( 
\frac{\ell \left( F\right) }{\ell \left( \pi _{\mathcal{F}}^{m}F\right) }%
\right) ^{\kappa -\varepsilon \left( n+\kappa -\alpha \right) }\mathrm{P}%
_{\kappa }^{\alpha }\left( \pi _{\mathcal{F}}^{m}F,\mathbf{1}_{\pi _{%
\mathcal{F}}^{m+1}F\setminus \pi _{\mathcal{F}}^{m}F}\sigma \right) \ell
\left( F\right) ^{-s}\sqrt{\left\vert F\right\vert _{\omega }}\left\Vert
g_{F}\right\Vert _{W_{\limfunc{dyad}}^{-s}\left( \omega \right) }
\end{eqnarray*}%
equals%
\begin{eqnarray*}
&&\sum_{m=1}^{\infty }\sum_{F\in \mathcal{F}}\left\Vert \mathbb{E}_{\pi _{%
\mathcal{F}}^{m+1}F;\kappa _{1}}^{\sigma }f\right\Vert _{\infty }\left( 
\frac{\ell \left( F\right) }{\ell \left( \pi _{\mathcal{F}}^{m}F\right) }%
\right) ^{\kappa -\varepsilon \left( n+\kappa -\alpha \right) -s}\ell \left(
\pi _{\mathcal{F}}^{m}F\right) ^{-s} \\
&&\ \ \ \ \ \ \ \ \ \ \ \ \ \ \ \ \ \ \ \ \ \ \ \ \ \ \ \ \ \ \times \left\{ 
\mathrm{P}_{\kappa }^{\alpha }\left( \pi _{\mathcal{F}}^{m}F,\mathbf{1}_{\pi
_{\mathcal{F}}^{m+1}F\setminus \pi _{\mathcal{F}}^{m}F}\sigma \right) \sqrt{%
\left\vert \pi _{\mathcal{F}}^{m}F\right\vert _{\omega }}\right\} \sqrt{%
\frac{\left\vert F\right\vert _{\omega }}{\left\vert \pi _{\mathcal{F}%
}^{m}F\right\vert _{\omega }}}\left\Vert g_{F}\right\Vert _{W_{\limfunc{dyad}%
}^{-s}\left( \omega \right) } \\
&\leq &\sum_{m=1}^{\infty }\sum_{F\in \mathcal{F}}\left\Vert \mathbb{E}_{\pi
_{\mathcal{F}}^{m+1}F;\kappa _{1}}^{\sigma }f\right\Vert _{\infty }\left( 
\frac{\ell \left( F\right) }{\ell \left( \pi _{\mathcal{F}}^{m}F\right) }%
\right) ^{\eta +\eta ^{\prime }}\ell \left( \pi _{\mathcal{F}}^{m}F\right)
^{-s}\left\{ \mathcal{V}_{2,\beta ,\varepsilon ^{\prime }}^{\alpha ,\kappa
}\ \sqrt{\left\vert \pi _{\mathcal{F}}^{m}F\right\vert _{\sigma }}\right\} 
\sqrt{\frac{\left\vert F\right\vert _{\omega }}{\left\vert \pi _{\mathcal{F}%
}^{m}F\right\vert _{\omega }}}\left\Vert g_{F}\right\Vert _{W_{\limfunc{dyad}%
}^{-s}\left( \omega \right) },
\end{eqnarray*}%
where we have used the pivotal stopping inequality, and written%
\begin{equation*}
\kappa -\varepsilon \left( n+\kappa -\alpha \right) -s+\theta _{\sigma }^{%
\func{rev}}=\eta +\eta ^{\prime },
\end{equation*}%
with $\eta ,\eta ^{\prime }>0$ to be chosen later. Note that this requires
the Alpert parameter $\kappa $ to satisfy%
\begin{equation}
\kappa >\frac{\varepsilon \left( n-\alpha \right) +s-\theta _{\sigma }^{%
\func{rev}}}{1-\varepsilon }.  \label{kappa needed}
\end{equation}%
Then by Cauchy-Schwarz we have%
\begin{eqnarray*}
\left\vert \sum_{F\in \mathcal{F}}\left\langle T_{\sigma }^{\alpha }\gamma
_{F},g_{F}\right\rangle _{\omega }\right\vert &\lesssim &\mathcal{V}%
_{2,\beta ,\varepsilon ^{\prime }}^{\alpha ,\kappa }\left(
\sum_{m=1}^{\infty }\sum_{F\in \mathcal{F}}\left\Vert \mathbb{E}_{\pi _{%
\mathcal{F}}^{m+1}F;\kappa _{1}}^{\sigma }f\right\Vert _{\infty }^{2}\ell
\left( \pi _{\mathcal{F}}^{m}F\right) ^{-2s}\left\vert \pi _{\mathcal{F}%
}^{m}F\right\vert _{\sigma }\left( \frac{\ell \left( F\right) }{\ell \left(
\pi _{\mathcal{F}}^{m}F\right) }\right) ^{2\eta }\frac{\left\vert
F\right\vert _{\omega }}{\left\vert \pi _{\mathcal{F}}^{m}F\right\vert
_{\omega }}\right) ^{\frac{1}{2}} \\
&&\ \ \ \ \ \ \ \ \ \ \ \ \ \ \ \times \left( \sum_{m=1}^{\infty }\sum_{F\in 
\mathcal{F}}\left( \frac{\ell \left( F\right) }{\ell \left( \pi _{\mathcal{F}%
}^{m}F\right) }\right) ^{2\eta ^{\prime }}\left\Vert g_{F}\right\Vert _{W_{%
\limfunc{dyad}}^{-s}\left( \omega \right) }^{2}\right) ^{\frac{1}{2}}.
\end{eqnarray*}

The square of the first factor satisfies%
\begin{eqnarray*}
&&\sum_{m=1}^{\infty }\sum_{F\in \mathcal{F}}\left\Vert \mathbb{E}_{\pi _{%
\mathcal{F}}^{m+1}F;\kappa _{1}}^{\sigma }f\right\Vert _{\infty
}^{2}\left\vert \pi _{\mathcal{F}}^{m}F\right\vert _{\sigma }\ell \left( \pi
_{\mathcal{F}}^{m}F\right) ^{-2s}\left( \frac{\ell \left( F\right) }{\ell
\left( \pi _{\mathcal{F}}^{m}F\right) }\right) ^{2\eta }\frac{\left\vert
F\right\vert _{\omega }}{\left\vert \pi _{\mathcal{F}}^{m}F\right\vert
_{\omega }} \\
&=&\sum_{F^{\prime }\in \mathcal{F}}\left\Vert \mathbb{E}_{\pi _{\mathcal{F}%
}F^{\prime };\kappa _{1}}^{\sigma }f\right\Vert _{\infty }^{2}\ell \left(
F^{\prime }\right) ^{-2s}\left\vert F^{\prime }\right\vert _{\sigma }\sum 
_{\substack{ F\in \mathcal{F}  \\ F\subset F^{\prime }}}\left( \frac{\ell
\left( F\right) }{\ell \left( F^{\prime }\right) }\right) ^{2\eta }\frac{%
\left\vert F\right\vert _{\omega }}{\left\vert F^{\prime }\right\vert
_{\omega }} \\
&\lesssim &\sum_{F^{\prime }\in \mathcal{F}}\left\Vert \mathbb{E}_{\pi _{%
\mathcal{F}}F^{\prime };\kappa _{1}}^{\sigma }f\right\Vert _{\infty
}^{2}\ell \left( F^{\prime }\right) ^{-2s}\left\vert F^{\prime }\right\vert
_{\sigma }=\sum_{F^{\prime \prime }\in \mathcal{F}}\left\Vert \mathbb{E}%
_{F^{\prime \prime };\kappa _{1}}^{\sigma }f\right\Vert _{\infty
}^{2}\sum_{F^{\prime }\in \mathfrak{C}_{\mathcal{F}}\left( F^{\prime \prime
}\right) }\ell \left( F^{\prime }\right) ^{-2s}\left\vert F^{\prime
}\right\vert _{\sigma } \\
&\lesssim &\sum_{F^{\prime \prime }\in \mathcal{F}}\left\Vert \mathbb{E}%
_{F^{\prime \prime };\kappa _{1}}^{\sigma }f\right\Vert _{\infty }^{2}\ell
\left( F^{\prime \prime }\right) ^{-2s}\left\vert F^{\prime \prime
}\right\vert _{\sigma }\lesssim \left\Vert f\right\Vert _{W_{\limfunc{dyad}%
}^{s}\left( \sigma \right) }^{2}\ ,
\end{eqnarray*}%
where the first inequality in the last line follows from the strong $\sigma $%
-Carleson condition, and the second inequality follows from the
Quasiorthogonality Lemma \ref{q lemma} since Lemma \ref{full} shows that $%
\mathcal{F}$ is a $\left( \beta ,\sigma \right) $-full grid for a
sufficiently small $\beta >0$. The square of the second factor satisfies%
\begin{equation*}
\sum_{m=1}^{\infty }\sum_{F\in \mathcal{F}}\left( \frac{\ell \left( F\right) 
}{\ell \left( \pi _{\mathcal{F}}^{m}F\right) }\right) ^{2\eta ^{\prime
}}\left\Vert g_{F}\right\Vert _{W_{\limfunc{dyad}}^{-s}\left( \omega \right)
}^{2}\lesssim \sum_{F\in \mathcal{F}}\left\Vert g_{F}\right\Vert _{W_{%
\limfunc{dyad}}^{-s}\left( \omega \right) }^{2}\leq \left\Vert g\right\Vert
_{W_{\limfunc{dyad}}^{-s}\left( \omega \right) }^{2}\ ,
\end{equation*}%
where we have used $\eta ^{\prime }>0$.

It remains to bound $\sum_{F\in \mathcal{F}}\left\langle T_{\sigma }^{\alpha
}\beta _{F},g_{F}\right\rangle _{\omega }$ where $\beta
_{F}=\sum_{m=1}^{\infty }\sum_{I:\ \pi _{\mathcal{F}}^{m}F\subsetneqq
I\subset \pi _{\mathcal{F}}^{m+1}F}\mathbf{1}_{\theta \left( I\right)
}\left( \mathbb{E}_{I;\kappa _{1}}^{\sigma }f\right) $. The difference
between the previous estimate and this one is that the averages $\mathbf{1}%
_{\pi _{\mathcal{F}}^{m+1}F\setminus \pi _{\mathcal{F}}^{m}F}\left\vert 
\mathbb{E}_{\pi _{\mathcal{F}}^{m+1}F;\kappa _{1}}^{\sigma }f\right\vert $
inside the Poisson kernel have been replaced with the sum of averages $%
\sum_{m=1}^{\infty }\sum_{I:\ \pi _{\mathcal{F}}^{m}F\subsetneqq I\subset
\pi _{\mathcal{F}}^{m+1}F}\mathbf{1}_{\theta \left( I\right) }\left\vert 
\mathbb{E}_{I;\kappa _{1}}^{\sigma }f\right\vert $, but where the sum is
taken over pairwise disjoint sets $\left\{ \theta \left( I\right) \right\}
_{\pi _{\mathcal{F}}^{m}F\subsetneqq I\subset \pi _{\mathcal{F}}^{m+1}F}$.
We start with%
\begin{eqnarray*}
&&\left\vert \sum_{F\in \mathcal{F}}\left\langle T_{\sigma }^{\alpha }\beta
_{F},g_{F}\right\rangle _{\omega }\right\vert \lesssim \sum_{F\in \mathcal{F}%
}\mathrm{P}_{\kappa }^{\alpha }\left( F,\sum_{m=1}^{\infty }\sum_{I:\ \pi _{%
\mathcal{F}}^{m}F\subsetneqq I\subset \pi _{\mathcal{F}}^{m+1}F}\mathbf{1}%
_{\theta \left( I\right) }\left\vert \mathbb{E}_{I;\kappa _{1}}^{\sigma
}f\right\vert \sigma \right) \ell \left( F\right) ^{-s}\sqrt{\left\vert
F\right\vert _{\omega }}\left\Vert g_{F}\right\Vert _{W_{\limfunc{dyad}%
}^{-s}\left( \omega \right) } \\
&=&\sum_{m=1}^{\infty }\sum_{F\in \mathcal{F}}\sum_{I:\ \pi _{\mathcal{F}%
}^{m}F\subsetneqq I\subset \pi _{\mathcal{F}}^{m+1}F}\left\Vert \mathbb{E}%
_{I;\kappa _{1}}^{\sigma }f\right\Vert _{\infty }\mathrm{P}_{\kappa
}^{\alpha }\left( F,\mathbf{1}_{\theta \left( I\right) }\sigma \right) \ell
\left( F\right) ^{-s}\sqrt{\left\vert F\right\vert _{\omega }}\left\Vert
g_{F}\right\Vert _{W_{\limfunc{dyad}}^{-s}\left( \omega \right) }\equiv S.
\end{eqnarray*}

Then we use%
\begin{eqnarray*}
\sum_{I:\ \pi _{\mathcal{F}}^{m}F\subsetneqq I\subset \pi _{\mathcal{F}%
}^{m+1}F}\left\Vert \mathbb{E}_{I;\kappa _{1}}^{\sigma }f\right\Vert
_{\infty }\mathrm{P}_{\kappa }^{\alpha }\left( F,\mathbf{1}_{\theta \left(
I\right) }\sigma \right) &\leq &\left( \sup_{I:\ \pi _{\mathcal{F}%
}^{m}F\subsetneqq I\subset \pi _{\mathcal{F}}^{m+1}F}\left\Vert \mathbb{E}%
_{I;\kappa _{1}}^{\sigma }f\right\Vert _{\infty }\right) \mathrm{P}_{\kappa
}^{\alpha }\left( F,\sum_{I:\ \pi _{\mathcal{F}}^{m}F\subsetneqq I\subset
\pi _{\mathcal{F}}^{m+1}F}\mathbf{1}_{\theta \left( I\right) }\sigma \right)
\\
&=&\left( \sup_{I:\ \pi _{\mathcal{F}}^{m}F\subsetneqq I\subset \pi _{%
\mathcal{F}}^{m+1}F}\left\Vert \mathbb{E}_{I;\kappa _{1}}^{\sigma
}f\right\Vert _{\infty }\right) \mathrm{P}_{\kappa }^{\alpha }\left( F,%
\mathbf{1}_{\pi _{\mathcal{F}}^{m+1}F\setminus \pi _{\mathcal{F}%
}^{m}F}\sigma \right) ,
\end{eqnarray*}%
and obtain that%
\begin{equation*}
S\leq \sum_{m=1}^{\infty }\sum_{F\in \mathcal{F}}\left( \sup_{I:\ \pi _{%
\mathcal{F}}^{m}F\subsetneqq I\subset \pi _{\mathcal{F}}^{m+1}F}\left\Vert 
\mathbb{E}_{I;\kappa _{1}}^{\sigma }f\right\Vert _{\infty }\right) \mathrm{P}%
_{\kappa }^{\alpha }\left( F,\mathbf{1}_{\pi _{\mathcal{F}}^{m+1}F\setminus
\pi _{\mathcal{F}}^{m}F}\sigma \right) \ell \left( F\right) ^{-s}\sqrt{%
\left\vert F\right\vert _{\omega }}\left\Vert g_{F}\right\Vert _{W_{\limfunc{%
dyad}}^{-s}\left( \omega \right) }\ .
\end{equation*}%
Now we define $G_{m}\left[ F\right] \in \left( \pi _{\mathcal{F}}^{m}F,\pi _{%
\mathcal{F}}^{m+1}F\right] $ so that $\sup_{I:\ \pi _{\mathcal{F}%
}^{m}F\subsetneqq I\subset \pi _{\mathcal{F}}^{m+1}F}\left\Vert \mathbb{E}%
_{I;\kappa _{1}}^{\sigma }f\right\Vert _{\infty }=\left\Vert \mathbb{E}%
_{G_{m}\left[ F\right] ;\kappa _{1}}^{\sigma }f\right\Vert _{\infty }$, and
dominate $S$ by%
\begin{eqnarray*}
S &\leq &\sum_{m=1}^{\infty }\sum_{F\in \mathcal{F}}\left\Vert \mathbb{E}%
_{G_{m}\left[ F\right] ;\kappa _{1}}^{\sigma }f\right\Vert _{\infty }\mathrm{%
P}_{\kappa }^{\alpha }\left( F,\mathbf{1}_{\pi _{\mathcal{F}%
}^{m+1}F\setminus \pi _{\mathcal{F}}^{m}F}\sigma \right) \ell \left(
F\right) ^{-s}\sqrt{\left\vert F\right\vert _{\omega }}\left\Vert
g_{F}\right\Vert _{W_{\limfunc{dyad}}^{-s}\left( \omega \right) } \\
&\lesssim &\sum_{m=1}^{\infty }\sum_{F\in \mathcal{F}}\left\Vert \mathbb{E}%
_{G_{m}\left[ F\right] ;\kappa _{1}}^{\sigma }f\right\Vert _{\infty }\left( 
\frac{\ell \left( F\right) }{\ell \left( G_{m}\left[ F\right] \right) }%
\right) ^{\eta }\mathrm{P}_{\kappa }^{\alpha }\left( G_{m}\left[ F\right] ,%
\mathbf{1}_{\pi _{\mathcal{F}}^{m+1}F\setminus \pi _{\mathcal{F}%
}^{m}F}\sigma \right) \ell \left( F\right) ^{-s}\sqrt{\left\vert
F\right\vert _{\omega }}\left\Vert g_{F}\right\Vert _{W_{\limfunc{dyad}%
}^{-s}\left( \omega \right) } \\
&=&\sum_{m=1}^{\infty }\sum_{F\in \mathcal{F}}\left\Vert \mathbb{E}_{G_{m}%
\left[ F\right] ;\kappa _{1}}^{\sigma }f\right\Vert _{\infty }\left( \frac{%
\ell \left( F\right) }{\ell \left( G_{m}\left[ F\right] \right) }\right)
^{\eta }\left\{ \mathrm{P}_{\kappa }^{\alpha }\left( G_{m}\left[ F\right] ,%
\mathbf{1}_{\pi _{\mathcal{F}}^{m+1}F\setminus \pi _{\mathcal{F}%
}^{m}F}\sigma \right) \left( \frac{\ell \left( \pi _{\mathcal{F}%
}^{m+1}F\right) }{\ell \left( G_{m}\left[ F\right] \right) }\right)
^{\varepsilon }\sqrt{\left\vert G_{m}\left[ F\right] \right\vert _{\omega }}%
\right\} \\
&&\ \ \ \ \ \ \ \ \ \ \ \ \ \ \ \ \ \ \ \ \ \ \ \ \ \ \ \ \ \ \ \ \ \ \times
\left( \frac{\ell \left( \pi _{\mathcal{F}}^{m+1}F\right) }{\ell \left( G_{m}%
\left[ F\right] \right) }\right) ^{-\varepsilon }\ell \left( F\right) ^{-s}%
\sqrt{\frac{\left\vert F\right\vert _{\omega }}{\left\vert G_{m}\left[ F%
\right] \right\vert _{\omega }}}\left\Vert g_{F}\right\Vert _{W_{\limfunc{%
dyad}}^{-s}\left( \omega \right) },
\end{eqnarray*}%
and then continue with%
\begin{eqnarray*}
&\lesssim &\mathcal{V}_{2,\beta ,\varepsilon ^{\prime }}^{\alpha ,\kappa
}\sum_{m=1}^{\infty }\sum_{F\in \mathcal{F}}\left\Vert \mathbb{E}_{G_{m}%
\left[ F\right] ;\kappa _{1}}^{\sigma }f\right\Vert _{\infty }\left( \frac{%
\ell \left( F\right) }{\ell \left( G\left[ F\right] \right) }\right) ^{\eta
}\ell \left( F\right) ^{-s}\left\{ \sqrt{\left\vert G_{m}\left[ F\right]
\right\vert _{\sigma }}\right\} \left( \frac{\ell \left( \pi _{\mathcal{F}%
}^{m+1}F\right) }{\ell \left( G_{m}\left[ F\right] \right) }\right)
^{-\varepsilon }\sqrt{\frac{\left\vert F\right\vert _{\omega }}{\left\vert
G_{m}\left[ F\right] \right\vert _{\omega }}}\left\Vert g_{F}\right\Vert
_{W_{\limfunc{dyad}}^{-s}\left( \omega \right) } \\
&=&\mathcal{V}_{2,\beta ,\varepsilon ^{\prime }}^{\alpha ,\kappa
}\sum_{m=1}^{\infty }\sum_{F\in \mathcal{F}}\left\Vert \mathbb{E}_{G\left[ F%
\right] ;\kappa _{1}}^{\sigma }f\right\Vert _{\infty }\ell \left( G_{m}\left[
F\right] \right) ^{-s}\sqrt{\left\vert G_{m}\left[ F\right] \right\vert
_{\sigma }}\left( \frac{\ell \left( F\right) }{\ell \left( G_{m}\left[ F%
\right] \right) }\right) ^{\eta -s}\left( \frac{\ell \left( \pi _{\mathcal{F}%
}^{m+1}F\right) }{\ell \left( G_{m}\left[ F\right] \right) }\right)
^{-\varepsilon }\sqrt{\frac{\left\vert F\right\vert _{\omega }}{\left\vert
G_{m}\left[ F\right] \right\vert _{\omega }}}\left\Vert g_{F}\right\Vert
_{W_{\limfunc{dyad}}^{-s}\left( \omega \right) } \\
&\approx &\mathcal{V}_{2,\beta ,\varepsilon ^{\prime }}^{\alpha ,\kappa
}\sum_{m=1}^{\infty }\sum_{F\in \mathcal{F}}\left\Vert \bigtriangleup _{G_{m}%
\left[ F\right] ;\kappa _{1}}^{\sigma }f\right\Vert _{W_{\limfunc{dyad}%
}^{s}\left( \sigma \right) }\left( \frac{\ell \left( F\right) }{\ell \left(
G_{m}\left[ F\right] \right) }\right) ^{\eta -s}\left( \frac{\ell \left( \pi
_{\mathcal{F}}^{m+1}F\right) }{\ell \left( G_{m}\left[ F\right] \right) }%
\right) ^{-\varepsilon }\sqrt{\frac{\left\vert F\right\vert _{\omega }}{%
\left\vert G_{m}\left[ F\right] \right\vert _{\omega }}}\left\Vert
g_{F}\right\Vert _{W_{\limfunc{dyad}}^{-s}\left( \omega \right) }.
\end{eqnarray*}%
Since there is geometric gain in the product 
\begin{equation*}
\left( \frac{\ell \left( F\right) }{\ell \left( G_{m}\left[ F\right] \right) 
}\right) ^{\eta -s}\left( \frac{\ell \left( \pi _{\mathcal{F}}^{m+1}F\right) 
}{\ell \left( G_{m}\left[ F\right] \right) }\right) ^{-\varepsilon }\sqrt{%
\frac{\left\vert F\right\vert _{\omega }}{\left\vert G_{m}\left[ F\right]
\right\vert _{\omega }}}\leq \left( \frac{\ell \left( F\right) }{\ell \left(
G_{m}\left[ F\right] \right) }\right) ^{\eta -s}\lesssim 2^{-m\left( \eta
-s\right) },
\end{equation*}%
provided $\eta >s$, an application of Cauchy-Schwarz finishes the proof
since $G_{m}\left[ F\right] $ is uniquely determined by $F$ and $m$ in the
tower $\left( \pi _{\mathcal{F}}^{m}F,\pi _{\mathcal{F}}^{m+1}F\right] $:%
\begin{eqnarray*}
&&\mathcal{V}_{2,\beta ,\varepsilon ^{\prime }}^{\alpha ,\kappa
}\sum_{m=1}^{\infty }2^{-m\left( \eta -s\right) }\sum_{F\in \mathcal{F}%
}\left\Vert \bigtriangleup _{G_{m}\left[ F\right] ;\kappa _{1}}^{\sigma
}f\right\Vert _{W_{\limfunc{dyad}}^{s}\left( \sigma \right) }\left\Vert
g_{F}\right\Vert _{W_{\limfunc{dyad}}^{-s}\left( \omega \right) } \\
&\leq &\mathcal{V}_{2,\beta ,\varepsilon ^{\prime }}^{\alpha ,\kappa
}\sum_{m=1}^{\infty }2^{-m\left( \eta -s\right) }\left( \sum_{F\in \mathcal{F%
}}\left\Vert \bigtriangleup _{G_{m}\left[ F\right] ;\kappa _{1}}^{\sigma
}f\right\Vert _{W_{\limfunc{dyad}}^{s}\left( \sigma \right) }^{2}\right) ^{%
\frac{1}{2}}\left( \sum_{F\in \mathcal{F}}\left\Vert g_{F}\right\Vert _{W_{%
\limfunc{dyad}}^{-s}\left( \omega \right) }^{2}\right) ^{\frac{1}{2}} \\
&\leq &\mathcal{V}_{2,\beta ,\varepsilon ^{\prime }}^{\alpha ,\kappa
}\sum_{m=1}^{\infty }2^{-m\left( \eta -s\right) }\left\Vert f\right\Vert
_{W_{\limfunc{dyad}}^{s}\left( \sigma \right) }\left\Vert g\right\Vert _{W_{%
\limfunc{dyad}}^{-s}\left( \omega \right) }=C\mathcal{V}_{2,\beta
,\varepsilon ^{\prime }}^{\alpha ,\kappa }\left\Vert f\right\Vert _{W_{%
\limfunc{dyad}}^{s}\left( \sigma \right) }\left\Vert g\right\Vert _{W_{%
\limfunc{dyad}}^{-s}\left( \omega \right) }\ .
\end{eqnarray*}%
Note that we have used $\eta >s$, which requires a bit more on $\kappa $
than was used in (\ref{kappa needed}), namely that%
\begin{equation*}
\kappa -\varepsilon \left( n+\kappa -\alpha \right) -s+\theta _{\sigma }^{%
\func{rev}}=\eta +\eta ^{\prime }>s,
\end{equation*}%
which requires 
\begin{equation}
\kappa >\frac{\varepsilon \left( n-\alpha \right) +2s-\theta _{\sigma }^{%
\func{rev}}}{1-\varepsilon }.  \label{kappa needed'}
\end{equation}
\end{proof}

\subsubsection{The diagonal form}

To handle the diagonal term $\mathsf{T}_{\limfunc{diagonal}}\left(
f,g\right) $, we decompose according to the stopping times $\mathcal{F}$,
which we recall are $\beta $-full for some $\beta >0$ by Lemma \ref{full},%
\begin{equation*}
\mathsf{T}_{\limfunc{diagonal}}\left( f,g\right) =\sum_{F\in \mathcal{F}}%
\mathsf{B}_{\Subset _{\mathbf{\rho }}}^{F}\left( f,g\right) \equiv
\left\langle T_{\sigma }^{\alpha }\left( \mathsf{P}_{\mathcal{C}%
_{F}}^{\sigma }f\right) ,\mathsf{P}_{\mathcal{C}_{F}^{\tau -\limfunc{shift}%
}}^{\omega }g\right\rangle _{\omega }^{\Subset _{\mathbf{\rho }}},
\end{equation*}%
and it is enough, using Cauchy-Schwarz and quasiorthogonality (\ref{quasi
orth}) in $f$, together with orthogonality (\ref{orth}) in both $f$ and $g$,
to prove the `below form' bound involving the usual cube testing constant,%
\begin{equation}
\left\vert \mathsf{B}_{\Subset _{\rho }}^{F}\left( f,g\right) \right\vert
\lesssim \left( \mathfrak{T}_{T^{\alpha }}^{s}+\sqrt{A_{2}^{\alpha }}\right)
\ \left( \ell \left( F\right) ^{-s}\left\Vert \mathbb{E}_{F;\kappa }^{\sigma
}f\right\Vert _{\infty }\sqrt{\left\vert F\right\vert _{\sigma }}+\left\Vert 
\mathsf{P}_{\mathcal{C}_{F}}^{\sigma }f\right\Vert _{W_{\limfunc{dyad}%
}^{s}\left( \sigma \right) }\right) \ \left\Vert \mathsf{P}_{\mathcal{C}%
_{F}^{\tau -\limfunc{shift}}}^{\omega }g\right\Vert _{W_{\limfunc{dyad}%
}^{-s}\left( \omega \right) }\ .  \label{below form bound}
\end{equation}%
Indeed, using quasiorthogonality, Lemma \ref{q lemma}, and orthogonality of
projections $\mathsf{P}_{\mathcal{C}_{F}}^{\sigma }f$ and $\mathsf{P}_{%
\mathcal{C}_{F}^{\tau -\limfunc{shift}}}^{\omega }g$ this then gives the
estimate,%
\begin{equation}
\left\vert \mathsf{T}_{\limfunc{diagonal}}\left( f,g\right) \right\vert
\lesssim \left( \mathfrak{T}_{T^{\alpha }}^{s}+\sqrt{A_{2}^{\alpha }}\right)
\left\Vert f\right\Vert _{W_{\limfunc{dyad}}^{s}\left( \sigma \right)
}\left\Vert g\right\Vert _{W_{\limfunc{dyad}}^{-s}\left( \omega \right) }.
\label{diag est}
\end{equation}

Thus at this point we have essentially reduced the proof of Theorem \ref%
{pivotal theorem} to

\begin{enumerate}
\item proving (\ref{below form bound}),

\item and controlling the triple polynomial testing condition (\ref{full
testing}) by the usual cube testing condition and the classical Muckenhoupt
condition.
\end{enumerate}

In the next section we address the first issue by proving the inequality (%
\ref{below form bound}) for the below forms $\mathsf{B}_{\Subset _{\rho
}}^{F}\left( f,g\right) $. In the final section, we address the second issue
and complete the proofs of our theorems by drawing together all of the
estimates.

\section{The Nazarov, Treil and Volberg reach for Alpert wavelets}

It will be convenient to denote our fractional singular integral operators
by $T^{\lambda }$, $0\leq \lambda <n$, instead of $T^{\alpha }$, thus
freeing up $\alpha $ for the familiar role of denoting multi-indices in $%
\mathbb{Z}_{+}^{n}$. Before getting started, we note that for a doubling
measure $\mu $, a cube $I$ and a polynomial $P$, we have%
\begin{equation*}
\left\Vert P\mathbf{1}_{I}\right\Vert _{L^{\infty }\left( \mu \right)
}=\sup_{x\in I}\left\vert P\left( x\right) \right\vert
\end{equation*}%
because $\mu $ charges all open sets, and so in particular, $\left\Vert P%
\mathbf{1}_{I}\right\Vert _{L^{\infty }\left( \sigma \right) }=\left\Vert P%
\mathbf{1}_{I}\right\Vert _{L^{\infty }\left( \omega \right) }$.

We will often follow the analogous arguments in \cite{AlSaUr}, and point out
the places where significant new approaches are needed. See \cite{AlSaUr}
for a review of the classical reach of Nazarov, Treil and Volberg using Haar
wavelet projections $\bigtriangleup _{I}^{\sigma }$, namely the beautiful
and ingenious `thinking outside the box' idea of the paraproduct / stopping
/ neighbour decomposition of Nazarov, Treil and Volberg \cite{NTV4} using
Haar wavelets.

When using weighted Alpert wavelet projections $\bigtriangleup _{I;\kappa
}^{\sigma }$ instead, the projection $\mathbb{E}_{I^{\prime };\kappa
}^{\sigma }\bigtriangleup _{I;\kappa }^{\sigma }f$ onto the child $I^{\prime
}\in \mathfrak{C}_{\mathcal{D}}\left( I\right) $ equals $M_{I^{\prime
};\kappa }\mathbf{1}_{I_{\pm }}$ where $M=M_{I^{\prime };\kappa }$ is a
polynomial of degree less than $\kappa $ restricted to $I^{\prime }$, and
hence no longer commutes with the operator $T_{\sigma }^{\lambda }$ - unless
it is the constant polynomial. We now recall the modifications used in \cite%
{AlSaUr}, where they obtained, 
\begin{eqnarray*}
&&\mathsf{B}_{\Subset _{\rho ,\varepsilon };\kappa }^{F}\left( f,g\right)
\equiv \sum_{\substack{ I\in \mathcal{C}_{F}\text{ and }J\in \mathcal{C}%
_{F}^{\tau -\limfunc{shift}}  \\ J\Subset _{\rho ,\varepsilon }I}}%
\left\langle T_{\sigma }^{\lambda }\bigtriangleup _{I;\kappa }^{\sigma
}f,\bigtriangleup _{J;\kappa }^{\omega }g\right\rangle _{\omega } \\
&=&\sum_{\substack{ I\in \mathcal{C}_{F}\text{ and }J\in \mathcal{C}_{F}^{%
\mathbf{\tau }-\limfunc{shift}}  \\ J\Subset _{\rho ,\varepsilon }I}}%
\left\langle T_{\sigma }^{\lambda }\left( \mathbf{1}_{I_{J}}\bigtriangleup
_{I;\kappa }^{\sigma }f\right) ,\bigtriangleup _{J;\kappa }^{\omega
}g\right\rangle _{\omega }+\sum_{\substack{ I\in \mathcal{C}_{F}\text{ and }%
J\in \mathcal{C}_{F}^{\mathbf{\tau }-\limfunc{shift}}  \\ J\Subset _{\rho
,\varepsilon }I}}\sum_{\theta \left( I_{J}\right) \in \mathfrak{C}_{\mathcal{%
D}}\left( I\right) \setminus \left\{ I_{J}\right\} }\left\langle T_{\sigma
}^{\lambda }\left( \mathbf{1}_{\theta \left( I_{J}\right) }\bigtriangleup
_{I;\kappa }^{\sigma }f\right) ,\bigtriangleup _{J;\kappa }^{\omega
}g\right\rangle _{\omega } \\
&\equiv &\mathsf{B}_{\func{home};\kappa }^{F}\left( f,g\right) +\mathsf{B}_{%
\limfunc{neighbour};\kappa }^{F}\left( f,g\right) .
\end{eqnarray*}%
They further decomposed the $\mathsf{B}_{\func{home};\kappa }^{F}$ form
using 
\begin{equation}
M_{I^{\prime }}=M_{I^{\prime };\kappa }\equiv \mathbf{1}_{I^{\prime
}}\bigtriangleup _{I;\kappa }^{\sigma }f=\mathbb{E}_{I^{\prime };\kappa
}^{\sigma }\bigtriangleup _{I;\kappa }^{\sigma }f,  \label{def M}
\end{equation}%
to obtain%
\begin{eqnarray*}
\mathsf{B}_{\func{home};\kappa }^{F}\left( f,g\right) &=&\sum_{\substack{ %
I\in \mathcal{C}_{F}\text{ and }J\in \mathcal{C}_{F}^{\tau -\limfunc{shift}} 
\\ J\Subset _{\rho ,\varepsilon }I}}\left\langle T_{\sigma }^{\lambda
}\left( M_{I_{J}}\mathbf{1}_{I_{J}}\right) ,\bigtriangleup _{J;\kappa
}^{\omega }g\right\rangle _{\omega } \\
&=&\sum_{\substack{ I\in \mathcal{C}_{F}\text{ and }J\in \mathcal{C}%
_{F}^{\tau -\limfunc{shift}}  \\ J\Subset _{\rho ,\varepsilon }I}}%
\left\langle M_{I_{J}}T_{\sigma }^{\lambda }\mathbf{1}_{I_{J}},%
\bigtriangleup _{J;\kappa }^{\omega }g\right\rangle _{\omega }+\sum 
_{\substack{ I\in \mathcal{C}_{F}\text{ and }J\in \mathcal{C}_{F}^{\tau -%
\limfunc{shift}}  \\ J\Subset _{\rho ,\varepsilon }I}}\left\langle \left[
T_{\sigma }^{\lambda },M_{I_{J}}\right] \mathbf{1}_{I_{J}},\bigtriangleup
_{J;\kappa }^{\omega }g\right\rangle _{\omega } \\
&=&\sum_{\substack{ I\in \mathcal{C}_{F}\text{ and }J\in \mathcal{C}%
_{F}^{\tau -\limfunc{shift}}  \\ J\Subset _{\rho ,\varepsilon }I}}%
\left\langle M_{I_{J}}T_{\sigma }^{\lambda }\mathbf{1}_{F},\bigtriangleup
_{J;\kappa }^{\omega }g\right\rangle _{\omega }-\sum_{\substack{ I\in 
\mathcal{C}_{F}\text{ and }J\in \mathcal{C}_{F}^{\tau -\limfunc{shift}}  \\ %
J\Subset _{\rho ,\varepsilon }I}}\left\langle M_{I_{J}}T_{\sigma }^{\lambda }%
\mathbf{1}_{F\setminus I_{J}},\bigtriangleup _{J;\kappa }^{\omega
}g\right\rangle _{\omega } \\
&&+\sum_{\substack{ I\in \mathcal{C}_{F}\text{ and }J\in \mathcal{C}%
_{F}^{\tau -\limfunc{shift}}  \\ J\Subset _{\rho ,\varepsilon }I}}%
\left\langle \left[ T_{\sigma }^{\lambda },M_{I_{J}}\right] \mathbf{1}%
_{I_{J}},\bigtriangleup _{J;\kappa }^{\omega }g\right\rangle _{\omega } \\
&\equiv &\mathsf{B}_{\limfunc{paraproduct};\kappa }^{F}\left( f,g\right) -%
\mathsf{B}_{\limfunc{stop};\kappa }^{F}\left( f,g\right) +\mathsf{B}_{%
\limfunc{commutator};\kappa }^{F}\left( f,g\right) .
\end{eqnarray*}%
Altogether then we have the weighted Alpert version of the Nazarov, Treil
and Volberg paraproduct decomposition that was obtained by Alexis, Sawyer
and Uriarte-Tuero in \cite{AlSaUr},%
\begin{equation*}
\mathsf{B}_{\Subset _{\mathbf{\rho },\varepsilon };\kappa }^{F}\left(
f,g\right) =\mathsf{B}_{\limfunc{paraproduct};\kappa }^{F}\left( f,g\right) +%
\mathsf{B}_{\limfunc{stop};\kappa }^{F}\left( f,g\right) +\mathsf{B}_{%
\limfunc{commutator};\kappa }^{F}\left( f,g\right) +\mathsf{B}_{\limfunc{%
neighbour};\kappa }^{F}\left( f,g\right) .
\end{equation*}

\subsection{The paraproduct form}

Following \cite{AlSaUr}, we first pigeonhole the sum over pairs $I$ and $J$
arising in the paraproduct form according to which child $K\in \mathfrak{C}_{%
\mathcal{D}}\left( I\right) $ contains $J$,%
\begin{equation*}
\mathsf{B}_{\limfunc{paraproduct};\kappa }^{F}\left( f,g\right) =\sum_{I\in 
\mathcal{C}_{F}}\sum_{K\in \mathfrak{C}_{\mathcal{D}}\left( I\right) }\sum 
_{\substack{ J\in \mathcal{C}_{F}^{\tau -\limfunc{shift}}:\ J\Subset _{\rho
,\varepsilon }I  \\ J\subset K}}\left\langle M_{K;\kappa }T_{\sigma
}^{\lambda }\mathbf{1}_{F},\bigtriangleup _{J;\kappa }^{\omega
}g\right\rangle _{\omega }.
\end{equation*}%
The form $\mathsf{B}_{\limfunc{paraproduct};\kappa }^{F}\left( f,g\right) $
will be handled using the telescoping property in part (2) of Theorem \ref%
{main1}, to sum the restrictions to a cube $J\in \mathcal{C}_{F}^{\tau -%
\limfunc{shift}}$ of the polynomials $M_{K;\kappa }$ on a child $K\in 
\mathfrak{C}_{\mathcal{D}}\left( I\right) $ of $I$, over the relevant cubes $%
I$, to obtain a restricted polynomial $\mathbf{1}_{J}P_{K;\kappa }$ that is
controlled by $E_{F}^{\sigma }\left\vert f\right\vert $, and then passing
the polynomial $M_{J;\kappa }$ over to $\bigtriangleup _{J;\kappa }^{\omega
}g$. More precisely, for each $J\in \mathcal{C}_{F}^{\mathbf{\tau }-\limfunc{%
shift}}$, let $I_{J}^{\natural }$ denote the smallest $L\in \mathcal{C}_{F}$
such that $J\Subset _{\rho ,\varepsilon }L$ provided it exists. Note that $J$
is at most $\tau $ levels below the bottom of the corona $\mathcal{C}_{F}$,
and since $\rho >2\tau $, we have that \emph{either} $\pi _{\mathcal{D}%
}^{\left( \rho \right) }J\in \mathcal{C}_{F}$ \emph{or} that $\pi _{\mathcal{%
D}}^{\left( \rho \right) }J\supsetneqq F$. Let $I_{J}^{\flat }$ denote the $%
\mathcal{D}$-child of $I_{J}^{\natural }$ that contains $J$, provided $%
I_{J}^{\natural }$ exists. We have 
\begin{equation}
\sum_{I\in \mathcal{C}_{F}:\ I_{J}^{\natural }\subset I}\mathbf{1}%
_{J}M_{K;\kappa }=\mathbf{1}_{J}\sum_{I\in \mathcal{C}_{F}:\ I_{J}^{\natural
}\subset I}M_{K;\kappa }=\mathbf{1}_{J}\left( \mathbb{E}_{I_{J}^{\flat
};\kappa }^{\sigma }f-\mathbb{E}_{F;\kappa }^{\sigma }f\right) \equiv
\left\{ 
\begin{array}{ccc}
\mathbf{1}_{J}P_{J;\kappa } & \text{ if } & I_{J}^{\natural }\text{ exists}
\\ 
0 & \text{ if } & I_{J}^{\natural }\text{ doesn't exist}%
\end{array}%
\right. .  \label{either event}
\end{equation}%
Then we write%
\begin{equation*}
\mathsf{B}_{\limfunc{paraproduct};\kappa }^{F}\left( f,g\right) =\sum 
_{\substack{ I\in \mathcal{C}_{F}\text{, }I_{J}^{\flat }\in \mathcal{C}_{F}%
\text{ and }J\in \mathcal{C}_{F}^{\tau -\limfunc{shift}}  \\ J\Subset _{%
\mathbf{\rho },\varepsilon }I}}\left\langle M_{I_{J};\kappa }T_{\sigma
}^{\lambda }\mathbf{1}_{F},\bigtriangleup _{J;\kappa }^{\omega
}g\right\rangle _{\omega }\ .
\end{equation*}

From (\ref{either event}) we obtain 
\begin{equation}
\left\Vert \mathbf{1}_{J}P_{J;\kappa }\right\Vert _{L^{\infty }\left( \sigma
\right) }\leq \left\Vert \mathbb{E}_{I_{J}^{\flat };\kappa }^{\sigma
}f\right\Vert _{L^{\infty }\left( \sigma \right) }+\left\Vert \mathbb{E}%
_{F;\kappa }^{\sigma }f\right\Vert _{L^{\infty }\left( \sigma \right) },
\label{unif bdd}
\end{equation}%
and so 
\begin{eqnarray*}
\mathsf{B}_{\limfunc{paraproduct};\kappa }^{F}\left( f,g\right)
&=&\sum_{J\in \mathcal{C}_{F}^{\tau -\limfunc{shift}}}\left\langle \mathbf{1}%
_{J}\left( \sum_{\substack{ I\in \mathcal{C}_{F}:\ I_{J}^{\flat }\in 
\mathcal{C}_{F}  \\ J\Subset _{\mathbf{\rho },\varepsilon }I}}\sum 
_{\substack{ I^{\prime }\in \mathfrak{C}_{\mathcal{D}}\left( I\right)  \\ %
J\subset I^{\prime }\text{ }}}M_{I^{\prime };\kappa }\right) T_{\sigma
}^{\lambda }\mathbf{1}_{F},\bigtriangleup _{J;\kappa }^{\omega
}g\right\rangle _{\omega } \\
&=&\sum_{J\in \mathcal{C}_{F}^{\tau -\limfunc{shift}}}\left\langle \mathbf{1}%
_{J}P_{J;\kappa }T_{\sigma }^{\lambda }\mathbf{1}_{F},\bigtriangleup
_{J;\kappa }^{\omega }g\right\rangle _{\omega }
\end{eqnarray*}%
to obtain%
\begin{eqnarray*}
\left\vert \mathsf{B}_{\limfunc{paraproduct};\kappa }^{F}\left( f,g\right)
\right\vert &=&\left\vert \sum_{J\in \mathcal{C}_{F}^{\tau -\limfunc{shift}%
}}\left\langle T_{\sigma }^{\lambda }\mathbf{1}_{F},P_{J;\kappa
}\bigtriangleup _{J;\kappa }^{\omega }g\right\rangle _{\omega }\right\vert
=\left\vert \left\langle T_{\sigma }^{\lambda }\mathbf{1}_{F},\sum_{J\in 
\mathcal{C}_{F}^{\tau -\limfunc{shift}}}P_{J;\kappa }\bigtriangleup
_{J;\kappa }^{\omega }g\right\rangle _{\omega }\right\vert \\
&\leq &\left\Vert T_{\sigma }^{\lambda }\mathbf{1}_{F}\right\Vert _{W_{\func{%
dyad}}^{s}\left( \mu \right) }\left\Vert \mathbb{E}_{F;\kappa }^{\sigma
}f\right\Vert _{L^{\infty }\left( \sigma \right) }\left\Vert \sum_{J\in 
\mathcal{C}_{F}^{\tau -\limfunc{shift}}}\frac{P_{J;\kappa }}{E_{F}^{\sigma
}\left\vert f\right\vert }\bigtriangleup _{J;\kappa }^{\omega }g\right\Vert
_{W_{\func{dyad}}^{-s}\left( \omega \right) } \\
&\leq &\mathfrak{T}_{T^{\lambda }}^{s}\ell \left( F\right) ^{-s}\sqrt{%
\left\vert F\right\vert _{\sigma }}\left\Vert \mathbb{E}_{F;\kappa }^{\sigma
}f\right\Vert _{L^{\infty }\left( \sigma \right) }\left\Vert \sum_{J\in 
\mathcal{C}_{F}^{\tau -\limfunc{shift}}}\frac{P_{J;\kappa }}{E_{F}^{\sigma
}\left\vert f\right\vert }\bigtriangleup _{J;\kappa }^{\omega }g\right\Vert
_{W_{\func{dyad}}^{-s}\left( \omega \right) }.
\end{eqnarray*}

Now we will use an almost orthogonality argument that reflects the fact that
for $J^{\prime }$ small compared to $J$, the function $M_{J^{\prime };\kappa
}\bigtriangleup _{J^{\prime };\kappa }^{\omega }g$ has vanishing $\omega $%
-mean, and the polynomial $\mathbf{1}_{J}P_{J;\kappa }^{\func{corona}%
}\bigtriangleup _{J;\kappa }^{\omega }g$ is relatively smooth at the scale
of $J^{\prime }$, together with the fact that the polynomials 
\begin{equation*}
R_{J;\kappa }\equiv \frac{P_{J;\kappa }}{E_{F}^{\sigma }\left\vert
f\right\vert }
\end{equation*}%
of degree at most $\kappa -1$, have $L^{\infty }$ norm uniformly bounded by
the constant $C$ appearing in (\ref{unif bdd}). We begin by writing 
\begin{eqnarray*}
&&\left\Vert \sum_{J\in \mathcal{C}_{F}^{\tau -\limfunc{shift}}}R_{J;\kappa
}\bigtriangleup _{J;\kappa }^{\omega }g\right\Vert _{W_{\func{dyad}%
}^{-s}\left( \omega \right) }^{2}=\left\langle \sum_{J\in \mathcal{C}%
_{F}^{\tau -\limfunc{shift}}}R_{J;\kappa }\bigtriangleup _{J;\kappa
}^{\omega }g,\sum_{J^{\prime }\in \mathcal{C}_{F}^{\tau -\limfunc{shift}%
}}R_{J^{\prime };\kappa }\bigtriangleup _{J^{\prime };\kappa }^{\omega
}g\right\rangle _{W_{\func{dyad}}^{-s}\left( \omega \right) } \\
&=&\left\{ \sum_{\substack{ J,J^{\prime }\in \mathcal{C}_{F}^{\tau -\limfunc{%
shift}}  \\ J^{\prime }\subset J}}+\sum_{\substack{ J,J^{\prime }\in 
\mathcal{C}_{F}^{\tau -\limfunc{shift}}  \\ J\subsetneqq J^{\prime }}}%
\right\} \left\langle R_{J;\kappa }\bigtriangleup _{J;\kappa }^{\omega
}g,R_{J^{\prime };\kappa }\bigtriangleup _{J^{\prime };\kappa }^{\omega
}g\right\rangle _{W_{\func{dyad}}^{-s}\left( \omega \right) }\equiv A+B.
\end{eqnarray*}%
By symmetry, it suffices to estimate term $A$. We will use the definitions 
\begin{eqnarray*}
\left\Vert f\right\Vert _{W_{\func{dyad}}^{s}\left( \mu \right) }^{2}
&\equiv &\sum_{Q\in \mathcal{D}}\ell \left( Q\right) ^{-2s}\left\Vert
\bigtriangleup _{Q;\kappa }^{\mu }f\right\Vert _{L^{2}\left( \mu \right)
}^{2}, \\
\left\langle f,h\right\rangle _{W_{\func{dyad}}^{s}\left( \mu \right) }
&\equiv &\sum_{Q\in \mathcal{D}}\ell \left( Q\right) ^{-2s}\left\langle
\bigtriangleup _{Q;\kappa }^{\mu }f,\bigtriangleup _{Q;\kappa }^{\mu
}h\right\rangle _{L^{2}\left( \mu \right) },
\end{eqnarray*}%
together with the fact that when $\kappa =1$, $\bigtriangleup _{Q;1}^{\omega
}R_{J;\kappa }\bigtriangleup _{J;\kappa }^{\omega }g$ has one vanishing
moment, to obtain%
\begin{eqnarray*}
A &=&\sum_{\substack{ J,J^{\prime }\in \mathcal{C}_{F}^{\tau -\limfunc{shift}%
}  \\ J^{\prime }\subset J}}\sum_{\substack{ Q\in \mathcal{D}  \\ Q\subset J 
}}\ell \left( Q\right) ^{2s}\int_{\mathbb{R}^{n}}\left( \bigtriangleup
_{Q;1}^{\omega }R_{J;\kappa }\bigtriangleup _{J;\kappa }^{\omega }g\right)
\left( \bigtriangleup _{Q;1}^{\omega }R_{J^{\prime };\kappa }\bigtriangleup
_{J^{\prime };\kappa }^{\omega }g\right) d\omega \\
&\lesssim &\sum_{J\in \mathcal{C}_{F}^{\tau -\limfunc{shift}}}\left\Vert
\bigtriangleup _{J;\kappa }^{\omega }g\right\Vert _{W_{\func{dyad}%
}^{-s}\left( \omega \right) }^{2}=\sum_{J\in \mathcal{C}_{F}^{\tau -\limfunc{%
shift}}}\ell \left( J\right) ^{2s}\left\Vert \bigtriangleup _{J;\kappa
}^{\omega }g\right\Vert _{L^{2}\left( \omega \right) }^{2}=\left\Vert 
\mathsf{P}_{\mathcal{C}_{F}^{\tau -\limfunc{shift}}}^{\omega }g\right\Vert
_{W_{\func{dyad}}^{-s}\left( \omega \right) }^{2}.
\end{eqnarray*}%
Note that here it is important to know that $W_{\func{dyad}}^{-s}\left(
\omega \right) $ equals both $W_{\mathcal{D};\kappa }^{-s}\left( \omega
\right) $ and $W_{\mathcal{D};1}^{-s}\left( \omega \right) $.

Indeed, if $J^{\prime }$ is small compared to $J$, and $J^{\prime }\subset
J_{J^{\prime }}\subset J$, then there are just three possibilities for $Q$,
namely $Q\cap J^{\prime }=\emptyset $, $Q\supsetneqq J^{\prime }$, and $%
Q\subset J^{\prime }$. If $Q\cap J^{\prime }=\emptyset $ then the integral
vanishes by support considerations. If $Q\supsetneqq J^{\prime }$ then the
integral vanishes since $\bigtriangleup _{Q;1}^{\omega }\left( R_{J;\kappa
}\bigtriangleup _{J;\kappa }^{\omega }g\right) $ is constant on $J^{\prime }$%
, $R_{J^{\prime };\kappa }\bigtriangleup _{J^{\prime };\kappa }^{\omega }g$
has $\omega $-mean $0$ on $J^{\prime }$, and $\bigtriangleup _{Q;1}^{\omega
} $ is a projection. Thus we are left with the case where $Q\subset
J^{\prime }\subset J$. We have 
\begin{eqnarray*}
&&\left\vert \int_{\mathbb{R}^{n}}\left( \bigtriangleup _{Q;1}^{\omega
}\left( R_{J;\kappa }\bigtriangleup _{J;\kappa }^{\omega }g\right) \right)
\left( \bigtriangleup _{Q;1}^{\omega }\left( R_{J^{\prime };\kappa
}\bigtriangleup _{J^{\prime };\kappa }^{\omega }g\right) \right) d\omega
\right\vert \\
&=&\left\Vert R_{J;\kappa }\bigtriangleup _{J;\kappa }^{\omega }g\right\Vert
_{L^{\infty }\left( \omega \right) }\left\Vert R_{J^{\prime };\kappa
}\bigtriangleup _{J^{\prime };\kappa }^{\omega }g\right\Vert _{L^{\infty
}\left( \omega \right) } \\
&&\times \left\vert \int_{\mathbb{R}^{n}}\left( \frac{\bigtriangleup
_{Q;1}^{\omega }\left( R_{J;\kappa }\bigtriangleup _{J;\kappa }^{\omega
}g\right) }{\left\Vert R_{J;\kappa }\bigtriangleup _{J;\kappa }^{\omega
}g\right\Vert _{L^{\infty }\left( \omega \right) }}\right) \left( \frac{%
\bigtriangleup _{Q;1}^{\omega }\left( R_{J^{\prime };\kappa }\bigtriangleup
_{J^{\prime };\kappa }^{\omega }g\right) }{\left\Vert R_{J^{\prime };\kappa
}\bigtriangleup _{J^{\prime };\kappa }^{\omega }g\right\Vert _{L^{\infty
}\left( \omega \right) }}\right) d\omega \right\vert ,
\end{eqnarray*}%
and now if $Q\subset J^{\prime }$, we get%
\begin{eqnarray*}
&&\int_{\mathbb{R}^{n}}\left( \frac{\bigtriangleup _{Q;1}^{\omega }\left(
R_{J;\kappa }\bigtriangleup _{J;\kappa }^{\omega }g\right) }{\left\Vert
R_{J;\kappa }\bigtriangleup _{J;\kappa }^{\omega }g\right\Vert _{L^{\infty
}\left( \omega \right) }}\right) \left( \frac{\bigtriangleup _{Q;1}^{\omega
}\left( R_{J^{\prime };\kappa }\bigtriangleup _{J^{\prime };\kappa }^{\omega
}g\right) }{\left\Vert R_{J^{\prime };\kappa }\bigtriangleup _{J^{\prime
};\kappa }^{\omega }g\right\Vert _{L^{\infty }\left( \omega \right) }}%
\right) d\omega \\
&=&\int_{\mathbb{R}^{n}}\left( \frac{\bigtriangleup _{Q;1}^{\omega
}R_{J;\kappa }\bigtriangleup _{J;\kappa }^{\omega }g}{\left\Vert R_{J;\kappa
}\bigtriangleup _{J;\kappa }^{\omega }g\right\Vert _{L^{\infty }\left(
\omega \right) }}\right) \left( \frac{R_{J^{\prime };\kappa }\bigtriangleup
_{J^{\prime };\kappa }^{\omega }g}{\left\Vert R_{J^{\prime };\kappa
}\bigtriangleup _{J^{\prime };\kappa }^{\omega }g\right\Vert _{L^{\infty
}\left( \omega \right) }}\right) d\omega
\end{eqnarray*}%
where $\bigtriangleup _{Q;1}^{\omega }$ has one vanishing mean, and hence%
\begin{eqnarray*}
&&\left\vert \int \left( \bigtriangleup _{Q;1}^{\omega }\left( R_{J;\kappa
}\bigtriangleup _{J;\kappa }^{\omega }g\right) \right) \left( \bigtriangleup
_{Q;1}^{\omega }\left( R_{J^{\prime };\kappa }\bigtriangleup _{J^{\prime
};\kappa }^{\omega }g\right) \right) d\omega \right\vert \\
&\leq &\frac{\ell \left( Q\right) }{\ell \left( J\right) }\left\Vert
R_{J;\kappa }\bigtriangleup _{J;\kappa }^{\omega }g\right\Vert _{L^{\infty
}\left( \omega \right) }\frac{\ell \left( Q\right) }{\ell \left( J^{\prime
}\right) }\left\Vert R_{J^{\prime };\kappa }\bigtriangleup _{J^{\prime
};\kappa }^{\omega }g\right\Vert _{L^{\infty }\left( \omega \right)
}\left\vert Q\right\vert _{\omega } \\
&\lesssim &\frac{\left\Vert \bigtriangleup _{J;\kappa }^{\omega
}g\right\Vert _{L^{2}\left( \omega \right) }}{\sqrt{\left\vert J\right\vert
_{\omega }}}\frac{\left\Vert \bigtriangleup _{J^{\prime };\kappa }^{\omega
}g\right\Vert _{L^{2}\left( \omega \right) }}{\sqrt{\left\vert J^{\prime
}\right\vert _{\omega }}}\frac{\ell \left( Q\right) }{\ell \left( J\right) }%
\frac{\ell \left( Q\right) }{\ell \left( J^{\prime }\right) }\left\vert
Q\right\vert _{\omega } \\
&=&\frac{\ell \left( Q\right) }{\ell \left( J\right) }\frac{\ell \left(
Q\right) }{\ell \left( J^{\prime }\right) }\sqrt{\frac{\left\vert
Q\right\vert _{\omega }}{\left\vert J\right\vert _{\omega }}}\sqrt{\frac{%
\left\vert Q\right\vert _{\omega }}{\left\vert J^{\prime }\right\vert
_{\omega }}}\left\Vert \bigtriangleup _{J;\kappa }^{\omega }g\right\Vert
_{L^{2}\left( \omega \right) }\left\Vert \bigtriangleup _{J^{\prime };\kappa
}^{\omega }g\right\Vert _{L^{2}\left( \omega \right) } \\
&\leq &\sqrt{\frac{\left\vert Q\right\vert _{\omega }}{\left\vert
J\right\vert _{\omega }}}\frac{\ell \left( Q\right) }{\ell \left( J\right) }%
\left\Vert \bigtriangleup _{J^{\prime };\kappa }^{\omega }g\right\Vert
_{L^{2}\left( \omega \right) }\left\Vert \bigtriangleup _{J;\kappa }^{\omega
}g\right\Vert _{L^{2}\left( \omega \right) },
\end{eqnarray*}%
by (\ref{analogue'}), i.e.%
\begin{eqnarray*}
\left\Vert R_{J^{\prime };\kappa }\bigtriangleup _{J^{\prime };\kappa
}^{\omega }g\right\Vert _{L^{\infty }\left( \omega \right) }\sqrt{\left\vert
J^{\prime }\right\vert _{\omega }} &\lesssim &\left\Vert \bigtriangleup
_{J^{\prime };\kappa }^{\omega }g\right\Vert _{L^{\infty }\left( \omega
\right) }\sqrt{\left\vert J^{\prime }\right\vert _{\omega }}\lesssim
\left\Vert \bigtriangleup _{J^{\prime };\kappa }^{\omega }g\right\Vert
_{L^{2}\left( \omega \right) }, \\
\left\Vert R_{J;\kappa }\bigtriangleup _{J;\kappa }^{\omega }g\right\Vert
_{L^{\infty }\left( \omega \right) }\sqrt{\left\vert J\right\vert _{\omega }}
&\lesssim &\left\Vert \bigtriangleup _{J;\kappa }^{\omega }g\right\Vert
_{L^{\infty }\left( \omega \right) }\sqrt{\left\vert J\right\vert _{\omega }}%
\lesssim \left\Vert \bigtriangleup _{J;\kappa }^{\omega }g\right\Vert
_{L^{2}\left( \omega \right) }.
\end{eqnarray*}%
Note that it was necessary to invoke the Haar wavelets with $\kappa =1$ in
order to obtain this inequality for $s\neq 0$.

We now note for use in the next estimate that%
\begin{eqnarray*}
&&\sum_{\substack{ Q\in \mathcal{D}  \\ Q\subset J^{\prime }}}\ell \left(
Q\right) ^{2s}\sqrt{\frac{\left\vert Q\right\vert _{\omega }}{\left\vert
J\right\vert _{\omega }}}\frac{\ell \left( Q\right) }{\ell \left( J\right) }%
=\sum_{m=0}^{\infty }\sum_{\substack{ Q\in \mathcal{D}  \\ Q\subset
J^{\prime }\text{ and }\ell \left( Q\right) =2^{-m}\ell \left( J^{\prime
}\right) }}\ell \left( Q\right) ^{2s}\sqrt{\frac{\left\vert Q\right\vert
_{\omega }}{\left\vert J\right\vert _{\omega }}}\frac{\ell \left( Q\right) }{%
\ell \left( J\right) } \\
&=&\sum_{m=0}^{\infty }\ell \left( J^{\prime }\right) ^{2s}2^{-2sm}\frac{%
2^{-m}\ell \left( J^{\prime }\right) }{\ell \left( J\right) }\sqrt{\frac{1}{%
\left\vert J\right\vert _{\omega }}}\sum_{\substack{ Q\in \mathcal{D}  \\ %
Q\subset J^{\prime }\text{ and }\ell \left( Q\right) =2^{-m}\ell \left(
J^{\prime }\right) }}1\sqrt{\left\vert Q\right\vert _{\omega }} \\
&\leq &\sum_{m=0}^{\infty }\ell \left( J^{\prime }\right) ^{2s}2^{-2sm}\frac{%
2^{-m}\ell \left( J^{\prime }\right) }{\ell \left( J\right) }\sqrt{\frac{1}{%
\left\vert J\right\vert _{\omega }}}2^{\frac{m}{2}}\sqrt{\sum_{\substack{ %
Q\in \mathcal{D}  \\ Q\subset J^{\prime }\text{ and }\ell \left( Q\right)
=2^{-m}\ell \left( J^{\prime }\right) }}\left\vert Q\right\vert _{\omega }}
\\
&=&\ell \left( J^{\prime }\right) ^{2s}\frac{\ell \left( J^{\prime }\right) 
}{\ell \left( J\right) }\sqrt{\frac{\left\vert J^{\prime }\right\vert
_{\omega }}{\left\vert J\right\vert _{\omega }}}\sum_{m=0}^{\infty }2^{\frac{%
m}{2}}2^{-2sm}2^{-m}\leq C_{s}\ell \left( J^{\prime }\right) ^{2s}\frac{\ell
\left( J^{\prime }\right) }{\ell \left( J\right) }\sqrt{\frac{\left\vert
J^{\prime }\right\vert _{\omega }}{\left\vert J\right\vert _{\omega }}}.
\end{eqnarray*}%
Thus recalling that we restricted attention to the case $J^{\prime }\subset
J $ by symmetry, we have%
\begin{eqnarray*}
&&\sum_{\substack{ J,J^{\prime }\in \mathcal{C}_{F}^{\tau -\limfunc{shift}} 
\\ J^{\prime }\subset J}}\sum_{\substack{ Q\in \mathcal{D}  \\ Q\subset J}}%
\ell \left( Q\right) ^{2s}\left\vert \int_{\mathbb{R}^{n}}\left(
\bigtriangleup _{Q;1}^{\omega }R_{J;\kappa }\bigtriangleup _{J;\kappa
}^{\omega }g\right) \left( \bigtriangleup _{Q;1}^{\omega }R_{J^{\prime
};\kappa }\bigtriangleup _{J^{\prime };\kappa }^{\omega }g\right) d\omega
\right\vert \\
&=&\sum_{\substack{ J,J^{\prime }\in \mathcal{C}_{F}^{\tau -\limfunc{shift}} 
\\ J^{\prime }\subset J}}\sum_{\substack{ Q\in \mathcal{D}  \\ Q\subset
J^{\prime }}}\ell \left( Q\right) ^{2s}\left\vert \int_{\mathbb{R}%
^{n}}\left( \bigtriangleup _{Q;1}^{\omega }R_{J;\kappa }\bigtriangleup
_{J;\kappa }^{\omega }g\right) \left( \bigtriangleup _{Q;1}^{\omega
}R_{J^{\prime };\kappa }\bigtriangleup _{J^{\prime };\kappa }^{\omega
}g\right) d\omega \right\vert \\
&\lesssim &\sum_{\substack{ J,J^{\prime }\in \mathcal{C}_{F}^{\tau -\limfunc{%
shift}}  \\ J^{\prime }\subset J}}\sum_{\substack{ Q\in \mathcal{D}  \\ %
Q\subset J^{\prime }}}\ell \left( Q\right) ^{2s}\sqrt{\frac{\left\vert
Q\right\vert _{\omega }}{\left\vert J\right\vert _{\omega }}}\frac{\ell
\left( Q\right) }{\ell \left( J\right) }\left\Vert \bigtriangleup
_{J^{\prime };\kappa }^{\omega }g\right\Vert _{L^{2}\left( \omega \right)
}\left\Vert \bigtriangleup _{J;\kappa }^{\omega }g\right\Vert _{L^{2}\left(
\omega \right) } \\
&\lesssim &\sum_{\substack{ J,J^{\prime }\in \mathcal{C}_{F}^{\tau -\limfunc{%
shift}}  \\ J^{\prime }\subset J}}\ell \left( J^{\prime }\right) ^{2s}\frac{%
\ell \left( J^{\prime }\right) }{\ell \left( J\right) }\sqrt{\frac{%
\left\vert J^{\prime }\right\vert _{\omega }}{\left\vert J\right\vert
_{\omega }}}\left\Vert \bigtriangleup _{J^{\prime };\kappa }^{\omega
}g\right\Vert _{L^{2}\left( \omega \right) }\left\Vert \bigtriangleup
_{J;\kappa }^{\omega }g\right\Vert _{L^{2}\left( \omega \right) }
\end{eqnarray*}%
which is%
\begin{eqnarray*}
&=&\sum_{m=0}^{\infty }2^{-m}2^{-2sm}\sum_{\substack{ J,J^{\prime }\in 
\mathcal{C}_{F}^{\tau -\limfunc{shift}}  \\ \ell \left( J^{\prime }\right)
=2^{-m}\ell \left( J\right) }}\ell \left( J\right) ^{2s}\sqrt{\frac{%
\left\vert J^{\prime }\right\vert _{\omega }}{\left\vert J\right\vert
_{\omega }}}\left\Vert \bigtriangleup _{J^{\prime };\kappa }^{\omega
}g\right\Vert _{L^{2}\left( \omega \right) }\left\Vert \bigtriangleup
_{J;\kappa }^{\omega }g\right\Vert _{L^{2}\left( \omega \right) } \\
&\leq &\sum_{m=0}^{\infty }2^{-m}2^{-2sm}\sum_{J\in \mathcal{C}_{F}^{\tau -%
\limfunc{shift}}}\ell \left( J\right) ^{2s}\sum_{J^{\prime }\in \mathcal{C}%
_{F}^{\mathbf{\tau }-\limfunc{shift}}:\ \ell \left( J^{\prime }\right)
=2^{-m}\ell \left( J\right) }\sqrt{\frac{\left\vert J^{\prime }\right\vert
_{\omega }}{\left\vert J\right\vert _{\omega }}}\left\Vert \bigtriangleup
_{J^{\prime };\kappa }^{\omega }g\right\Vert _{L^{2}\left( \omega \right)
}\left\Vert \bigtriangleup _{J;\kappa }^{\omega }g\right\Vert _{L^{2}\left(
\omega \right) } \\
&\lesssim &\sum_{m=0}^{\infty }2^{-m}2^{-2sm}\sum_{J\in \mathcal{C}%
_{F}^{\tau -\limfunc{shift}}}\ell \left( J\right) ^{2s}\sqrt{\sum_{J^{\prime
}\in \mathcal{C}_{F}^{\mathbf{\tau }-\limfunc{shift}}:\ \ell \left(
J^{\prime }\right) =2^{-m}\ell \left( J\right) }\frac{\left\vert J^{\prime
}\right\vert _{\omega }}{\left\vert J\right\vert _{\omega }}\left\Vert
\bigtriangleup _{J;\kappa }^{\omega }g\right\Vert _{L^{2}\left( \omega
\right) }^{2}}\sqrt{\sum_{J^{\prime }:\ \ell \left( J^{\prime }\right)
=2^{-m}\ell \left( J\right) }\left\Vert \bigtriangleup _{J^{\prime };\kappa
}^{\omega }g\right\Vert _{L^{2}\left( \omega \right) }^{2}}
\end{eqnarray*}%
which is at most%
\begin{eqnarray*}
&\lesssim &\sum_{m=0}^{\infty }2^{-m}2^{-2sm}\sum_{J\in \mathcal{C}%
_{F}^{\tau -\limfunc{shift}}}\ell \left( J\right) ^{2s}\left\Vert
\bigtriangleup _{J;\kappa }^{\omega }g\right\Vert _{L^{2}\left( \omega
\right) }\sqrt{\sum_{J^{\prime }\in \mathcal{C}_{F}^{\tau -\limfunc{shift}%
}:\ \ell \left( J^{\prime }\right) =2^{-m}\ell \left( J\right) }\left\Vert
\bigtriangleup _{J^{\prime };\kappa }^{\omega }g\right\Vert _{L^{2}\left(
\omega \right) }^{2}} \\
&\lesssim &\sum_{m=0}^{\infty }2^{-m}2^{-2sm}\sqrt{\sum_{J\in \mathcal{C}%
_{F}^{\tau -\limfunc{shift}}}\ell \left( J\right) ^{2s}\left\Vert
\bigtriangleup _{J;\kappa }^{\omega }g\right\Vert _{L^{2}\left( \omega
\right) }^{2}}\sqrt{\sum_{J\in \mathcal{C}_{F}^{\tau -\limfunc{shift}}}\ell
\left( J\right) ^{2s}\sum_{J^{\prime }\in \mathcal{C}_{F}^{\tau -\limfunc{%
shift}}:\ \ell \left( J^{\prime }\right) =2^{-m}\ell \left( J\right)
}\left\Vert \bigtriangleup _{J^{\prime };\kappa }^{\omega }g\right\Vert
_{L^{2}\left( \omega \right) }^{2}} \\
&\lesssim &\sqrt{\sum_{J\in \mathcal{C}_{F}^{\tau -\limfunc{shift}}}\ell
\left( J\right) ^{2s}\left\Vert \bigtriangleup _{J;\kappa }^{\omega
}g\right\Vert _{L^{2}\left( \omega \right) }^{2}}\sum_{m=0}^{\infty
}2^{-m}2^{-2sm}2^{ms}\sqrt{\sum_{J\in \mathcal{C}_{F}^{\tau -\limfunc{shift}%
}}\sum_{J^{\prime }\in \mathcal{C}_{F}^{\tau -\limfunc{shift}}:\ \ell \left(
J^{\prime }\right) =2^{-m}\ell \left( J\right) }\ell \left( J^{\prime
}\right) ^{2s}\left\Vert \bigtriangleup _{J^{\prime };\kappa }^{\omega
}g\right\Vert _{L^{2}\left( \omega \right) }^{2}} \\
&\leq &\left\Vert \mathsf{P}_{\mathcal{C}_{F}^{\tau -\limfunc{shift}%
}}^{\omega }g\right\Vert _{W_{\limfunc{dyad}}^{-s}\left( \omega \right)
}\sum_{m=0}^{\infty }2^{-m}2^{-2sm}2^{ms}\sqrt{\sum_{J^{\prime }\in \mathcal{%
C}_{F}^{\tau -\limfunc{shift}}}\ell \left( J^{\prime }\right)
^{2s}\left\Vert \bigtriangleup _{J^{\prime };\kappa }^{\omega }g\right\Vert
_{L^{2}\left( \omega \right) }^{2}}=C_{s}\left\Vert \mathsf{P}_{\mathcal{C}%
_{F}^{\tau -\limfunc{shift}}}^{\omega }g\right\Vert _{W_{\limfunc{dyad}%
}^{-s}\left( \omega \right) }^{2}
\end{eqnarray*}%
provided that $s>-1$. Altogether we have shown%
\begin{eqnarray*}
\left\vert \mathsf{B}_{\limfunc{paraproduct};\kappa }^{F}\left( f,g\right)
\right\vert &\lesssim &\mathfrak{T}_{T^{\lambda }}^{s}\ell \left( F\right)
^{-s}\sqrt{\left\vert F\right\vert _{\sigma }}\left\Vert \mathbb{E}%
_{F;\kappa }^{\sigma }f\right\Vert _{L^{\infty }\left( \sigma \right)
}\left\Vert \sum_{J\in \mathcal{C}_{F}^{\tau -\limfunc{shift}}}\frac{%
P_{J;\kappa }}{E_{F}^{\sigma }\left\vert f\right\vert }\bigtriangleup
_{J;\kappa }^{\omega }g\right\Vert _{W_{\func{dyad}}^{-s}\left( \omega
\right) } \\
&\lesssim &\mathfrak{T}_{T^{\lambda }}^{s}\ell \left( F\right) ^{-s}\sqrt{%
\left\vert F\right\vert _{\sigma }}\left\Vert \mathbb{E}_{F;\kappa }^{\sigma
}f\right\Vert _{L^{\infty }\left( \sigma \right) }C_{s}\left\Vert \mathsf{P}%
_{\mathcal{C}_{F}^{\tau -\limfunc{shift}}}^{\omega }g\right\Vert _{W_{%
\limfunc{dyad}}^{-s}\left( \omega \right) }.
\end{eqnarray*}

For future reference we record the fact that the main inequality proved
above,%
\begin{equation*}
\left\Vert \sum_{J\in \mathcal{C}_{F}^{\tau -\limfunc{shift}}}R_{J;\kappa
}\bigtriangleup _{J;\kappa }^{\omega }g\right\Vert _{W_{\func{dyad}%
}^{-s}\left( \omega \right) }^{2}\lesssim \sum_{J\in \mathcal{C}_{F}^{\tau -%
\limfunc{shift}}}\left\Vert \bigtriangleup _{J;\kappa }^{\omega
}g\right\Vert _{W_{\func{dyad}}^{-s}\left( \omega \right) }^{2},
\end{equation*}%
continues to hold if $\mathcal{C}_{F}^{\tau -\limfunc{shift}}$ is replaced
by an arbitrary subset $\mathcal{J}$ of the dyadic grid $\mathcal{D}$, and
if the polynomials $R_{J;\kappa }$ are replaced with any family of
polynomials $\left\{ W_{J;\kappa }\right\} _{J\in \mathcal{J}}$ such that
for all $J\in \mathcal{J}$,%
\begin{eqnarray*}
\deg W_{J;\kappa } &<&\kappa , \\
\int_{J}W_{J;\kappa }d\omega &=&0, \\
\sup_{J}\left\vert W_{J;\kappa }\right\vert &\leq &C.
\end{eqnarray*}%
More precisely the above arguments prove,%
\begin{equation}
\left\Vert \sum_{J\in \mathcal{J}}W_{J;\kappa }\bigtriangleup _{J;\kappa
}^{\omega }g\right\Vert _{W_{\func{dyad}}^{-s}\left( \omega \right)
}^{2}\lesssim \sum_{J\in \mathcal{J}}\left\Vert \bigtriangleup _{J;\kappa
}^{\omega }g\right\Vert _{W_{\func{dyad}}^{-s}\left( \omega \right) }^{2}.
\label{continues}
\end{equation}

Since the weighted Sobolev wavelets $\left\{ \bigtriangleup _{J;\kappa
}^{\omega }\right\} _{J\in \mathcal{D}}$ are pairwise orthogonal, we have 
\begin{equation*}
\sum_{J\in \mathcal{C}_{F}^{\tau -\limfunc{shift}}}\left\Vert \bigtriangleup
_{J;\kappa }^{\omega }g\right\Vert _{W_{\limfunc{dyad}}^{-s}\left( \omega
\right) }^{2}=\left\Vert \mathsf{P}_{\mathcal{C}_{F}^{\tau -\limfunc{shift}%
}}^{\omega }g\right\Vert _{W_{\limfunc{dyad}}^{-s}\left( \omega \right)
}^{2},
\end{equation*}%
and so we obtain 
\begin{equation*}
\left\vert \mathsf{B}_{\limfunc{paraproduct};\kappa }^{F}\left( f,g\right)
\right\vert \lesssim \mathfrak{T}_{T^{\lambda }}^{s}\ell \left( F\right)
^{-s}\left\Vert \mathbb{E}_{F;\kappa }^{\sigma }f\right\Vert _{L^{\infty
}\left( \sigma \right) }\sqrt{\left\vert F\right\vert _{\sigma }}\
\left\Vert \mathsf{P}_{\mathcal{C}_{F}^{\tau -\limfunc{shift}}}^{\omega
}g\right\Vert _{W_{\limfunc{dyad}}^{-s}\left( \omega \right) }
\end{equation*}%
as required by (\ref{below form bound}).

Finally using orthogonality of $\left\{ \mathsf{P}_{\mathcal{C}_{F}^{\tau -%
\limfunc{shift}}}^{\omega }\right\} _{F\in \mathcal{F}}$ and
quasiorthogonality in Lemma \ref{q lemma}, we obtain%
\begin{equation}
\left\vert \mathsf{B}_{\limfunc{paraproduct};\kappa }\left( f,g\right)
\right\vert \lesssim \sqrt{A_{2}^{\lambda }}\left\Vert f\right\Vert _{W_{%
\limfunc{dyad}}^{s}\left( \sigma \right) }\left\Vert g\right\Vert _{W_{%
\limfunc{dyad}}^{-s}\left( \omega \right) },\ \ \ \ \ \left\vert
s\right\vert <1.  \label{para est}
\end{equation}

We next turn to the commutator inner products $\left\langle \left[ T_{\sigma
}^{\lambda },M_{I^{\prime };\kappa }\right] \mathbf{1}_{I^{\prime
}},\bigtriangleup _{J;\kappa }^{\omega }g\right\rangle _{\omega }$ arising
in $\mathsf{B}_{\limfunc{commutator};\kappa }^{F}\left( f,g\right) $,
followed by the neighbour and stopping inner products.

\begin{remark}
The arguments for the commutator, neighbour and stopping forms follow
closely the analogous arguments in \cite{AlSaUr} where the case $s=0$ is
handled. Nevertheless, there are differences arising when $s\neq 0$, and so
we give complete details for the convenience of the reader.
\end{remark}

\subsection{The commutator form\label{Sub commutator form}}

Fix $\kappa \geq 1$. In this subsection we use $\alpha $ to denote a
multiindex in $\mathbb{R}^{n}$, and so we will instead use $\lambda $ to
denote the fractional order of the Calderon-Zygmund operator. Assume now
that $K^{\lambda }$ is a general standard $\lambda $-fractional kernel in $%
\mathbb{R}^{n}$, and $T^{\lambda }$ is the associated Calder\'{o}n-Zygmund
operator, and that $P_{\alpha ,a,I^{\prime }}\left( x\right) =\left( \frac{%
x-a}{\ell \left( I^{\prime }\right) }\right) ^{\alpha }=\left( \frac{%
x_{1}-a_{1}}{\ell \left( I^{\prime }\right) }\right) ^{\alpha _{1}}...\left( 
\frac{x_{n}-a_{n}}{\ell \left( I^{\prime }\right) }\right) ^{\alpha _{n}}$,
where $\left\vert \alpha \right\vert \leq \kappa -1$ and $I^{\prime }\in 
\mathfrak{C}_{\mathcal{D}}\left( I\right) $, $I\in \mathcal{C}_{F}$. We
recall from \cite{AlSaUr} the formula%
\begin{equation*}
x^{\alpha }-y^{\alpha }=\sum_{k=1}^{n}\left( x_{k}-y_{k}\right)
\dsum\limits_{\left\vert \beta \right\vert +\left\vert \gamma \right\vert
=\left\vert \alpha \right\vert -1}c_{\alpha ,\beta ,\gamma }x^{\beta
}y^{\gamma }.
\end{equation*}%
Continuing to follow \cite{AlSaUr}, we then have%
\begin{eqnarray}
&&\mathbf{1}_{I^{\prime }}\left( x\right) \left[ P_{\alpha ,a,I^{\prime }}\
,T_{\sigma }^{\lambda }\right] \mathbf{1}_{I^{\prime }}\left( x\right) =%
\mathbf{1}_{I^{\prime }}\left( x\right) \int_{\mathbb{R}^{n}}K^{\lambda
}\left( x-y\right) \left\{ P_{\alpha ,a,I^{\prime }}\left( x\right)
-P_{\alpha ,a,I^{\prime }}\left( y\right) \right\} \mathbf{1}_{I^{\prime
}}\left( y\right) d\sigma \left( y\right)  \label{continue} \\
&=&\mathbf{1}_{I^{\prime }}\left( x\right) \int_{\mathbb{R}^{n}}K^{\lambda
}\left( x-y\right) \left\{ \sum_{k=1}^{n}\left( \frac{x_{k}-y_{k}}{\ell
\left( I^{\prime }\right) }\right) \dsum\limits_{\left\vert \beta
\right\vert +\left\vert \gamma \right\vert =\left\vert \alpha \right\vert
-1}c_{\alpha ,\beta ,\gamma }\left( \frac{x-a}{\ell \left( I^{\prime
}\right) }\right) ^{\beta }\left( \frac{y-a}{\ell \left( I^{\prime }\right) }%
\right) ^{\gamma }\right\} \mathbf{1}_{I^{\prime }}\left( y\right) d\sigma
\left( y\right)  \notag \\
&=&\sum_{k=1}^{n}\dsum\limits_{\left\vert \beta \right\vert +\left\vert
\gamma \right\vert =\left\vert \alpha \right\vert -1}c_{\alpha ,\beta
,\gamma }\mathbf{1}_{I^{\prime }}\left( x\right) \left[ \int \Phi
_{k}^{\lambda }\left( x-y\right) \left\{ \left( \frac{y-a}{\ell \left(
I^{\prime }\right) }\right) ^{\gamma }\right\} \mathbf{1}_{I^{\prime
}}\left( y\right) d\sigma \left( y\right) \right] \left( \frac{x-a}{\ell
\left( I^{\prime }\right) }\right) ^{\beta },  \notag
\end{eqnarray}%
where $\Phi _{k}^{\lambda }\left( x-y\right) =K^{\lambda }\left( x-y\right)
\left( \frac{x_{k}-y_{k}}{\ell \left( I^{\prime }\right) }\right) $. So $%
\left[ P_{\alpha ,a,I^{\prime }},T_{\sigma }^{\lambda }\right] \mathbf{1}%
_{I^{\prime }}\left( x\right) $ is a `polynomial' of degree $\left\vert
\alpha \right\vert -1$ with \emph{variable} coefficients. Now we take the
inner product of the commutator with $\bigtriangleup _{J;\kappa }^{\omega }g$
for some $J\subset I^{\prime }$, and split the inner product into two pieces,%
\begin{eqnarray}
&&\left\langle \left[ P_{\alpha ,a,I^{\prime }}\ ,T_{\sigma }^{\lambda }%
\right] \mathbf{1}_{I^{\prime }},\bigtriangleup _{J;\kappa }^{\omega
}g\right\rangle _{\omega }=\int_{\mathbb{R}^{n}}\left[ P_{\alpha
,a,I^{\prime }}\ ,T_{\sigma }^{\lambda }\right] \mathbf{1}_{I^{\prime
}}\left( x\right) \bigtriangleup _{J;\kappa }^{\omega }g\left( x\right)
d\omega \left( x\right)  \label{two pieces} \\
&=&\int_{\mathbb{R}^{n}}\left[ P_{\alpha ,a,I^{\prime }}\ ,T_{\sigma
}^{\lambda }\right] \mathbf{1}_{I^{\prime }\setminus 2J}\left( x\right)
\bigtriangleup _{J;\kappa }^{\omega }g\left( x\right) d\omega \left(
x\right) +\int_{\mathbb{R}^{n}}\left[ P_{\alpha ,a,I^{\prime }}\ ,T_{\sigma
}^{\lambda }\right] \mathbf{1}_{I^{\prime }\cap 2J}\left( x\right)
\bigtriangleup _{J;\kappa }^{\omega }g\left( x\right) d\omega \left( x\right)
\notag \\
&\equiv &\func{Int}^{\lambda ,\natural }\left( J\right) +\func{Int}^{\lambda
,\flat }\left( J\right) ,  \notag
\end{eqnarray}%
where we are suppressing the dependence on $\alpha $ and $I^{\prime }$. For
the first term $\func{Int}^{\lambda ,\natural }\left( J\right) $ we write%
\begin{equation*}
\func{Int}^{\lambda ,\natural }\left( J\right)
=\sum_{k=1}^{n}\dsum\limits_{\left\vert \beta \right\vert +\left\vert \gamma
\right\vert =\left\vert \alpha \right\vert -1}c_{\alpha ,\beta ,\gamma }%
\func{Int}_{k,\beta ,\gamma }^{\lambda ,\natural }\left( J\right) ,
\end{equation*}%
where with the choice $a=c_{J}$ the center of $J$, we define%
\begin{eqnarray}
\func{Int}_{k,\beta ,\gamma }^{\lambda ,\natural }\left( J\right) &\equiv
&\int_{J}\left[ \int_{I^{\prime }\setminus 2J}\Phi _{k}^{\lambda }\left(
x-y\right) \left( \frac{y-c_{J}}{\ell \left( I^{\prime }\right) }\right)
^{\gamma }d\sigma \left( y\right) \right] \left( \frac{x-c_{J}}{\ell \left(
I^{\prime }\right) }\right) ^{\beta }\bigtriangleup _{J;\kappa }^{\omega
}g\left( x\right) d\omega \left( x\right)  \label{coefficients} \\
&=&\int_{I^{\prime }\setminus 2J}\left\{ \int_{J}\Phi _{k}^{\lambda }\left(
x-y\right) \left( \frac{x-c_{J}}{\ell \left( I^{\prime }\right) }\right)
^{\beta }\bigtriangleup _{J;\kappa }^{\omega }g\left( x\right) d\omega
\left( x\right) \right\} \left( \frac{y-c_{J}}{\ell \left( I^{\prime
}\right) }\right) ^{\gamma }d\sigma \left( y\right) .  \notag
\end{eqnarray}%
While these integrals need no longer vanish, we will show they are suitably
small, using that the function $\left( \frac{x-c_{J}}{\ell \left( I^{\prime
}\right) }\right) ^{\beta }\bigtriangleup _{J;\kappa }^{\omega }g\left(
x\right) $ is supported in $J$ and has vanishing $\omega $-means up to order 
$\kappa -\left\vert \beta \right\vert -1$, and that the function $\Phi
_{k}^{\lambda }\left( z\right) $ is appropriately smooth away from $z=0$,%
\begin{equation*}
\left\vert \nabla ^{m}\Phi _{k}^{\lambda }\left( z\right) \right\vert \leq
C_{m,n}\frac{1}{\left\vert z\right\vert ^{m+n-\lambda -1}\ell \left(
I^{\prime }\right) }.
\end{equation*}

Indeed, we have the following estimate for the integral in braces in (\ref%
{coefficients}), keeping in mind that $y\in I^{\prime }\setminus 2J$ and $%
x\in J$, where the term in the second line below vanishes because $h\left(
x\right) =\left( \frac{x-c_{J}}{\ell \left( I^{\prime }\right) }\right)
^{\beta }\bigtriangleup _{J;\kappa }^{\omega }g\left( x\right) $ has
vanishing $\omega $-means up to order $\kappa -1-\left\vert \beta
\right\vert $, and the fourth line uses that $h\left( x\right) $ is
supported in $J$:%
\begin{eqnarray*}
&&\left\vert \int_{\mathbb{R}^{n}}\Phi _{k}^{\lambda }\left( x-y\right)
h\left( x\right) d\omega \left( x\right) \right\vert \\
&=&\left\vert \int_{\mathbb{R}^{n}}\left\{ \sum_{m=0}^{\kappa -\left\vert
\beta \right\vert -1}\frac{1}{m!}\left( \left( x-c_{J}\right) \cdot \nabla
\right) ^{m}\Phi _{k}^{\lambda }\left( c_{J}-y\right) \right\} h\left(
x\right) d\omega \left( x\right) \right. \\
&&+\left. \int_{\mathbb{R}^{n}}\frac{1}{\left( \kappa -\left\vert \beta
\right\vert \right) !}\left( \left( x-c_{J}\right) \cdot \nabla \right)
^{\kappa -\left\vert \beta \right\vert }\Phi _{k}^{\lambda }\left( \eta
_{J}^{\omega }\right) h\left( x\right) d\omega \left( x\right) \right\vert \\
&\lesssim &\left\Vert h\right\Vert _{L^{1}\left( \omega \right) }\frac{\ell
\left( J\right) ^{\kappa -\left\vert \beta \right\vert }}{\left[ \ell \left(
J\right) +\limfunc{dist}\left( y,J\right) \right] ^{\kappa -\left\vert \beta
\right\vert +n-\lambda -1}\ell \left( I^{\prime }\right) } \\
&\lesssim &\left( \frac{\ell \left( J\right) }{\ell \left( I^{\prime
}\right) }\right) ^{\left\vert \beta \right\vert }\frac{\ell \left( J\right)
^{\kappa -\left\vert \beta \right\vert }}{\left[ \ell \left( J\right) +%
\limfunc{dist}\left( y,J\right) \right] ^{\kappa -\left\vert \beta
\right\vert +n-\lambda -1}\ell \left( I^{\prime }\right) }\sqrt{\left\vert
J\right\vert _{\omega }}\left\vert \widehat{g}\left( J\right) \right\vert ,
\end{eqnarray*}%
since 
\begin{eqnarray*}
\left\Vert h\right\Vert _{L^{1}\left( \omega \right) } &=&\int_{J}\left\vert
\left( \frac{x-c_{J}}{\ell \left( I^{\prime }\right) }\right) ^{\beta
}\bigtriangleup _{J;\kappa }^{\omega }g\left( x\right) \right\vert d\omega
\left( x\right) \\
&\leq &\left( \frac{\ell \left( J\right) }{\ell \left( I^{\prime }\right) }%
\right) ^{\left\vert \beta \right\vert }\left\Vert \bigtriangleup _{J;\kappa
}^{\omega }g\right\Vert _{L^{1}\left( \omega \right) }\lesssim \left( \frac{%
\ell \left( J\right) }{\ell \left( I^{\prime }\right) }\right) ^{\left\vert
\beta \right\vert }\sqrt{\left\vert J\right\vert _{\omega }}\left\vert 
\widehat{g}\left( J\right) \right\vert .
\end{eqnarray*}

Now recall the orthonormal basis $\left\{ h_{I;\kappa }^{\sigma ,a}\right\}
_{a\in \Gamma }$ of $L_{I;\kappa }^{2}\left( \sigma \right) $ for any $I\in 
\mathcal{D}$. For a $\mathcal{D}$-child $I^{\prime }$ of a cube $I$, we
consider the polynomial 
\begin{equation*}
Q_{I^{\prime };\kappa }^{\sigma }\equiv \mathbf{1}_{I^{\prime }}\sum_{a\in
\Gamma }c_{a}h_{I;\kappa }^{\sigma ,a}
\end{equation*}%
where $c_{a}=\frac{\left\langle f,h_{I;\kappa }^{\sigma ,a}\right\rangle }{%
\left\vert \widehat{f}\left( I\right) \right\vert }$, so that $Q_{I^{\prime
};\kappa }^{\sigma }$ is a renormalization of the polynomial $M_{I^{\prime
};\kappa }$ introduced earlier. We have%
\begin{equation*}
\left\vert \widehat{f}\left( I\right) \right\vert \mathbf{1}_{I^{\prime
}}\sum_{a\in \Gamma }c_{a}h_{I;\kappa }^{\sigma ,a}=\left\vert \widehat{f}%
\left( I\right) \right\vert \mathbf{1}_{I^{\prime }}Q_{I^{\prime };\kappa
}^{\sigma }=\mathbb{E}_{I^{\prime };\kappa }^{\sigma }f-\mathbf{1}%
_{I^{\prime }}\mathbb{E}_{I;\kappa }^{\sigma }f,
\end{equation*}%
where%
\begin{equation*}
\sum_{a\in \Gamma }\left\vert c_{a}\right\vert ^{2}=\sum_{a\in \Gamma
}\left\vert \frac{\left\langle f,h_{I;\kappa }^{\sigma ,a}\right\rangle }{%
\left\vert \widehat{f}\left( I\right) \right\vert }\right\vert ^{2}=1.
\end{equation*}

Recall also that from (\ref{analogue'}) we have 
\begin{equation*}
\left\vert \widehat{f}\left( I\right) \right\vert \left\Vert Q_{I^{\prime
};\kappa }^{\sigma }\right\Vert _{\infty }=\left\Vert \mathbf{1}_{I^{\prime
}}\sum_{a\in \Gamma }\left\langle f,h_{I;\kappa }^{\sigma ,a}\right\rangle
h_{I;\kappa }^{\sigma ,a}\right\Vert _{\infty }=\left\Vert \mathbf{1}%
_{I^{\prime }}\bigtriangleup _{I;\kappa }^{\sigma }f\right\Vert _{\infty
}\approx \frac{\left\vert \widehat{f}\left( I\right) \right\vert }{\sqrt{%
\left\vert I\right\vert _{\sigma }}}.
\end{equation*}%
Hence for $c_{J}\in J\subset I^{\prime }$, if we write%
\begin{equation}
Q_{I^{\prime };\kappa }^{\sigma }\left( y\right) =\sum_{\left\vert \alpha
\right\vert <\kappa }b_{\alpha }\left( \frac{y-c_{J}}{\ell \left( I^{\prime
}\right) }\right) ^{\alpha }=\sum_{\left\vert \alpha \right\vert <\kappa
}b_{\alpha }P_{\alpha ,c_{J}}\left( y\right) ,  \label{def Q}
\end{equation}%
and then rescale to the unit cube and invoke the fact that any two norms on
a finite dimensional vector space are equivalent, we obtain%
\begin{equation}
\sum_{\left\vert \alpha \right\vert <\kappa }\left\vert b_{\alpha
}\right\vert \approx \left\Vert Q_{I^{\prime };\kappa }^{\sigma }\right\Vert
_{\infty }\approx \frac{1}{\sqrt{\left\vert I\right\vert _{\sigma }}},\ \ \
\ \ I^{\prime }\in \mathfrak{C}_{\mathcal{D}}\left( I\right) .
\label{we obtain}
\end{equation}

We then write 
\begin{equation}
\left\langle \left[ Q_{I^{\prime };\kappa },T_{\sigma }^{\lambda }\right] 
\mathbf{1}_{I^{\prime }\setminus 2J},\bigtriangleup _{J;\kappa }^{\omega
}g\right\rangle _{\omega }=\sum_{\left\vert \alpha \right\vert <\kappa
}b_{\alpha }\left\langle \left[ P_{\alpha ,c_{J}},T_{\sigma }^{\lambda }%
\right] \mathbf{1}_{I^{\prime }\setminus 2J},\bigtriangleup _{J;\kappa
}^{\omega }g\right\rangle _{\omega }\ ,  \label{then write}
\end{equation}%
and note that%
\begin{eqnarray*}
&&\left\vert \left\langle \left[ Q_{I^{\prime };\kappa },T_{\sigma
}^{\lambda }\right] \mathbf{1}_{I^{\prime }\setminus 2J},\bigtriangleup
_{J;\kappa }^{\omega }g\right\rangle _{\omega }\right\vert \leq
\sum_{\left\vert \alpha \right\vert <\kappa }\left\vert b_{\alpha
}\left\langle \left[ P_{\alpha ,c_{J}},T_{\sigma }^{\lambda }\right] \mathbf{%
1}_{I^{\prime }\setminus 2J},\bigtriangleup _{J;\kappa }^{\omega
}g\right\rangle _{\omega }\right\vert \\
&\lesssim &\frac{1}{\sqrt{\left\vert I\right\vert _{\sigma }}}%
\max_{\left\vert \alpha \right\vert <\kappa }\left\vert \left\langle \left[
P_{\alpha ,c_{J}},T_{\sigma }^{\lambda }\right] \mathbf{1}_{I^{\prime
}\setminus 2J},\bigtriangleup _{J;\kappa }^{\omega }g\right\rangle _{\omega
}\right\vert ,
\end{eqnarray*}%
so that it remains to estimate each inner product $\func{Int}^{\lambda
,\natural }\left( J\right) =\left\langle \left[ P_{\alpha ,c_{J}},T_{\sigma
}^{\lambda }\right] \mathbf{1}_{I^{\prime }\setminus 2J},\bigtriangleup
_{J;\kappa }^{\omega }g\right\rangle _{\omega }$ as follows:

\begin{equation*}
\left\vert \func{Int}^{\lambda ,\natural }\left( J\right) \right\vert
=\left\vert \sum_{k=1}^{n}\dsum\limits_{\left\vert \beta \right\vert
+\left\vert \gamma \right\vert =\left\vert \alpha \right\vert -1}c_{\alpha
,\beta ,\gamma }\func{Int}_{k,\beta ,\gamma }^{\lambda ,\natural }\left(
J\right) \right\vert \lesssim \max_{\left\vert \beta \right\vert +\left\vert
\gamma \right\vert =\left\vert \alpha \right\vert -1}\left\vert \func{Int}%
_{k,\beta ,\gamma }^{\lambda ,\natural }\left( J\right) \right\vert ,
\end{equation*}%
where $\left\vert \beta \right\vert +\left\vert \gamma \right\vert
=\left\vert \alpha \right\vert -1$, and the estimates above imply,%
\begin{eqnarray*}
&&\left\vert \func{Int}_{k,\beta ,\gamma }^{\lambda ,\natural }\left(
J\right) \right\vert \leq \int_{I^{\prime }\setminus 2J}\left\vert \int \Phi
_{k}^{\lambda }\left( x-y\right) \left( \frac{x-c_{J}}{\ell \left( I^{\prime
}\right) }\right) ^{\beta }\bigtriangleup _{J;\kappa }^{\omega }g\left(
x\right) d\omega \left( x\right) \right\vert \left\vert \left( \frac{y-c_{J}%
}{\ell \left( I^{\prime }\right) }\right) ^{\gamma }\right\vert d\sigma
\left( y\right) \\
&\lesssim &\int_{I^{\prime }\setminus 2J}\left( \frac{\ell \left( J\right) }{%
\ell \left( I^{\prime }\right) }\right) ^{\left\vert \beta \right\vert }%
\frac{\ell \left( J\right) ^{\kappa -\left\vert \beta \right\vert }}{\left[
\ell \left( J\right) +\limfunc{dist}\left( y,J\right) \right] ^{\kappa
-\left\vert \beta \right\vert +n-\lambda -1}\ell \left( I\right) }\sqrt{%
\left\vert J\right\vert _{\omega }}\left\vert \widehat{g}\left( J\right)
\right\vert \left( \frac{\ell \left( J\right) +\limfunc{dist}\left(
y,J\right) }{\ell \left( I^{\prime }\right) }\right) ^{\left\vert \gamma
\right\vert }d\sigma \left( y\right) \\
&=&\int_{I^{\prime }\setminus 2J}\left( \frac{\ell \left( J\right) }{\ell
\left( I^{\prime }\right) }\right) ^{\left\vert \alpha \right\vert -1}\left( 
\frac{\ell \left( J\right) }{\ell \left( J\right) +\limfunc{dist}\left(
y,J\right) }\right) ^{\kappa -\left\vert \alpha \right\vert +1}\frac{1}{%
\left[ \ell \left( J\right) +\limfunc{dist}\left( y,J\right) \right]
^{n-\lambda -1}\ell \left( I^{\prime }\right) }\sqrt{\left\vert J\right\vert
_{\omega }}\left\vert \widehat{g}\left( J\right) \right\vert d\sigma \left(
y\right) \\
&=&\left( \frac{\ell \left( J\right) }{\ell \left( I^{\prime }\right) }%
\right) ^{\left\vert \alpha \right\vert -1}\sqrt{\left\vert J\right\vert
_{\omega }}\left\vert \widehat{g}\left( J\right) \right\vert \left\{
\int_{I^{\prime }\setminus 2J}\left( \frac{\ell \left( J\right) }{\ell
\left( J\right) +\limfunc{dist}\left( y,J\right) }\right) ^{\kappa
-\left\vert \alpha \right\vert +1}\frac{1}{\left[ \ell \left( J\right) +%
\limfunc{dist}\left( y,J\right) \right] ^{n-\lambda -1}\ell \left( I^{\prime
}\right) }d\sigma \left( y\right) \right\} .
\end{eqnarray*}

Now we fix $t\in \mathbb{N}$, and estimate the sum of $\left\vert \func{Int}%
^{\lambda ,\natural }\left( J\right) \right\vert $ over $J\subset I^{\prime
} $ with $\ell \left( J\right) =2^{-t}\ell \left( I^{\prime }\right) $ by
splitting the integration in $y$ according to the size of $\ell \left(
J\right) +\limfunc{dist}\left( y,J\right) $, to obtain the following bound:%
\begin{eqnarray*}
&&\sum_{J\subset I^{\prime }:\ \ell \left( J\right) =2^{-t}\ell \left(
I^{\prime }\right) }\left\vert \func{Int}^{\lambda ,\natural }\left(
J\right) \right\vert \\
&\lesssim &2^{-t\left( \left\vert \alpha \right\vert -1\right)
}\sum_{J\subset I^{\prime }:\ \ell \left( J\right) =2^{-t}\ell \left(
I^{\prime }\right) }\sqrt{\left\vert J\right\vert _{\omega }}\left\vert 
\widehat{g}\left( J\right) \right\vert \left\{ \int_{I^{\prime }\setminus
2J}\left( \frac{\ell \left( J\right) }{\ell \left( J\right) +\limfunc{dist}%
\left( y,J\right) }\right) ^{\kappa -\left\vert \alpha \right\vert +1}\frac{%
d\sigma \left( y\right) }{\left[ \ell \left( J\right) +\limfunc{dist}\left(
y,J\right) \right] ^{n-\lambda -1}\ell \left( I^{\prime }\right) }\right\} \\
&\lesssim &2^{-t\left( \left\vert \alpha \right\vert -1\right)
}\sum_{J\subset I^{\prime }:\ \ell \left( J\right) =2^{-t}\ell \left(
I^{\prime }\right) }\sqrt{\left\vert J\right\vert _{\omega }}\left\vert 
\widehat{g}\left( J\right) \right\vert \left\{
\sum_{s=1}^{t}\int_{2^{s+1}J\setminus 2^{s}J}\left( 2^{-s}\right) ^{\kappa
-\left\vert \alpha \right\vert +1}\frac{d\sigma \left( y\right) }{\left(
2^{s}\ell \left( J\right) \right) ^{n-\lambda -1}\ell \left( I^{\prime
}\right) }\right\} \\
&\lesssim &2^{-t\left\vert \alpha \right\vert }\sum_{J\subset I^{\prime }:\
\ell \left( J\right) =2^{-t}\ell \left( I^{\prime }\right) }\sqrt{\left\vert
J\right\vert _{\omega }}\left\vert \widehat{g}\left( J\right) \right\vert
\sum_{s=1}^{t}\left( 2^{-s}\right) ^{\kappa -\left\vert \alpha \right\vert
+1}2^{-s\left( n-\lambda -1\right) }\frac{\left\vert 2^{s}J\right\vert
_{\sigma }}{\ell \left( J\right) ^{n-\lambda }},
\end{eqnarray*}%
which gives upon pigeonholing the sum in $J$ according to membership in the
grandchildren of $I^{\prime }$ at depth $t-s$: 
\begin{eqnarray*}
&&\sum_{J\in \mathfrak{C}_{\mathcal{D}}^{\left( t\right) }\left( I^{\prime
}\right) }\left\vert \func{Int}^{\lambda ,\natural }\left( J\right)
\right\vert \\
&\lesssim &2^{-t\left\vert \alpha \right\vert }\sum_{J\in \mathfrak{C}_{%
\mathcal{D}}^{\left( t\right) }\left( I^{\prime }\right) }\sqrt{\left\vert
J\right\vert _{\omega }}\left\vert \widehat{g}\left( J\right) \right\vert
\sum_{s=1}^{t}\left( 2^{-s}\right) ^{\kappa -\left\vert \alpha \right\vert
+n-\lambda }\frac{\left\vert 2^{s}J\right\vert _{\sigma }}{\ell \left(
J\right) ^{n-\lambda }} \\
&=&2^{-t\left\vert \alpha \right\vert }\sum_{s=1}^{t}\left( 2^{-s}\right)
^{\kappa -\left\vert \alpha \right\vert }\sum_{K\in \mathfrak{C}_{\mathcal{D}%
}^{\left( t-s\right) }\left( I^{\prime }\right) }\sum_{J\in \mathfrak{C}_{%
\mathcal{D}}^{\left( s\right) }\left( K\right) }\sqrt{\left\vert
J\right\vert _{\omega }}\left\vert \widehat{g}\left( J\right) \right\vert 
\frac{\left\vert 2^{s}J\right\vert _{\sigma }}{\ell \left( K\right)
^{n-\lambda }} \\
&\lesssim &2^{-t\left\vert \alpha \right\vert }\sum_{s=1}^{t}\left(
2^{-s}\right) ^{\kappa -\left\vert \alpha \right\vert }\sum_{K\in \mathfrak{C%
}_{\mathcal{D}}^{\left( t-s\right) }\left( I^{\prime }\right) }\frac{%
\left\vert 3K\right\vert _{\sigma }}{\ell \left( K\right) ^{n-\lambda }}%
\sum_{J\in \mathfrak{C}_{\mathcal{D}}^{\left( s\right) }\left( K\right) }%
\sqrt{\left\vert J\right\vert _{\omega }}\left\vert \widehat{g}\left(
J\right) \right\vert \\
&\lesssim &2^{-t\left\vert \alpha \right\vert }\sum_{s=1}^{t}\left(
2^{-s}\right) ^{\kappa -\left\vert \alpha \right\vert }\sum_{K\in \mathfrak{C%
}_{\mathcal{D}}^{\left( t-s\right) }\left( I^{\prime }\right) }\frac{%
\left\vert 3K\right\vert _{\sigma }}{\ell \left( K\right) ^{n-\lambda }}%
\sqrt{\left\vert K\right\vert _{\omega }}\sqrt{\sum_{J\in \mathfrak{C}_{%
\mathcal{D}}^{\left( s\right) }\left( K\right) }\left\vert \widehat{g}\left(
J\right) \right\vert ^{2}}.
\end{eqnarray*}%
Now we use the $A_{2}^{\lambda }$ condition and doubling for $\sigma $ to
obtain the bound 
\begin{equation*}
\frac{\left\vert 3K\right\vert _{\sigma }}{\ell \left( K\right) ^{n-\lambda }%
}\sqrt{\left\vert K\right\vert _{\omega }}\lesssim \frac{\sqrt{\left\vert
3K\right\vert _{\sigma }\left\vert 3K\right\vert _{\omega }}}{\ell \left(
3K\right) ^{n-\lambda }}\sqrt{\left\vert 3K\right\vert _{\sigma }}\leq \sqrt{%
A_{2}^{\lambda }}\sqrt{\left\vert 3K\right\vert _{\sigma }}\lesssim \sqrt{%
A_{2}^{\lambda }}\sqrt{\left\vert K\right\vert _{\sigma }}.
\end{equation*}%
Thus we have%
\begin{eqnarray*}
\sum_{J\in \mathfrak{C}_{\mathcal{D}}^{\left( t\right) }\left( I^{\prime
}\right) }\left\vert \func{Int}^{\lambda ,\natural }\left( J\right)
\right\vert &\lesssim &2^{-t\left\vert \alpha \right\vert }\sqrt{%
A_{2}^{\lambda }}\sum_{s=1}^{t}\left( 2^{-s}\right) ^{\kappa -\left\vert
\alpha \right\vert }\sum_{K\in \mathfrak{C}_{\mathcal{D}}^{\left( t-s\right)
}\left( I^{\prime }\right) }\sqrt{\left\vert K\right\vert _{\sigma }}\sqrt{%
\sum_{J\in \mathfrak{C}_{\mathcal{D}}^{\left( s\right) }\left( K\right)
}\left\vert \widehat{g}\left( J\right) \right\vert ^{2}} \\
&\lesssim &2^{-t\left\vert \alpha \right\vert }\sqrt{A_{2}^{\lambda }}%
\sum_{s=1}^{t}\left( 2^{-s}\right) ^{\kappa -\left\vert \alpha \right\vert }%
\sqrt{\left\vert I^{\prime }\right\vert _{\sigma }}\sqrt{\sum_{J\in 
\mathfrak{C}_{\mathcal{D}}^{\left( t\right) }\left( I^{\prime }\right)
}\left\vert \widehat{g}\left( J\right) \right\vert ^{2}} \\
&\lesssim &2^{-t}\sqrt{A_{2}^{\lambda }}\sqrt{\left\vert I^{\prime
}\right\vert _{\sigma }}\sqrt{\sum_{J\in \mathfrak{C}_{\mathcal{D}}^{\left(
t\right) }\left( I^{\prime }\right) }\left\vert \widehat{g}\left( J\right)
\right\vert ^{2}},
\end{eqnarray*}%
since $1\leq \left\vert \alpha \right\vert \leq \kappa -1$ (the commutator
vanishes if $\left\vert \alpha \right\vert =0$) shows that both $%
2^{-t\left\vert \alpha \right\vert }\leq 2^{-t}$ and 
\begin{equation*}
\kappa -\left\vert \alpha \right\vert \geq 1>0.
\end{equation*}

We now claim the same estimate holds for the sum of $\left\vert \func{Int}%
^{\lambda ,\flat }\left( J\right) \right\vert $ over $J\subset I^{\prime }$
with $\ell \left( J\right) =2^{-t}\ell \left( I^{\prime }\right) $, namely%
\begin{eqnarray*}
&&\left\vert \func{Int}_{k,\beta ,\gamma }^{\lambda ,\flat }\left( J\right)
\right\vert \lesssim \left\vert \int_{J}\left( \int_{2J}\Phi _{k}^{\lambda
}\left( x-y\right) \left( \frac{y-c_{J}}{\ell \left( I^{\prime }\right) }%
\right) ^{\gamma }d\sigma \left( y\right) \right) \left( \frac{x-c_{J}}{\ell
\left( I^{\prime }\right) }\right) ^{\beta }\bigtriangleup _{J;\kappa
}^{\omega }g\left( x\right) d\omega \left( x\right) \right\vert \\
&\leq &\int_{J}\left( \int_{2J}\frac{1}{\ell \left( I^{\prime }\right)
\left\vert x-y\right\vert ^{n-\lambda -1}}\left\vert \frac{y-c_{J}}{\ell
\left( I^{\prime }\right) }\right\vert ^{\left\vert \gamma \right\vert
}d\sigma \left( y\right) \right) \left\vert \frac{x-c_{J}}{\ell \left(
I^{\prime }\right) }\right\vert ^{\left\vert \beta \right\vert }\left\vert
\bigtriangleup _{J;\kappa }^{\omega }g\left( x\right) \right\vert d\omega
\left( x\right) \\
&\lesssim &\left( \frac{\ell \left( J\right) }{\ell \left( I^{\prime
}\right) }\right) ^{\left\vert \gamma \right\vert +\left\vert \beta
\right\vert }\frac{1}{\ell \left( I^{\prime }\right) \sqrt{\left\vert
J\right\vert _{\omega }}}\left\vert \widehat{g}\left( J\right) \right\vert
\int_{J}\int_{2J}\frac{d\sigma \left( y\right) d\omega \left( x\right) }{%
\left\vert x-y\right\vert ^{n-\lambda -1}}.
\end{eqnarray*}%
In order to estimate the double integral using the $A_{2}^{\lambda }$
condition we cover the band $\left\vert x-y\right\vert \leq C2^{-m}\ell
\left( J\right) $ by a collection of cubes $Q\left( z_{m},C2^{-m}\ell \left(
J\right) \right) \times Q\left( z_{m},C2^{-m}\ell \left( J\right) \right) $
in $CJ\times CJ$ with centers $\left( z_{m},z_{m}\right) $ and bounded
overlap. Then we have%
\begin{eqnarray*}
&&\int_{J}\int_{2J}\frac{d\sigma \left( y\right) d\omega \left( x\right) }{%
\left\vert x-y\right\vert ^{n-\lambda -1}}\leq \sum_{m=0}^{\infty
}\diint\limits_{\substack{ x,y\in 2J  \\ \left\vert x-y\right\vert \approx
2^{-m}\ell \left( J\right) }}\frac{d\sigma \left( y\right) d\omega \left(
x\right) }{\left( 2^{-m}\ell \left( J\right) \right) ^{n-\lambda -1}} \\
&\approx &\sum_{m=0}^{\infty }\sum_{Q\left( z_{m},C2^{-m}\ell \left(
J\right) \right) \times Q\left( z_{m},C2^{-m}\ell \left( J\right) \right)
}\int_{Q\left( z_{m},C2^{-m}\ell \left( J\right) \right) \times Q\left(
z_{m},C2^{-m}\ell \left( J\right) \right) }\frac{d\sigma \left( y\right)
d\omega \left( x\right) }{\left( 2^{-m}\ell \left( J\right) \right)
^{n-\lambda -1}} \\
&\leq &\frac{1}{\ell \left( J\right) ^{n-\lambda -1}}\sum_{m=0}^{\infty
}2^{m\left( n-\lambda -1\right) }\sum_{z_{m}}\left\vert Q\left(
z_{m},C2^{-m}\ell \left( J\right) \right) \right\vert _{\sigma }\left\vert
Q\left( z_{m},C2^{-m}\ell \left( J\right) \right) \right\vert _{\omega } \\
&\leq &\frac{1}{\ell \left( J\right) ^{n-\lambda -1}}\sum_{m=0}^{\infty
}2^{m\left( n-\lambda -1\right) }\sum_{z_{m}}\sqrt{\left\vert Q\left(
z_{m},C2^{-m}\ell \left( J\right) \right) \right\vert _{\sigma }\left\vert
Q\left( z_{m},C2^{-m}\ell \left( J\right) \right) \right\vert _{\omega }}%
\sqrt{A_{2}^{\lambda }}\left( 2^{-m}\ell \left( J\right) \right) ^{n-\lambda
} \\
&\lesssim &\sqrt{A_{2}^{\lambda }}\ell \left( J\right) \sum_{m=0}^{\infty
}2^{m\left( n-\lambda -1\right) }2^{-m\left( n-\lambda \right) }\sqrt{%
\left\vert CJ\right\vert _{\sigma }\left\vert CJ\right\vert _{\omega }}%
\lesssim \sqrt{A_{2}^{\lambda }}\ell \left( J\right) \sqrt{\left\vert
J\right\vert _{\sigma }\left\vert J\right\vert _{\omega }}.
\end{eqnarray*}%
Thus altogether, we have%
\begin{eqnarray*}
\left\vert \func{Int}_{k,\beta ,\gamma }^{\lambda ,\flat }\left( J\right)
\right\vert &\lesssim &\left( \frac{\ell \left( J\right) }{\ell \left(
I^{\prime }\right) }\right) ^{\left\vert \gamma \right\vert +\left\vert
\beta \right\vert }\frac{1}{\ell \left( I^{\prime }\right) \sqrt{\left\vert
J\right\vert _{\omega }}}\left\vert \widehat{g}\left( J\right) \right\vert 
\sqrt{A_{2}^{\lambda }}\ell \left( J\right) \sqrt{\left\vert J\right\vert
_{\sigma }\left\vert J\right\vert _{\omega }} \\
&\leq &\sqrt{A_{2}^{\lambda }}\left( \frac{\ell \left( J\right) }{\ell
\left( I^{\prime }\right) }\right) ^{\left\vert \alpha \right\vert
}\left\vert \widehat{g}\left( J\right) \right\vert \sqrt{\left\vert
J\right\vert _{\sigma }}\leq \sqrt{A_{2}^{\lambda }}\frac{\ell \left(
J\right) }{\ell \left( I^{\prime }\right) }\left\vert \widehat{g}\left(
J\right) \right\vert \sqrt{\left\vert J\right\vert _{\sigma }},
\end{eqnarray*}%
since $\left\vert \alpha \right\vert \geq 1$ (otherwise the commutator
vanishes). Now%
\begin{eqnarray*}
\sum_{J\in \mathfrak{C}_{\mathcal{D}}^{\left( t\right) }\left( I^{\prime
}\right) }\left\vert \func{Int}^{\lambda ,\flat }\left( J\right) \right\vert
&\lesssim &\sum_{J\in \mathfrak{C}_{\mathcal{D}}^{\left( t\right) }\left(
I^{\prime }\right) }\sqrt{A_{2}^{\lambda }}\frac{\ell \left( J\right) }{\ell
\left( I^{\prime }\right) }\left[ \ell \left( J\right) ^{-s}\ell \left(
J\right) ^{s}\right] \left\vert \widehat{g}\left( J\right) \right\vert \sqrt{%
\left\vert J\right\vert _{\sigma }} \\
&\leq &\ell \left( I^{\prime }\right) ^{-s}2^{-t\left( 1-s\right) }\sqrt{%
A_{2}^{\lambda }}\sqrt{\left\vert I^{\prime }\right\vert _{\sigma }}\sqrt{%
\sum_{J\in \mathfrak{C}_{\mathcal{D}}^{\left( t\right) }\left( I^{\prime
}\right) }\ell \left( J\right) ^{2s}\left\vert \widehat{g}\left( J\right)
\right\vert ^{2}},
\end{eqnarray*}%
and so altogether we have%
\begin{eqnarray*}
\sum_{J\in \mathfrak{C}_{\mathcal{D}}^{\left( t\right) }\left( I^{\prime
}\right) }\left\vert \left\langle \left[ P_{\alpha ,c_{J}},T_{\sigma
}^{\lambda }\right] \mathbf{1}_{I^{\prime }},\bigtriangleup _{J;\kappa
}^{\omega }g\right\rangle _{\omega }\right\vert &=&\sum_{J\in \mathfrak{C}_{%
\mathcal{D}}^{\left( t\right) }\left( I^{\prime }\right) }\left\vert \func{%
Int}^{\lambda }\left( J\right) \right\vert \\
&\leq &\sum_{J\in \mathfrak{C}_{\mathcal{D}}^{\left( t\right) }\left(
I^{\prime }\right) }\left\vert \func{Int}^{\lambda ,\natural }\left(
J\right) \right\vert +\sum_{J\in \mathfrak{C}_{\mathcal{D}}^{\left( t\right)
}\left( I^{\prime }\right) }\left\vert \func{Int}^{\lambda ,\flat }\left(
J\right) \right\vert \\
&\lesssim &2^{-t\left( 1-s\right) }\ell \left( I^{\prime }\right) ^{-s}\sqrt{%
A_{2}^{\lambda }}\sqrt{\left\vert I^{\prime }\right\vert _{\sigma }}\sqrt{%
\sum_{J\in \mathfrak{C}_{\mathcal{D}}^{\left( t\right) }\left( I^{\prime
}\right) }\ell \left( J\right) ^{2s}\left\vert \widehat{g}\left( J\right)
\right\vert ^{2}}.
\end{eqnarray*}%
Finally then we obtain from this and (\ref{we obtain}),%
\begin{eqnarray*}
\sum_{J\in \mathfrak{C}_{\mathcal{D}}^{\left( t\right) }\left( I^{\prime
}\right) }\left\vert \left\langle \left[ Q_{I^{\prime };\kappa },T_{\sigma
}^{\lambda }\right] \mathbf{1}_{I^{\prime }},\bigtriangleup _{J;\kappa
}^{\omega }g\right\rangle _{\omega }\right\vert &=&\sum_{J\in \mathfrak{C}_{%
\mathcal{D}}^{\left( t\right) }\left( I^{\prime }\right) }\left\vert
\sum_{\left\vert \alpha \right\vert \leq \kappa -1}b_{\alpha }\left\langle %
\left[ P_{\alpha ,c_{J}},T_{\sigma }^{\lambda }\right] \mathbf{1}_{I^{\prime
}},\bigtriangleup _{J;\kappa }^{\omega }g\right\rangle _{\omega }\right\vert
\\
&\lesssim &2^{-t\left( 1-s\right) }\ell \left( I^{\prime }\right) ^{-s}\sqrt{%
A_{2}^{\lambda }}\sqrt{\sum_{J\in \mathfrak{C}_{\mathcal{D}}^{\left(
t\right) }\left( I^{\prime }\right) }\ell \left( J\right) ^{2s}\left\vert 
\widehat{g}\left( J\right) \right\vert ^{2}}.
\end{eqnarray*}

Now using $M_{I_{J};\kappa }=\left\vert \widehat{f}\left( I\right)
\right\vert Q_{I_{J};\kappa }$, and applying the above estimates with $%
I^{\prime }=I_{J}$, we can sum over $t$ and $I\in \mathcal{C}_{F}$ to obtain%
\begin{eqnarray*}
&&\left\vert \mathsf{B}_{\limfunc{commutator};\kappa }^{F}\left( f,g\right)
\right\vert \leq \sum_{\substack{ I\in \mathcal{C}_{F}\text{ and }J\in 
\mathcal{C}_{F}^{\tau -\limfunc{shift}}  \\ J\Subset _{\rho ,\varepsilon }I}}%
\left\vert \widehat{f}\left( I\right) \right\vert \left\vert \left\langle %
\left[ T_{\sigma }^{\lambda },Q_{I_{J};\kappa }\right] \mathbf{1}%
_{I_{_{J}}},\bigtriangleup _{J;\kappa }^{\omega }g\right\rangle _{\omega
}\right\vert \\
&\lesssim &\sum_{t=r}^{\infty }\sum_{I\in \mathcal{C}_{F}}2^{-t\left(
1-s\right) }\ell \left( I^{\prime }\right) ^{-s}\sqrt{A_{2}^{\lambda }}%
\left\vert \widehat{f}\left( I\right) \right\vert \sqrt{\sum_{J\in \mathfrak{%
C}_{\mathcal{D}}^{\left( t\right) }\left( I_{J}\right) \text{ and }J\in 
\mathcal{C}_{F}^{\tau -\limfunc{shift}}}\ell \left( J\right) ^{2s}\left\vert 
\widehat{g}\left( J\right) \right\vert ^{2}} \\
&\lesssim &\sqrt{A_{2}^{\lambda }}\sum_{t=r}^{\infty }2^{-t\left( 1-s\right)
}\sqrt{\sum_{I\in \mathcal{C}_{F}}\ell \left( I\right) ^{-2s}\left\vert 
\widehat{f}\left( I\right) \right\vert ^{2}}\sqrt{\sum_{I\in \mathcal{C}%
_{F}}\sum_{J\in \mathfrak{C}_{\mathcal{D}}^{\left( t\right) }\left(
I_{J}\right) \text{ and }J\in \mathcal{C}_{F}^{\tau -\limfunc{shift}}}\ell
\left( J\right) ^{2s}\left\vert \widehat{g}\left( J\right) \right\vert ^{2}}
\\
&\lesssim &\sqrt{A_{2}^{\lambda }}\sum_{t=r}^{\infty }2^{-t}\left\Vert 
\mathsf{P}_{\mathcal{C}_{F}}^{\sigma }f\right\Vert _{W_{\limfunc{dyad}%
}^{s}\left( \sigma \right) }\left\Vert \mathsf{P}_{\mathcal{C}_{F}^{\tau -%
\limfunc{shift}}}^{\omega }g\right\Vert _{W_{\limfunc{dyad}}^{-s}\left(
\omega \right) }\lesssim \sqrt{A_{2}^{\lambda }}\left\Vert \mathsf{P}_{%
\mathcal{C}_{F}}^{\sigma }f\right\Vert _{W_{\limfunc{dyad}}^{s}\left( \sigma
\right) }\left\Vert \mathsf{P}_{\mathcal{C}_{F}^{\mathbf{\tau }-\limfunc{%
shift}}}^{\omega }g\right\Vert _{W_{\limfunc{dyad}}^{-s}\left( \omega
\right) }.
\end{eqnarray*}%
Thus the commutator form $\mathsf{B}_{\limfunc{commutator};\kappa
}^{F}\left( f,g\right) \,$is controlled by $A_{2}^{\lambda }$ alone.

Finally using orthogonality of both $\left\{ \mathsf{P}_{\mathcal{C}%
_{F}}^{\sigma }f\right\} _{F\in \mathcal{F}}$ and $\left\{ \mathsf{P}_{%
\mathcal{C}_{F}^{\tau -\limfunc{shift}}}^{\omega }\right\} _{F\in \mathcal{F}%
}$, we obtain%
\begin{equation}
\left\vert \mathsf{B}_{\limfunc{commutator};\kappa }\left( f,g\right)
\right\vert \lesssim \sqrt{A_{2}^{\lambda }}\left\Vert f\right\Vert _{W_{%
\limfunc{dyad}}^{s}\left( \sigma \right) }\left\Vert g\right\Vert _{W_{%
\limfunc{dyad}}^{-s}\left( \omega \right) },\ \ \ \ \ \left\vert
s\right\vert <1.  \label{comm est}
\end{equation}

\subsection{The neighbour form}

In this form we can obtain the required bound, which uses only the $%
A_{2}^{\lambda }$ constant, by taking absolute values inside the sum, and
following \cite{AlSaUr} again, we argue as in the case of Haar wavelets in 
\cite[end of Subsection 8.4]{SaShUr7}. We begin with $M_{I^{\prime };\kappa
}=\mathbf{1}_{I^{\prime }}\bigtriangleup _{I;\kappa }^{\sigma }f$ as in (\ref%
{def M}) to obtain 
\begin{eqnarray*}
\left\vert \mathsf{B}_{\limfunc{neighbour};\kappa }^{F}\left( f,g\right)
\right\vert &\leq &\sum_{\substack{ I\in \mathcal{C}_{F}\text{ and }J\in 
\mathcal{C}_{F}^{\tau -\limfunc{shift}}  \\ J\Subset _{\rho ,\varepsilon }I}}%
\sum_{\theta \left( I_{J}\right) \in \mathfrak{C}_{\mathcal{D}}\left(
I\right) \setminus \left\{ I_{J}\right\} }\left\vert \left\langle T_{\sigma
}^{\lambda }\left( \mathbf{1}_{\theta \left( I_{J}\right) }\bigtriangleup
_{I;\kappa }^{\sigma }f\right) ,\bigtriangleup _{J;\kappa }^{\omega
}g\right\rangle _{\omega }\right\vert \\
&\leq &\sum_{\substack{ I\in \mathcal{C}_{F}\text{ and }J\in \mathcal{C}%
_{F}^{\tau -\limfunc{shift}}  \\ J\Subset _{\rho ,\varepsilon }I}}%
\sum_{I^{\prime }\equiv \theta \left( I_{J}\right) \in \mathfrak{C}_{%
\mathcal{D}}\left( I\right) \setminus \left\{ I_{J}\right\} }\left\vert
\left\langle T_{\sigma }^{\lambda }\left( M_{I^{\prime };\kappa }\mathbf{1}%
_{I^{\prime }}\right) ,\bigtriangleup _{J;\kappa }^{\omega }g\right\rangle
_{\omega }\right\vert .
\end{eqnarray*}%
We now control this by the pivotal bound (\ref{piv bound}) on the inner
product with $\nu =\left\Vert M_{I^{\prime };\kappa }\right\Vert _{L^{\infty
}\left( \sigma \right) }\mathbf{1}_{I^{\prime }}d\sigma $, and then
estimating by the usual Poisson kernel,%
\begin{eqnarray*}
\left\vert \left\langle T_{\sigma }^{\lambda }\left( M_{I^{\prime };\kappa }%
\mathbf{1}_{I^{\prime }}\right) ,\bigtriangleup _{J;\kappa }^{\omega
}g\right\rangle _{\omega }\right\vert &\lesssim &\mathrm{P}_{\kappa
}^{\lambda }\left( J,\left\Vert M_{I^{\prime };\kappa }\right\Vert
_{L^{\infty }\left( \sigma \right) }\mathbf{1}_{I^{\prime }}\sigma \right)
\ell \left( J\right) ^{-s}\sqrt{\left\vert J\right\vert _{\omega }}%
\left\Vert \bigtriangleup _{J;\kappa }^{\omega }g\right\Vert _{W_{\limfunc{%
dyad}}^{-s}\left( \omega \right) } \\
&\leq &\left\Vert M_{I^{\prime };\kappa }\right\Vert _{L^{\infty }\left(
\sigma \right) }\mathrm{P}^{\lambda }\left( J,\mathbf{1}_{I^{\prime }}\sigma
\right) \ell \left( J\right) ^{-s}\sqrt{\left\vert J\right\vert _{\omega }}%
\left\Vert \bigtriangleup _{J;\kappa }^{\omega }g\right\Vert _{W_{\limfunc{%
dyad}}^{-s}\left( \omega \right) },
\end{eqnarray*}%
and the estimate $\left\Vert M_{I^{\prime };\kappa }\right\Vert _{L^{\infty
}\left( \sigma \right) }\approx \frac{1}{\sqrt{\left\vert I^{\prime
}\right\vert _{\sigma }}}\left\vert \widehat{f}\left( I\right) \right\vert $
from (\ref{analogue'}), along with (\ref{e.Jsimeq}), namely%
\begin{equation*}
\mathrm{P}_{m}^{\lambda }(J,\sigma \mathbf{1}_{K\setminus I})\lesssim \left( 
\frac{\ell \left( J\right) }{\ell \left( I\right) }\right) ^{m-\varepsilon
\left( n+m-\lambda \right) }\mathrm{P}_{m}^{\alpha }(I,\sigma \mathbf{1}%
_{K\setminus I}),
\end{equation*}%
to obtain%
\begin{eqnarray*}
&&\left\vert \mathsf{B}_{\limfunc{neighbour};\kappa }^{F}\left( f,g\right)
\right\vert \lesssim \sum_{\substack{ I\in \mathcal{C}_{F}\text{ and }J\in 
\mathcal{C}_{F}^{\tau -\limfunc{shift}}  \\ J\Subset _{\rho ,\varepsilon }I}}%
\sum_{I^{\prime }\equiv \theta \left( I_{J}\right) \in \mathfrak{C}_{%
\mathcal{D}}\left( I\right) \setminus \left\{ I_{J}\right\} }\frac{%
\left\vert \widehat{f}\left( I\right) \right\vert }{\sqrt{\left\vert
I^{\prime }\right\vert _{\sigma }}}\mathrm{P}^{\lambda }\left( J,\mathbf{1}%
_{I^{\prime }}\sigma \right) \ell \left( J\right) ^{-s}\sqrt{\left\vert
J\right\vert _{\omega }}\left\Vert \bigtriangleup _{J;\kappa }^{\omega
}g\right\Vert _{W_{\limfunc{dyad}}^{-s}\left( \omega \right) } \\
&=&\sum_{I\in \mathcal{C}_{F}}\sum_{\substack{ I_{0},I_{\theta }\in 
\mathfrak{C}_{\mathcal{D}}\left( I\right)  \\ I_{0}\neq I_{\theta }}}\sum 
_{\substack{ J\in \mathcal{C}_{F}^{\tau -\limfunc{shift}}  \\ J\Subset
_{\rho ,\varepsilon }I\text{ and }J\subset I_{0}}}\frac{\left\vert \widehat{f%
}\left( I\right) \right\vert }{\sqrt{\left\vert I_{\theta }\right\vert
_{\sigma }}}\mathrm{P}^{\lambda }\left( J,\mathbf{1}_{I_{\theta }}\sigma
\right) \ell \left( J\right) ^{-s}\sqrt{\left\vert J\right\vert _{\omega }}%
\left\Vert \bigtriangleup _{J;\kappa }^{\omega }g\right\Vert _{W_{\limfunc{%
dyad}}^{-s}\left( \omega \right) } \\
&=&\sum_{t=r}^{\infty }\sum_{I\in \mathcal{C}_{F}}\sum_{\substack{ %
I_{0},I_{\theta }\in \mathfrak{C}_{\mathcal{D}}\left( I\right)  \\ I_{0}\neq
I_{\theta }}}\sum_{\substack{ J\in \mathcal{C}_{F}^{\tau -\limfunc{shift}}%
\text{ and }\ell \left( J\right) =2^{-t}\ell \left( I\right)  \\ J\Subset
_{\rho ,\varepsilon }I\text{ and }J\subset I_{0}}}\frac{\left\vert \widehat{f%
}\left( I\right) \right\vert }{\sqrt{\left\vert I_{\theta }\right\vert
_{\sigma }}}\mathrm{P}^{\lambda }\left( J,\mathbf{1}_{I_{\theta }}\sigma
\right) \ell \left( J\right) ^{-s}\sqrt{\left\vert J\right\vert _{\omega }}%
\left\Vert \bigtriangleup _{J;\kappa }^{\omega }g\right\Vert _{W_{\limfunc{%
dyad}}^{-s}\left( \omega \right) } \\
&\lesssim &\sum_{t=r}^{\infty }\sum_{I\in \mathcal{C}_{F}}\sum_{\substack{ %
I_{0},I_{\theta }\in \mathfrak{C}_{\mathcal{D}}\left( I\right)  \\ I_{0}\neq
I_{\theta }}}\sum_{\substack{ J\in \mathcal{C}_{F}^{\tau -\limfunc{shift}}%
\text{ and }\ell \left( J\right) =2^{-t}\ell \left( I\right)  \\ J\Subset
_{\rho ,\varepsilon }I\text{ and }J\subset I_{0}}}\frac{\left\vert \widehat{f%
}\left( I\right) \right\vert }{\sqrt{\left\vert I_{\theta }\right\vert
_{\sigma }}}\left\{ \left( 2^{-t}\right) ^{1-\varepsilon \left( n+1-\lambda
\right) -s}\ell \left( I\right) ^{s}\mathrm{P}^{\lambda }\left( I_{0},%
\mathbf{1}_{I_{\theta }}\sigma \right) \right\} \\
&&\ \ \ \ \ \ \ \ \ \ \ \ \ \ \ \ \ \ \ \ \ \ \ \ \ \ \ \ \ \ \times \sqrt{%
\left\vert J\right\vert _{\omega }}\left\Vert \bigtriangleup _{J;\kappa
}^{\omega }g\right\Vert _{W_{\limfunc{dyad}}^{-s}\left( \omega \right) } \\
&=&\sum_{I\in \mathcal{C}_{F}}\sum_{\substack{ I_{0},I_{\theta }\in 
\mathfrak{C}_{\mathcal{D}}\left( I\right)  \\ I_{0}\neq I_{\theta }}}%
\sum_{t=r}^{\infty }A\left( I,I_{0},I_{\theta },t\right) ,
\end{eqnarray*}%
where%
\begin{equation*}
A\left( I,I_{0},I_{\theta },t\right) =\left( 2^{-t}\right) ^{1-\varepsilon
\left( n+1-\lambda \right) -s}\frac{\ell \left( I\right) ^{s}\left\vert 
\widehat{f}\left( I\right) \right\vert }{\sqrt{\left\vert I^{\prime
}\right\vert _{\sigma }}}\mathrm{P}^{\lambda }\left( I_{0},\mathbf{1}%
_{I_{\theta }}\sigma \right) \sum_{\substack{ J\in \mathcal{C}_{F}^{\mathbf{%
\tau }-\limfunc{shift}}\text{ and }\ell \left( J\right) =2^{-t}\ell \left(
I\right)  \\ J\Subset _{\mathbf{\rho },\varepsilon }I\text{ and }J\subset
I_{0}}}\sqrt{\left\vert J\right\vert _{\omega }}\ell \left( J\right)
^{s}\left\Vert \bigtriangleup _{J;\kappa }^{\omega }g\right\Vert _{W_{%
\limfunc{dyad}}^{-s}\left( \omega \right) }.
\end{equation*}%
Now recall that the case $s=0$ of the following estimate was proved in \cite[%
see from the bottom of page 120 to the top of page 122]{SaShUr7},%
\begin{equation*}
\left\vert \sum_{I\in \mathcal{C}_{F}}\sum_{\substack{ I_{0},I_{\theta }\in 
\mathfrak{C}_{\mathcal{D}}\left( I\right)  \\ I_{0}\neq I_{\theta }}}%
\sum_{s=r}^{\infty }A\left( I,I_{0},I_{\theta },s\right) \right\vert
\lesssim \sqrt{A_{2}^{\lambda }}\left\Vert \mathsf{P}_{\mathcal{C}%
_{F}}^{\sigma }f\right\Vert _{W_{\limfunc{dyad}}^{s}\left( \sigma \right)
}\left\Vert \mathsf{P}_{\mathcal{C}_{F}^{\tau -\limfunc{shift}}}^{\omega
}g\right\Vert _{W_{\limfunc{dyad}}^{-s}\left( \omega \right) },
\end{equation*}%
where the quantity $A\left( I,I_{0},I_{\theta },t\right) $ was defined there
with $s=0$ (in our notation) by%
\begin{equation*}
\left( 2^{-t}\right) ^{1-\varepsilon \left( n+1-\lambda \right) }\left\vert
E_{I_{\theta }}^{\sigma }\bigtriangleup _{I;1}^{\sigma }f\right\vert \mathrm{%
P}^{\lambda }\left( I_{0},\mathbf{1}_{I_{\theta }}\sigma \right) \sum 
_{\substack{ J\in \mathcal{C}_{F}^{\tau -\limfunc{shift}}\text{ and }\ell
\left( J\right) =2^{-t}\ell \left( I\right)  \\ J\Subset _{\rho ,\varepsilon
}I\text{ and }J\subset I_{0}}}\sqrt{\left\vert J\right\vert _{\omega }}%
\left\Vert \bigtriangleup _{J;\kappa }^{\omega }g\right\Vert _{L^{2}\left(
\omega \right) }.
\end{equation*}%
When $\sigma $ is doubling, the reader can check that $\left\vert
E_{I_{\theta }}^{\sigma }\bigtriangleup _{I;1}^{\sigma }f\right\vert \approx 
\frac{\left\vert \widehat{f}\left( I\right) \right\vert }{\sqrt{\left\vert
I_{\theta }\right\vert _{\sigma }}}$, and then that the proof in \cite[see
from the bottom of page 120 to the top of page 122]{SaShUr7} applies almost
verbatim to our situation when $\left\vert \varepsilon \left( n+1-\lambda
\right) +s\right\vert <1$. This proves the required bound for the neighbour
form,%
\begin{equation*}
\left\vert \mathsf{B}_{\limfunc{neighbour};\kappa }^{F}\left( f,g\right)
\right\vert \lesssim \sqrt{A_{2}^{\lambda }}\left\Vert \mathsf{P}_{\mathcal{C%
}_{F}}^{\sigma }f\right\Vert _{W_{\limfunc{dyad}}^{s}\left( \sigma \right)
}\left\Vert \mathsf{P}_{\mathcal{C}_{F}^{\tau -\limfunc{shift}}}^{\omega
}g\right\Vert _{W_{\limfunc{dyad}}^{-s}\left( \omega \right) },\ \ \ \ \
\left\vert s\right\vert <1,
\end{equation*}%
since we can always take $\varepsilon \left( n+1-\lambda \right) $ as small
as we wish.

Finally using orthogonality of both $\left\{ \mathsf{P}_{\mathcal{C}%
_{F}}^{\sigma }f\right\} _{F\in \mathcal{F}}$ and $\left\{ \mathsf{P}_{%
\mathcal{C}_{F}^{\tau -\limfunc{shift}}}^{\omega }\right\} _{F\in \mathcal{F}%
}$, we obtain%
\begin{equation}
\left\vert \mathsf{B}_{\limfunc{neighbour};\kappa }\left( f,g\right)
\right\vert \lesssim \sqrt{A_{2}^{\lambda }}\left\Vert f\right\Vert _{W_{%
\limfunc{dyad}}^{s}\left( \sigma \right) }\left\Vert g\right\Vert _{W_{%
\limfunc{dyad}}^{-s}\left( \omega \right) },\ \ \ \ \ \left\vert
s\right\vert <1.  \label{neigh est}
\end{equation}

\subsection{The stopping form}

To bound the stopping form, we follow the argument for the Haar stopping
form due to Nazarov, Treil and Volberg as in \cite{AlSaUr}. However, we only
need the classical $1$-pivotal constant $\mathcal{V}_{2}^{\lambda ,1}$
(defined in (\ref{both pivotal k}) for $\kappa =1$) for this, not the $%
\left( \beta ,\varepsilon ^{\prime }\right) $-strong $1$-pivotal constant $%
\mathcal{V}_{2,\beta ,\varepsilon ^{\prime }}^{\lambda ,1}$ (defined in (\ref%
{strong pivotal}) for $\kappa =1$ and $\varepsilon =\varepsilon ^{\prime }$%
), since there is no use of quasiorthogonality in this argument, only
orthogonality. In view of this consideration, we apply the $1$-pivotal
stopping time construction to the Alpert projection $\mathsf{P}_{\mathcal{C}%
_{F};\kappa }^{\sigma }f$ in order to obtain a further corona decomposition $%
\mathcal{C}_{F}=\bigcup_{H\in \mathcal{H}\left( F\right) }\mathcal{C}_{H}$
in which we obtain the stopping control bound%
\begin{equation}
\mathrm{P}_{1}^{\lambda }\left( I,\mathbf{1}_{H}\sigma \right)
^{2}\left\vert I\right\vert _{\omega }\leq \Gamma \left\vert I\right\vert
_{\sigma }\lesssim 2\mathcal{V}_{2}^{\lambda ,1}\left( \sigma ,\omega
\right) \left\vert I\right\vert _{\sigma },\ \ \ \ \ I\in \mathcal{C}_{H}\
,\ H\in \mathcal{H}\left( F\right) ,\ F\in \mathcal{F}.
\label{stop con bound}
\end{equation}

Recall that%
\begin{equation*}
\left\vert \widehat{f}\left( I\right) \right\vert Q_{I^{\prime };\kappa }=%
\mathbb{E}_{I^{\prime };\kappa }^{\sigma }f-\mathbf{1}_{I^{\prime }}\mathbb{E%
}_{I;\kappa }^{\sigma }f.
\end{equation*}%
We begin the proof by pigeonholing the ratio of side lengths of $I$ and $J$
in the stopping form:%
\begin{eqnarray*}
&&\mathsf{B}_{\limfunc{stop};\kappa }^{F}\left( f,g\right) \equiv \sum_{I\in 
\mathcal{C}_{F}}\sum_{I^{\prime }\in \mathfrak{C}_{\mathcal{D}}\left(
I\right) }\sum_{\substack{ J\in \mathcal{C}_{F}^{\tau -\limfunc{shift}}  \\ %
J\subset I^{\prime }\text{ and }J\Subset _{\rho ,\varepsilon }I}}\left\vert 
\widehat{f}\left( I\right) \right\vert \left\langle Q_{I^{\prime };\kappa
}T_{\sigma }^{\lambda }\mathbf{1}_{F\setminus I^{\prime }},\bigtriangleup
_{J;\kappa }^{\omega }g\right\rangle _{\omega } \\
&=&\sum_{H\in \mathcal{H}\left( F\right) }\sum_{I\in \mathcal{C}%
_{H}}\sum_{I^{\prime }\in \mathfrak{C}_{\mathcal{D}}\left( I\right) }\sum 
_{\substack{ J\in \mathcal{C}_{F}^{\tau -\limfunc{shift}}\cap \mathcal{C}%
_{H}  \\ J\subset I^{\prime }\text{ and }J\Subset _{\rho ,\varepsilon }I}}%
\left\vert \widehat{f}\left( I\right) \right\vert \left\langle Q_{I^{\prime
};\kappa }T_{\sigma }^{\lambda }\mathbf{1}_{H\setminus I^{\prime
}},\bigtriangleup _{J;\kappa }^{\omega }g\right\rangle _{\omega } \\
&=&\sum_{t=0}^{\infty }\sum_{H\in \mathcal{H}\left( F\right) }\sum_{I\in 
\mathcal{C}_{H}}\sum_{I^{\prime }\in \mathfrak{C}_{\mathcal{D}}\left(
I\right) }\sum_{\substack{ J\in \mathcal{C}_{F}^{\tau -\limfunc{shift}}\cap 
\mathcal{C}_{H}\ \text{and }\ell \left( J\right) =2^{-t}\ell \left( I\right) 
\\ J\subset I^{\prime }\text{ and }J\Subset _{\rho ,\varepsilon }I}}%
\left\vert \widehat{f}\left( I\right) \right\vert \left\langle Q_{I^{\prime
};\kappa }T_{\sigma }^{\lambda }\mathbf{1}_{H\setminus I^{\prime
}},\bigtriangleup _{J;\kappa }^{\omega }g\right\rangle _{\omega } \\
&\equiv &\sum_{s=0}^{\infty }\mathsf{B}_{\limfunc{stop};\kappa ,t}^{F}\left(
f,g\right) \ ,
\end{eqnarray*}%
where using Lemma \ref{mod Alpert}, we have $\left\Vert \left\vert
\bigtriangleup _{J;\kappa }^{\omega }g\right\vert \right\Vert _{W_{\limfunc{%
dyad}}^{-s}\left( \omega \right) }\lesssim \left\Vert \bigtriangleup
_{J;\kappa }^{\omega }g\right\Vert _{W_{\limfunc{dyad}}^{-s}\left( \omega
\right) }$. Thus using the fact that $Q_{I^{\prime };\kappa }\bigtriangleup
_{J;\kappa }^{\omega }$ has one vanishing moment of order $0$, we have 
\begin{eqnarray*}
&&\left\vert \mathsf{B}_{\limfunc{stop};\kappa ,t}^{F}\left( f,g\right)
\right\vert \leq \sum_{H\in \mathcal{H}\left( F\right) }\sum_{I\in \mathcal{C%
}_{H}}\sum_{I^{\prime }\in \mathfrak{C}_{\mathcal{D}}\left( I\right) }\sum 
_{\substack{ J\in \mathcal{C}_{F}^{\tau -\limfunc{shift}}\cap \mathcal{C}_{H}%
\text{ and }\ell \left( J\right) =2^{-t}\ell \left( I\right)  \\ J\subset
I^{\prime }\text{ and }J\Subset _{\rho ,\varepsilon }I}}\left\vert \widehat{f%
}\left( I\right) \right\vert \ \left\vert \left\langle T_{\sigma }^{\lambda }%
\mathbf{1}_{H\setminus I^{\prime }},Q_{I^{\prime };\kappa }\bigtriangleup
_{J;\kappa }^{\omega }g\right\rangle _{\omega }\right\vert \\
&\lesssim &\sum_{H\in \mathcal{H}\left( F\right) }\sum_{I\in \mathcal{C}%
_{H}}\sum_{I^{\prime }\in \mathfrak{C}_{\mathcal{D}}\left( I\right) }\sum 
_{\substack{ J\in \mathcal{C}_{F}^{\tau -\limfunc{shift}}\cap \mathcal{C}_{H}%
\text{ and }\ell \left( J\right) =2^{-t}\ell \left( I\right)  \\ J\subset
I^{\prime }\text{ and }J\Subset _{\rho ,\varepsilon }I}}\left\vert \widehat{f%
}\left( I\right) \right\vert \frac{1}{\sqrt{\left\vert I^{\prime
}\right\vert _{\sigma }}}\ \mathrm{P}_{1}^{\lambda }\left( J,\mathbf{1}%
_{H\setminus I^{\prime }}\sigma \right) \ell \left( J\right) ^{-s}\sqrt{%
\left\vert J\right\vert _{\omega }}\left\Vert \bigtriangleup _{J;\kappa
}^{\omega }g\right\Vert _{W_{\limfunc{dyad}}^{-s}\left( \omega \right) } \\
&\leq &\sum_{H\in \mathcal{H}\left( F\right) }\sum_{I\in \mathcal{C}%
_{H}}\sum_{I^{\prime }\in \mathfrak{C}_{\mathcal{D}}\left( I\right) }\sum 
_{\substack{ J\in \mathcal{C}_{F}^{\tau -\limfunc{shift}}\cap \mathcal{C}_{H}%
\text{ and }\ell \left( J\right) =2^{-t}\ell \left( I\right)  \\ J\subset
I^{\prime }\text{ and }J\Subset _{\rho ,\varepsilon }I}}\ell \left( I\right)
^{-s}\left\vert \widehat{f}\left( I\right) \right\vert \frac{1}{\sqrt{%
\left\vert I^{\prime }\right\vert _{\sigma }}}\ \left( 2^{-t}\right)
^{1-\varepsilon \left( n+1-\lambda \right) -s} \\
&&\ \ \ \ \ \ \ \ \ \ \ \ \ \ \ \ \ \ \ \ \ \ \ \ \ \ \ \ \ \ \ \ \ \ \
\times \mathrm{P}_{1}^{\lambda }\left( I,\mathbf{1}_{H\setminus I^{\prime
}}\sigma \right) \sqrt{\left\vert J\right\vert _{\omega }}\left\Vert
\bigtriangleup _{J;\kappa }^{\omega }g\right\Vert _{W_{\limfunc{dyad}%
}^{-s}\left( \omega \right) } \\
&\lesssim &\left( 2^{-t}\right) ^{1-\varepsilon \left( n+1-\lambda \right)
-s}\sqrt{\sum_{H\in \mathcal{H}\left( F\right) }\sum_{I\in \mathcal{C}%
_{H}}\sum_{I^{\prime }\in \mathfrak{C}_{\mathcal{D}}\left( I\right) }\sum 
_{\substack{ J\in \mathcal{C}_{F}^{\tau -\limfunc{shift}}\cap \mathcal{C}_{H}%
\text{ and }\ell \left( J\right) =2^{-t}\ell \left( I\right)  \\ J\subset
I^{\prime }\text{ and }J\Subset _{\rho ,\varepsilon }I}}\ell \left( I\right)
^{-2s}\left\vert \widehat{f}\left( I\right) \right\vert ^{2}\frac{1}{%
\left\vert I^{\prime }\right\vert _{\sigma }}\ \mathrm{P}_{1}^{\lambda
}\left( I,\mathbf{1}_{H\setminus I^{\prime }}\sigma \right) ^{2}\left\vert
J\right\vert _{\omega }} \\
&&\ \ \ \ \ \ \ \ \ \ \ \ \ \ \ \ \ \ \ \ \ \ \ \ \ \ \ \ \ \ \ \ \ \ \
\times \sqrt{\sum_{H\in \mathcal{H}\left( F\right) }\sum_{I\in \mathcal{C}%
_{H}}\sum_{I^{\prime }\in \mathfrak{C}_{\mathcal{D}}\left( I\right) }\sum 
_{\substack{ J\in \mathcal{C}_{F}^{\tau -\limfunc{shift}}\cap \mathcal{C}_{H}%
\text{ and }\ell \left( J\right) =2^{-t}\ell \left( I\right)  \\ J\subset
I^{\prime }\text{ and }J\Subset _{\rho ,\varepsilon }I}}\left\Vert
\bigtriangleup _{J}^{\omega }g\right\Vert _{W_{\limfunc{dyad}}^{-s}\left(
\omega \right) }^{2}}\ .
\end{eqnarray*}%
Now we note that%
\begin{equation*}
\sum_{H\in \mathcal{H}\left( F\right) }\sum_{I\in \mathcal{C}%
_{H}}\sum_{I^{\prime }\in \mathfrak{C}_{\mathcal{D}}\left( I\right) }\sum 
_{\substack{ J\in \mathcal{C}_{F}^{\tau -\limfunc{shift}}\text{and }\ell
\left( J\right) =2^{-s}\ell \left( I\right)  \\ J\subset I^{\prime }\text{
and }J\Subset _{\rho ,\varepsilon }I}}\left\Vert \bigtriangleup _{J}^{\omega
}g\right\Vert _{W_{\limfunc{dyad}}^{-s}\left( \omega \right) }^{2}\leq
\sum_{J\in \mathcal{C}_{F}^{\tau -\limfunc{shift}}}\left\Vert \bigtriangleup
_{J}^{\omega }g\right\Vert _{W_{\limfunc{dyad}}^{-s}\left( \omega \right)
}^{2}=\left\Vert \mathsf{P}_{\mathcal{C}_{F}^{\tau -\limfunc{shift}%
}}^{\omega }g\right\Vert _{W_{\limfunc{dyad}}^{-s}\left( \omega \right)
}^{2}\ ,
\end{equation*}%
and use the stopping control bound $\mathrm{P}_{1}^{\lambda }\left( I,%
\mathbf{1}_{H}\sigma \right) ^{2}\left\vert I\right\vert _{\omega }\leq
\Gamma \left\vert I\right\vert _{\sigma }\lesssim 2\mathcal{V}_{2}^{\lambda
,1}\left( \sigma ,\omega \right) \left\vert I^{\prime }\right\vert _{\sigma
} $ in the corona $\mathcal{C}_{H}$, to obtain%
\begin{eqnarray*}
\left\vert \mathsf{B}_{\limfunc{stop};\kappa ,t}^{F}\left( f,g\right)
\right\vert &\lesssim &\left( 2^{-t\left( 1-s\right) }\right)
^{1-\varepsilon \left( n+1-\lambda \right) }\mathcal{V}_{2}^{\lambda
,1}\left( \sigma ,\omega \right) \\
&&\times \sqrt{\sum_{H\in \mathcal{H}\left( F\right) }\sum_{I\in \mathcal{C}%
_{H}}\sum_{I^{\prime }\in \mathfrak{C}_{\mathcal{D}}\left( I\right) }\sum 
_{\substack{ J\in \mathcal{C}_{F}^{\tau -\limfunc{shift}}\text{and }\ell
\left( J\right) =2^{-t}\ell \left( I\right)  \\ J\subset I^{\prime }\text{
and }J\Subset _{\rho ,\varepsilon }I}}\ell \left( I\right) ^{-2s}\left\vert 
\widehat{f}\left( I\right) \right\vert ^{2}\frac{\left\vert J\right\vert
_{\sigma }}{\left\vert I^{\prime }\right\vert _{\sigma }}}\left\Vert \mathsf{%
P}_{\mathcal{C}_{F}^{\tau -\limfunc{shift}}}^{\omega }g\right\Vert _{W_{%
\limfunc{dyad}}^{-s}\left( \omega \right) } \\
&\lesssim &\left( 2^{-t\left( 1-s\right) }\right) ^{1-\varepsilon \left(
n+1-\lambda \right) }\mathcal{V}_{2}^{\lambda ,1}\left( \sigma ,\omega
\right) \sqrt{\sum_{H\in \mathcal{H}\left( F\right) }\sum_{I\in \mathcal{C}%
_{H}}\ell \left( I\right) ^{-2s}\left\vert \widehat{f}\left( I\right)
\right\vert ^{2}}\left\Vert \mathsf{P}_{\mathcal{C}_{F}^{\tau -\limfunc{shift%
}}}^{\omega }g\right\Vert _{W_{\limfunc{dyad}}^{-s}\left( \omega \right) } \\
&\lesssim &\left( 2^{-t}\right) ^{1-\varepsilon \left( n+1-\lambda \right)
-s}\mathcal{V}_{2}^{\lambda ,1}\left( \sigma ,\omega \right) \left\Vert 
\mathsf{P}_{\mathcal{C}_{F}}^{\sigma }f\right\Vert _{W_{\limfunc{dyad}%
}^{s}\left( \sigma \right) }\left\Vert \mathsf{P}_{\mathcal{C}_{F}^{\tau -%
\limfunc{shift}}}^{\omega }g\right\Vert _{W_{\limfunc{dyad}}^{-s}\left(
\omega \right) }.
\end{eqnarray*}%
Finally then we sum in $t$ to obtain%
\begin{eqnarray*}
&&\left\vert \mathsf{B}_{\limfunc{stop};\kappa }^{F}\left( f,g\right)
\right\vert \leq \mathcal{V}_{2,\varepsilon }^{\lambda ,1}\left( \sigma
,\omega \right) \sum_{t=0}^{\infty }\left\vert \mathsf{B}_{\limfunc{stop}%
;\kappa ,t}^{F}\left( f,g\right) \right\vert \\
&\lesssim &\mathcal{V}_{2,\varepsilon }^{\lambda ,1}\left( \sigma ,\omega
\right) \sum_{t=0}^{\infty }\left( 2^{-t}\right) ^{1-\varepsilon \left(
n+1-\lambda \right) -s}\left\Vert \mathsf{P}_{\mathcal{C}_{F}}^{\sigma
}f\right\Vert _{W_{\limfunc{dyad}}^{s}\left( \sigma \right) }\left\Vert 
\mathsf{P}_{\mathcal{C}_{F}^{\tau -\limfunc{shift}}}^{\omega }g\right\Vert
_{W_{\limfunc{dyad}}^{-s}\left( \omega \right) } \\
&\lesssim &\mathcal{V}_{2,\varepsilon }^{\lambda ,1}\left( \sigma ,\omega
\right) \left\Vert \mathsf{P}_{\mathcal{C}_{F}}^{\sigma }f\right\Vert _{W_{%
\limfunc{dyad}}^{s}\left( \sigma \right) }\left\Vert \mathsf{P}_{\mathcal{C}%
_{F}^{\tau -\limfunc{shift}}}^{\omega }g\right\Vert _{W_{\limfunc{dyad}%
}^{-s}\left( \omega \right) },
\end{eqnarray*}%
if we take $0<\varepsilon <\frac{1+s}{n+1-\lambda }$.

Finally using orthogonality of both $\left\{ \mathsf{P}_{\mathcal{C}%
_{F}}^{\sigma }f\right\} _{F\in \mathcal{F}}$ and $\left\{ \mathsf{P}_{%
\mathcal{C}_{F}^{\tau -\limfunc{shift}}}^{\omega }\right\} _{F\in \mathcal{F}%
}$, we obtain%
\begin{equation}
\left\vert \mathsf{B}_{\limfunc{stop};\kappa }\left( f,g\right) \right\vert
\lesssim \sqrt{A_{2}^{\lambda }}\left\Vert f\right\Vert _{W_{\limfunc{dyad}%
}^{s}\left( \sigma \right) }\left\Vert g\right\Vert _{W_{\limfunc{dyad}%
}^{-s}\left( \omega \right) },\ \ \ \ \ \left\vert s\right\vert <1.
\label{stop est}
\end{equation}

\section{Conclusion of the proofs}

Collecting all the estimates proved above, namely (\ref{routine}), (\ref{top
control}), (\ref{unif bound'}), (\ref{far below est}), Lemma \ref{doub piv},
(\ref{diag est}), (\ref{para est}), (\ref{comm est}), (\ref{neigh est}) and (%
\ref{stop est}), we obtain just as in \cite{AlSaUr} that for any dyadic grid 
$\mathcal{D}$, and any admissible truncation of $T^{\alpha }$,

\begin{eqnarray*}
\left\vert \left\langle T_{\sigma }^{\alpha }\mathsf{P}_{\func{good}}^{%
\mathcal{D}}f,\mathsf{P}_{\func{good}}^{\mathcal{D}}g\right\rangle _{\omega
}\right\vert &\leq &C\left( \mathfrak{TR}_{T^{\alpha }}^{\kappa ,s}\left(
\sigma ,\omega \right) +\mathfrak{TR}_{T^{\alpha ,\ast }}^{\kappa ,-s}\left(
\omega ,\sigma \right) +A_{2}^{\alpha }\left( \sigma ,\omega \right)
+\varepsilon _{3}\mathfrak{N}_{T^{\alpha }}\left( \sigma ,\omega \right)
\right) \\
&&\times \left\Vert \mathsf{P}_{\func{good}}^{\mathcal{D}}f\right\Vert _{W_{%
\limfunc{dyad}}^{s}\left( \sigma \right) }\left\Vert \mathsf{P}_{\func{good}%
}^{\mathcal{D}}g\right\Vert _{W_{\limfunc{dyad}}^{-s}\left( \omega \right) }.
\end{eqnarray*}%
Thus for any admissible truncation of $T^{\alpha }$, using the above two
theorems, we obtain 
\begin{eqnarray}
\mathfrak{N}_{T^{\alpha }}\left( \sigma ,\omega \right) &\leq &C\sup_{%
\mathcal{D}}\frac{\left\vert \left\langle T_{\sigma }^{\alpha }\mathsf{P}_{%
\func{good}}^{\mathcal{D}}f,\mathsf{P}_{\func{good}}^{\mathcal{D}%
}g\right\rangle _{\omega }\right\vert }{\left\Vert \mathsf{P}_{\func{good}}^{%
\mathcal{D}}f\right\Vert _{W_{\limfunc{dyad}}^{s}\left( \sigma \right)
}\left\Vert \mathsf{P}_{\func{good}}^{\mathcal{D}}g\right\Vert _{W_{\limfunc{%
dyad}}^{-s}\left( \omega \right) }}  \label{absorb} \\
&\leq &C\left( \mathfrak{TR}_{T^{\alpha }}^{\kappa ,s}\left( \sigma ,\omega
\right) +\mathfrak{TR}_{T^{\alpha ,\ast }}^{\kappa ,-s}\left( \omega ,\sigma
\right) +A_{2}^{\alpha }\left( \sigma ,\omega \right) \right) +C\varepsilon
_{3}\mathfrak{N}_{T^{\alpha }}\left( \sigma ,\omega \right) .  \notag
\end{eqnarray}

Our next task is to use the doubling hypothesis to replace the triple $%
\kappa $-testing constants by the usual cube testing constants, and we
follow almost verbatim the argument in \cite{AlSaUr}\ for the case $s=0$.
Recall that the $\kappa $-cube testing conditions use the $Q$-normalized
monomials $m_{Q}^{\beta }\left( x\right) \equiv \mathbf{1}_{Q}\left(
x\right) \left( \frac{x-c_{Q}}{\ell \left( Q\right) }\right) ^{\beta }$, for
which we have $\left\Vert m_{Q}^{\beta }\right\Vert _{L^{\infty }}\approx 1$.

\begin{theorem}
\label{no tails}Suppose that $\sigma $ and $\omega $ are locally finite
positive Borel measures on $\mathbb{R}^{n}$, with $\sigma $ doubling, and
let $\kappa \in \mathbb{N}$. If $T^{\alpha }$ is a bounded operator from $%
W^{s}\left( \sigma \right) $ to $W^{s}\left( \omega \right) $, then\ for
every $0<\varepsilon _{2}<1$, there is a positive constant $C\left( \kappa
,\varepsilon _{2}\right) $ such that 
\begin{equation*}
\mathfrak{TR}_{T^{\alpha }}^{\kappa ,s}\left( \sigma ,\omega \right) \leq
C\left( \kappa ,\varepsilon _{2}\right) \left[ \mathfrak{T}_{T^{\alpha
}}^{s}\left( \sigma ,\omega \right) +\sqrt{A_{2}^{\alpha }\left( \sigma
,\omega \right) }\right] +\varepsilon _{2}\mathfrak{N}_{T^{\alpha }}\left(
\sigma ,\omega \right) \ ,\ \ \ \ \ \kappa \geq 1,
\end{equation*}%
and where the constants $C\left( \kappa ,\varepsilon _{2}\right) $ depend
only on $\kappa $ and $\varepsilon $, and not on the operator norm $%
\mathfrak{N}_{T^{\alpha }}\left( \sigma ,\omega \right) $.
\end{theorem}

\begin{proof}
Fix a dyadic cube $I$. If $P$ is an $I$-normalized polynomial of degree less
than $\kappa $ on the cube $I$, i.e. $\left\Vert P\right\Vert _{L^{\infty
}}\approx 1$, then we can approximate $P$ by a step function 
\begin{equation*}
S\equiv \sum_{I^{\prime }\in \mathfrak{C}_{\mathcal{D}}^{\left( m\right)
}\left( I\right) }a_{I^{\prime }}\mathbf{1}_{I^{\prime }},
\end{equation*}%
satisfying%
\begin{equation*}
\left\Vert S-\mathbf{1}_{I}P\right\Vert _{L^{\infty }\left( \sigma \right) }<%
\frac{\varepsilon _{2}}{2}\ ,
\end{equation*}%
provided we take $m\geq 1$ sufficiently large depending on $n$ and $\kappa $%
, but independent of the cube $I$. Then using the above lemma with $C2^{%
\frac{m}{2}}\varepsilon _{1}\leq \frac{\varepsilon _{2}}{2}$, and the
estimate $\left\vert a_{I^{\prime }}\right\vert \lesssim \left\Vert
P\right\Vert _{L^{\infty }}\lesssim 1$, we have%
\begin{eqnarray*}
&&\left\Vert T_{\sigma }^{\alpha }\mathbf{1}_{I}P\right\Vert _{W_{\limfunc{%
dyad}}^{s}\left( \sigma \right) }\leq \left\Vert \sum_{I^{\prime }\in 
\mathfrak{C}_{\mathcal{D}}^{\left( m\right) }\left( I\right) }a_{I^{\prime
}}T_{\sigma }^{\alpha }\mathbf{1}_{I^{\prime }}\right\Vert _{W_{\limfunc{dyad%
}}^{s}\left( \sigma \right) }+\left\Vert T_{\sigma }^{\alpha }\left[ \left(
S-P\right) \mathbf{1}_{I}\right] \right\Vert _{W_{\limfunc{dyad}}^{s}\left(
\sigma \right) }\sqrt{\int_{3I}\left\vert T_{\sigma }^{\alpha }\left[ \left(
S-P\right) \mathbf{1}_{I}\right] \right\vert ^{2}d\omega } \\
&\leq &C\sum_{I^{\prime }\in \mathfrak{C}_{\mathcal{D}}^{\left( m\right)
}\left( I\right) }\left\vert a_{I^{\prime }}\right\vert \left\Vert T_{\sigma
}^{\alpha }\mathbf{1}_{I^{\prime }}\right\Vert _{W_{\limfunc{dyad}%
}^{s}\left( \sigma \right) }+\frac{\varepsilon _{2}}{2}\mathfrak{N}%
_{T^{\alpha }}\left( \sigma ,\omega \right) \sqrt{\left\vert I\right\vert
_{\sigma }} \\
&\leq &C\sum_{I^{\prime }\in \mathfrak{C}_{\mathcal{D}}^{\left( m\right)
}\left( Q\right) }\mathfrak{T}_{T^{\alpha }}^{s}\left( \sigma ,\omega
\right) \sqrt{\left\vert I^{\prime }\right\vert _{\sigma }}+\frac{%
\varepsilon _{2}}{2}\mathfrak{N}_{T^{\alpha }}\left( \sigma ,\omega \right) 
\sqrt{\left\vert I\right\vert _{\sigma }}\leq C\left\{ \mathfrak{T}%
_{T^{\alpha }}^{s}\left( \sigma ,\omega \right) +\frac{\varepsilon _{2}}{2}%
\mathfrak{N}_{T^{\alpha }}\left( \sigma ,\omega \right) \right\} \sqrt{%
\left\vert I\right\vert _{\sigma }}.
\end{eqnarray*}
\end{proof}

Combining this with (\ref{absorb}) we obtain%
\begin{equation*}
\mathfrak{N}_{T^{\alpha }}\left( \sigma ,\omega \right) \leq C\left( 
\mathfrak{T}_{T^{\alpha }}^{s}\left( \sigma ,\omega \right) +\mathfrak{T}%
_{T^{\alpha ,\ast }}^{-s}\left( \omega ,\sigma \right) +\sqrt{A_{2}^{\alpha
}\left( \sigma ,\omega \right) }\right) +C\left( \varepsilon
_{2}+\varepsilon _{3}\right) \mathfrak{N}_{T^{\alpha }}\left( \sigma ,\omega
\right) .
\end{equation*}%
Since $\mathfrak{N}_{T^{\alpha }}\left( \sigma ,\omega \right) <\infty $ for
each truncation, we may absorb the final summand on the right into the left
hand side provided $C\left( \varepsilon _{2}+\varepsilon _{3}\right) <\frac{1%
}{2}$, to obtain%
\begin{equation*}
\mathfrak{N}_{T^{\alpha }}\left( \sigma ,\omega \right) \lesssim \mathfrak{T}%
_{T^{\alpha }}^{s}\left( \sigma ,\omega \right) +\mathfrak{T}_{T^{\alpha
,\ast }}^{-s}\left( \omega ,\sigma \right) +\sqrt{A_{2}^{\alpha }\left(
\sigma ,\omega \right) }.
\end{equation*}%
By the definition of boundedness of $T^{\alpha }$ in (\ref{def bounded}),
this completes the proof of Theorem \ref{pivotal theorem}.

\section{Appendix: necessity of the classical pivotal condition}

Here we derive the necessity of the classical $1$-pivotal condition from the
classical $A_{2}$ condition and local testing of the vector fractional Riesz
transform with doubling measures. Recall that 
\begin{equation*}
\mathrm{P}^{\alpha }\left( Q,\mu \right) =\int_{\mathbb{R}^{n}}\frac{\ell
\left( Q\right) ^{1}}{\left( \ell \left( Q\right) +\left\vert
y-c_{Q}\right\vert \right) ^{n+1-\alpha }}d\mu \left( y\right) ,
\end{equation*}%
and%
\begin{eqnarray*}
\left( \mathcal{V}_{2}^{\alpha }\left( \sigma ,\omega \right) \right) ^{2}
&=&\sup_{Q\supset \dot{\cup}Q_{r}}\frac{1}{\left\vert Q\right\vert _{\sigma }%
}\sum_{r=1}^{\infty }\mathrm{P}^{\alpha }\left( Q_{r},\mathbf{1}_{Q}\sigma
\right) ^{2}\left\vert Q_{r}\right\vert _{\omega }\ , \\
\left( \mathcal{V}_{2}^{\alpha ,\ast }\left( \sigma ,\omega \right) \right)
^{2} &=&\sup_{Q\supset \dot{\cup}Q_{r}}\frac{1}{\left\vert Q\right\vert
_{\omega }}\sum_{r=1}^{\infty }\mathrm{P}^{\alpha }\left( Q_{r},\mathbf{1}%
_{Q}\omega \right) ^{2}\left\vert Q_{r}\right\vert _{\sigma }=\left( 
\mathcal{V}_{2}^{\alpha }\left( \omega ,\sigma \right) \right) ^{2}\ .
\end{eqnarray*}%
Recall also the classical fractional Muckenhoupt condition on the measure
pair,%
\begin{equation*}
A_{2}^{\alpha }\left( \sigma ,\omega \right) \equiv \sup_{Q\in \mathcal{Q}%
^{n}}\frac{\left\vert Q\right\vert _{\omega }\left\vert Q\right\vert
_{\sigma }}{\left\vert Q\right\vert ^{2\left( 1-\frac{\alpha }{n}\right) }}%
<\infty ,
\end{equation*}%
as well as the Sobolev $\mathbf{1}$-testing and $\mathbf{1}^{\ast }$-testing
conditions for the operator $\mathbf{R}^{\alpha ,n}$,%
\begin{eqnarray*}
\left\Vert T_{\sigma }^{\alpha }\mathbf{1}_{I}\right\Vert _{W^{s}\left(
\omega \right) } &\leq &\mathfrak{T}_{\mathbf{R}^{\alpha ,n,s}}\left( \sigma
,\omega \right) \sqrt{\left\vert I\right\vert _{\sigma }}\ell \left(
I\right) ^{-s},\ \ \ \ \ I\in \mathcal{Q}^{n}, \\
\left\Vert T_{\omega }^{\alpha ,\ast }\mathbf{1}_{I}\right\Vert
_{W^{-s}\left( \sigma \right) } &\leq &\mathfrak{T}_{\mathbf{R}^{\alpha
,n,-s,\ast }}\left( \omega ,\sigma \right) \sqrt{\left\vert I\right\vert
_{\omega }}\ell \left( I\right) ^{s},\ \ \ \ \ I\in \mathcal{Q}^{n},
\end{eqnarray*}%
taken over the family of indicator test functions $\left\{ \mathbf{1}%
_{I}\right\} _{I\in \mathcal{Q}^{n}}$.

\begin{lemma}
\label{pivotal}For $0\leq \alpha <n$ and $\left\vert s\right\vert $
sufficiently small, we have%
\begin{equation*}
\mathcal{V}_{2,s}^{\alpha }\left( \sigma ,\omega \right) \lesssim \mathfrak{T%
}_{\mathbf{R}^{\alpha ,n,s}}\left( \sigma ,\omega \right) +\sqrt{%
A_{2}^{\alpha }\left( \sigma ,\omega \right) }.
\end{equation*}
\end{lemma}

In order to prove this theorem we will model our argument on some of the
material from \cite{SaShUr9}. We first recall the definition of strong
energy reversal from \cite{SaShUr7}. We say that a vector $\mathbf{T}%
^{\alpha }=\left\{ T_{\ell }^{\alpha }\right\} _{\ell =1}^{2}$ of $\alpha $%
-fractional transforms in the plane has \emph{strong} reversal of $\omega $%
-energy on a cube $J$ if there is a positive constant $C_{0}$ such that for
all $2\leq \gamma \leq 2^{\mathbf{r}\left( 1-\varepsilon \right) }$ and for
all positive measures $\mu $ supported outside $\gamma J$, we have the
inequality%
\begin{equation}
\mathbb{E}_{J}^{\omega }\left[ \left( \mathbf{x}-\mathbb{E}_{J}^{\omega }%
\mathbf{x}\right) ^{2}\right] \left( \frac{\mathrm{P}^{\alpha }\left( J,\mu
\right) }{\left\vert J\right\vert ^{\frac{1}{n}}}\right) ^{2}=\mathsf{E}%
\left( J,\omega \right) ^{2}\mathrm{P}^{\alpha }\left( J,\mu \right)
^{2}\leq C_{0}\ \mathbb{E}_{J}^{\omega }\left\vert \mathbf{T}^{\alpha }\mu -%
\mathbb{E}_{J}^{\omega }\mathbf{T}^{\alpha }\mu \right\vert ^{2}.
\label{will fail}
\end{equation}

We now introduce a \emph{stronger} notion of energy reversal which we call
extreme energy reversal. We say that a vector $\mathbf{T}^{\alpha }=\left\{
T_{\ell }^{\alpha }\right\} _{\ell =1}^{2}$ of $\alpha $-fractional
transforms in the plane has \emph{extreme} reversal of $\omega $-energy on a
cube $J$ if there is a Haar function $h_{J}^{\omega }\left( x\right) $ and a
positive constant $C_{0}$, such that for all $2\leq \gamma \leq 2^{\mathbf{r}%
\left( 1-\varepsilon \right) }$ and for all positive measures $\mu $
supported outside $\gamma J$, we have the inequality,%
\begin{eqnarray}
&&\mathbb{E}_{J}^{\omega }\left[ \left( \mathbf{x}-\mathbb{E}_{J}^{\omega }%
\mathbf{x}\right) ^{2}\right] \left( \frac{\mathrm{P}^{\alpha }\left( J,\mu
\right) }{\left\vert J\right\vert ^{\frac{1}{n}}}\right) ^{2}\left\vert
J\right\vert _{\omega }=\mathsf{E}\left( J,\omega \right) ^{2}\mathrm{P}%
^{\alpha }\left( J,\mu \right) ^{2}\left\vert J\right\vert _{\omega }
\label{will fail extreme} \\
&&\ \ \ \ \ \ \ \ \ \ \leq C\left\vert \int_{J}\int_{\mathbb{R}^{n}\setminus
\gamma J}\left[ \mathbf{K}^{\alpha }\left( x,y\right) -\mathbf{K}^{\alpha
}\left( c_{J},y\right) \right] h_{J}^{\omega }\left( x\right) d\mu \left(
y\right) d\omega \left( x\right) \right\vert ^{2}.  \notag
\end{eqnarray}%
Clearly extreme reversal of energy implies strong reversal of energy.

\subsection{Fractional Riesz transforms}

Now we compute for $\beta $ real that%
\begin{eqnarray*}
\bigtriangleup \left\vert x\right\vert ^{\beta } &=&\nabla \cdot \nabla
\left\vert x\right\vert ^{2\frac{\beta }{2}}=\nabla \cdot \left\{ \frac{%
\beta }{2}\left\vert x\right\vert ^{2\left( \frac{\beta }{2}-1\right)
}2x\right\} =\beta \nabla \cdot \left\{ x\left\vert x\right\vert ^{2\frac{%
\beta -2}{2}}\right\} \\
&=&\beta \left\{ \left( \nabla \cdot x\right) \left\vert x\right\vert ^{2%
\frac{\beta -2}{2}}+x\cdot \nabla \left\vert x\right\vert ^{2\frac{\beta -2}{%
2}}\right\} =\beta \left\{ n\left\vert x\right\vert ^{2\frac{\beta -2}{2}%
}+x\cdot \frac{\beta -2}{2}\left\vert x\right\vert ^{2\left( \frac{\beta -2}{%
2}-1\right) }2x\right\} \\
&=&\beta \left\{ n\left\vert x\right\vert ^{\beta -2}+\left( \beta -2\right)
\left\vert x\right\vert ^{2}\left\vert x\right\vert ^{\beta -4}\right\}
=\beta \left( n+\beta -2\right) \left\vert x\right\vert ^{\beta -2}.
\end{eqnarray*}%
The case of interest for us is when $\beta =\alpha -n+1$, since then 
\begin{equation}
\bigtriangleup \left\vert x\right\vert ^{\beta }=\nabla \cdot \nabla
\left\vert x\right\vert ^{\alpha -n+1}=\nabla \cdot \nabla \left\vert
x\right\vert ^{\alpha -n+1}=c_{\alpha ,n}\nabla \cdot \mathbf{K}^{\alpha
,n}\left( x\right) ,  \label{interest}
\end{equation}%
where $\mathbf{K}^{\alpha ,n}$ is the vector convolution kernel of the $%
\alpha $-fractional Riesz transform $\mathbf{R}^{\alpha ,n}$. We conclude
that $\bigtriangleup \left\vert x\right\vert ^{\beta }$ is of one sign for
all $x$, provided $\beta \neq 0$ and $n+\beta -2\neq 0$, i.e. $\alpha \notin
\left\{ 1,n-1\right\} $. The case $\alpha =1$ is not included since $%
\left\vert x\right\vert ^{\alpha -n+1}=\left\vert x\right\vert ^{2-n}$ is
the fundamental solution of the Laplacian for $n>2$ and constant for $n=2$.
The case $\alpha =n-1$ is not included since $\left\vert x\right\vert
^{\alpha -n+1}=1$ is constant.

Thus $z\in J$, we have from (\ref{interest}) and (\ref{conclude}), and with $%
\mathbf{I}^{\alpha +1,n}\mu \left( z\right) \equiv \int_{\mathbb{R}%
^{n}}\left\vert z-y\right\vert ^{\alpha +1-n}d\mu \left( y\right) $ denoting
the convolution of $\left\vert x\right\vert ^{\alpha +1-n}$ with $\mu $,
that 
\begin{equation}
\left\vert \nabla \mathbf{R}^{\alpha ,n}\mu \left( z\right) \right\vert
\gtrsim \left\vert \limfunc{trace}\nabla \mathbf{R}^{\alpha ,n}\mu \left(
z\right) \right\vert =\left\vert \bigtriangleup \mathbf{I}^{\alpha +1,n}\mu
\left( z\right) \right\vert \approx \int \left\vert y-z\right\vert ^{\alpha
-n-1}d\mu \left( y\right) \approx \frac{\mathrm{P}^{\alpha }\left( J,\mu
\right) }{\ell \left( J\right) },  \label{L control}
\end{equation}%
where we assume that the positive measure $\mu $ is supported outside the
expanded cube $\gamma J$.

Recall that the trace of a matrix is invariant under conjugation by
rotations, and hence is the sum of the eigenvalues of a symmetric matrix. We
now claim that for every $z\in J$, the full matrix gradient $\nabla \mathbf{R%
}^{\alpha ,n}\mu \left( z\right) $ has at least $1$ eigenvalue of size at
least $c\frac{\mathrm{P}^{\alpha }\left( J,\mu \right) }{\ell \left(
J\right) }$. Indeed, if all eigenvalues of the matrix $\nabla \mathbf{R}%
^{\alpha ,n}\mu \left( z\right) $ have size at most $c\frac{\mathrm{P}%
^{\alpha }\left( J,\mu \right) }{\ell \left( J\right) }$, then $\left\vert
\nabla \mathbf{R}^{\alpha ,n}\mu \left( z\right) \right\vert \leq c\frac{%
\mathrm{P}^{\alpha }\left( J,\mu \right) }{\ell \left( J\right) }$, which
contradicts (\ref{L control}) if $c$ is chosen small enough. This proves our
claim, and moreover, it satisfies the quantitative quadratic estimate%
\begin{equation}
\left\vert \xi \cdot \nabla \mathbf{R}^{\alpha ,n}\mu \left( z\right) \xi
\right\vert \geq c\frac{\mathrm{P}^{\alpha }\left( J,\mu \right) }{\ell
\left( J\right) }\left\vert \xi \right\vert ^{2},\ \ \ \ \ \xi \in \mathsf{S}%
_{z},\ \text{for }z\in J.  \label{form est}
\end{equation}%
where $\mathsf{S}_{z}\equiv \limfunc{Span}\mathbf{v}_{z}$, for some $\mathbf{%
v}_{z}\in \mathbb{S}^{n-1}$. Thus to each $z$ in $J$, there corresponds a
unit vector $\mathbf{v}_{z}$ for which%
\begin{equation*}
\left\vert \mathbf{v}_{z}\cdot \nabla \mathbf{R}^{\alpha ,n}\mu \left(
z\right) \mathbf{v}_{z}\right\vert \geq c\frac{\mathrm{P}^{\alpha }\left(
J,\mu \right) }{\ell \left( J\right) }.
\end{equation*}%
However, for $w\in J$ we have%
\begin{eqnarray*}
&&\left\vert \mathbf{v}_{z}\cdot \nabla \mathbf{R}^{\alpha ,n}\mu \left(
z\right) \mathbf{v}_{z}-\mathbf{v}_{z}\cdot \nabla \mathbf{R}^{\alpha ,n}\mu
\left( w\right) \mathbf{v}_{z}\right\vert \leq \left\Vert \nabla \mathbf{R}%
^{\alpha ,n}\mu \left( z\right) -\nabla \mathbf{R}^{\alpha ,n}\mu \left(
w\right) \right\Vert  \\
&\leq &\left\Vert \nabla ^{2}\mathbf{R}^{\alpha ,n}\mu \left( \theta
_{z,w}\right) \right\Vert \left\vert z-w\right\vert \leq \int \left\vert
\nabla ^{2}\mathbf{K}^{\alpha ,n}\left( \theta _{z,w},y\right) \right\vert
d\mu \left( y\right) \ \left\vert z-w\right\vert  \\
&\leq &\int_{\mathbb{R}^{n}\setminus \gamma J}\frac{1}{\left\vert \theta
_{z,w}-y\right\vert ^{n-\alpha +2}}d\mu \left( y\right) \ \left\vert
z-w\right\vert =\int_{\mathbb{R}^{n}\setminus \gamma J}\frac{1}{\left\vert
\theta _{z,w}-y\right\vert ^{n-\alpha +1}}\frac{\ell \left( J\right) }{%
\left\vert \theta _{z,w}-y\right\vert }d\mu \left( y\right) \ \frac{%
\left\vert z-w\right\vert }{\ell \left( J\right) } \\
&\leq &C_{\gamma }\int_{\mathbb{R}^{n}\setminus \gamma J}\frac{1}{\left\vert
\theta _{z,w}-y\right\vert ^{n-\alpha +1}}d\mu \left( y\right) \ \frac{%
\left\vert z-w\right\vert }{\ell \left( J\right) }\lesssim C_{\gamma }\frac{%
\mathrm{P}^{\alpha }\left( J,\mu \right) }{\ell \left( J\right) }\frac{%
\left\vert z-w\right\vert }{\ell \left( J\right) },
\end{eqnarray*}%
since $\frac{\ell \left( J\right) }{\left\vert \theta _{z,w}-y\right\vert }%
\leq C_{\gamma }$. Thus there is a fixed $m$ such that for each $m^{th}$
order grandchild $J^{\prime }\in \mathfrak{C}_{\mathcal{D}}^{\left( m\right)
}\left( J\right) $, we have upon replacing $z$ by $c_{J^{\prime }}$ in the
displays above,%
\begin{equation}
\left\vert \mathbf{v}_{c_{J^{\prime }}}\cdot \nabla \mathbf{R}^{\alpha
,n}\mu \left( w\right) \mathbf{v}_{c_{J^{\prime }}}\right\vert \geq c\frac{%
\mathrm{P}^{\alpha }\left( J,\mu \right) }{\ell \left( J\right) },\ \ \ \ \
w\in J^{\prime }\text{,}  \label{direction}
\end{equation}%
i.e. we can use the same unit vector $\mathbf{v}_{c_{J^{\prime }}}$ in place
of $v_{z}$ for all $z\in J^{\prime }$.

\subsection{Extreme reversal of energy}

\begin{lemma}
\label{extreme rev ener}Let $0\leq \alpha <n$. Then the $\alpha $-fractional
Riesz transform $\mathbf{R}^{\alpha ,n}=\left\{ R_{\ell }^{n,\alpha
}\right\} _{\ell =1}^{n}$ has extreme reversal of $\omega $-energy (\ref%
{will fail extreme}) on \emph{all} cubes $J$ provided $\gamma $ is chosen
large enough depending only on $n$ and $\alpha $.
\end{lemma}

\begin{proof}
It suffices to show that%
\begin{eqnarray*}
&&\mathbb{E}_{J}^{\omega }\left[ \left( \mathbf{x}-\mathbb{E}_{J}^{\omega }%
\mathbf{x}\right) ^{2}\right] \left( \frac{\mathrm{P}^{\alpha }\left( J,\mu
\right) }{\left\vert J\right\vert ^{\frac{1}{n}}}\right) ^{2}\left\vert
J\right\vert _{\omega }=\mathsf{E}\left( J,\omega \right) ^{2}\mathrm{P}%
^{\alpha }\left( J,\mu \right) ^{2}\left\vert J\right\vert _{\omega } \\
&&\ \ \ \ \ \ \ \ \ \ \leq C\left\vert \int_{J}\int_{\mathbb{R}^{n}\setminus
\gamma J}\mathbf{v}_{c_{J}}\cdot \left[ \mathbf{K}^{\alpha }\left(
x,y\right) -\mathbf{K}^{\alpha }\left( c_{J},y\right) \right] h_{J}^{\omega
}\left( x\right) d\mu \left( y\right) d\omega \left( x\right) \right\vert
^{2},
\end{eqnarray*}%
where 
\begin{equation*}
h_{J}^{\omega }\left( x\right) =\sum_{K\in \mathfrak{C}\left( J\right) }a_{K}%
\mathbf{1}_{K}\left( x\right) \text{, where }\left\{ 
\begin{array}{ccc}
a_{K}>0 & \text{ if } & K\text{ lies to the right of center} \\ 
a_{K}<0 & \text{ if } & K\text{ lies to the left of center}%
\end{array}%
\right. ,
\end{equation*}%
and without loss of generality $\mathbf{v}_{c_{J}}=\mathbf{e}_{1}$. To see
this we compute,%
\begin{eqnarray*}
&&\int_{J}\left[ K_{1}^{\alpha ,n}\left( x,y\right) -K_{1}^{\alpha ,n}\left(
c_{J},y\right) \right] h_{J}^{\omega }\left( x\right) d\omega \left( x\right)
\\
&=&\int_{J}\left[ \frac{x_{1}-y_{1}}{\left\vert x-y\right\vert ^{n-\alpha +1}%
}-\frac{\left( c_{J}\right) _{1}-y_{1}}{\left\vert c_{J}-y\right\vert
^{n-\alpha +1}}\right] h_{J}^{\omega }\left( x\right) d\omega \left( x\right)
\\
&=&\int_{J}\left( x_{1}-y_{1}\right) \left\{ \frac{1}{\left\vert
x-y\right\vert ^{n-\alpha +1}}-\frac{1}{\left\vert c_{J}-y\right\vert
^{n-\alpha +1}}\right\} h_{J}^{\omega }\left( x\right) d\omega \left(
x\right) +\int_{J}\left\{ \frac{x_{1}-\left( c_{J}\right) _{1}}{\left\vert
c_{J}-y\right\vert ^{n-\alpha +1}}\right\} h_{J}^{\omega }\left( x\right)
d\omega \left( x\right) \\
&\equiv &A+B.
\end{eqnarray*}%
Now in term $B$ we have $\left( x_{1}-\left( c_{J}\right) _{1}\right)
h_{J}^{\omega }\left( x\right) $ is of one sign and so%
\begin{equation*}
\left\vert B\right\vert =\left\vert \int_{J}\left\{ \frac{x_{1}-\left(
c_{J}\right) _{1}}{\left\vert c_{J}-y\right\vert ^{n-\alpha +1}}\right\}
h_{J}^{\omega }\left( x\right) d\omega \left( x\right) \right\vert =\int_{J}%
\frac{\left\vert x_{1}-\left( c_{J}\right) _{1}\right\vert }{\left\vert
c_{J}-y\right\vert ^{n-\alpha +1}}\left\vert h_{J}^{\omega }\left( x\right)
\right\vert d\omega \left( x\right) \geq c\frac{\ell \left( J\right) }{%
\left\vert c_{J}-y\right\vert ^{n-\alpha +1}}\sqrt{\left\vert J\right\vert },
\end{equation*}%
because $\omega $\ is doubling. On the other hand,%
\begin{eqnarray*}
\left\vert A\right\vert &\leq &\int_{J}\left\vert x_{1}-y_{1}\right\vert
\left\vert \frac{1}{\left\vert x-y\right\vert ^{n-\alpha +1}}-\frac{1}{%
\left\vert c_{J}-y\right\vert ^{n-\alpha +1}}\right\vert \left\vert
h_{J}^{\omega }\left( x\right) \right\vert d\omega \left( x\right) \\
&\lesssim &\frac{\ell \left( J\right) ^{2}}{\left\vert c_{J}-y\right\vert
^{n-\alpha +2}}\sqrt{\left\vert J\right\vert _{\omega }}=\frac{\ell \left(
J\right) }{\left\vert c_{J}-y\right\vert }\frac{\ell \left( J\right) }{%
\left\vert c_{J}-y\right\vert ^{n-\alpha +1}}\sqrt{\left\vert J\right\vert
_{\omega }}\leq C\frac{1}{\gamma }\frac{\ell \left( J\right) }{\left\vert
c_{J}-y\right\vert ^{n-\alpha +1}}\sqrt{\left\vert J\right\vert _{\omega }}
\end{eqnarray*}%
and so for $\gamma >1$ chosen sufficiently large, we obtain%
\begin{eqnarray*}
\left\vert \int_{J}\left[ K_{1}^{\alpha ,n}\left( x,y\right) -K_{1}^{\alpha
,n}\left( c_{J},y\right) \right] h_{J}^{\omega }\left( x\right) d\omega
\left( x\right) \right\vert &\gtrsim &\left\vert B\right\vert -\left\vert
A\right\vert \geq \left( c-C\frac{1}{\gamma }\right) \frac{\ell \left(
J\right) }{\left\vert c_{J}-y\right\vert ^{n-\alpha +1}}\sqrt{\left\vert
J\right\vert _{\omega }} \\
&\geq &\frac{c}{2}\frac{\ell \left( J\right) }{\left\vert c_{J}-y\right\vert
^{n-\alpha +1}}\sqrt{\left\vert J\right\vert _{\omega }}.
\end{eqnarray*}

Since $\int_{J}\left[ K_{1}^{\alpha ,n}\left( x,y\right) -K_{1}^{\alpha
,n}\left( c_{J},y\right) \right] h_{J}^{\omega }\left( x\right) d\omega
\left( x\right) $ is also of one sign, it follows that%
\begin{eqnarray*}
&&\left\vert \int_{J}\int_{\mathbb{R}^{n}\setminus \gamma J}\mathbf{v}%
_{c_{J}}\cdot \left[ \mathbf{K}^{\alpha }\left( x,y\right) -\mathbf{K}%
^{\alpha }\left( c_{J},y\right) \right] h_{J}^{\omega }\left( x\right) d\mu
\left( y\right) d\omega \left( x\right) \right\vert \\
&=&\int_{\mathbb{R}^{n}\setminus \gamma J}\left\vert \int_{J}\left[
K_{1}^{\alpha ,n}\left( x,y\right) -K_{1}^{\alpha ,n}\left( c_{J},y\right) %
\right] h_{J}^{\omega }\left( x\right) d\omega \left( x\right) \right\vert
d\mu \left( y\right) \\
&\geq &\int_{\mathbb{R}^{n}\setminus \gamma J}\frac{c}{2}\frac{\ell \left(
J\right) }{\left\vert c_{J}-y\right\vert ^{n-\alpha +1}}\sqrt{\left\vert
J\right\vert _{\omega }}d\mu \left( y\right) =\frac{c}{2}\sqrt{\left\vert
J\right\vert _{\omega }}\mathrm{P}^{\alpha }\left( J,\mathbf{1}_{\mathbb{R}%
^{n}\setminus \gamma J}\right) ,
\end{eqnarray*}%
which proves the extreme reversal of energy.
\end{proof}

\subsection{Proof of Lemma \protect\ref{pivotal}}

\begin{proof}
For each $r$ we first write $\mathrm{P}^{\alpha }\left( Q_{r},\mathbf{1}%
_{Q}\sigma \right) =\mathrm{P}^{\alpha }\left( Q_{r},\mathbf{1}_{\gamma
Q_{r}}\sigma \right) +\mathrm{P}^{\alpha }\left( Q_{r},\mathbf{1}%
_{Q\setminus \gamma Q_{r}}\sigma \right) $ to obtain,%
\begin{equation*}
\frac{1}{\left\vert Q\right\vert _{\sigma }}\sum_{r=1}^{\infty }\mathrm{P}%
^{\alpha }\left( Q_{r},\mathbf{1}_{Q}\sigma \right) ^{2}\left\vert
Q_{r}\right\vert _{\omega }\leq \frac{1}{\left\vert Q\right\vert _{\sigma }}%
\sum_{r=1}^{\infty }\mathrm{P}^{\alpha }\left( Q_{r},\mathbf{1}_{\gamma
Q_{r}}\sigma \right) ^{2}\left\vert Q_{r}\right\vert _{\omega }+\frac{1}{%
\left\vert Q\right\vert _{\sigma }}\sum_{r=1}^{\infty }\mathrm{P}^{\alpha
}\left( Q_{r},\mathbf{1}_{Q\setminus \gamma Q_{r}}\sigma \right)
^{2}\left\vert Q_{r}\right\vert _{\omega }\ ,
\end{equation*}%
and since $\sigma $ is doubling we have $\mathrm{P}_{\kappa }^{\alpha
}\left( Q_{r},\mathbf{1}_{\gamma Q_{r}}\sigma \right) \approx \frac{%
\left\vert Q_{r}\right\vert _{\sigma }}{\left\vert Q_{r}\right\vert ^{1-%
\frac{\lambda }{n}}}$, and hence the first term on the right hand side is
dominated by%
\begin{equation*}
\frac{1}{\left\vert Q\right\vert _{\sigma }}\sum_{r=1}^{\infty }\left( \frac{%
\left\vert Q_{r}\right\vert _{\sigma }}{\left\vert Q_{r}\right\vert ^{1-%
\frac{\lambda }{n}}}\right) ^{2}\left\vert Q_{r}\right\vert _{\omega
}\lesssim \frac{1}{\left\vert Q\right\vert _{\sigma }}\sum_{r=1}^{\infty
}\left\vert Q_{r}\right\vert _{\sigma }\left( \frac{\left\vert
Q_{r}\right\vert _{\sigma }\left\vert Q_{r}\right\vert _{\omega }}{%
\left\vert Q_{r}\right\vert ^{2\left( 1-\frac{\lambda }{n}\right) }}\right)
\leq A_{2}^{\alpha }\left( \sigma ,\omega \right) .
\end{equation*}%
To handle the second term involving $\mathrm{P}_{1}^{\lambda }\left( Q_{r},%
\mathbf{1}_{\mathbb{R}^{n}\setminus \gamma Q_{r}}\sigma \right) $ we will
use extreme reversal of energy for the vector $\alpha $-fractional Riesz
transform $\mathbf{R}^{\alpha ,n}$ in $\mathbb{R}^{n}$, which uses the
doubling property of $\omega $.

We first illustrate the constuction by assuming that the unit vector $%
\mathbf{e}_{1}$ in the direction of the positive $x_{1}$-axis, is the
direction $\mathbf{v}_{c_{Q^{r}}}$ in which we have the lower bound (\ref%
{direction}). Then we fix a special unit Haar function $h_{Q_{r}}^{\omega }$%
, i.e. $\bigtriangleup _{Q_{r}}^{\omega }h_{Q_{r}}^{\omega
}=h_{Q_{r}}^{\omega }$ and $\left\Vert h_{Q_{r}}^{\omega }\right\Vert
_{L^{2}\left( \omega \right) }=1$, satisfying%
\begin{equation*}
h_{Q_{r}}^{\omega }\left( x\right) =\sum_{K\in \mathfrak{C}\left( I\right)
}a_{K}\mathbf{1}_{K}\left( x\right) \text{, where }\left\{ 
\begin{array}{ccc}
a_{K}>0 & \text{ if } & K\text{ lies to the right of center} \\ 
a_{K}<0 & \text{ if } & K\text{ lies to the left of center}%
\end{array}%
\right. ,
\end{equation*}%
where for a cube $Q$ centered at the origin, we\ say a child $K$ lies to the
right of center if $K$ is contained in the half space where $x_{1}\geq 0$.
In the general case that $\mathbf{v}_{c_{Q^{r}}}$ is not equal to $\mathbf{e}%
_{1}$, we choose a child $Q_{r}^{\prime }\in \mathfrak{C}_{\mathcal{D}%
}\left( Q_{r}\right) $ of $Q_{r}$ that lies to one side of the hyperplane
through $c_{Q^{r}}$ with normal vector $\mathbf{v}_{c_{Q^{r}}}$. Then if $%
Q_{r}^{\prime \prime }$ denotes the reflection of $Q_{r}^{\prime }$ across
the center $c_{Q^{r}}$ of the cube $Q_{r}$, we define%
\begin{equation*}
h_{Q_{r}}^{\omega }\left( x\right) =a_{Q_{r}^{\prime }}\mathbf{1}%
_{Q_{r}^{\prime }}\left( x\right) -a_{Q_{r}^{\prime \prime }}\mathbf{1}%
_{Q_{r}^{\prime \prime }}\left( x\right) .
\end{equation*}

Then we compute,%
\begin{eqnarray*}
&&\left\Vert \mathbf{1}_{Q_{r}}\left( \mathbf{R}^{\alpha ,n}\mathbf{1}%
_{Q\setminus \gamma Q_{r}}\sigma \left( x\right) -E_{Q_{r}}^{\omega }\left[ 
\mathbf{R}_{\sigma }^{\alpha ,n}\mathbf{1}_{Q\setminus \gamma Q_{r}}\right]
\right) \right\Vert _{W_{\func{dyad}}^{s}\left( \omega \right) }^{2} \\
&=&\sum_{J\in \mathcal{D}}\ell \left( J\right) ^{-2s}\left\Vert
\bigtriangleup _{J}^{\omega }\left( \mathbf{1}_{Q_{r}}\left( \mathbf{R}%
_{\sigma }^{\alpha ,n}\mathbf{1}_{Q\setminus \gamma Q_{r}}\left( x\right)
-E_{Q_{r}}^{\omega }\left[ \mathbf{R}_{\sigma }^{\alpha ,n}\mathbf{1}%
_{Q\setminus \gamma Q_{r}}\right] \right) \right) \right\Vert _{L^{2}\left(
\omega \right) }^{2} \\
&\geq &\ell \left( Q_{r}\right) ^{-2s}\left\Vert \bigtriangleup
_{Q_{r}}^{\omega }\left( \mathbf{R}_{\sigma }^{\alpha ,n}\mathbf{1}%
_{Q\setminus \gamma Q_{r}}\left( x\right) -E_{Q_{r}}^{\omega }\left[ \mathbf{%
R}_{\sigma }^{\alpha ,n}\mathbf{1}_{Q\setminus \gamma Q_{r}}\right] \right)
\right\Vert _{L^{2}\left( \omega \right) }^{2} \\
&\geq &\ell \left( Q_{r}\right) ^{-2s}\left\vert \left\langle \mathbf{R}%
_{\sigma }^{\alpha ,n}\mathbf{1}_{Q\setminus \gamma Q_{r}},h_{Q_{r}}^{\omega
}\right\rangle _{\omega }\right\vert ^{2}=\ell \left( Q_{r}\right)
^{-2s}\left\vert \int_{Q_{r}}\left( \int_{Q\setminus \gamma Q_{r}}\mathbf{K}%
_{\sigma }^{\alpha ,n}\left( x,y\right) d\sigma \left( y\right) \right)
h_{Q_{r}}^{\omega }\left( x\right) d\omega \left( x\right) \right\vert ^{2},
\end{eqnarray*}%
which equals,%
\begin{equation*}
\ell \left( Q_{r}\right) ^{-2s}\left\vert \int_{Q_{r}}\int_{Q\setminus
\gamma Q_{r}}\left[ \mathbf{K}_{\sigma }^{\alpha ,n}\left( x,y\right) -%
\mathbf{K}_{\sigma }^{\alpha ,n}\left( c_{Q_{r}},y\right) \right]
h_{Q_{r}}^{\omega }\left( x\right) d\sigma \left( y\right) d\omega \left(
x\right) \right\vert ^{2}
\end{equation*}

Since $\omega $ is doubling, we have $\mathsf{E}\left( Q_{r},\omega \right)
\approx 1$, and from the\emph{\ extreme energy reversal} inequality (\ref%
{will fail extreme}), we have%
\begin{equation}
\left\vert \int_{Q_{r}}\int_{Q\setminus \gamma Q_{r}}\left[ \mathbf{K}%
_{\sigma }^{\alpha ,n}\left( x,y\right) -\mathbf{K}_{\sigma }^{\alpha
,n}\left( c_{Q_{r}},y\right) \right] h_{Q_{r}}^{\omega }\left( x\right)
d\sigma \left( y\right) d\omega \left( x\right) \right\vert ^{2}\gtrsim 
\mathrm{P}^{\alpha }\left( Q_{r},\mathbf{1}_{Q\setminus \gamma Q_{r}}\sigma
\right) ^{2}\left\vert Q_{r}\right\vert _{\omega }\ .  \label{if}
\end{equation}%
With this inequality in hand, we have%
\begin{eqnarray*}
&&\frac{1}{\left\vert Q\right\vert _{\sigma }}\sum_{r=1}^{\infty }\mathrm{P}%
^{\alpha }\left( Q_{r},\mathbf{1}_{Q\setminus \gamma Q_{r}}\sigma \right)
^{2}\left\vert Q_{r}\right\vert _{\omega } \\
&\lesssim &\frac{1}{\left\vert Q\right\vert _{\sigma }}\sum_{r=1}^{\infty
}\left\vert \int_{Q_{r}}\int_{Q\setminus \gamma Q_{r}}\left[ \mathbf{K}%
_{\sigma }^{\alpha ,n}\left( x,y\right) -\mathbf{K}_{\sigma }^{\alpha
,n}\left( c_{Q_{r}},y\right) \right] h_{Q_{r}}^{\omega }\left( x\right)
d\sigma \left( y\right) d\omega \left( x\right) \right\vert ^{2} \\
&\lesssim &\frac{1}{\left\vert Q\right\vert _{\sigma }}\sum_{r=1}^{\infty
}\ell \left( Q_{r}\right) ^{2s}\left\Vert \mathbf{1}_{Q_{r}}\left( \mathbf{R}%
^{\alpha ,n}\mathbf{1}_{\gamma Q_{r}}\sigma \left( x\right)
-E_{Q_{r}}^{\omega }\left[ \mathbf{R}_{\sigma }^{\alpha ,n}\mathbf{1}%
_{\gamma Q_{r}}\right] \right) \right\Vert _{W_{\func{dyad}}^{s}\left(
\omega \right) }^{2} \\
&&+\frac{1}{\left\vert Q\right\vert _{\sigma }}\sum_{r=1}^{\infty }\ell
\left( Q_{r}\right) ^{2s}\left\Vert \mathbf{1}_{Q_{r}}\left( \mathbf{R}%
^{\alpha ,n}\mathbf{1}_{Q}\sigma \left( x\right) -E_{Q_{r}}^{\omega }\left[ 
\mathbf{R}_{\sigma }^{\alpha ,n}\mathbf{1}_{Q}\right] \right) \right\Vert
_{W_{\func{dyad}}^{s}\left( \omega \right) }^{2},
\end{eqnarray*}%
which is dominated by%
\begin{eqnarray*}
&\lesssim &\frac{1}{\left\vert Q\right\vert _{\sigma }}\sum_{r=1}^{\infty
}\ell \left( Q_{r}\right) ^{2s}\left\Vert \mathbf{1}_{Q_{r}}\right\Vert _{W_{%
\func{dyad}}^{s}\left( \sigma \right) }^{2}+\frac{1}{\left\vert Q\right\vert
_{\sigma }}\sum_{r=1}^{\infty }\ell \left( Q\right) ^{2s}\left\Vert \mathbf{1%
}_{Q}\right\Vert _{W_{\func{dyad}}^{s}\left( \sigma \right) }^{2} \\
&\lesssim &\frac{1}{\left\vert Q\right\vert _{\sigma }}\mathfrak{T}_{\mathbf{%
R}^{\alpha ,n,s}}\left( \sigma ,\omega \right) ^{2}\left\vert Q\right\vert
_{\sigma }=\mathfrak{T}_{\mathbf{R}^{\alpha ,n,s}}\left( \sigma ,\omega
\right) ^{2}.
\end{eqnarray*}%
Thus altogether we have 
\begin{equation*}
\mathcal{V}_{2}^{\alpha ,1}\left( \sigma ,\omega \right) ^{2}\lesssim 
\mathfrak{T}_{\mathbf{R}^{\alpha ,n,s}}\left( \sigma ,\omega \right)
^{2}+A_{2}^{\alpha }\left( \sigma ,\omega \right) ,
\end{equation*}%
which completes the proof of Lemma \ref{pivotal}.
\end{proof}


\begin{thebibliography}{DuLiSaVeWiYa}
\bibitem[AlSaUr]{AlSaUr} \textsc{M. Alexis, E. Sawyer and I. Uriarte-Tuero,}%
\textit{\ A weak to strong type }$T1$\textit{\ theorem for general smooth
Calder\'{o}n-Zygmund operators with doubling weights, II} \texttt{%
arXiv:2111.06277}.

\bibitem[AlSaUr2]{AlSaUr2} \textsc{M. Alexis, E. T. Sawyer and I.
Uriarte-Tuero,} \textit{Tops of dyadic grids}, \texttt{arXiv:2201.02897.}

\bibitem[DaJo]{DaJo} \textsc{Guy David, Jean-Lin Journ\'{e},} \textit{A
boundedness criterion for generalized Calder\'{o}n-Zygmund operators,} Ann.
of Math. (2) \textbf{120} (1984), 371--397.

\bibitem[DiWiWI]{DilWiWi} \textsc{Francesco Di Plinio, Brett D. Wick and
Tyler Williams,} \textit{Wavelet Representation of Singular Integral
Operators}, to appear in Math. Ann., \texttt{arXiv:2009.01212}.

\bibitem[DuLiSaVeWiYa]{DuLiSaVeWiYa} \textsc{Xuan Thinh Duong, Ji Li, Eric
Sawyer, Naga Manasa Vempati, Brett Wick, Dongyong Yang.} \textit{A two
weight inequality for Calder\'{o}n--Zygmund operators on spaces of
homogeneous type with applications}. Journal of Functional Analysis. \textbf{%
281} (2021).

\bibitem[HaSa]{HaSa} \textsc{Y.-S. Han and E. Sawyer,} \textit{%
Littlewood-Paley theory on spaces of homogeneous type and the classical
function spaces, }Memoirs A.M.S. Number \textbf{530}, (1994).

\bibitem[Hyt]{Hyt} \textsc{Tuomas Hyt\"{o}nen, }\textit{The two weight
inequality for the Hilbert transform with general measures, \texttt{%
arXiv:1312.0843v2}.}

\bibitem[KaLiPeWa]{KaLiPeWa} \textsc{A. Kareima, J. Li, C. Pereyra and L.
Ward, }\textit{Haar bases on quasi-metric measure spaces, and dyadic
structure theorems for function spaces on product spaces of homogeneous type,%
} \textit{\texttt{arXiv:1509.03761}}.

\bibitem[Lac]{Lac} \textsc{Michael T. Lacey, }\textit{\ Two weight
inequality for the Hilbert transform: A real variable characterization, II},
Duke Math. J. Volume \textbf{163}, Number 15 (2014), 2821-2840.

\bibitem[LaSaShUr3]{LaSaShUr3} \textsc{Michael T. Lacey, Eric T. Sawyer,
Chun-Yen Shen, and Ignacio Uriarte-Tuero,} \textit{Two weight inequality for
the Hilbert transform: A real variable characterization I}, Duke Math. J,
Volume \textbf{163}, Number 15 (2014), 2795-2820.

\bibitem[LaSaShUrWi]{LaSaShUrWi} \textsc{Michael T. Lacey, Eric T. Sawyer,
Chun-Yen Shen, Ignacio Uriarte-Tuero and Brett D. Wick,} \textit{Two weight
inequalities for the Cauchy transform from }$\mathbb{R}$ to $\mathbb{C}_{+}$%
, \textit{\texttt{arXiv:1310.4820v4}}.

\bibitem[LaSaUr1]{LaSaUr1} \textsc{Michael T. Lacey, Eric T. Sawyer and
Ignacio Uriarte-Tuero,} \textit{A characterization of two weight norm
inequalities for maximal singular integrals with one doubling measure,}
Analysis \& PDE, Vol. \textbf{5} (2012), No. 1, 1-60.

\bibitem[LaWi]{LaWi} \textsc{Michael T. Lacey and Brett D. Wick,} \textit{%
Two weight inequalities for Riesz transforms: uniformly full dimension
weights}, \textit{\texttt{arXiv:1312.6163v1,v2,v3}}.

\bibitem[NTV4]{NTV4} \textsc{F. Nazarov, S. Treil and A. Volberg,} \textit{%
Two weight estimate for the Hilbert transform and corona decomposition for
non-doubling measures}, preprint (2004) \texttt{arxiv:1003.1596.}

\bibitem[Pee]{Pee} \textsc{J. Peetre,} \textit{New thoughts on Besov spaces}%
, Mathematics Department, Duke University, 1976.

\bibitem[RaSaWi]{RaSaWi} \textsc{Robert Rahm, Eric T. Sawyer and Brett D.
Wick,} \textit{Weighted Alpert wavelets}, Journal of Fourier Analysis and
Applications (IF1.\textbf{273}), Pub Date : 2020-11-23, DOI:
10.1007/s00041-020-09784-0, \texttt{arXiv:1808.01223v2}.

\bibitem[SaUr]{SaUr} \textsc{E. Sawyer and I. Uriarte-Tuero,} \textit{%
Control of the bilinear indicator cube testing property, }\texttt{%
arXiv:1910.09869}.

\bibitem[Saw6]{Saw6} \textsc{E. Sawyer,} \textit{A }$T1$\textit{\ theorem
for general Calder\'{o}n-Zygmund operators with comparable doubling weights
and optimal cancellation conditions}, \texttt{arXiv:1906.05602v10}, to
appear in Journal d'Analyse.

\bibitem[SaShUr7]{SaShUr7} \textsc{Eric T. Sawyer, Chun-Yen Shen and Ignacio
Uriarte-Tuero,} \textit{A} \textit{two weight theorem for }$\alpha $\textit{%
-fractional singular integrals with an energy side condition}, Revista Mat.
Iberoam. \textbf{32} (2016), no. 1, 79-174.

\bibitem[SaShUr9]{SaShUr9} \textsc{Eric T. Sawyer, Chun-Yen Shen and Ignacio
Uriarte-Tuero,} A \textit{two weight fractional singular integral theorem
with side conditions, energy and }$k$\textit{-energy dispersed,} Harmonic
Analysis, Partial Differential Equations, Complex Analysis, Banach Spaces,
and Operator Theory (Volume 2) (Celebrating Cora Sadosky's life), Springer
2017 (see also \texttt{arXiv:1603.04332v2}).

\bibitem[SaShUr12]{SaShUr12} \textsc{Eric T. Sawyer, Chun-Yen Shen and
Ignacio Uriarte-Tuero,} \textit{A two weight local $Tb$ theorem for the
Hilbert transform}, to appear in Revista Mat. Iberoam 2021.

\bibitem[SeUl]{SeUl} \textsc{Andreas Seeger, Tino Ullrich,} \textit{Haar
projection numbers and failure of unconditional convergence in Sobolev spaces%
}, Math. Z. \textbf{285} (2017), 91 -- 119.

\bibitem[Ste]{Ste} \textsc{Elias M. Stein,} \textit{Singular integrals and
differentiability properties of functions}, Princeton University Press,
Princeton, N.J. 1970.

\bibitem[Ste2]{Ste2} \textsc{E. M. Stein,} \textit{Harmonic Analysis:
real-variable methods, orthogonality, and oscillatory integrals},\textit{\ }%
Princeton University Press, Princeton, N. J., 1993.

\bibitem[Tri]{Tri} \textsc{H. Triebel}, \textit{Theory of function spaces},
Monographs in Math., vol. \textbf{78}, Birkhauser, Verlag, Basel, 1983.

\bibitem[Vol]{Vol} \textsc{A. Volberg,} \textit{Calder\'{o}n-Zygmund
capacities and operators on nonhomogeneous spaces,} CBMS .
\end{thebibliography}
\end{document}